\documentclass[reqno,a4paper,12pt]{amsart}
\usepackage{newtxtext,newtxmath}
\usepackage[english]{babel}
\usepackage{url}
\usepackage[
colorlinks=true, linkcolor = black, urlcolor=black, citecolor=blue, anchorcolor=blue]{hyperref}

\usepackage{xcolor,color}
\usepackage{geometry}
\usepackage[all, cmtip]{xy}
\usepackage{enumitem}
\usepackage{booktabs}
\usepackage{tikz}
\usepackage{tikz-cd}
\tikzcdset{arrow style=tikz, diagrams={>=stealth}}

\usepackage{mathtools, 
amsthm, amsfonts}
\usepackage{cleveref}
\usepackage{graphicx}
\usepackage{cases}

\calclayout
\numberwithin{equation}{section}
\theoremstyle{plain}
\newtheorem{lemma}{Lemma}[section]
\newtheorem{proposition}[lemma]{Proposition}
\newtheorem{proposition/definition}[lemma]{Proposition/Definition}
\newtheorem{theorem}[lemma]{Theorem}
\newtheorem{corollary}[lemma]{Corollary}

\newtheorem*{theorem*}{Theorem}

\theoremstyle{definition}
\newtheorem{definition}[lemma]{Definition}
\newtheorem{remark}[lemma]{Remark}
\newtheorem{example}[lemma]{Example}

\DeclareRobustCommand{\SkipTocEntry}[5]{}

\DeclareMathOperator{\im}{im}
\DeclareMathOperator{\rank}{rank}

\newcommand{\bfi}{\mathbf{i}}

\newcommand{\bfm}{\mathbf{m}}

\newcommand{\bfs}{\mathbf{s}}
\newcommand{\bft}{\mathbf{t}}
\newcommand{\bfu}{\mathbf{u}}

\newcommand{\calD}{\mathcal{D}}
\newcommand{\calE}{\mathcal{E}}
\newcommand{\calF}{\mathcal{F}}

\newcommand{\calI}{\mathcal{I}}

\newcommand{\calK}{\mathcal{K}}
\newcommand{\calL}{\mathcal{L}}

\newcommand{\calN}{\mathcal{N}}

\newcommand{\calR}{\mathcal{R}}

\newcommand{\bbR}{\mathbb{R}}

\newcommand{\frakX}{\mathfrak{X}}

\newcommand{\frakg}{\mathfrak{g}}

\newcommand{\ra}{\rightarrow}
\newcommand{\xra}{\xrightarrow}
\newcommand{\la}{\leftarrow}
\newcommand{\xla}{\xleftarrow}
\newcommand{\da}{\downarrow}

\newcommand{\rra}{\rightrightarrows}

\newcommand{\Ra}{\Rightarrow}
\newcommand{\pr}{\textnormal{pr}}

\allowdisplaybreaks

\title{Fat Lie theory}

\makeatletter
\def\author@andify{%
\nxandlist {\unskip ,\penalty-1 \space\ignorespaces}%
{\unskip {} \@@and~}%
{\unskip \penalty-2 \space \@@and~}%
}
\makeatother

\author[L.~Obster]{Lennart Obster}
\address{CMUC, University of Coimbra, Department of Mathematics, Portugal.}
\email{\href{mailto:lennartzegthoi@hotmail.com}{lennartzegthoi at hotmail.com}}

\begin{document}

\date{\today}
\maketitle

\vspace{-0.5cm}

\begin{abstract}
We discuss a new point of view of representation theory of Lie groupoids and algebroids: fat Lie theory. The category of fat extensions is introduced, as well as the category of abstract $2$-term representations up to homotopy (ruths) --- the intrinsic objects behind usual (split) $2$-term ruths. We obtain a one-to-one correspondence between them, and relate to the well-known equivalence between $2$-term ruths and VB-groupoids/algebroids. On the other hand, we show that fat extensions of groupoids correspond to general linear PB-groupoids. The differentiation procedure of fat extensions is discussed, as well as the functorial aspects of all mentioned correspondences. In particular, we upgrade the one-to-one correspondence between general linear PB-groupoids and VB-groupoids of Cattafi and Garmendia to an equivalence of categories. Fat extensions are intimately related to another notion we introduce: core extensions. We show that they correspond to vertically/horizontally core-transitive double groupoids, generalising work by Brown, Jotz-Lean and Mackenzie. This way, we also realise regular fat extensions as general linear double groupoids.
\end{abstract} 

\setcounter{tocdepth}{1}
\tableofcontents

\vspace{-1cm}

\addtocontents{toc}{\SkipTocEntry}
\section*{Acknowledgements}
First, I would like to thank Ioan M\u{a}rcu\c{t}, Jo\~ao Nuno Mestre and Luca Vitagliano for their many helpful comments on this work. I also thank Miquel Cueca, Sebasti\'an Daza-K\"uhn and \v{Z}an Grad for our discussions on this topic. The author acknowledges financial support by the Centre for Mathematics of the University of Coimbra (CMUC, \url{https://doi.org/10.54499/UID/00324/2025}) under the Portuguese Foundation for Science and Technology (FCT), Grants UID/00324/2025 and UID/PRR/00324/2025. I acknowledge FCT for support under the Ph.D. Scholarship UI/BD/152718/2022. This article is based upon work from COST Action CaLISTA CA21109 supported by COST (European Cooperation in Science and Technology). \url{www.cost.eu}

\newpage

\section*{Introduction}

The representation theory of Lie groupoids and algebroids has been widely studied by now, and many different viewpoints have been proposed; see, for example, \cite{Transversemeasures,CrainicFernandeschar,VBalgebroidsruths,LAruths,LGruths,VBgroupoidsruths,Wolbert,Deformationsgroupoids,GLStefaniMatias,Bouaziz,PB-groupoids}. The adjoint representation can be seen as a $2$-term representation up to homotopy \cite{LAruths,LGruths} (we abbreviate representation up to homotopy to ruth from now on), and as a VB-groupoid/algebroid \cite{PradinesVBLA,PradinesVBLG,Mackenziebook} through the tangent bundle. The two points of view are related in a very precise way, for there is an equivalence of categories between the categories of $2$-term ruths and VB-groupoids/algebroids that sends the adjoint ruth of a groupoid/algebroid to (an isomorphic copy of) the tangent bundle of the groupoid/algebroid \cite{VBalgebroidsruths,VBgroupoidsruths,VBMorita}.

In this work, we propose another point of view, axiomatising a structure that has been around, but has not yet been explicitly mentioned in the literature. We introduce a category of \textit{fat extensions}, both for Lie groupoids and for Lie algebroids, that is equivalent to the category of $2$-terms ruths and of VB-groupoids/algebroids. The structure has indeed been around in different contexts, for the main example is the first \textit{jet groupoid} or first \textit{jet algebroid}, together with its canonical cochain complex representation; see, for example, \cite{CrainicFernandeschar,BlaomInfinitesimal,MariaThesis,IvanMariaMarius,BlaomCartangroupoids}. But in its generality presented here, it seems that the type of structure we will consider appeared first in \cite{VBalgebroidsruths} as the \textit{fat algebroid} of a VB-algebroid, and in \cite{VBgroupoidsruths} as the \textit{fat groupoid} of a VB-groupoid. Especially the fat algebroid and its generalisations have been studied in different contexts \cite{ShengLAdeformations,DrummondEgea,Papantonis,MeinrenkenPike,LucaLaPastina,ShengCategorification,LiBlandMeinrenkenManin}. The fat groupoid appears, for example, in \cite{LucaLaPastina,LiBlandMeinrenkenManin,PB-groupoids}.

The main goal of this work is to establish equivalences between the category of fat extensions and the following other models presented in the literature: 
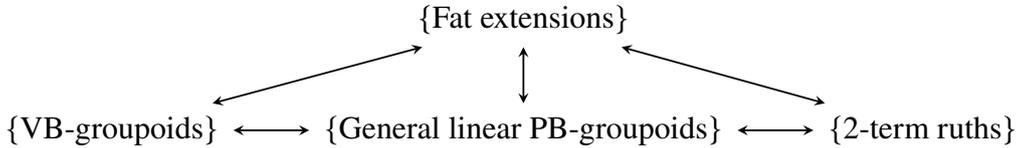
\begin{center}
    \begin{tikzcd}
        \phantom{A} & \{\textnormal{Fat extensions}\} \ar[dl, leftrightarrow] \ar[d, leftrightarrow] \ar[dr, leftrightarrow] & \phantom{A} \\
         \{\textnormal{VB-groupoids}\} \ar[r, leftrightarrow] & \{\textnormal{General linear PB-groupoids}\} \ar[r, leftrightarrow] & \{\textnormal{$2$-term ruths}\} 
    \end{tikzcd}
\end{center}
We also explain the connection to \cite{Wolbert} and \cite{GLStefaniMatias}. The discussion about the infinitesimal analogue of the correspondences is kept shorter, and we only briefly comment on the notion of  ``general linear PB-algebroid''. In turn, we also discuss the Lie functor --- the differentiation procedure. The work presented here reveals many new concepts that become directly related to the representation theory of Lie groupoids and algebroids. On the other hand, we motivate throughout that, for many purposes, fat Lie theory is a natural and useful perspective on the representation theory in Lie theory.

We start the motivation of this work through two tangential projects that will be mentioned several times. One is joint work with Jo\~ao Nuno Mestre and Luca Vitagliano (to appear in \cite{Multiplicativetensors}) and one is joint work with Ioan M\u{a}rcu\c{t} and Jo\~ao Nuno Mestre (to appear in \cite{Homotopyoperators}). 

The topic of \cite{Multiplicativetensors} is to develop further the theory of (infinitesimally) multiplicative tensors on Lie groupoids and algebroids. In that project, we consider tensor products of VB-groupoids and VB-algebroids, as well as their associated linear cohomology theory (see also \cite{BursztynCabreraOrtiz,BursztynCabrera,VanEsthomogeneous,BursztynDrummond,DrummondEgea}). While working on that, it became clear that (infinitesimally) multiplicative tensors can be thought of as natural objects defined on the fat groupoid and the fat algebroid. In fact, there are at least two natural ways to take tensor products of fat extensions, and one in particular will show up in this work. Namely, morphisms of fat extensions can naturally be viewed as, what we call here, \textit{invariant} cochains. Moreover, from the description presented in this work, it is immediate that maps of fat extensions define $1$-cocycles in a canonically defined cochain complex --- a $3$-term ruth. Such reformulations to the theory of fat extensions show that the language is especially well-suited for such representation theoretic aspects like taking tensor products. While we make these points more precise, especially when discussing the functorial aspects, we discuss the relationship between tensor products of VB-groupoids, $2$-term ruths and fat extensions in greater detail in \cite{Multiplicativetensors}.

Although we summarise the results already in this work, we develop the theory of \textit{abstract ruths} for Lie groupoids in \cite{Homotopyoperators} (this should be compared to the infinitesimal counterpart discussed in \cite{LAruthsasfgp}). The main goal of that project, however, is to showcase different perspectives on the construction of certain homotopy operators for ruths, and how to do so \textit{explicitly}. This should be compared to \cite{TangVillatoro}. General homotopy theoretic methods (perturbation theory) are used to construct the desired homotopy operators for ruths, and this is especially satisfactory when understood abstractly. However, for the applications we have in mind, it is important to understand the analytical properties of the homotopy operators. Therefore, it is useful to pass to \textit{split} ruths (which become ruths in the ``usual sense''). Understanding the interplay between abstract and split ruths then becomes crucial. We start this work with a summary on abstract ruths, and we put an emphasis on abstract $2$-term ruths. 

Motivation for fat Lie theory is presented already in the literature: usually, when the fat groupoid or algebroid is used, it is because the language of fat extensions helps in thinking about the problem at hand. Take, for example, the appearance of the fat groupoid in \cite{PB-groupoids}. The \textit{frame bundle} of a VB-groupoid, one of the main objects of study in \cite{PB-groupoids}, is defined using the fat groupoid. The reason is simple: the structure of the frame bundle in terms of fat extensions uses only actions that are naturally associated to the fat groupoid and its canonical representations. Another example is by reinterpretation of \cite{CamiloMariusWeil,MeinrenkenPike}. Namely, the \textit{Weil algebra} of a Lie algebroid, the infinitesimal counterpart of the \textit{Bott-Shulman-Stasheff complex} of a Lie groupoid \cite{BSS}, is efficient to write down using the jet algebroid and its canonical representations. The conditions appearing in the definition of the Weil algebra of \cite{CamiloMariusWeil} are natural in terms of the jet algebroid. The work \cite{MeinrenkenPike} presents generalisations, and this has to do with possible generalisations of this work to the setting of LA-groupoids (or quasi LA-groupoids; see \cite{AlvarezCueca}) and double algebroids (see also \cite{Matchedpairs}, \cite{Zanbrocov} and \cite{Ruioan}). A clear future direction of this work is therefore to describe ``enriched'' versions of the equivalences described above.

On the other hand, we will make first steps towards ``higher'' versions of the equivalences discussed above. In \cite{Highervectorbundles}, a higher version of the equivalence between VB-groupoids and $2$-term ruths is discussed. More precisely, it is proved that the categories of so-called higher vector bundles (VB-$n$-groupoids), equipped with a normal weakly flat cleavage, and split ($n+1$)-term ruths are equivalent. Our reformulation of ruths seems to suggest that the appearance of ``normal weakly flat cleavages'' has to do with the fact that split ruths are considered instead of abstract ruths. So, a clear future direction of this work is to use the present work as a stepping stone towards formulating the correct categories and establish the equivalences above in such a ``higher'' generality. This also relates to \cite{Multiplicativetensors}, for tensor products of VB-groupoids and VB-algebroids should be considered in such a higher context.

We now give a description of the contents of the work. 

\newpage

\subsection*{Description of the work}

Before we explain the contents, let us give a brief introduction to the fat extension of a VB-groupoid $V_G$. The fat groupoid 
\begin{equation*}
    \widehat V_G \rra M
\end{equation*}
can be seen as the Lie groupoid of ``pointwise bihorizontal distributions'' or ``pointwise linear bisections''. That is, the bisections of $\widehat V_G$ comprise the \textit{linear bisections} of $V_G$. The groupoid structure of $\widehat V_G$ is reminiscent of the usual group structure of bisections of a Lie groupoid. Moreover, the canonical action of the group of bisections on a groupoid defines a linear groupoid action of $\widehat V_G$ on $V_G$. The restriction of this action to the (core) $2$-term cochain complex $C_M \ra V_M$ of $V_G$, which sits naturally inside $V_G$ as an action VB-groupoid $C_M \ltimes V_M$, reveals the structure of a \textit{cochain complex representation} of $\widehat V_G$ on $C_M \ra V_M$. The short exact sequence of VB-groupoids
\begin{center}
\begin{tikzcd}
1 \ar[r] & C_M \ltimes V_M \ar[r] \ar[d, shift left] \ar[d, shift right] & V_G \ar[r] \ar[d, shift left] \ar[d, shift right] & 0_G \ar[r] \ar[d, shift left] \ar[d, shift right] & 1 \\
1 \ar[r] & V_M \ar[r] & V_M \ar[r] & 0_M \ar[r] & 1
\end{tikzcd}
\end{center}
descends, by passing to pointwise linear bisections, to a short exact sequence of Lie groupoids
\begin{center}
\begin{tikzcd}
1 \ar[r] & \textnormal{H}(V_M,C_M) \ar[r] & \widehat V_G \ar[r] & G \ar[r] & 1.
\end{tikzcd}
\end{center}
On the one hand, this reveals that $\widehat V_G$ is an extension of $G$ by the \textit{bundle of invertible homotopies} 
\begin{equation*}
\textnormal{H}(V_M,C_M) \coloneq \{h_x \in \textnormal{Hom}(V_M,C_M) \mid 1 + \partial h \textnormal{ is invertible}\}.
\end{equation*}
On the other hand, it shows that the two pieces of data on $\widehat V_G$ defined above --- the cochain complex representation and the extension --- are compatible in a precise way (we come back to this later). This identifies $\widehat V_G$ as a \textit{fat extension} of $G$ over $C_M \ra V_M$.

We now discuss the contents of this work. As mentioned above, our first goal in Section \ref{sec: Ruths as differential graded modules} is to explain an abstract point of view on ruths of Lie groupoids. Our viewpoint is related to the remark made after Definition 3.9 in \cite{LGruths} and also with the approach on ruths of Lie groupoids taken in \cite{Stefanithesis}. However, our approach is to view ruths of Lie groupoids as intrinsically defined objects. Namely, as in \cite{LAruthsasfgp} in the case of Lie algebroids, we view a ruth as a finitely generated projective differential graded module --- see Definition \ref{def: ruths as dgms}. The reason for including the discussion is that, with this point of view on ruths, the equivalence of categories described in \cite{VBgroupoidsruths,VBMorita} can be understood canonically as the functor
\begin{equation*}
\{\textnormal{VB-groupoids}\} \ra \{\textnormal{Abstract $2$-term ruths}\} \qquad V_G \mapsto C_\textnormal{VB} V_G^* \cong C_\textnormal{proj} V_G
\end{equation*}
using the VB-complex from \cite{VanEsthomogeneous} or the projectable complex from \cite{VBgroupoidsruths,Deformationsgroupoids}. Although a more complete discussion will appear in \cite{Homotopyoperators}, we believe the point of view we take is useful, and clarifies that all equivalences appearing in this work are canonical and come with a ``split'' version. At least in the $2$-term case, these split categories are equivalent to the intrinsic models via the forgetful functor. We will comment on this several times.

After the introduction on abstract ruths, we formulate the category of fat extensions in Section \ref{sec: The fat extension of a VB-groupoid} using the language of VB-groupoids. That is, we study properties of the fat groupoid of a VB-groupoid, and we will explain what structure turns it into a fat extension --- see Definition \ref{def: fat extensions}. We also explain that we can encode the data of a fat extension using the \textit{fat category} instead. We call such objects \textit{fat category extensions} ---  see Definition \ref{def: fat category extension}. Fat category extensions are similar to fat extensions, but they are the intrinsic models for \textit{weak representations} (see \cite{Wolbert}). Having both fat extensions and fat category extensions in mind will turn out to be useful. Afterwards, we will describe the typical constructions of VB-groupoids in terms of fat extensions. For example, the dual of a fat extension (see Definition \ref{def: the dual of a fat extension}) dualises simply its cochain complex representation, and the analogue of fiber products of VB-groupoids become a type of space of ``invertible blockmatrices'' for fat extensions (see Definition \ref{def: fat fiber product}). In essence, such phenomena seem to make the theory of fat extensions suited for considering (infinitesimally) multiplicative tensors, for example.

This introduction to the theory of fat extensions, together with the initial step of reformulating the theory of ruths of groupoids, will allow us to formulate the canonical one-to-one correspondences of categories
\begin{center}
\begin{tikzcd}
\{\textnormal{VB-groupoids}\} \ar[r, leftrightarrow] & \{\textnormal{Fat extensions}\} \ar[r] & \{\textnormal{Abstract $2$-term ruths}\}.
\end{tikzcd}
\end{center}
The functorial aspects of these equivalences are also discussed, but delayed until Section \ref{sec: functorial aspects}. As mentioned above, the step from VB-groupoids to abstract $2$-term ruths is to associate to a VB-groupoid either its VB-complex or its projectable complex. The step from VB-groupoids to fat extensions is discussed in Section \ref{sec: fat extensions}. Associating an abstract $2$-term ruth of a fat extension is by taking its \textit{invariant complex} (see Definition \ref{def: invariant cochains}). In other words, we reformulate the abstract $2$-term ruth of a VB-groupoid using the fat groupoid (see Proposition \ref{prop: cochain map to representation up to homotopy of fat groupoid}). The complex turns out to be a type of \textit{relative complex} as introduced in \cite{Mariarelative}. When interpreting the jet groupoid as a fat extension, the result is the adjoint ruth. This complex can be seen as the deformation complex of the Lie groupoid (see \cite{Deformationsgroupoids}), and it is the global counterpart of the deformation complex of a Lie algebroid in terms of the jet algebroid (see \cite{ShengLAdeformations}). A small section, Section \ref{sec: Deformation theory for Lie groupoids using the jet groupoid}, is dedicated to this new description of the deformation complex.

We then proceed: in \cite{PB-groupoids}, it is shown that there is a one-to-one correspondence between \textit{general linear PB-groupoids} (PB-groupoids with structural strict Lie $2$-groupoid the general linear strict Lie $2$-groupoid) and VB-groupoids. The fat groupoid is used explicitly in \cite{PB-groupoids}; in hindsight, the fat groupoid shows up in \cite{PB-groupoids} because of passing through the correspondence with the category of fat extensions:
\begin{center}
\begin{tikzcd}
\{\textnormal{VB-groupoids}\} \ar[r] & \{\textnormal{Fat extensions}\} \ar[r] & \{\textnormal{General linear PB-groupoids}\}.
\end{tikzcd}
\end{center}
We will go through these correspondences in detail in Section \ref{sec: PB-groupoids}, and slightly reformulate the correspondence between VB-groupoids and general linear PB-groupoids. Moreover, part of this work is to formulate the category of general linear PB-groupoids that turns the above correspondences into equivalences of categories. This is done in Section \ref{sec: functorial aspects}. The step
\begin{equation*}
\{\textnormal{General linear PB-groupoids}\} \ra \{\textnormal{Abstract $2$-term ruths}\}
\end{equation*}
produces from a PB-groupoid a type of complex of $\textnormal{GL}$-equivariant cochains.

Now, this relation between general linear PB-groupoids and fat extensions is rather similar to the discussion of \cite{Corediagram} (and \cite{CorediagramLA}), where double groupoids (double algebroids) are discussed and how double transitive groupoids (double transitive algebroids) are equivalently described by so-called \textit{core diagrams}. The suspicion seems justified, for general linear PB-groupoids sometimes come with a \textit{gauge double groupoid}. However, the double groupoids in question are not double transitive groupoids. Still, a generalisation of \cite{Corediagram} shows that horizontally or vertically \textit{core-transitive} double groupoids are equivalently described by, what we call here, \textit{core extensions} (see Definition \ref{def: core extension}). The generalisation should be seen as a clarification of the concluding remarks in \cite{Corediagram}. Core extensions are generalisations of core diagrams of \cite{Corediagram}, but they are also generalisations of so-called \textit{crossed modules} of Lie groupoids as in \cite{MackenzieClassification,CrossedAndrouli,CamilleWagemann}. The proof of the generalisation of the main result of \cite{Corediagram} is postponed to Appendix \ref{app: double groupoids}. There we also show that, as expected, core extensions of general linear double groupoids are fat extensions. In the main body, in Section \ref{sec: PB-groupoids}, we do explain what the concepts mentioned above mean.

After the discussion on the many equivalent ways of organising the information of the above categories in the setting of Lie groupoids, we turn to the infinitesimal side of things in Section \ref{sec: The infinitesimal picture}. That discussion is kept brief, but the essential points will be made. In particular, a new model for the linear complex of a VB-algebroid will be discussed in terms of an invariant complex of the fat algebroid (see Definition \ref{def: invariant cochains fat extensions LA}). The discussion on the equivalences between $2$-term ruths, VB-algebroids and fat extensions is short, but rather complete. Although clear indications will be given about the steps to ``PB-algebroids'' and ``core-transitive double algebroids'', these concepts are not fully developed and deserve a more thorough treatment elsewhere. The concept of core extension for Lie algebroids is also introduced (see Definition \ref{def: core extension LA}). After that, in Section \ref{sec: Fat Lie theory}, we discuss the Fat Lie theory to differentiate from global fat extensions to infinitesimal fat extensions. Part of that discussion appears in \cite{DrummondEgea}, but we elaborate on some aspects. We also write the Van Est map of \cite{CamiloFlorianruths,VanEsthomogeneous} in terms of invariant cochains.

Our discussion in Appendix \ref{app: double groupoids} points towards an interesting new direction: can we describe PB-groupoids/algebroids in terms of ``core data''? With an eye on defining geometric structures on Lie groupoids and algebroids, this seems an interesting direction that will be pursued elsewhere.

The content of the sections is as follows:
\begin{itemize}
\item In Section \ref{sec: Notation for VB-groupoids} we recall some of the basic structure behind VB-groupoids. For the reader that is new to this topic, see, for example, \cite{PradinesVBLA,PradinesVBLG,Mackenziebook,VBalgebroidsruths,GrabowskiRotkiewicz,VBalejandrohenriquematias,VBgroupoidsruths,GrabowskaGrabowskiRavanpak}
\item In Section \ref{sec: Ruths as differential graded modules} we explain our point of view on abstract ruths. We give a rather complete summary, but to keep the discussion short, the proofs of the statements presented there are kept to a minimum. The equivalence between VB-groupoids and abstract $2$-term ruths is Theorem \ref{thm: equivalence VB-groupoids and ruths}. In Proposition \ref{prop: VB complex} we also demonstrate how abstract ruths can be ``split'' to recover the usual notion of ruths. 
\item In Section \ref{sec: The fat extension of a VB-groupoid} we give the necessary background on the structure behind the fat groupoid of a VB-groupoid. In particular, we give an account on the bundle of invertible homotopies (see Section \ref{sec: the bundle of invertible homotopies}). The main definition of fat (category) extension of groupoids will be discussed here (Definition \ref{def: fat extensions} and Definition \ref{def: fat category extension}), and so will examples of fat extensions (see Section \ref{sec: examples of fat extensions}).
\item In Section \ref{sec: fat extensions} we explain the one-to-one correspondence between VB-groupoids and fat extensions (Proposition \ref{prop: essential surjectivity of fat construction}) and we do the same for split $2$-term ruths and fat extensions (see Section \ref{sec: From 2-term ruths to fat extensions}, but also Section \ref{sec: Relation to Representations up to homotopy}) as well as for abstract $2$-term ruths and fat extensions (See Corollary \ref{cor: invariant complex fat extension} and Proposition \ref{prop: invariant cochains in the fat case}).
\item In Section \ref{sec: PB-groupoids} we define (general linear) PB-groupoids. Our setup is slightly different from \cite{PB-groupoids}, but we comment on the differences. In the rest of the section, we set the correspondences between general linear PB-groupoids and fat extensions (See Proposition \ref{prop: one to one correspondence fat extensions and PB-groupoids}), as well as between fat extensions and general linear double groupoids (in the regular case; the details for double groupoids are left for Appendix \ref{app: double groupoids}). Here, we also present the model of $2$-term ruths in terms of general linear PB-groupoids (Proposition \ref{prop: ruth of PB-groupoid}). Since our definition of general linear PB-groupoid is slightly different than in \cite{PB-groupoids}, we also show how general linear PB-groupoids correspond to VB-groupoids (Proposition \ref{prop: essential surjectivity of GL of VB-groupoid}).
\item In Section \ref{sec: comments} we make some other types of comments on the correspondences. Firstly, the vanishing theorem for $2$-term ruths of \cite{LGruths,VanEsthomogeneous} is proved using the invariant complex of a fat extension (Proposition \ref{prop: vanishing for proper groupoids}). Secondly, the ``split'' category of fat extensions is introduced. That is, we define the notion of \textit{fat splitting} (see Definition \ref{def: fat splittings}), which corresponds to a cleavage on the level of VB-groupoids (see Proposition \ref{prop: cleavage and fat splitting}). Thirdly, we explain the connection to the work \cite{GLStefaniMatias} (see Section \ref{sec: Fat extensions as a functor to the quasi general linear Lie 2-groupoid}).
\item In Section \ref{sec: functorial aspects} we go into the functorial aspects. We introduce the definition of a fat extension through the language of VB-groupoids. That way, we will naturally set the equivalence of categories between VB-groupoids and fat extensions (Theorem \ref{thm: equivalence of VBLG and fat}). We also comment on the equivalence of categories between split $2$-term ruths and fat extensions. Afterwards, using the language of VB-groupoids again, we will introduce morphisms of general linear PB-groupoids, and show that the categories of VB-groupoids and general linear PB-groupoids are equivalent (Theorem \ref{thm: equivalence VB-groupoids and PB-groupoids}). We will then also comment on the equivalence between general linear PB-groupoids and fat extensions (Theorem \ref{thm: equivalence fat extensions and PB-groupoids}).
\item In Section \ref{sec: The infinitesimal picture} we explain the infinitesimal picture. That is, we introduce fat extensions using the language of VB-algebroids, we give an abstract definition of fat extension, and then we go through the correspondences between abstract and split $2$-term ruths, VB-algebroids and fat extensions (Proposition \ref{prop: correspondence fat extensions VB-algebroids} and Theorem \ref{thm: equivalence of VBLA and fat}). Except for a small comment, ``general linear PB-algebroids'' and ``general linear double algebroids'' are not discussed.
\item In Section \ref{sec: Fat Lie theory} we explain how, abstractly, fat extensions of groupoids differentiate, in a functorial way, to fat extensions of algebroids (Proposition \ref{prop: Lie functor fat extensions}). Then we explicitly work out the details for the case of the fat groupoid of a VB-groupoid (see Section \ref{sec: The fat algebroid of the fat groupoid of a VB-groupoid}; for this step, see also \cite{DrummondEgea}). After that, we write the Van Est map for invariant cochains in terms of the fat groupoid of a VB-groupoid (Theorem \ref{thm: Van Est theorem of invariant cochains}).
\item In Section \ref{sec: Deformation theory for Lie groupoids using the jet groupoid} we make a small comment that we get a new model for the deformation cohomology of a Lie groupoid. In particular, we explain there that deformations of Lie groupoids give rise to $2$-cocycles (this is entirely analogous to the construction of \cite{Deformationsgroupoids}).
\item In Appendix \ref{app: double groupoids} we introduce the notion of vertical/horizontal core extensions and prove that they correspond to vertically/horizontally core-transitive double groupoids (Theorem \ref{thm: horizontally core transtive double Lie groupoids equivalent to core extensions}). There is also a discussion on PB-groupoids that come with a gauge double groupoid. Proposition \ref{prop: fat extension as core extension} shows that regular fat extensions correspond to a type of core extension.
\end{itemize}

\section{Notation for VB-groupoids}\label{sec: Notation for VB-groupoids}
We develop here the relevant notation for VB-groupoids that we use.

\subsection{Notation for groupoids}
We denote the structure maps of a groupoid $G \rra M$ by $(\bfs,\bft,\bfm,\bfi,\bfu)$, but we often simplify the notation to
\begin{equation*}
\bfm(g,g') = gg' \qquad \bfi(g) = g^{-1} \qquad \bfu(x) = 1_x.
\end{equation*}
We also use the division map of $G$, denoted by
\begin{equation*}
\overline\bfm(g,g') = g(g')^{-1}.
\end{equation*}
The groupoid $G$ has an underlying simplicial manifold (the nerve):
\begin{center}
\begin{tikzcd}
    G^{(\bullet)}: \cdots \textnormal{ } G^{(2)} \ar[r, shift left, shift left, shift left, shift left] \ar[r] \ar[r, shift right, shift right, shift right, shift right] & \ar[l, shift left, shift left] \ar[l, shift right, shift right] G \ar[r, shift right, shift right] \ar[r, shift left, shift left] & \ar[l] M.
\end{tikzcd}
\end{center}
We denote the face maps by\footnote{That is, $d_k$ ``skips'' the $k$-th basepoint.}
\begin{equation*}
d_k(g_1,\dots,g_\bullet) \coloneq
\begin{cases}
(g_2,\dots,g_\bullet) & \textnormal{if } k=0 \\
(g_1,\dots,g_kg_{k+1},\dots,g_\bullet) & \textnormal{if } 1 \le k \le \bullet-1 \\
(g_1,\dots,g_{\bullet-1}) & \textnormal{if } k=\bullet
\end{cases}
\end{equation*}
(for $\bullet=1$, $d_0 = \bfs$ and $d_1 = \bft$) and the degeneracy maps by\footnote{That is, $u_k$ ``repeats'' the $k$-th basepoint.}
\begin{equation*}
u_k(g_1,\dots,g_\bullet) \coloneq (g_1,\dots,g_k,1_{\bfs g_k},g_{k+1},\dots,g_\bullet) \textnormal{ for all } 0 \le k \le \bullet.
\end{equation*}
We put $C^\bullet G = C^\infty G^{(\bullet)}$ and
\begin{equation*}
\delta \coloneq (-1)^\bullet\textstyle\sum_{k=0}^{\bullet+1} (-1)^k d_k^*.
\end{equation*}
Then $CG$ becomes a differential graded algebra with $\delta$ and the product\footnote{If we view $f$ and $g$ as having degree $\bullet$, and $f'$ and $g'$ as having degree $\bullet'$, we can interpret the sign as a Koszul sign rule. Similarly, our sign convention for $\delta$ can be explained this way.}
\begin{equation*}
f \cdot f'(g,g') \coloneq (-1)^{\bullet \cdot \bullet'} f(g)f'(g'),
\end{equation*}
where $f \in C^\bullet G$, $f' \in C^{\bullet'}G$ and $(g,g') \in G^{(\bullet+\bullet')}$. We will also use the normalised (sub)complex
\begin{equation*}
    (C^\bullet G)_\textnormal{N} \coloneq \{f \in C^\bullet G \mid \textnormal{for all } 0 \le \ell \le \bullet-1, \nu_\ell f = 0\}
\end{equation*}
where $\nu_\ell = u_\ell^*$. Importantly, the inclusion 
\begin{equation*}
    (CG)_\textnormal{N} \hookrightarrow CG
\end{equation*}
is a quasi-isomorphism (see \cite{EilenbergMaclane} for groups; see also Remark \ref{rema: Dold-Kan correspondence ruths}).

Lastly, we write
\begin{equation*}
\bfs^{(\bullet)}(g_1,\dots,g_\bullet) \coloneq \bfs g_\bullet \qquad \bft^{(\bullet)}(g_1,\dots,g_\bullet) \coloneq \bft g_1 \qquad \pr_k^{(\bullet)}(g_1,\dots,g_\bullet) \coloneq g_k
\end{equation*}
and if $E_M \ra M$ is a vector bundle, we write
\begin{equation*}
C^\bullet(G; E) \coloneq \Gamma (\bft^{(\bullet)})^*E_M.
\end{equation*}

\subsection{Notation for VB-groupoids}
We denote a VB-groupoid over $G \rra M$ by $V_G \rra V_M$:
\begin{center}
\begin{tikzcd}
    V_G \ar[r, shift left] \ar[r, shift right] \ar[d] & V_M \ar[d] \\
    G \ar[r, shift left] \ar[r, shift right] & M
\end{tikzcd}
\end{center}
The structure maps for $V_G$ and for $G$ are denoted by the same symbols. We then write, for all $g \in G$,
\begin{equation*}
\bfs_g: (V_G)_g \ra (V_M)_{\bfs g} \qquad \bft_g: (V_G)_g \ra (V_M)_{\bft g}
\end{equation*}
for the induced linear maps on fibers. The scalar multiplication of $V_G$ and of $V_M$ are both denoted by $h^\lambda$, but we often simplify to
\begin{equation*}
h^\lambda v = \lambda v.
\end{equation*}
For all $\lambda \in \bbR_{>0}$ and $(v,w) \in V_G^{(2)}$ we then have
\begin{equation*}
\lambda (v\cdot w) = (\lambda v) \cdot (\lambda w).
\end{equation*}
That is, scalar multiplication by $\lambda \in \bbR_{>0}$ is a groupoid isomorphism $V_G \ra V_G$ (see, for example, \cite{VBalejandrohenriquematias}).

For the core of the VB-groupoid we always use the right-core
\begin{equation*}
C_M \coloneqq \ker \bfs|_M
\end{equation*}
and we denote
\begin{equation*}
\partial \coloneq C_M \xra{\bft} V_M
\end{equation*}
for the $2$-term complex associated to $V_G$.

Writing
\begin{equation*}
\ell_g: (C_M)_{\bfs g} \xra{-0_g \cdot c^{-1}} (V_G)_g \qquad r_g: (C_M)_{\bft g} \xra{c \cdot 0_g} (V_G)_g,
\end{equation*}
the dual VB-groupoid $V_G^* \rra C_M^*$ has source and target map given by $\bfs = \ell_g^*$ and $\bft = r_g^*$, and the other structure maps are determined by
\begin{equation*}
\bfm(\varphi_g,\varphi_h)(v \cdot w) = \varphi_gv + \varphi_hw \qquad \bfi(\varphi_g)(v^{-1}) = -\varphi_gv \qquad \bfu\varphi_x = \varphi_x|_{(C_M)_x}.
\end{equation*}
In particular, $V_G^*$ has
\begin{equation*}
\partial^*: V_M^* \xra{\bft^*} C_M^*
\end{equation*}
as its associated $2$-term complex.

Lastly, we mean by a cleavage $\Sigma$ a unital splitting of the exact sequence
\begin{center}
\begin{tikzcd}
0 \ar[r] & \bft^*C_M \ar[r, "r_g"] & V_G \ar[r, "\bfs"] & \ar[l, "\Sigma", bend left] \bfs^*V_M \ar[r] & 0.
\end{tikzcd}
\end{center} 
The cleavage $\Sigma$ induces a unital splitting $\overline\Sigma$
\begin{center}
\begin{tikzcd}
0 \ar[r] & \bfs^*C_M \ar[r, "\ell_g"] & V_G \ar[r, "\bft"] & \ar[l, "\overline\Sigma", bend left] \bft^*V_M \ar[r] & 0
\end{tikzcd}
\end{center} 
given by $\overline\Sigma_gv = (\Sigma_{g^{-1}}v)^{-1}$. The duals of these exact sequences are the corresponding exact sequences for the target and the source map, respectively, of the dual VB-groupoid.

\subsection{The fat groupoid and the fat extension of a VB-groupoid}
One of the main points of this work is the observation that, equivalently, we can describe a VB-groupoid $V_G \rra V_M$ through its so-called fat groupoid $\widehat V_G \rra M$, where
\begin{equation*}
\widehat V_G = \{H_g: (V_M)_{\bfs g} \ra (V_G)_g \mid \bfs_g H_g = 1 \textnormal{ and } \bft_g H_g \textnormal{ is a linear isomorphism}\}
\end{equation*}
(see \cite{VBgroupoidsruths}). Although we leave the full discussion for Section \ref{sec: The fat extension of a VB-groupoid}, here is a summary. A VB-groupoid turns out to be equivalent to the data of, what we call, a \textit{fat extension}. In short, the fat extension is defined out of the Lie groupoid map
\begin{equation*}
\widehat V_G \ra G \qquad H_g \mapsto g,
\end{equation*}
which is a surjective submersion. The kernel of this map is a bundle of Lie groups, denoted $\textnormal{H}(V_M,C_M) \subset \textnormal{Hom}(V_M,C_M)$, and called the \textit{bundle of invertible homotopies} for reasons that will be clarified in Section \ref{sec: the bundle of invertible homotopies}. Its structure of bundle of Lie groups only depends on the cochain complex $C_M \ra V_M$, which, in turn, is canonically a cochain complex representation of $\textnormal{H}(V_M,C_M)$. But $C_M \ra V_M$ is also canonically a cochain complex representation of $\widehat V_G$, and the short exact sequence of Lie groupoids over $M$
\begin{center}
\begin{tikzcd}
1 \ar[r] &\textnormal{H}(V_M,C_M) \ar[r] & \widehat V_G \ar[r] & G \ar[r] & 1
\end{tikzcd}
\end{center}
is compatible with these cochain complex representations. By this we mean that, on the one hand, the inclusion $\textnormal{H}(V_M,C_M) \hookrightarrow \widehat V_G$ is equivariant with respect to the cochain complex representations and, on the other hand, the two natural conjugation actions of $\widehat V_G$ on $\textnormal{H}(V_M,C_M)$ agree. Together, the cochain complex representation of $\widehat V_G$ on $C_M \ra V_M$, the exact sequence above, and the compatibility between these two structures that we described is called the fat extension associated to $V_G$. So, the extension is of a very particular form, and our abstract definition of fat extension keeps track of this form.

\subsection{The linear complex and the VB-subcomplex}
The components of the nerve of a VB-groupoid $V_G$ are vector bundles: for all $\bullet \ge 0$,
\begin{equation*}
V_G^{(\bullet)} \ra G^{(\bullet)}
\end{equation*}
is a vector bundle (with structure defined ``componentwise''). The simplicial maps of $V_G$ are vector bundle maps. In other words, the nerve of $V_G$ is a simplicial vector bundle over the nerve of $G$. In particular, the $1$-homogeneous cochains\footnote{We wrote $\bbR_{>0}$ for the open interval $(0,\infty) \subset \bbR$.}
\begin{equation}\label{eq: linear complex of a VB groupoid}
C^\bullet_{\textnormal{lin}} V_G = \{f \in C^\infty V_G^{(\bullet)} \mid \textnormal{ for all } \lambda \in \bbR_{>0}, \textnormal{ } \lambda^*f=\lambda \cdot f\} \subset C^\bullet V_G
\end{equation}
(i.e. the linear functions on the nerve) form a subcomplex called the linear complex.

The VB-complex is a subcomplex of the linear complex given by\footnote{The conditions for $f \in C^\bullet_{\textnormal{lin}} V_G$ to be an element of $C^\bullet_{\textnormal{VB}} V_G$ can be rewritten as
\begin{equation*}
f(0_{g_1},v_2,\dots,v_\bullet) = 0 \textnormal{ and } f(0_{g_1} \cdot v_2,\dots,v_{\bullet+1}) = f(v_2,\dots,v_{\bullet+1}).
\end{equation*}
}
\begin{equation*}
C_{\textnormal{VB}}^\bullet V_G = \{f \in C^\bullet_{\textnormal{lin}} V_G \mid f(0_{g_1},v_2,\dots,v_\bullet) = 0 \textnormal{ and } \delta f(0_{g_1},v_2,\dots, v_{\bullet+1}) = 0\}.
\end{equation*}
Using this notation, the inclusion $C_\textnormal{VB} V_G \hookrightarrow C_\textnormal{lin} V_G$ is a quasi-isomorphism (see \cite{VanEsthomogeneous,VBMorita} and Remark \ref{rema: inclusion VB-complex is a qiso}). Moreover, $C_\textnormal{VB} V_G^*$ is isomorphic to the complex of projectable cochains
\begin{equation*}
C_{\textnormal{proj}}^\bullet V_G = \{c: G^{(\bullet)} \ra V_G \mid c(g_1,\dots,g_\bullet) \in (V_G)_{g_1} \textnormal{ and } \bfs c(g_1,\dots,g_\bullet) = \bfs c(1_{\bft g_2}, g_2, \dots, g_\bullet)\}
\end{equation*}
with differential given by
\begin{equation*}
\delta = (-1)^\bullet(-\overline\bfm(d_1^*, d_0^*) + \textstyle\sum_{k=2}^{\bullet+1} (-1)^k d_k^*).
\end{equation*}
Without giving explicitly the differential, this complex appeared in \cite{VBgroupoidsruths}. For $V_G = TG$, we have $C_{\textnormal{proj}} V_G = C_{\textnormal{def}} G$, and this complex appeared in \cite{Deformationsgroupoids}.

In terms of a representation up to homotopy (we write ruth for short), $C_\textnormal{VB} V_G^*$ is isomorphic to the ruth $\calE_G$ associated to $V_G$
\begin{equation*}
\calE_G^\bullet = C^\bullet(G; C_M) \oplus C^{\bullet-1}(G; V_M)
\end{equation*}
with differential given by\footnote{The structure of the ruth $\calE_G$ is denoted by $\partial,R_1,R_2$. Defining this structure out of $V_G$ uses a cleavage; we will recall the construction in Section \ref{sec: Relation to Representations up to homotopy}.}
\begin{equation*}
\delta = \partial + R_1 + R_2 + (-1)^\bullet\textstyle\sum_{k=1}^{\bullet+1} (-1)^k d_k^*.
\end{equation*}
We view the dual ruth $\calE_G^*$ as being concentrated in positive degrees as well (this is a non-standard convention; we come back to this in Remark \ref{rema: dual of ruth}).

We describe a new model for the VB-complex in Section \ref{sec: fat extensions}, which comes from understanding VB-groupoids as a fat extension. In the next section, Section \ref{sec: Ruths as differential graded modules}, we adopt a point of view on ruths which allows us to see all the different models of VB-complexes canonically as ruths. The model $\calE_G$ will then be called a split ruth.

\section{Abstract representations up to homotopy of Lie groupoids}\label{sec: Ruths as differential graded modules}

This section is kept short for it will appear in \cite{Homotopyoperators}. Still, we would like to emphasise the main points. We think the point of view we take is useful, and we stress on the natural appearance of the fat groupoid in the main example (see Proposition \ref{prop: VB complex}).

\subsection{The definition of an abstract ruth}

We take the following approach to ruths of Lie groupoids (for Lie algebroids, see \cite{LAruthsasfgp}):

\begin{definition}\label{def: ruths as dgms}
An \textit{abstract representation up to homotopy} (ruth) $(\calE_G)_\textnormal{N}$ is a finitely generated projective right differential graded $(CG)_\textnormal{N}$-module. In particular, the differential
\begin{equation*}
\delta: (\calE_G^\bullet)_\textnormal{N} \ra (\calE_G^{\bullet+1})_\textnormal{N}
\end{equation*}
satisfies the Leibniz rule: for all $c \in (\calE_G)_\textnormal{N}$ and $f \in (CG)_\textnormal{N}$,
\begin{equation*}
\delta(c \cdot f) = \delta(c) \cdot f + (-1)^c c \cdot \delta f.
\end{equation*}
\end{definition}

\begin{remark}
    Using the inclusion 
    \begin{equation*}
        (CG)_\textnormal{N} \hookrightarrow CG
    \end{equation*}
    an abstract ruth $(\calE_G)_\textnormal{N}$ gives rise to a finitely generated projective differential graded $CG$-module 
    \begin{equation*}
        \calE_G \coloneq (\calE_G)_\textnormal{N} \otimes_{(CG)_\textnormal{N}} CG.
    \end{equation*}
    Here, the elements of $\calE_G$ are graded by total degree and generated by pure tensors 
    \begin{equation*}
        c \otimes f \in \calE_G
    \end{equation*}
    where $c \in (\calE_G)_\textnormal{N}$ and $f \in CG$. Whenever $c \in (\calE_G)_\textnormal{N}$, $f \in C_NG$ and $f' \in CG$, we have 
    \begin{equation*}
        (c \cdot f) \otimes f' = c \otimes (f \cdot f')
    \end{equation*}
    and the differential on $\calE_G$ is such that
    \begin{equation*}
        \delta(c \otimes f) = \delta(c) \otimes f + (-1)^c c \otimes \delta f.
    \end{equation*}
    We will often work at the level of $CG$-modules, but it is important to incorporate a subcomplex of normalised cochains $(\calE_G)_\textnormal{N}$ of $\calE_G$ (in this way). We will elaborate on this point in Section \ref{sec: the normalised complex of an abstract ruth}.
\end{remark}

To clarify: a $CG$-module $\calE_G$ is finitely generated projective if we can find a split short exact sequence of $CG$-modules\footnote{Here, $C^\bullet G[-\ell] = C^{\bullet-\ell}G$, so $CG[-\ell]$ is concentrated in degrees $\bullet \ge \ell$.}
\begin{center}
\begin{tikzcd}
    0 \ar[r] & \calR \ar[r] & \textstyle\bigoplus_{k \in I} CG[-\ell_k] \ar[r] & \calE_G \ar[r] & 0
\end{tikzcd}
\end{center}
where $I$ is a finite indexing set, and all $\ell_k$ are non-negative. Alternatively, this is equivalent to the existence of a $CG$-module $\calR$ such that $\calR \oplus \calE_G$ is a free $CG$-module.

\begin{example}\label{exa: VB-complex as an abstract ruth}
    The main example we consider in this work is the VB-complex, or the projectable complex, of a VB-groupoid $V_G$:
    \begin{equation*}
        C_\textnormal{VB} V_G \cong C_\textnormal{proj} V_G^*.
    \end{equation*}
    (see Section \ref{sec: Towards a Serre-Swan theorem}). Notice that the normalised complex is the subcomplex
    \begin{equation*}
        (C_\textnormal{VB}^\bullet V_G)_\textnormal{N} = \textstyle\bigcap_{\ell=0}^{\bullet-1} \ker \nu_\ell,
    \end{equation*}
    where $\nu_\ell = u_\ell^*$ for all $0 \le \ell \le \bullet-1$ (here, $u_\ell$ are the degeneracy maps of $V_G$). For $C_\textnormal{proj} V_G$, $\nu_0$ takes the form ($\bfu$ and $\bfs$ denote the unit and source map of $V_G$, respectively)
    \begin{equation*}
        \nu_0: C_\textnormal{proj}^\bullet V_G \ra C_\textnormal{proj}^{\bullet-1} V_G \qquad \nu_0c(g_2,\cdots,g_\bullet) = c(1_{\bft g_2},g_2,\cdots,g_\bullet) - \bfu_{1_{\bft g_2}} \bfs_{1_{\bft g_2}} c(1_{\bft g_2},g_2,\cdots,g_\bullet)
    \end{equation*} 
    and $\nu_\ell = u_\ell^*$ for all $1 \le \ell \le \bullet-1$ (here, $u_\ell$ are the degeneracy maps of $G$).
\end{example}

\subsection{Locally free sheaves of $CG$-modules}

In fact, we conjecture that a type of Serre-Swan theorem relates higher vector bundles with abstract ruths (see Remark \ref{rema: Serre-Swan}). A weak version at least holds: the global sections functor sets an equivalence of categories between ``locally free sheaves of $CG$-modules'' and finitely generated projective $CG$-modules. We discuss next how we think of a (locally free) sheaf of $CG$-modules.

\begin{remark}
We can view $CG$ as a graded sheaf (over $M$) of right $CG$-modules as follows. For each open set $U$ of $M$, we set
\begin{equation*}
G^{(\bullet)}U = (\bft^{(\bullet)})^{-1} U = \{(g_1,\dots,g_\bullet) \in G^{(\bullet)} \mid \bft g_1 \in U\}.
\end{equation*}
We then write
\begin{equation*}
C^\bullet G(U) = C^\infty(G^{(\bullet)}U).
\end{equation*}
The product of $CG$ defines the structure of right $CG$-module on this sheaf. That is, for $f \in C^\bullet G(U)$ and $f' \in C^\bullet G$, we define $f \cdot f' \in C^{\bullet+\bullet'} G(U)$ by
\begin{equation*}
(f \cdot f')(g,g') \coloneq (-1)^{\bullet \cdot \bullet'} f(g) \cdot f'(g').
\end{equation*}
When we view $CG$ as a sheaf in this way, the differential is only defined on the global sections. There are ways to think of the differential as being local which is interesting for developing further representation theory for non-Hausdorff Lie groupoids. This will be discussed in \cite{Homotopyoperators}.
\end{remark}

Consider now a graded sheaf of $CG$-modules $\calE_G$. That is, for all open sets $U \subset M$, we have a right $CG$-module $\calE_G(U)$ (altogether satisfying sheaf-like properties). Locally free means that, for all $x \in M$, there is an open subset $U \ni x$ of $M$ such that
\begin{equation*}
    \calE_G(U) \cong \textstyle\bigoplus_{k \in I} CG[-\ell_k](U)
\end{equation*}
as sheaves of $CG$-modules, where $I$ is a finite indexing set and all $\ell_k$ are non-negative. The following result is a first step towards a type of Serre-Swan theorem.

\begin{proposition}\label{prop: Serre-Swan}
The global sections functor establishes an equivalence of categories
\begin{equation*}
\{\textnormal{locally free sheaves of $CG$-modules}\} \xra{\sim} \{\textnormal{finitely generated projective $CG$-modules}\}.
\end{equation*}
\end{proposition}

Although the proof is relatively straightforward, we discuss the proof only in \cite{Homotopyoperators}. In fact, the main ingredients of the proof will be discussed next.

\begin{remark}\label{rema: ideal of ruth}
Notice that $CG$ is canonically filtered by (differential, left and right) ideals:
\begin{equation*}
    \calI_k \coloneq \textstyle\bigoplus_{\bullet \ge k} C^\bullet G.
\end{equation*}
We write $\calI_k^\bullet = \calI_k \cap C^\bullet G$. These ideals canonically define a filtration on a locally free sheaf of $CG$-modules / finitely generated projective $CG$-module $\calE_G$:
\begin{equation*}
    \calF_k \coloneq \calE_G \cdot \calI_k = \textstyle\bigoplus_{\bullet \ge k} \{c \in \calE_G^\bullet \mid c = \textstyle\sum_m c_m \cdot f_m \textnormal{ for some } c_m \in \calE_G^{\bullet-\bullet'} \textnormal{ and } f_m \in \calI_k^{\bullet'}\}.
\end{equation*}
We write $\calF_k^\bullet = \calF_k \cap \calE_G^\bullet$.
\end{remark}

\begin{proposition}\label{prop: underlying complex of vector bundles}
Let $\calE_G$ be a ruth. For all $\bullet \ge 0$,
\begin{equation*}
\Gamma E_M^\bullet = \calF^\bullet_0/\calF^\bullet_1
\end{equation*}
forms the space of sections of a vector bundle $E_M^\bullet$ over $M$. Moreover, the differential $\delta$ descends to a cochain complex of vector bundles:
\begin{center}
\begin{tikzcd}
    E_M^0 \ar[r,"\partial"] & E_M^1 \ar[r,"\partial"] & \cdots \ar[r,"\partial"] & E_M^{n-1} \ar[r, "\partial"] & E_M^n
\end{tikzcd}
\end{center}
(We assume that $E_M^{\bullet>n} = 0$; in this case we call $\calE_G$ an \textit{$(n+1)$-term abstract ruth}).
\end{proposition}

\begin{remark}
The quotient map $(CG)_\textnormal{N} \ra (CG)_\textnormal{N} / \calI_1 = C^\infty M$ comes from the unit map
\begin{equation*}
1_M \hookrightarrow G
\end{equation*}
viewed as a groupoid map (here, $1_M$ denotes the unit groupoid). Now, we can always pullback an abstract ruth along a Lie groupoid map, and what Proposition \ref{prop: underlying complex of vector bundles} shows is that the resulting abstract ruth $\calE_G|_M$ can be seen as a complex of vector bundles. 
\end{remark}

\begin{remark}\label{rema: dual of ruth}
Notice that we are assuming that (abstract) ruths are non-negatively graded. In particular, when dualising an abstract ruth, we shift it to non-negative degrees. To be clear, given an $(n+1)$-term abstract ruth $\calE_G$,
\begin{equation*}
(\calE_G^*)^k \coloneq \textnormal{Hom}^k_{CG\textnormal{-lin}}(\calE_G, CG) = \{\varphi: \calE_G \ra CG[k] \mid \textnormal{for all } f \in CG, \textnormal{ } \varphi(c \cdot f) = \varphi(c) \cdot f\}
\end{equation*}
is the dual ruth, which is concentrated in non-positive degrees. The differential is given by\footnote{Here and later, the notation $[\cdot,\cdot]$ is used for the graded commutator. In this case, $[\delta,\varphi]$ means $\delta \varphi - (-1)^\varphi \varphi \delta$.}
\begin{equation*}
\delta^*\varphi \coloneq [\delta, \varphi]
\end{equation*}
and it is a right $CG$-module by setting ($\iota \coloneq (-1)^{\tfrac{\bullet(\bullet+1)}{2}}\bfi^*$ is pullback by the inversion map\footnote{Notice that $[\delta,\iota] = \delta \iota - \iota \delta = 0$.})
\begin{equation*}
(\varphi \cdot f)c \coloneq (-1)^{\varphi \cdot f} (\iota f) \cdot (\varphi c).
\end{equation*}
We will, however, consider $\calE_G^*[-n]$, which introduces a sign $(-1)^n$ in the equation for $\delta$. Instead of writing $\calE_G^*[-n]$, we simply write $\calE_G^*$. So, the dual construction we consider is a ``shifted duality'' and in the $2$-term case it aligns with the dualisation of VB-groupoids. For higher terms, we expect there is a precise relation with \cite{RonchiZhu} (see Remark \ref{rema: Serre-Swan}).
\end{remark}

\begin{example}\label{exa: dual of VB as ruth}
    As mentioned implicitly above, we have an isomorphism of abstract ruths
    \begin{equation*}
        C_\textnormal{VB} V_G \cong (C_\textnormal{VB} V_G^*)^* 
    \end{equation*}
    which, explicitly, is given by sending $f \in C_\textnormal{VB} V_G$ to 
    \begin{equation*}
        \varphi_f(c)(g_1,\cdots,g_{f+c-1}) \coloneq (-1)^{(f-1)c + \tfrac{f(f+1)}{2}} c(f((-)_{g_f^{-1}}, g_{f-1}^{-1}, \dots, g_1^{-1})^{-1}, g_{f+1}, \dots, g_{f+c-1}).
    \end{equation*}
\end{example}

\subsection{The normalised complex of an abstract ruth}\label{sec: the normalised complex of an abstract ruth}

Here, we elaborate on our choice to define abstract ruths as in Definition \ref{def: ruths as dgms} as $(CG)_\textnormal{N}$-modules.

\begin{remark}\label{rema: Dold-Kan correspondence ruths}
The cosimplicial version of the Dold-Kan correspondence \cite{Dold,Kan} sets an equivalence of categories between (non-negatively graded)  cochain complexes and cosimplicial objects (in an abelian category). The cosimplicial object (say, as a cosimplicial vector space)
\begin{center}
\begin{tikzcd}
C^\infty M \ar[r, shift right, shift right] \ar[r, shift left, shift left] & \ar[l] C^1 G \ar[r, shift left, shift left, shift left, shift left] \ar[r] \ar[r, shift right, shift right, shift right, shift right] & \ar[l, shift left, shift left] \ar[l, shift right, shift right] C^2 G \textnormal{ } \cdots: C^\bullet G
\end{tikzcd}
\end{center}
corresponds this way to the cochain complex $(CG)_\textnormal{N}$. A crucial point is that the inclusion 
\begin{equation*}
(CG)_\textnormal{N} \hookrightarrow CG
\end{equation*}
is a quasi-isomorphism, and especially the way in which it is. In short, the degeneracy maps provide homotopy data
\begin{equation*}
    \eta_m \coloneq (-1)^{m+1}\nu_{\bullet-m-1} = (-1)^{m+1}u_{\bullet-m-1}^* \qquad \pi_m \coloneq 1 - [\delta,\eta_m]
\end{equation*}
that show that each inclusion in the filtration 
\begin{center}
\begin{tikzcd}
(C^\bullet G)_\textnormal{N} \ar[r, shift left, hookrightarrow] & \ar[l, shift left, "\pi_{\bullet-1}"]\textstyle\bigcap_{\ell=1}^{\bullet-1} \ker \nu_\ell \ar[r, shift left, hookrightarrow] & \ar[l, shift left, "\pi_{\bullet-2}"] \cdots \ar[r, shift left, hookrightarrow] & \ar[l, shift left, "\pi_1"] \ker \nu_{\bullet-1} \ar[r, shift left, hookrightarrow] & \ar[l, shift left, "\pi_0"] C^\bullet G
\end{tikzcd}
\end{center}
is a quasi-isomorphism.\footnote{This argument works for general cosimplicial objects.} The VB-complex $C_\textnormal{VB} V_G$ of a VB-groupoid $V_G$ is similarly filtered
\begin{center}
\begin{tikzcd}
(C_\textnormal{VB}^\bullet V_G)_\textnormal{N} \ar[r, shift left, hookrightarrow] & \ar[l, shift left, "\pi_{\bullet-1}"]\textstyle\bigcap_{\ell=1}^{\bullet-1} \ker \nu_\ell \ar[r, shift left, hookrightarrow] & \ar[l, shift left, "\pi_{\bullet-2}"] \cdots \ar[r, shift left, hookrightarrow] & \ar[l, shift left, "\pi_1"] \ker \nu_{\bullet-1} \ar[r, shift left, hookrightarrow] & \ar[l, shift left, "\pi_0"] C_\textnormal{VB}^\bullet V_G
\end{tikzcd}
\end{center}
and each inclusion is a quasi-isomorphism by the same argument as above. However, realising
\begin{equation*}
    C_\textnormal{VB} V_G \cong (C_\textnormal{VB} V_G)_\textnormal{N} \otimes_{(CG)_\textnormal{N}} CG,
\end{equation*}     
this fact can be seen as a consequence of the above argument for $(CG)_\textnormal{N} \hookrightarrow CG$.
\end{remark}

By virtue of the way in which we defined abstract ruths, we have the following general statement.

\begin{proposition}\label{prop: inclusion of normalised subcomplex is qiso for abstract ruth}
    Given an abstract ruth $(\calE_G)_\textnormal{N}$, the inclusion 
    \begin{equation*}
    (\calE_G)_\textnormal{N} \hookrightarrow \calE_G
    \end{equation*}
    is a $((CG)_\textnormal{N} \hookrightarrow CG)$-linear quasi-isomorphism.
\end{proposition}

So, if we wish to view an abstract ruth as a $CG$-module, Definition \ref{def: ruths as dgms} automatically incorporates a sensible notion of ``normalised subcomplex''. In the next section (Section \ref{sec: split ruths}) we show how to recover the usual notion of a ruth (called split ruth in this work). We conclude with two remarks.

\begin{remark}\label{rema: inclusion VB-complex is a qiso}
Observe the similarities of the above discussion with the result that the inclusion
\begin{equation*}
C_\textnormal{VB} V_G \hookrightarrow C_\textnormal{lin} V_G
\end{equation*}
is a quasi-isomorphism of complexes (see \cite{VanEsthomogeneous,VBMorita}). As made precise in \cite{VBMorita}, this fact can be proven via the same strategy explained above. Indeed, $C_\textnormal{lin} V_G$ is canonically filtered  
\begin{center}
\begin{tikzcd}
C_\textnormal{VB}^\bullet V_G \ar[r, shift left, hookrightarrow] & \ar[l, shift left, "\pi_{\bullet-1}"] \calK^\bullet_{\bullet-1} \ar[r, shift left, hookrightarrow] & \ar[l, shift left, "\pi_{\bullet-2}"] \cdots \ar[r, shift left, hookrightarrow] & \ar[l, shift left, "\pi_1"]  \calK^\bullet_1 \ar[r, shift left, hookrightarrow] & \ar[l, shift left, "\pi_0"] C_\textnormal{lin}^\bullet V_G
\end{tikzcd}
\end{center}
by the subcomplexes 
\begin{equation*}
\calK^\bullet_m = \{f \in C_\textnormal{lin}^\bullet V_G \mid \begin{array}{l} f(0_{g_1}, \dots, 0_{g_{\bullet-m+1}}, v_{\bullet-m+2}, \dots, v_\bullet) = 0 \\ \delta f(0_{g_1},\cdots,0_{g_{\bullet-m+1}},v_{\bullet-m+2}, \dots, v_{\bullet+1}) = 0 \end{array}\}
\end{equation*}
and the homotopy data is defined via a (unital) cleavage $\Sigma$. Explicitly, given $f \in K_m^\bullet$,
\begin{equation*}
\eta_mf(v_1, \dots, v_{\bullet-1}) \coloneq (-1)^\bullet f(\Sigma_{g_{\bullet-1}^{-1} \cdots g_1^{-1}}\bft_{g_1} v_1, v_1, \dots, v_{\bullet-1})
\end{equation*}
and then, as before, we set $\pi_m \coloneq 1 - [\delta,\eta_m]$.
\end{remark}

\begin{remark}\label{rema: VB-complex and Dold-Kan}
Given a VB-groupoid $V_G \rra V_M$, the cosimplicial object
\begin{center}
\begin{tikzcd}
C^\infty_\textnormal{lin} V_M \ar[r, shift right, shift right] \ar[r, shift left, shift left] & \ar[l] C^1_\textnormal{lin} V_G \ar[r, shift left, shift left, shift left, shift left] \ar[r] \ar[r, shift right, shift right, shift right, shift right] & \ar[l, shift left, shift left] \ar[l, shift right, shift right] C^2_\textnormal{lin} V_G \textnormal{ } \cdots: C_\textnormal{lin}^\bullet V_G
\end{tikzcd}
\end{center}
is related to the cochain complex $(C_\textnormal{lin} V_G)_\textnormal{N}$ via the Dold-Kan correspondence. But, usually, we can not view $C_\textnormal{VB} V_G \subset C_\textnormal{lin} V_G$ as a cosimplicial subobject even though
\begin{equation*}
    (C_\textnormal{VB} V_G)_\textnormal{N} \subset (C_\textnormal{lin} V_G)_\textnormal{N}
\end{equation*}
is a subcomplex. Instead, $(C_\textnormal{VB} V_G)_\textnormal{N}$ corresponds to the cosimplicial subobject $C_\textnormal{VB,cs} V_G$ consisting of those $f \in C_\textnormal{lin}^\bullet V_G$ for which
\begin{equation*}
    f(0_{g_1} \cdot c^{-1}, v_2, \dots, v_\bullet) = f(c^{-1}, v_2, \dots, v_\bullet),
\end{equation*}
and such that, for all $2 \le k \le \bullet$,
\begin{equation*}
    f(0_{g_1}, \dots, 0_{g_{k-1}}, 0_{g_k} \cdot c^{-1}, v_{k+1}, \dots, v_\bullet) = f(0_{g_1}, \dots, 0_{g_{k-1}g_k}, c^{-1}, v_{k+1}, \dots, v_\bullet).
\end{equation*}
Indeed, we have
\begin{equation*}
    (C_\textnormal{VB,cs} V_G)_\textnormal{N} = (C_\textnormal{VB} V_G)_\textnormal{N}.
\end{equation*}
Notice that, $C_\textnormal{VB, cs} V_G \subset C_\textnormal{lin} V_G$ is not a $CG$-submodule and, for all $\bullet \ge 1$, $C_\textnormal{VB, cs}^\bullet V_G \subset C_\textnormal{lin}^\bullet V_G$ is not a $C^\infty G^{(\bullet)}$-submodule in general.
\end{remark}

\subsection{Split ruths}\label{sec: split ruths}

Let $\calE_G = (\calE_G)_\textnormal{N} \otimes_{(CG)_\textnormal{N}} CG$ be an abstract ruth. There is a non-canonical, filtration preserving, $CG$-linear isomorphism of $\calE_G^\bullet$ with the sum of the $\bullet$-antidiagonal of
\begin{center}
\begin{tikzcd}
C^2(G; E_M^0) & C^2(G; E_M^1) & C^2(G; E_M^2) \\
C^1(G; E_M^0) & C^1(G; E_M^1) & C^1(G; E_M^2) \\
C^0(G; E_M^0) & C^0(G; E_M^1) & C^0(G; E_M^2)
\end{tikzcd}
\end{center}
We consider the bigraded object above to be filtered by the rows: $\calF_k$ consists of those elements of $\calE_G$ that lie in a row $\ge k$. We can view a component in degree $(k,\bullet-k)$\footnote{We use the matrix convention, so $k$ denotes the row and $\bullet-k$ the column.} as
\begin{equation*}
\calF_k^\bullet/\calF_{k+1}^\bullet = C^k(G;E_M^{\bullet-k}).
\end{equation*}
Now, $\delta$ decomposes into maps\footnote{Notice that $\partial=\delta_0$ is determined by the previously defined $\partial$.}
\begin{equation*}
    \delta_\ell \coloneq C^k(G; E_M^{\bullet-k}) \ra C^{k+\ell}(G; E_M^{\bullet-k+1-\ell})
\end{equation*}
that diagrammatically are given as follows:\footnote{We only drew the maps $\partial$, $\delta_1$ and $\delta_2$.}
\begin{center}
\begin{tikzcd}
C^2(G; E_M^0) \ar[r] & C^2(G; E_M^1) \ar[r] & C^2(G; E_M^2) \\
C^1(G; E_M^0) \ar[r, dotted] \ar[u] & C^1(G; E_M^1) \ar[r, dotted] \ar[u] & C^1(G; E_M^2) \ar[u] \\
C^0(G; E_M^0) \ar[r, "\partial"] \ar[u, "\delta_1"] & C^0(G; E_M^1) \ar[r] \ar[u] \ar[luu, "\delta_2"] & C^0(G; E_M^2) \ar[u] \ar[luu]
\end{tikzcd}
\end{center}

While this seems to realise the abstract ruth as a ruth in the usual sense, this is only true if the splitting preserves normalised cochains.

\begin{definition}\label{def: splitting of a ruth}
A \textit{splitting} of an $(n+1)$-term abstract ruth $\calE_G = (\calE_G)_\textnormal{N} \otimes_{(CG)_\textnormal{N}} CG$ is a $CG$-linear isomorphism
\begin{equation*}
\calE_G^\bullet \cong C(G;E_M)^\bullet = \textstyle\bigoplus_{k=0}^\bullet C^k(G; E_M^{\bullet-k})
\end{equation*}
that, for all $\bullet \le n$, induces the identity map on $\Gamma E_M^\bullet$. The splitting is called \textit{normalised} or \textit{unital} if it preserves normalised cochains.
\end{definition}

\begin{remark}\label{rema: splitting is choice of section of projection to underlying graded vector bundle}
    A (unital) splitting of an $(n+1)$-term abstract ruth $\calE_G$ is automatically filtration preserving. Moreover, it is equivalently described by a choice of sections of the $n$ projections
    \begin{equation*}
        (\calE_G^{1 \le \bullet \le n})_\textnormal{N} \ra \Gamma E_M^\bullet. 
    \end{equation*}
\end{remark}

We adopt the following terminology if we want to emphasise the difference:

\begin{definition}\label{def: split ruth}
A \textit{split ruth} is a graded vector bundle $E_M^{\bullet \ge 0}$ together with a differential $\delta = \textstyle\sum_{\ell \ge 0} \delta_\ell$ on the graded $CG$-module
\begin{equation*}
    \calE_G^\bullet = C(G;E_M)^\bullet = \textstyle\bigoplus_{k=0}^\bullet C^k(G; E_M^{\bullet-k})
\end{equation*}
turning it into a differential graded $CG$-module, i.e. such that the Leibniz rule
\begin{equation*}
    \delta(c \cdot f) = \delta c \cdot f + (-1)^c c \cdot \delta f
\end{equation*}
is satisfied. We assume $\calE_G$ is \textit{unital}, meaning that $\delta$ preserves normalised cochains.
\end{definition}

\begin{remark}\label{rema: structure maps R}
Given a split ruth, we can define maps $R=(R_\ell)_{\ell \ge 0}$, where
\begin{equation*}
    R_\ell \in \Gamma \textnormal{End}^{\ell,1-\ell} E_M \qquad \textnormal{End}^{\ell,\ell'} E_M \coloneq \textnormal{Hom}((\bfs^{(\ell)})^*E_M^\bullet,(\bft^{(\ell)})^*E_M^{\bullet+\ell'})
\end{equation*}
are defined as
\begin{equation*}
    R_{\ell \neq 1} \coloneq \delta_\ell \qquad R_1 \coloneq \delta_1 - (-1)^\bullet\textstyle\sum_{k=1}^{\bullet+1} (-1)^k d_k^*.
\end{equation*}
This uses the Leibniz rule and the fact that all $CG$-linear maps
\begin{equation*}
C^\bullet(G;E_M^{\bullet'}) \ra C^{\bullet+\ell}(G;E_M^{\bullet'+\ell'})
\end{equation*}
(so of bidegree $(\ell,\ell')$) come uniquely from elements of $\textnormal{End}^{\ell,\ell'} E_M$ (see \cite{LGruths}). The correspondence is such that the Koszul sign rule applies: given $R_{\ell,\ell'} \in \textnormal{End}^{\ell,\ell'} E_M$, the associated $CG$-linear map is given by\footnote{As mentioned in Section \ref{sec: Notation for VB-groupoids}, we think of $g \in G^{(\ell)}$ as having degree $\ell$.}
\begin{equation*}
R_{\ell,\ell'} \cdot f(g,g') = (-1)^{\ell \cdot f} R_{\ell,\ell'}(g)f(g')
\end{equation*}
That $\delta^2 = 0$ then translates into the Maurer-Cartan equation
\begin{equation*}
\delta_\textnormal{ruth} R + \tfrac{1}{2}[R,R]_\textnormal{ruth} = 0.
\end{equation*}
The Maurer-Cartan equation can be taken literally as follows: the dgLa of ruths, that controls deformations, is given by
\begin{equation*}
C^\bullet_\textnormal{ruth} \coloneq \textstyle\bigoplus_{\ell \ge 0} \textnormal{End}^{\ell,\bullet-\ell} E_M
\end{equation*}
together with the graded Lie bracket and differential defined through the differential graded algebra\footnote{Notice that $d_0$ and $d_{\bullet+1}$ are not appearing in the equation for $\delta_\textnormal{ruth}$.}
\begin{equation*}
R_\ell \cdot R_{\ell'}(g,g') = (-1)^{\ell \cdot R_{\ell'}} R_\ell(g) R_{\ell'}(g') \qquad \delta_\textnormal{ruth} = (-1)^\bullet \textstyle\sum_{k=1}^\ell (-1)^k d_k^*
\end{equation*}
That is, the graded Lie bracket $[\cdot,\cdot]_\textnormal{ruth}$ is the commutator. This differential graded algebra appeared in \cite{Tensorruth}. The deformation complex (which is a dgLa) of a ruth $R = (R_\ell)$ is then
\begin{equation*}
C_\textnormal{def}^\bullet R = C^\bullet_\textnormal{ruth}
\end{equation*}
together with the same graded Lie bracket, but with
\begin{equation*}
\delta_R \coloneq \delta_\textnormal{ruth} + [R,\cdot].
\end{equation*}
(see also \cite{LaPastinathesis}). Notice that the Maurer-Cartan equation above gives rise to \textit{non-unital ruths}; we didn't take into account the condition that $\delta$ preserves normalised cochains. It is readily verified that this happens precisely when $R_1(1)=1$ and $u_k^*R_{\ell \ge 2}=0$ for all degeneracies $u_k$.
\end{remark}

\subsection{Towards a Serre-Swan theorem}\label{sec: Towards a Serre-Swan theorem}

Definitions of ruths as finitely generated projective differential graded $CG$-modules can already be found in a remark in \cite{LGruths} (after Definition 3.9 there), and in \cite{Stefanithesis}. However, it seems that in the literature ruths of Lie groupoids have not properly been studied as ``abstract'' or ``canonical'' objects. The recognition of a normalised complex as being (part of) the structure of a ruth is important for achieving this. That abstract ruths appear canonically in examples becomes especially helpful in more complicated examples, but we observe here that the VB-complex of a VB-groupoid is a valid (universal) example of a $2$-term abstract ruth. In essence, this was proved in \cite{VBgroupoidsruths}.

\begin{proposition}\label{prop: VB complex}
Let $V_G$ be a VB-groupoid. Then $C_\textnormal{VB} V_G$ is a $2$-term abstract ruth. Moreover, (unital) cleavages correspond to (unital) splittings.
\end{proposition}
\begin{proof}
That $C_{\textnormal{VB}} V_G$ is finitely generated and projective as a $CG$-module can be seen by picking a cleavage $\Sigma$, for $\Sigma$ sets an isomorphism
\begin{equation*}
C_{\textnormal{VB}}^\bullet V_G \cong C^\bullet(G; V_M^*) \oplus C^{\bullet-1}(G; C_M^*) \qquad f \mapsto (f_{V_M}, f_{C_M}),
\end{equation*}
where\footnote{We use here that, if $f \in C^\bullet_\textnormal{VB} V_G$, then $f(v_{g_1},\dots,v_{g_\bullet}) \in \bbR$ only depends on $(v_1,g_2,\dots,g_\bullet)$.}
\begin{equation*}
f_{V_M}(g_1,\dots,g_\bullet)v \coloneq f(\overline\Sigma_{g_1} v, g_2, \dots, g_\bullet) \qquad f_{C_M}(g_2,\dots,g_\bullet)c \coloneq f(-c^{-1}, g_2, \dots, g_\bullet).
\end{equation*}
Given a (unital) splitting of the abstract ruth $C_{\textnormal{VB}} V_G$, the (unital) cleavage $\overline\Sigma$ (a section of $\bft$) is recovered as the unique linear map for which
\begin{equation*}
C^1_\textnormal{VB} V_G \cong C^1(G; V_M^*) \oplus \Gamma C_M^* \xra{\pr} C^1(G; V_M^*) \qquad f \mapsto f \circ \overline\Sigma.
\end{equation*}
That $\overline\Sigma$ is a splitting follows by considering $f = f_{V_M} \cdot f'$ for all $f_{V_M} \in \Gamma V_M^*$ and $f' \in C^1G$. If the ruth splitting is unital, then for all normalised $f$,
\begin{equation*}
f \circ \overline\Sigma_1 = 0.
\end{equation*}
So, $\overline\Sigma_1$ maps into the units of $V_G$, and therefore it is the unit map. This proves the statement.
\end{proof}

We will discuss a slightly different proof, which shows directly that $C_{\textnormal{VB}} V_G$ can be seen as a locally free sheaf of $CG$-modules. The main point is that, locally, we can use a special type of cleavage $H$, one for which all maps $\bft_g H_g$ are invertible. We can see such sections as local sections of the surjective submersion\footnote{See Section \ref{sec: Notation for VB-groupoids} for a short introduction to the fat groupoid $\widehat V_G$. In Section \ref{sec: The fat extension of a VB-groupoid} we give a more thorough introduction}
\begin{equation*}
\widehat V_G \ra G \qquad H_g \mapsto g.
\end{equation*}
That is, $H$ is a smooth family of linear maps
\begin{equation*}
g \mapsto (V_M)_{\bfs g} \xra{H_g} (V_G)_g
\end{equation*}
such that, for all $g \in G$, $\bfs_g H_g = 1$ and $\bft_g H_g$ is a linear isomorphism.

Now, in the above isomorphism using the cleavage $\Sigma$, we used the linear isomorphisms
\begin{equation*}
(C_M)_{\bfs g} \oplus (V_M)_{\bft g} \ra (V_G)_g \qquad (c,v) \mapsto \ell_gc + \overline\Sigma_gv.
\end{equation*}
Given a local section $H$ of $\widehat V_G \ra G$, we can instead use the linear isomorphism
\begin{equation*}
(C_M)_{\bfs g} \oplus (V_M)_{\bft g} \ra (V_G)_g \qquad (c,v) \mapsto \ell_gc + H_g(\bft_g H_g)^{-1}v.
\end{equation*}
In fact, this observation leads to a whole new description of $C_\textnormal{VB} V_G$ in terms of the fat groupoid $\widehat V_G$. This will be made more precise in Section \ref{sec: From fat extensions to ruths using the fat groupoid}.

\begin{proof}[Alternative proof of Proposition \ref{prop: VB complex}]
As said before, we will show that $C_\textnormal{VB} V_G$ is locally free as a sheaf of $CG$-modules (see Proposition \ref{prop: Serre-Swan}). In fact, given $U \subset M$ and a local section $H$ of $\widehat V_G \ra G$ defined over $U$, then $H$ defines an isomorphism
\begin{equation*}
    C_{\textnormal{VB}}^\bullet V_G(U) \cong C^\bullet(G; V_M^*)(U) \oplus C^{\bullet-1}(G; C_M^*)(U) \qquad f \mapsto (f_{V_M}, f_{C_M})
\end{equation*}
where 
\begin{equation*}
    f_{V_M}(g_1,\dots,g_\bullet)v \coloneq f(H_{g_1}(\bft H_{g_1})^{-1}v, g_2, \dots, g_\bullet) \qquad f_{C_M}(g_2,\dots,g_\bullet)c \coloneq f(-c^{-1}, g_2, \dots, g_\bullet).
    \end{equation*}
For the inverse map, we construct out of $f_{V_M} \in C^1(G; V_M^*)$ a linear function $f \in C_{\textnormal{lin}}^1 V_G$ by setting
\begin{equation*}
    f(v) \coloneq f_{V_M}(g)\bft_gv.
\end{equation*}
Since $C^1(G; V_M^*) \cong \Gamma V_M \otimes C^\infty G$, we see that every such function is locally generated by sections of $V_M^*$ (sitting in degree $0$). Now, given a local section $H$ of $\widehat V_G \ra G$, a section $f_{C_M}$ of $C_M^*$ gives rise to the linear function $f \in C_{\textnormal{lin}}^1 V_G$ defined as
\begin{equation*}
f(v) \coloneq f_{C_M}\omega_g^{H_g}v,
\end{equation*}
where
\begin{equation*}
\omega_g^{H_g}: (V_G)_g \ra (C_M)_{\bfs g} \qquad v \mapsto r_g((H_g(\bft_gH_g)^{-1}\bft_gv)^{-1} - v^{-1}).
\end{equation*}
It is readily verified that both types of linear functions $f \in C^1_\textnormal{lin} V_G$ are elements of $C^1_{\textnormal{VB}} V_G$. Similarly, or extending $CG$-linearly, we obtain elements $C^\bullet_{\textnormal{VB}} V_G$ out of (local) sections of $C^\bullet(G; V_M^*)$ and out of (local) sections of $C^{\bullet-1}(G; C_M^*)$. Now, given $f \in C^1_\textnormal{VB} V_G$, we have
\begin{equation*}
f(v) = f_{C_M}\omega_g^{H_g}v + f_{V_M}(g)\bft_gv,
\end{equation*}
and similarly we can decompose elements of $C_\textnormal{VB}^\bullet V_G$. This shows that the above maps are inverse to each other, so $C_{\textnormal{VB}} V_G$ is locally free as a sheaf of $CG$-modules.
\end{proof}

The equivalence of categories between VB-groupoids and $2$-term ruths can now be understood through the assignment
\begin{equation*}
V_G \mapsto C_\textnormal{VB} V_G^* \coloneq C_\textnormal{VB}(V_G^*) \cong (C_\textnormal{VB} V_G)^*.
\end{equation*}
The assignment on morphisms uses, for $\Phi: V_G \ra V_H$ a VB-groupoid map, the dual map
\begin{equation*}
\Phi^*: V_H^* \ra V_G^*
\end{equation*}
and then the pullback map
\begin{equation*}
(\Phi^*)^*_\textnormal{VB}: C_\textnormal{VB} V_G^* \ra C_\textnormal{VB} V_H^*.
\end{equation*}
(or this pullback map first and then the dual map of ruths). This defines a functor, and it is an equivalence of categories \cite{VBgroupoidsruths,VBMorita}:
\begin{theorem}\label{thm: equivalence VB-groupoids and ruths}
The functor
\begin{equation*}
\{\textnormal{VB-groupoids}\} \ra \{\textnormal{Abstract $2$-term ruths}\} \qquad V_G \mapsto C_\textnormal{VB} V_G^* \qquad \Phi \mapsto (\Phi^*)^*_\textnormal{VB}
\end{equation*}
sets an equivalence of categories.
\end{theorem}


\begin{remark}\label{rema: Serre-Swan}
Perhaps we should identify Theorem \ref{thm: equivalence VB-groupoids and ruths} truly as a Serre-Swan type theorem. That is, we can think of $C_\textnormal{VB} V_G^*$ (or $C_\textnormal{proj} V_G$) as a type of ``space of sections'' of the VB-groupoid $V_G$. A current work in progress is to understand generalisations of this fact. Motivated by \cite{Highervectorbundles}, we expect to set a canonical Serre-Swan correspondence between higher vector bundles and abstract ruths. In doing so, we believe it is important to investigate the link between the choice of a ``normal pre-cleavage'' (or a normal weakly flat cleavage) for higher vector bundles in \cite{Highervectorbundles} to the above unital splittings of abstract ruths. Indeed, in \cite{Highervectorbundles} it is shown that the choice of a normal weakly flat cleavage of a higher vector bundle (that admits such a cleavage) defines a ruth, and the normal and weakly flat conditions are used to show that this ruth is unital. 
\end{remark}

\begin{remark}
Related to the above remark is how the tensor product constructions of VB-groupoids and $2$-term ruths relate (for the latter, see \cite{Tensorruth}). This is part of the work in progress \cite{Multiplicativetensors}.
\end{remark}

\subsection{Cochain complex representations}

We end this section by mentioning a special type of $2$-term ruth that is central to this work:

\begin{definition}\label{def: cochain complex representation}
Let $G \rra M$ be a Lie groupoid. A \textit{cochain complex representation}, or strict ruth, is a split ruth $\calE_G$ such that $R_{\ell \ge 2} = 0$.
\end{definition}
A cochain complex representation of $G$ on $C_M \ra V_M$ is a pair of representations of $G$ on $C_M$ and $V_M$, such that the map $C_M \ra V_M$ is equivariant with respect to the actions. We write the abstract ruth associated to a cochain complex representation by
\begin{equation*}
C(G; C_M \ra V_M) = C_\textnormal{VB}(G \ltimes (V_M^* \ra C_M^*)).
\end{equation*}

There are at least two more ways of describing a cochain complex representation that we want to discuss.

\begin{remark}\label{rema: GL of a cochain complex representation}
Given two vector bundles $C_M$ and $V_M$, there is a type of general linear groupoid given by
\begin{equation*}
\textnormal{GL}(C_M,V_M) \coloneq \textnormal{GL } C_M \times_{M \times M} \textnormal{GL } V_M
\end{equation*}
which we can view as a fiber product of Lie groupoids.\footnote{We think here of $M \times M$ as the pair groupoid, and we use the (surjective) anchor maps.} Now, the differential induces a map
\begin{equation*}
[\cdot, \partial]: \textnormal{GL}(C_M,V_M) \ra \textnormal{Hom}(\pr_1^*C_M,\pr_2^*V_M) \qquad \Phi = (\Phi^{C_M}, \Phi^{V_M}) \mapsto \Phi^{V_M}\partial - \partial\Phi^{C_M}.
\end{equation*}
We can interpret the preimage of zero
\begin{equation*}
\textnormal{Aut}(C_M \ra V_M) \coloneq \{\Phi \in \textnormal{GL}(C_M,V_M) \mid \partial \Phi^{C_M} = \Phi^{V_M} \partial\}
\end{equation*}
as cochain isomorphisms
\begin{center}
\begin{tikzcd}
(C_M)_x \ar[r, "\partial"] \ar[d, "\Phi^{C_M}"] & (V_M)_x \ar[d, "\Phi^{V_M}"] \\
(C_M)_y \ar[r, "\partial"] & (V_M)_y
\end{tikzcd}
\end{center}
However, while $\textnormal{Aut}(C_M \ra V_M)$ is a subgroupoid of $\textnormal{GL}(C_M,V_M)$, it does not define a smooth Lie subgroupoid in general. In fact, the above commutator map is the restriction of a linear map, and then it is readily verified that it has constant rank precisely when $\partial$ has constant rank. In any case, we can think of a cochain complex representation as a Lie groupoid map
\begin{equation*}
G \ra \textnormal{GL}(C_M,V_M) \dasharrow \textnormal{Hom}(\pr_1^*C_M,\pr_2^*V_M)
\end{equation*}
such that postcomposing with the map $[\cdot,\partial]$ is zero.
\end{remark}

\begin{remark}\label{rema: cochain complex representation as groupoid map}
Another way to organise the information of a cochain complex representation of a Lie groupoid $G$ on $C_M \ra V_M$ is as a Lie groupoid map in the following way. Consider first the general linear groupoid from Remark \ref{rema: GL of a cochain complex representation}:
\begin{equation*}
\textnormal{GL}(C_M,V_M) = \textnormal{GL } C_M \times_{M \times M} \textnormal{GL } V_M.
\end{equation*}
This groupoid acts (on the right) on ``differentials'' $\textnormal{Hom}(C_M,V_M)$ by conjugating the differentials: given $\Psi_{y,x} = (\Psi_{y,x}^{C_M}, \Psi_{y,x}^{V_M}) \in \textnormal{GL}(C_M,V_M)$ and $\partial_y \in \textnormal{Hom}(C_M,V_M)$,
\begin{equation*}
\partial_y \cdot \Psi_{y,x} \coloneq (\Psi_{y,x}^{V_M})^{-1} \partial_y \Psi_{y,x}^{C_M}.
\end{equation*}
We write the resulting action groupoid by
\begin{equation*}
    \textnormal{GL}(C_M,V_M)_M \coloneq \textnormal{Hom}(C_M,V_M) \rtimes \textnormal{GL}(C_M,V_M).
\end{equation*}
A cochain complex representation of $G$ on $C_M \ra V_M$ can now be seen as a groupoid map
\begin{equation*}
G \ra \textnormal{GL}(C_M,V_M)_M
\end{equation*}
whose basemap is the differential:
\begin{equation*}
\partial: M \ra \textnormal{Hom}(C_M,V_M).
\end{equation*}
If we postcompose this map with the groupoid projection $\textnormal{GL}(C_M,V_M)_M \ra \textnormal{GL}(C_M,V_M)$, then we recover the map we considered in Remark \ref{rema: GL of a cochain complex representation}.
\end{remark}

\section{The fat extension of a VB-groupoid}\label{sec: The fat extension of a VB-groupoid}

We start this section by introducing the fat groupoid of a VB-groupoid. In this generality, it seems to have first appeared in \cite{VBgroupoidsruths}.

\subsection{The fat groupoid}\label{sec: the fat groupoid}
Let $V_G \rra V_M$ be a VB-groupoid. Consider the surjective linear map
\begin{equation*}
\textnormal{Hom}(\bfs^*V_M, V_G) \ra \textnormal{End } V_M \qquad H_g \mapsto \bfs_g H_g
\end{equation*}
over $\bfs: G \ra M$, and the preimage of the identity section
\begin{equation*}
 \widehat V_G^\textnormal{cat} = \{H_g: (V_M)_{\bfs g} \ra (V_G)_g \mid \bfs_g H_g = 1\}.
\end{equation*}
Then $\widehat V_G^\textnormal{cat}$ is an affine bundle over $G$\footnote{The fiber over $g$ is isomorphic to $\textnormal{Hom}((V_M)_{\bfs g},(C_M)_{\bft g})$.} and it is a Lie category (without boundary) over $M$ (a thorough discussion on Lie categories is presented in \cite{Zanbroscats}). The structure maps are given as follows: we set $\widehat\bfs H_g = \bfs g$, $\widehat\bft H_g = \bft g$ and
\begin{equation*}
(H_g \cdot H_{g'})v = H_g(\bft_{g'} H_{g'}v) \cdot H_{g'}v.
\end{equation*}
The invertible elements of this Lie category are given by the following open subset of $\widehat V_G^\textnormal{cat}$:
\begin{equation*}
    \widehat V_G = \{H_g \in \widehat V_G^\textnormal{cat} \mid \bft_g H_g \textnormal{ is a linear isomorphism}\}
\end{equation*}
which becomes a Lie groupoid with
\begin{equation*}
(H_g)^{-1}v = (H_g(\bft_g H_g)^{-1}v)^{-1}.
\end{equation*}

\begin{definition}\label{def: fat groupoid of VB-groupoid}
Given a VB-groupoid $V_G$, we call $\widehat V_G \rra M$ the fat groupoid (of $V_G$).
\end{definition}

We can think of the elements of $\widehat V_G$ as coming from (fiberwise) linear bisections of $V_G$.

\begin{remark}\label{rema: core and linear bisections}
Notice that there is a one-to-one correspondence between bisections of $\widehat V_G$ and linear bisections of $V_G$. Indeed, a linear bisection $\sigma$ of $V_G$ is a bisection
\begin{equation*}
V_M \ra V_G \qquad x \mapsto \sigma x
\end{equation*}
that is a vector bundle map over a bisection $\sigma$ of $G$.\footnote{Alternatively, it is a $0$-homogeneous bisection, i.e. $\lambda^*\sigma (= \lambda^{-1} \circ \sigma \circ \lambda) = \sigma$.} We can therefore, equivalently, view $\sigma$ as the bisection
\begin{equation*}
M \ra \widehat V_G \qquad x \mapsto (V_M)_x \xra{\sigma x} (V_G)_{\sigma x}
\end{equation*}
of $\widehat V_G$. There are also ``core bisections'' of $V_G$ (over the identity bisection of $G$): given $\chi \in \Gamma C_M$, we can construct a bisection $\sigma_\chi$ of $V_G$ by setting
\begin{equation*}
    \sigma_\chi(v_x) = 1_{v_x} + \chi(x).
\end{equation*}
Applying the target to this expression gives the translation $v_x \mapsto v_x + \partial\chi(x)$.
\end{remark}
\begin{remark}
In \cite{VBgroupoidsruths} the fat groupoid is denoted by $\widehat G$. However, we prefer to keep track of the dependence on $V_G$ in the notation.
\end{remark}

We end the section by mentioning a few examples. In Section \ref{sec: examples of fat extensions} we discuss more examples and go into more detail.

\begin{example}\label{exa: jet is fat}
Probably the most well-known example of a ``fat groupoid'' is the jet groupoid $J^1G$:
\begin{equation*}
J^1G = \{j^1_x\sigma \mid \sigma \in \textnormal{Bis}_\textnormal{loc} G\}.
\end{equation*}
In words, its elements consist of all possible values of $1$-jets of local bisections $\sigma$. The product is induced by the usual product of (local) bisections. We can realise $J^1G$ as the fat groupoid $\widehat{TG}$ of the VB-groupoid $TG$. Indeed, every $1$-jet of a local bisection takes the form
\begin{equation*}
H_g: T_{\bfs g}M \ra T_gG,
\end{equation*}
where $d\bfs_g H_g = 1$ and $d\bft_g H_g$ is a linear isomorphism. On the other hand, we can integrate such a linear map $H_g$ to a local bisection $\sigma$ with the property that $j^1_{\bfs g}\sigma = H_g$. This follows from standard proofs on the existence of local bisections through elements of $G$.
\end{example}

\begin{example}\label{exa: subgroupoid of jet}
Let $H \rra N$ be a Lie subgroupoid of $G \rra M$. The jet groupoid $J^1G$ contains the Lie subgroupoid
\begin{equation*}
J^1_H G \coloneq \{j^1_y\sigma \in J^1G \mid y \in N, \im \sigma|_N \subset H\} \rra N.
\end{equation*}
This subgroupoid is equal to $J^1H$ only if $N=M$. How are $J^1H$ and $J^1G$ related if $N \neq M$? For now, we only observe that $J^1_H G$ and $J^1H$ are related via a restriction map
\begin{equation*}
J^1_H G \ra J^1H.
\end{equation*}
A similar phenomenon appears when considering $J^1_\calN(G,H)$, the fat groupoid of
\begin{equation*}
\calN(G,H) \rra \calN(M,N),
\end{equation*}
the normal bundle, which is a VB-groupoid over $H$.\footnote{The Lie groupoid structure of the normal bundle $\calN(G,H)$ is induced by that of $TG$.} Namely, there is again a natural map
\begin{equation*}
J^1_HG \ra J^1_\calN(G,H).
\end{equation*}
In Section \ref{sec: examples of fat extensions}, we will further clarify the relation between $J^1 H$, $J^1G$ and $J^1_\calN(G,H)$.
\end{example}

\begin{example}\label{exa: general linear}
Let $E_M \ra M$ be a vector bundle. The general linear groupoid
\begin{equation*}
\textnormal{GL } E_M = \{\Psi_{y,x}: E_x \xra{\sim} E_y \textnormal{ is a linear isomorphism}\}.
\end{equation*}
is the fat groupoid of the pair groupoid $E_M \times E_M \rra E_M$, seen as a VB-groupoid over $M \times M \rra M$.
\end{example}

The last example shows the extension of the ``classical'' analogues of the correspondences we aim to describe in Sections \ref{sec: fat extensions} and \ref{sec: PB-groupoids}:
\begin{center}
\begin{tikzcd}[column sep = small]
\{\textnormal{f.g.p. $C^\infty M$-mods}\} & \ar[l, "\textnormal{Serre-Swan}"'] \{\textnormal{Vector bundles}\} \ar[r, "\textnormal{Fr}"] \ar[d, "\textnormal{GL}"] & \{\textnormal{Principal $\textnormal{GL}$-bundles}\} \ar[ld, "\textnormal{Gauge groupoid}"] \\
& \ar[ul] \{\textnormal{fat extensions of $M \times M$ over $1_{E_M}$}\} &
\end{tikzcd}
\end{center}
The goal of the coming two sections (Sections \ref{sec: the bundle of invertible homotopies} and \ref{sec: canonical representations of the fat groupoid}) is to discuss two essential pieces of information: the bundle of invertible homotopies, and the canonical representations of the fat groupoid. This is enough to understand the basic structure of fat extensions (see Section \ref{sec: fat extensions 1}). In the sections afterwards we look at several consequences of the structure behind the fat groupoid and discuss, in this order, the fat category (Section \ref{sec: the fat category}), examples (Section \ref{sec: examples of fat extensions}), other models and dualisation of the fat groupoid (Section \ref{sec: other models of the fat groupoid}) and, lastly, the fat pairing (Section \ref{sec: fat pairing}).

\subsection{The bundle of invertible homotopies of a $2$-term complex}\label{sec: the bundle of invertible homotopies}
As mentioned before, if $V_G$ is a VB-groupoid, then the groupoid map
\begin{equation*}
\widehat V_G \ra G \qquad H_g \mapsto g
\end{equation*}
is a surjective submersion. The kernel of this map is a bundle of Lie groups, and it can be identified with the open set of $\textnormal{Hom}(V_M,C_M)$ given by
\begin{equation*}
\textnormal{H}(V_M,C_M) = \{h_x: (V_M)_x \ra (C_M)_x \mid 1 + \partial h_x \textnormal{ is invertible}\}
\end{equation*}
(via $h_x \mapsto 1+h_x$). The unit is the zero map and
\begin{equation*}
h_x \cdot h_x' = h_x + h_x' + h_x\partial h_x' \qquad h_x^{-1} = -h_x(1+\partial h_x)^{-1} = -(1+h_x\partial)^{-1}h_x.
\end{equation*}
Elements $h_x \in \textnormal{H}(V_M,C_M)$ are homotopies, for they show that the isomorphism\footnote{Notice that we have defined two maps, $1 + \partial h_x: (V_M)_x \ra (V_M)_x$ and $1 + h_x\partial: (C_M)_x \ra (C_M)_x$. The latter map is a linear isomorphism by the five-lemma.}
\begin{equation*}
1 + [\partial, h_x]
\end{equation*}
(from $C_M \ra V_M$ to itself) is homotopic to the identity map. The product can then be interpreted as a ``composition of homotopies'' (see also Section \ref{sec: From 2-term ruths to fat extensions}):
\begin{lemma}\label{lemm: homotopy composition}
Suppose we are given complexes $C_1$, $C_2$ and $C_3$ (with differentials $\partial$) with the following homotopy data: let\footnote{We use a subscript $ji$ to denote the relevant map $C_i \ra C_j$.}
\begin{center}
\begin{tikzcd}
C_1 \ar[d] \ar[r, "\Phi", shift left] \ar[r, "\Psi"', shift right] & C_2 \ar[d] \ar[r, "\Phi", shift left] \ar[r, "\Psi"', shift right] & C_3 \ar[d] \\
C_1 \ar[r, "\Phi", shift left] \ar[r, "\Psi"', shift right] \ar[ru, "\eta"] & C_2 \ar[r, "\Phi", shift left] \ar[r, "\Psi"', shift right] \ar[ru, "\eta"] & C_3
\end{tikzcd}
\end{center}
be maps satisfying
\begin{equation*}
[\partial,\Phi] = 0 \qquad [\partial,\Psi] = 0 \qquad \Phi - \Psi = [\partial, \eta].
\end{equation*}
Then
\begin{equation*}
\Phi_{32}\Phi_{21} - \Psi_{32}\Psi_{21} = [\partial, \eta_{32}\Phi_{21} + \Psi_{32}\eta_{21}] = [\partial, \eta_{32}\Psi_{21} + \Phi_{32}\eta_{21}].
\end{equation*}
Moreover, $[\partial, \eta_{32}\eta_{21}] = 0$ if and only if
\begin{equation*}
\eta_{32}\Phi_{21} + \Psi_{32}\eta_{21} = \eta_{32}\Psi_{21} + \Phi_{32}\eta_{21}.
\end{equation*}
In that case, we call this map the composition of homotopies (of $\eta_{32}$ and $\eta_{21}$).
\end{lemma}

The composition of homotopies of $\textnormal{H}(V_M,C_M)$ is well-defined (and is equal to either side of the above formula). This composition actually defines a structure of Lie category on $\textnormal{Hom}(V_M,C_M)$,\footnote{More precisely, $\textnormal{Hom}(V_M,C_M)$ is a bundle of Lie monoids.} and $\textnormal{H}(V_M,C_M)$ form precisely the invertible elements of this Lie category. Notice that the structure of $\textnormal{H}(V_M,C_M)$ as a bundle of Lie groups is independent of the fat groupoid $\widehat V_G$.

\begin{definition}
Let
\begin{equation*}
C_M \xra{\partial} V_M
\end{equation*}
be a $2$-term complex. We call the bundle of Lie groups
\begin{equation*}
\textnormal{H}(V_M,C_M) = \{h_x: (V_M)_x \ra (C_M)_x \mid 1 + \partial h_x \textnormal{ is invertible}\}
\end{equation*}
the \textit{bundle of invertible homotopies} of the $2$-term complex $C_M \ra V_M$.
\end{definition}

\begin{remark}
    Although we write elements of $\textnormal{H}(V_M,C_M)$ as $h_x$, its product shows that a good way to think of the elements $h_x$ is as $1+h_x$. This shows that the product of $\textnormal{H}(V_M,C_M)$ is determined by a structure of Lie category on $\textnormal{Hom}(V_M,C_M)$:
    \begin{equation*}
        h_x \cdot h_x' \coloneq h_x \partial h_x'.
    \end{equation*}
    This product defines a Lie category, and we will elaborate on its structure in Section \ref{sec: the fat category}.
\end{remark}

That the elements of $\textnormal{H}(V_M,C_M)$ give rise to isomorphisms of the complex $C_M \ra V_M$ can be interpreted as a cochain complex representation of $\textnormal{H}(V_M,C_M)$ on $C_M \ra V_M$:
\begin{equation*}
\textnormal{H}(V_M,C_M) \ra \textnormal{GL}(C_M, V_M)_M \qquad h_x \mapsto (1+ h_x\partial, 1+\partial h_x).
\end{equation*}
This canonical representation of $\textnormal{H}(V_M,C_M)$ plays a central role in this work.

Before we move on, notice that elements $h_x \in \textnormal{Hom}(V_M,C_M)$ formally have an inverse:
\begin{equation*}
(1+\partial h_x)^{-1} = \textstyle\sum_{k \ge 0} (-1)^k (\partial h_x)^k.
\end{equation*}
In particular, if $\partial h_x$ is nilpotent, i.e. $(\partial h_x)^k = 0$ for some $k$, then $1+\partial h_x$ has a smooth inverse.

\begin{remark}\label{rema: perturbation theory}
With cohomological perturbation theory in mind (see, for example, \cite{PerturbationMarius,LiblandMeinrenkenVanEst,MariaMeinrenkenVanEst} and references therein), we can think of $\partial$ as a perturbation of the zero differential. The identity map
\begin{center}
\begin{tikzcd}
(C_M)_x \ar[d, "="]& \ar[ld, dotted, "h_x"'] (V_M)_x \ar[d, "="] \\ (C_M)_x & (V_M)_x
\end{tikzcd}
\end{center}
can then be perturbed, using $h_x \in \textnormal{H}(V_M,C_M)$, to an isomorphism
\begin{center}
\begin{tikzcd}
(C_M)_x \ar[d, "\varphi_x"] \ar[r, "\partial"]& \ar[ld, dotted] (V_M)_x \ar[d, "\varphi_x"] \\ (C_M)_x \ar[r, "\partial"] & (V_M)_x
\end{tikzcd}
\end{center}
In fact, we can view $\textnormal{H}(V_M,C_M)$ as a level set (of a section of $\textnormal{Hom}(C_M,V_M) \ra M$; the differential) of the canonical projection
\begin{equation*}
\textnormal{Pert}(C_M,V_M) \ra \textnormal{Hom}(C_M,V_M)
\end{equation*}
where
\begin{equation*}
\textnormal{Pert}(C_M,V_M) = \{(\partial,h) \in \textnormal{Hom}(C_M,V_M) \oplus \textnormal{Hom}(V_M,C_M) \mid 1+\partial h \textnormal{ is invertible}\}.
\end{equation*}
We can also view $\textnormal{Pert}(C_M,V_M)$ as a bundle of Lie groups, with the same product as we defined for $\textnormal{H}(C_M,V_M)$. This groupoid structure becomes very relevant in Section \ref{sec: PB-groupoids}.

Heuristically, a ``small'' perturbation $\partial$ comes with fewer restrictions on $h_x$, while a ``big'' perturbation imposes more restrictions on $h_x$. For example, if $\partial = 0$, then $\textnormal{H}(V_M,C_M) = \textnormal{Hom}(V_M,C_M)$. But if $\partial$ is of full rank, then $h_x \in \textnormal{H}(V_M,C_M)$ is determined by isomorphisms $\varphi_x$ and left or right inverses of $\partial$.
\end{remark}

The last observation of the above remark can be made much more precise as follows.

\begin{proposition}\label{prop: bundle of Lie groups is locally trivial}
Let $\bbR^k \xra{d} \bbR^\ell$ be a linear map. The Lie group
\begin{equation*}
\textnormal{H}(\bbR^\ell, \bbR^k) = \{h: \bbR^\ell \ra \bbR^k \mid 1 + dh \textnormal{ is invertible}\}
\end{equation*}
embeds into $\textnormal{GL}(k+\ell)$ via
\begin{equation*}
\textnormal{H}(\bbR^\ell, \bbR^k) \ra \textnormal{GL}(k+\ell) \qquad h \mapsto \begin{pmatrix} 1 & h \\ 0 & 1+dh \end{pmatrix}.
\end{equation*}
In fact, if $r = \rank d$, then $\textnormal{H}(\bbR^\ell, \bbR^k)$ is isomorphic to the semidirect product associated to the action of $\textnormal{GL}(r)$ on upper diagonal blockmatrices\footnote{The non-trivial blocks are of the form $(k-r) \times r$, $(k-r) \times (\ell-r)$ and $r \times (\ell-r)$.} given by
\begin{equation*}
g \cdot
\begin{pmatrix}
1 & h_1 & h_2 \\ 0 & 1 & h_3 \\ 0 & 0 & 1
\end{pmatrix}
=
\begin{pmatrix}
1 & h_1g^{-1} & h_2 \\ 0 & 1 & gh_3 \\ 0 & 0 & 1
\end{pmatrix}.
\end{equation*}
In particular, if $C_M \xra{\partial} V_M$ is a $2$-term complex, then $\partial$ is regular if and only if $\textnormal{H}(V_M,C_M)$ is a Lie group bundle.\footnote{A Lie group bundle is a locally trivial bundle of Lie groups.}
\end{proposition}
\begin{proof}
The first statement is readily verified. To prove the second statement, we conjugate $d$ so that $d: \bbR^r \oplus \bbR^{k-r} \ra \bbR^r \oplus \bbR^{\ell-r}$ is in the canonical form $(x,y) \mapsto (x,0)$.\footnote{If we conjugate $d$ to $AdB$, then the two models of $\textnormal{H}(\bbR^\ell,\bbR^k)$ are isomorphic via $h \mapsto B^{-1}hA^{-1}$.} The subgroup of $\textnormal{GL}(k+\ell)$ corresponding to $\textnormal{H}(\bbR^k,\bbR^\ell)$ consists now of invertible matrices of the form
\begin{equation*}
\begin{pmatrix}
1 & 0 & h_4 & h_3 \\ 0 & 1 & h_1 & h_2 \\ 0 & 0 & 1 + h_4 & h_3 \\ 0 & 0 & 0 & 1
\end{pmatrix}.
\end{equation*}
Notice that $h_1$, $h_2$ and $h_3$ can be taken arbitrary. The constraint on $h_4$ is that $1 + h_4$ is invertible. Now, the map that sends a matrix as above to $1+h_4$ is a Lie group homomorphism onto $\textnormal{GL}(r)$ that is a surjective submersion. Its kernel consist of elements of the form
\begin{equation*}
\begin{pmatrix}
1 & 0 & 0 & h_3 \\ 0 & 1 & h_1 & h_2 \\ 0 & 0 & 1 & h_3 \\ 0 & 0 & 0 & 1
\end{pmatrix},
\end{equation*}
and this group is isomorphic to the group of upper triangular blockmatrices (by eliminating, from such elements, the first row and column). A direct computation shows that the (conjugation) action that defines the resulting semidirect product is given as in the statement.

The last statement follows by the constant rank theorem: if $\partial$ is regular, then local frames adapted to $\partial$ trivialise $\textnormal{H}(V_M,C_M)$ locally.
\end{proof}

\begin{remark}\label{rema: invertible homotopy bundle is semidirect product non-canonically}
Notice that, if $\partial$ is regular, we can split (non-canonically)
\begin{equation*}
V_M \cong \textnormal{coker } \partial \oplus \im \partial \qquad C_M \cong \ker \partial \oplus \im \partial.
\end{equation*}
Therefore, there is a global non-canonical isomorphism between $\textnormal{H}(V_M,C_M)$ and a type of semidirect product as in the above statement. Indeed, we can form a semidirect product (of Lie group bundles) of the isotropy Lie group bundle of $\textnormal{GL}(\im \partial)$, and the Lie group bundle of upper triangular blockmatrices, where now
\begin{equation*}
h_1 \in \textnormal{Hom}(\im \partial, \ker \partial) \qquad h_2 \in \textnormal{Hom}(\textnormal{coker } \partial, \ker \partial) \qquad h_3 \in \textnormal{Hom}(\textnormal{coker } \partial, \im \partial).
\end{equation*}

In general, given a $2$-term complex $C_M \ra V_M$, we can represent elements of $\textnormal{H}(V_M,C_M) \ni h_x$ as blockmatrices
\begin{equation*}
\begin{pmatrix}
1 & h_x \\ 0 & 1 + \partial h_x
\end{pmatrix} \in \textnormal{GL}(C_M \oplus V_M).
\end{equation*}
Another way is to use instead
\begin{equation*}
\begin{pmatrix}
1 + h_x \partial & h_x \\ 0 & 1
\end{pmatrix} \in \textnormal{GL}(C_M \oplus V_M)
\end{equation*}
which is related to the previous matrix as follows:
\begin{equation*}
\begin{pmatrix}
1 + h_x \partial & h_x \\ 0 & 1
\end{pmatrix} = \begin{pmatrix}
1 & h_x \\ 0 & 1 + \partial h_x
\end{pmatrix}
\begin{pmatrix}
1 + h_x\partial & h_x \partial h_x (1 + \partial h_x)^{-1} \\ 0 & (1 + \partial h_x)^{-1}
\end{pmatrix}.
\end{equation*}
Similarly, we can depict elements of $\textnormal{Pert}(C_M,V_M)$, but such elements ``carry a differential''.
\end{remark}

Under the identification of $\textnormal{Hom}(V_M,C_M)$ with $\textnormal{Hom}(C_M^*,V_M^*)$ (by taking the dual), the structure of $\textnormal{H}(V_M,C_M)$ takes its ``opposite'' form:
\begin{equation*}
h_x \cdot_\textnormal{op} h_x' \coloneq h_x' \cdot h_x = h_x + h_x'+ h_x' \partial h_x \qquad (h_x)_\textnormal{op}^{-1} \coloneq h_x^{-1} = -h_x(1+ \partial h_x)^{-1}.
\end{equation*}
As happens for groups, bundles of groups are isomorphic to their opposites via the inverse:

\begin{proposition}\label{prop: dual of invertible homotopy bundle}
Let $C_M \ra V_M$ be a $2$-term complex. The canonical map
\begin{equation*}
\textnormal{H}(V_M,C_M) \ra \textnormal{H}(C_M^*,V_M^*) \qquad h_x \mapsto (h_x^{-1})^*
\end{equation*}
is an isomorphism of bundles of Lie groups.
\end{proposition}

In the coming section, Section \ref{sec: canonical representations of the fat groupoid}, we define the well-known canonical cochain complex representation of the fat groupoid $\widehat V_G$ of a VB-groupoid $V_G$. This cochain complex representation is compatible with the exact sequence of Lie groupoids\footnote{The short exact sequence can be seen as a restriction of the short exact sequence of Lie categories
\begin{center}
\begin{tikzcd}[ampersand replacement = \&]
1 \ar[r] \& \textnormal{Hom}(V_M,C_M) \ar[r] \& \widehat V_G^\textnormal{cat} \ar[r] \& G \ar[r] \& 1.
\end{tikzcd}
\end{center}
}
\begin{center}
\begin{tikzcd}
1 \ar[r] & \textnormal{H}(V_M,C_M) \ar[r] & \widehat V_G \ar[r] & G \ar[r] & 1
\end{tikzcd}
\end{center}
in two ways. One is that the inclusion $\textnormal{H}(V_M,C_M) \hookrightarrow \widehat V_G$ is equivariant with respect to the canonical cochain complex representations. Another is that the conjugation action of $\widehat V_G$ on $\textnormal{H}(V_M,C_M)$ is induced by the representation on $\textnormal{Hom}(V_M,C_M)$. This is precisely what encodes the fat extension associated to $V_G$ (see the next Section \ref{sec: fat extensions 1} and Section \ref{sec: fat extensions}).

\subsection{Canonical representations of the fat groupoid}\label{sec: canonical representations of the fat groupoid}
The action of $V_G$ on itself induces an action of $\widehat V_G$ on $V_G$: given $H_g \in \widehat V_G$ and $v \in (V_G)_{g'}$ with $(g,g') \in G^{(2)}$,
\begin{equation*}
H_g \cdot v \coloneq H_g(\bft_g v) \cdot v.
\end{equation*}
In turn, this induces actions of $\widehat V_G$ on $C_M$ and $V_M$: given $H_g \in \widehat V_G$, $c \in (C_M)_{\bfs g}$ and $v \in (V_M)_{\bfs g}$,
\begin{equation}\label{eq: representations of the fat groupoid actions}
H_g \cdot c \coloneq r_{g^{-1}}(H_g(\partial c) \cdot c) \qquad H_g \cdot v \coloneq \bft_g (H_g v).
\end{equation}

\begin{remark}\label{rema: linear actions induce actions of fat groupoid}
More generally, a linear action of $V_G$ on a vector bundle $V_P \ra P$ induces an action of $\widehat V_G$ on $V_P$. By linear, we mean that the action map
\begin{equation*}
V_G \ltimes V_P \ra V_P
\end{equation*}
is linear (over an action of $G$ on $P$). A more thorough discussion (including a discussion on VB-Morita equivalences in terms of bibundles, the notion of Morita equivalence in terms of ruths, and the notion of Morita equivalence in terms of fat extensions) will appear in \cite{Homotopyoperators}.
\end{remark}

It is readily verified that the $2$-term complex
\begin{equation*}
\partial: C_M \xra{\bft} V_M
\end{equation*}
is equivariant with respect to the actions, so we can see $C_M \ra V_M$ as a cochain complex representation of $\widehat V_G$. In turn, there is a dual cochain complex representation on\footnote{Recall that, with our convention, $V_M^*$ sits in degree $0$ and $C_M^*$ sits in degree $1$.}
\begin{equation*}
\partial^*: V_M^* \xra{\bft^*} C_M^*
\end{equation*}
and the representations can also be put together to give a representation on $\textnormal{Hom}(V_M,C_M) = V_M^* \otimes C_M$. To be clear, $\widehat V_G$ acts on $\textnormal{Hom}(V_M,C_M)$ via the representation of $C_M$ and via the representation of $V_M$. Denoting these actions by $\cdot_{C_M}$ and $\cdot_{V_M}$, respectively, the conjugation action is given by
\begin{equation*}
H_g \cdot h_{\bfs g} = H_g \cdot_{V_M} h_{\bfs g} \cdot_{C_M} H_g^{-1}.
\end{equation*}
When using this action, we use the notation from the right hand side to avoid confusion. This representation descends to an action on $\textnormal{H}(V_M,C_M) \subset \textnormal{Hom}(V_M,C_M)$.\footnote{For $h_{\bfs g} \in \textnormal{H}(V_M,C_M)$ and $H_g \in \widehat V_G$, the map $\bft_g(H_g \cdot h_{\bfs g})$ is given by $(\bft_g H_g)(1 + \partial h_{\bfs g})$.}

\begin{proposition}\label{prop: representation H(V_M,C_M) is action by conjugation}
The inclusion map
\begin{equation*}
\textnormal{H}(V_M,C_M) \hookrightarrow \widehat V_G \qquad h_x \mapsto 1+h_x
\end{equation*}
is equivariant with respect to the canonical cochain complex representations on $C_M \ra V_M$. Moreover, the conjugation action of $\widehat V_G$ on $\textnormal{H}(V_M,C_M)$, induced by the representation on $\textnormal{Hom}(V_M,C_M)$, coincides with the conjugation action of $\widehat V_G$ on $\textnormal{H}(V_M,C_M)$ by viewing $\textnormal{H}(V_M,C_M)$ as a normal subgroupoid of $\widehat V_G$.
\end{proposition}
\begin{proof}
The first statement is readily verified. For the second statement we do a direct computation:
\begin{align*}
(H_g \cdot (1+h_{\bfs g}) \cdot H_g^{-1})v &= H_g(1+\partial h_{\bfs g})(H_g^{-1} \cdot v) \cdot (H_g^{-1} \cdot v + h_{\bfs g}(H_g^{-1} \cdot v)) \cdot H_g^{-1}(v) \\
&= H_g(1+\partial h_{\bfs g})(H_g^{-1} \cdot v) \cdot (H_g^{-1}(v) + h_{\bfs g}(H_g^{-1} \cdot v) \cdot 0_{g^{-1}}) \\
&= v + H_g(\partial h_{\bfs g}(H_g^{-1} \cdot v)) \cdot h_{\bfs g}(H_g^{-1} \cdot v) \cdot 0_{g^{-1}} \\
&= (1 + H_g \cdot_{V_M} h_{\bfs g} \cdot_{C_M} H_g^{-1})v.
\end{align*}
This proves the statement.
\end{proof}

\subsection{Fat extensions}\label{sec: fat extensions 1}

As we discussed now the structure behind the fat groupoid of a VB-groupoid that makes it into a \textit{fat extension}, we already put here the abstract definition:

\begin{definition}\label{def: fat extensions}
Let $G \rra M$ be a Lie groupoid, and consider a $2$-term complex
\begin{equation*}
C_M \xra{\partial} V_M.
\end{equation*}
A \textit{fat extension} of $G$, with underlying complex $C_M \ra V_M$, consists of a Lie groupoid
\begin{equation*}
F_G \rra M,
\end{equation*}
that comes with a cochain complex representation on $C_M \ra V_M$, and a short exact sequence of Lie groupoids (over $M$)
\begin{center}
\begin{tikzcd}
1 \ar[r] & \textnormal{H}(V_M,C_M) \ar[r] & F_G \ar[r] & G \ar[r] & 1.
\end{tikzcd}
\end{center}
Moreover, the cochain complex representation of $F_G$ on $C_M \ra V_M$ restricts to the canonical cochain complex representation of $\textnormal{H}(V_M,C_M)$, and the two natural conjugation actions of $F_G$ on $\textnormal{H}(V_M,C_M)$ agree. That is, denoting by $\cdot_{V_M}$ and $\cdot_{C_M}$ the actions of $F_G$ on $\textnormal{H}(V_M,C_M)$ induced by the representations, we have that, for all $H_g \in F_G$ and $h_{\bfs g} \in \textnormal{H}(V_M,C_M)$,
\begin{equation*}
H_g \cdot h_{\bfs g} \cdot H_g^{-1} = H_g \cdot_{V_M} h_{\bfs g} \cdot_{C_M} H_g^{-1}.
\end{equation*}
\end{definition}

We will simply write $F_G$ for a fat extension. The underlying groupoid $F_G$ is called the fat groupoid of the fat extension.

To understand the structure and examples of fat extensions, it is helpful to keep the reference with VB-groupoids, and so, for now, we study the above structure in terms of VB-groupoids.

\subsection{Fat category extensions}\label{sec: the fat category}

Before we move to the examples, we observe that $\widehat V_G^\textnormal{cat}$ can be recovered from $\widehat V_G$ and its cochain complex representation. In fact, this leads to yet another equivalent way to encode the structure of a fat extension in terms of a \textit{fat category extension}. This defines the intrinsic notion of weak representation as defined in \cite{Wolbert}. 

To start, observe that $\textnormal{Hom}(\bfs^*V_M, \bft^*C_M) \rra M$ is a Lie category with $\bfs h_g = \bfs g$, $\bft h_g = \bft g$ and
\begin{equation*}
h_g \cdot h_{g'} \coloneq h_g \partial h_{g'}.
\end{equation*}
This is not a ``VB-category'' over $G$, because this product is not linear but bilinear. The fat groupoid $\widehat V_G$ acts on the vector bundle $\textnormal{Hom}(\bfs^*V_M, \bft^*C_M)$, both on the left and on the right, via the two representations on $C_M$ and $V_M$. We denote these two actions by $\cdot_{C_M}$ and $\cdot_{V_M}$ for clarity. We then get a groupoid structure on
\begin{equation*}
\textnormal{Hom}(\bfs^*V_M, \bft^*C_M) \times_G \widehat V_G \rra M
\end{equation*}
by setting $\bfs(h_g, H_g) = \bfs g$, $\bft(h_g, H_g) = \bft g$ and
\begin{equation*}
(h_g, H_g) \cdot (h_{g'}, H_{g'}) \coloneq (h_g \cdot h_{g'} + H_g \cdot_{C_M} h_{g'} + h_g \cdot_{V_M} H_{g'}, H_g \cdot H_{g'}).
\end{equation*}
The connection with $\widehat V_G^\textnormal{cat}$ becomes clear when associating to $(h_g, H_g)$ the map
\begin{equation*}
r_gh_g + H_g: (V_M)_{\bfs g} \ra (V_G)_g \qquad v \mapsto r_gh_gv + H_gv
\end{equation*}
which satisfies $\bfs_g(r_gh_g + H_g) = 1$.

\begin{proposition}\label{prop: fat category from fat groupoid}
The map
\begin{equation*}
\textnormal{Hom}(\bfs^*V_M, \bft^*C_M) \times_G \widehat V_G \ra \widehat V_G^\textnormal{cat} \qquad (h_g, H_g) \mapsto r_gh_g + H_g
\end{equation*}
is a surjective submersion of Lie categories. Moreover, the embedding
\begin{equation*}
\textnormal{H}(V_M,C_M) \hookrightarrow \textnormal{Hom}(\bfs^*V_M, \bft^*C_M) \times_G \widehat V_G \qquad h_x \mapsto (-h_x,1+h_x)
\end{equation*}
defines free and proper left and right actions whose quotients are both given by $\widehat V_G^\textnormal{cat}$, and the above map defines both quotient maps.
\end{proposition}
\begin{proof}
To see that the first map is a surjective submersion follows, for example, by constructing a local section from a local section $H$ of $\widehat V_G \ra G$:
\begin{equation*}
\widehat V_G^\textnormal{cat} \ra \textnormal{Hom}(\bfs^*V_M, \bft^*C_M) \times_G \widehat V_G \qquad H_g' \mapsto (r_{g^{-1}}(H_g' - H_g), H_g).
\end{equation*}
The second statement follows by realising that, whenever
\begin{equation*}
r_gh_g + H_g = r_gh_g' + H_g'
\end{equation*}
there is a unique element $h_{\bft g} \in \textnormal{H}(V_M,C_M)$ such that 
\begin{equation*}
(-h_{\bft g}, 1+h_{\bft g}) \cdot (h_g,H_g) = (h_g',H_g'),
\end{equation*}
namely, the unique element $h_{\bft g}$ such that $(1+h_{\bft g}) \cdot H_g = H_g'$ (and similarly for the right action). This proves the statement.
\end{proof}

\begin{remark}\label{rema: exact sequence fat category}
The above statement describes an ``exact sequence of Lie categories''
\begin{center}
\begin{tikzcd}
1 \ar[r] & \textnormal{H}(V_M,C_M) \ar[r] & \textnormal{Hom}(\bfs^*V_M, \bft^*C_M) \times_G \widehat V_G \ar[r] & \widehat V_G^\textnormal{cat} \ar[r] & 1.
\end{tikzcd}
\end{center}
Namely, while it is true that $\textnormal{H}(V_M,C_M)$ is the kernel, in the sense that it comprises the elements mapping to units in $\widehat V_G$, it is the kernel in a more general, more appropriate sense. Being more precise here leads to the above statement: whenever two elements coincide in the image of the map
\begin{equation*}
\textnormal{Hom}(\bfs^*V_M, \bft^*C_M) \times_G \widehat V_G \ra \widehat V_G^\textnormal{cat}
\end{equation*}
they are related under the equivalence obtained from the inclusion by $\textnormal{H}(V_M,C_M)$. That is, the action map from $\textnormal{H}(V_M,C_M) \times_M (\textnormal{Hom}(\bfs^*V_M, \bft^*C_M) \times_G \widehat V_G)$ to
\begin{equation*}
 (\textnormal{Hom}(\bfs^*V_M, \bft^*C_M) \times_G \widehat V_G) \times_{\widehat V_G^\textnormal{cat}} (\textnormal{Hom}(\bfs^*V_M, \bft^*C_M) \times_G \widehat V_G)
\end{equation*}
is an isomorphism (and similarly for the right action), meaning that the action is free and proper.
\end{remark}

\begin{remark}\label{rema: fat category from fat groupoid}
From the previous statement, or rather its proof, we saw that $\widehat V_G^\textnormal{cat}$ is a quotient with respect to the (left) action of $\textnormal{H}(V_M,C_M)$ on $\textnormal{Hom}(\bfs^*V_M, \bft^*C_M) \times_G \widehat V_G$ given by
\begin{equation*}
h_{\bft g} \cdot (h_g, H_g) \coloneq (h_g - h_{\bft g} \cdot H_g, (1 + h_{\bft g}) \cdot H_g).
\end{equation*}
\end{remark}

\begin{remark}\label{rema: cochain complex representation of fat extension}
The Lie category $\widehat V_G$ also has $C_M \ra V_M$ as a cochain complex representation. By representation we mean here a smooth functor to the endomorphism Lie category
\begin{equation*}
\textnormal{End } E_M \coloneq \{E_x \ra E_y \mid x,y \in M\}
\end{equation*}
where $E_M$ is a vector bundle. It is defined exactly as in \eqref{eq: representations of the fat groupoid actions}.
\end{remark}

The above discussion shows that the fat category of a fat extension can be defined as follows:

\begin{definition}\label{def: abstract definition fat category}
Let $F_G$ be a fat extension. Consider the Lie category
\begin{equation*}
\textnormal{Hom}(\bfs^*V_M, \bft^*C_M) \times_G F_G \rra M
\end{equation*}
The fat category $F_G^\textnormal{cat} \rra M$ (of $F_G$) is defined as the quotient of
\begin{equation*}
\textnormal{Hom}(\bfs^*V_M, \bft^*C_M) \times_G F_G
\end{equation*}
by the action induced by the inclusion\footnote{Notice that here we identify $h_x \in \textnormal{H}(V_M,C_M)$ with its image in $F_G$. If $F_G = \widehat V_G$, then this element equals $1+h_x$.}
\begin{equation*}
\textnormal{H}(V_M,C_M) \hookrightarrow \textnormal{Hom}(\bfs^*V_M, \bft^*C_M) \times_G F_G \qquad h_x \mapsto (-h_x,h_x).
\end{equation*}
\end{definition}

Now, apart from the cochain complex representation, the fat category $\widehat V_G^\textnormal{cat}$ comes with a type of ``comparison'' map 
\begin{equation*}
h: \widehat V_G^\textnormal{cat} \times_G \widehat V_G^\textnormal{cat} \ra \textnormal{Hom}(\bfs^*V_M,\bft^*C_M) \qquad (H_g,H_g') \mapsto r_{g^{-1}}(H_g' - H_g)
\end{equation*}
that is equivariant for both $\widehat V_G^\textnormal{cat}$ actions (considering the diagonal action on $\widehat V_G^\textnormal{cat} \times_G \widehat V_G^\textnormal{cat}$). It can be interpreted as a homotopy, for, given $H_g,H_g' \in F_G^\textnormal{cat}$, $c \in (C_M)_{\bfs g}$ and $v \in (V_M)_{\bfs g}$, we have
\begin{equation*}
    h(H_g,H_g') \partial c = H_g' \cdot c - H_g \cdot c \qquad \partial h(H_g,H_g') v = H_g' \cdot v - H_g \cdot v.
\end{equation*}
Moreover, the map $h$ restricts to the ``canonical'' comparison map of $\textnormal{Hom}(V_M,C_M)$. In particular, it restricts to the zero map on the diagonal embedding 
\begin{equation*}
    \widehat V_G^\textnormal{cat} \hookrightarrow \widehat V_G^\textnormal{cat} \times_G \widehat V_G^\textnormal{cat}.
\end{equation*}
This defines the structure of a fat category extension:
\begin{definition}\label{def: fat category extension}
Let $G \rra M$ be a Lie groupoid, and consider a $2$-term complex
\begin{equation*}
C_M \xra{\partial} V_M.
\end{equation*}
A \textit{fat category extension} of $G$, with underlying complex $C_M \ra V_M$, consists of a Lie category
\begin{equation*}
F_G^\textnormal{cat} \rra M,
\end{equation*}
that comes with a cochain complex representation on $C_M \ra V_M$, a short exact sequence of Lie categories (over $M$)\footnote{Recall that this means that $\textnormal{Hom}(V_M,C_M)$ defines a free and proper action on $F_G^\textnormal{cat}$.}
\begin{center}
\begin{tikzcd}
1 \ar[r] & \textnormal{Hom}(V_M,C_M) \ar[r] & F_G^\textnormal{cat} \ar[r] & G \ar[r] & 1
\end{tikzcd}
\end{center}
and a smooth map 
\begin{equation*}
    h: F_G^\textnormal{cat} \times_G F_G^\textnormal{cat} \ra \textnormal{Hom}(\bfs^*V_M,\bft^*C_M).
\end{equation*}
The cochain complex representation of $F_G^\textnormal{cat}$ on $C_M \ra V_M$ restricts to the canonical cochain complex representation of $\textnormal{Hom}(V_M,C_M)$. Moreover, $h$ restricts to the canonical comparison map of $\textnormal{Hom}(V_M,C_M)$ (taking the difference). Lastly, $h$ is multiplicative, meaning that, for all $H_g,H_g' \in F_G^\textnormal{cat}$, $c \in (C_M)_{\bfs g}$ and $v \in (V_M)_{\bfs g}$,
\begin{equation*}
    h(H_g,H_g') \partial c = H_g' \cdot c - H_g \cdot c \qquad \partial h(H_g,H_g') v = H_g' \cdot v - H_g \cdot v
\end{equation*}
and it is equivariant with respect to the canonical left/right (diagonal) actions of $F_G^\textnormal{cat}$.
\end{definition}

\begin{remark}\label{rema: splitting of fat category extension}
    The above definition can be seen to define the intrinsic object behind a ``weak representation'' of \cite{Wolbert}. Namely, as was already observed in \cite{VBgroupoidsruths}, a splitting $\Sigma$ of a VB-groupoid $V_G$ can be seen as a section
    \begin{equation*}
        \Sigma: G \ra \widehat V_G^\textnormal{cat}
    \end{equation*}
    of the projection $V_G^\textnormal{cat} \ra G$ (that is not necessarily a functor). Therefore, such a map should be seen as ``splitting'' a fat category extension. Decomposing $\widehat V_G^\textnormal{cat}$ using $\Sigma$, this leads precisely to the axioms of a weak representation as in \cite{Wolbert}. There, it is proved that the category of weak representations is equivalent to the category of (split) $2$-term ruths and to the category of VB-groupoids.  
\end{remark}

\begin{remark}
    As we will see in Section \ref{sec: functorial aspects}, the comparison map
    \begin{equation*}
        h: F_G^\textnormal{cat} \times_G F_G^\textnormal{cat} \ra \textnormal{Hom}(\bft^*V_M,\bfs^*C_M)
    \end{equation*}
    from above can be seen as the ``identity map'' of the fat category extensions $F_G^\textnormal{cat}$.  
\end{remark}

From the discussion in terms of VB-groupoids it is already clear that we can pass from a fat category extension $F_G^\textnormal{cat}$ to $F_G$ by taking the invertible elements of $F_G^\textnormal{cat}$. Indeed, the properties of the map $h$ imply that, for all $H_g \in F_G$ and $h_{\bfs g} \in \textnormal{H}(V_M,C_M)$,   
\begin{equation*}
    H_g \cdot h_{\bfs g} \cdot H_g^{-1} = h(H_g \cdot h_{\bfs g} \cdot H_g^{-1}, 0_{\bft g}) = H_g \cdot_{C_M} h_{\bfs g} \cdot_{V_M} H_g^{-1}.
\end{equation*}
So, to end this section, we conclude:

\begin{proposition}\label{prop: fat category extension correspondence}
    The assignment
    \begin{equation*}
    \{\textnormal{Fat category extensions}\} \ra \{\textnormal{Fat extensions}\} \qquad F_G^\textnormal{cat} \mapsto F_G
    \end{equation*}
    sets a one-to-one correspondence, and the assignment
    \begin{equation*}
    \{\textnormal{Fat extensions}\} \ra \{\textnormal{Fat category extensions}\} \qquad F_G \mapsto F_G^\textnormal{cat}
    \end{equation*}
    is inverse to it (up to canonical isomorphisms).
\end{proposition}

We now move to a more detailed discussion on examples of fat extensions.

\subsection{Examples of fat extensions}\label{sec: examples of fat extensions}

In this section we describe, for well-known constructions of VB-groupoids, the associated fat extensions
\begin{center}
\begin{tikzcd}
1 \ar[r] & \textnormal{H}(V_M,C_M) \ar[r] & \widehat V_G \ar[r]& G \ar[r] & 1.
\end{tikzcd}
\end{center}
Naturally, we will see how to define the constructions on an abstract level for fat extensions.

To start, recall that we can view $2$-term cochain complexes as VB-groupoids (via the Dold-Kan correspondence). The fat extension associated to this VB-groupoid keeps track of the cochain complex itself, but it also comes with a fat groupoid, namely, the bundle of invertible homotopies:

\begin{example}\label{exa: bundle of invertible homotopies as fat extension}
Notice that $\textnormal{H}(V_M,C_M)$, together with its canonical cochain complex representation $C_M \ra V_M$, itself comes with a short exact sequence:
\begin{center}
\begin{tikzcd}
1 \ar[r] & \textnormal{H}(V_M,C_M) \ar[r] & \textnormal{H}(V_M,C_M) \ar[r] & 1_M \ar[r] & 1,
\end{tikzcd}
\end{center}
where $1_M$ denotes the unit groupoid of $M$. This exact sequence appears when considering
\begin{equation*}
C_M \ltimes V_M \rra V_M
\end{equation*}
as a VB-groupoid over $1_M$; here, $C_M$ acts on $V_M$ via the differential:
\begin{equation*}
c \cdot v = \partial c + v.
\end{equation*}
Its underlying $2$-term complex is $C_M \ra V_M$, and its fat groupoid is $\textnormal{H}(V_M,C_M)$. It is now readily verified that the associated exact sequence is the one written above.
\end{example}

Given a single vector bundle $E_M$, recall that there is a slightly different fat groupoid we can associate to it: the general linear groupoid (see \ref{exa: general linear}).
\begin{example}\label{exa: general linear fat extension}
In Example \ref{exa: general linear} we already mentioned that $\textnormal{GL } E_M$ is the fat groupoid of the pair groupoid $E_M \times E_M \rra E_M$. In this case,
\begin{equation*}
\textnormal{H}(E_M,E_M) = \textnormal{Aut } E_M = \{\Psi_x: (E_M)_x \xra{\sim} (E_M)_x\}
\end{equation*}
and the representation on $E_M \xra{=} E_M$ is the canonical one: given $\Psi_{y,x} \in \textnormal{GL } E_M$ and $e \in (E_M)_x$,
\begin{equation*}
\Psi_{y,x} \cdot e \coloneq \Psi_{y,x}e.
\end{equation*}
\end{example}

\begin{remark}\label{rema: differential is map of fat extensions GL groupoids}
As mentioned in Remark \ref{rema: GL of a cochain complex representation}, a differential $C_M \ra V_M$ defines the map
\begin{equation*}
[\cdot, \partial]: \textnormal{GL}(C_M,V_M) = \textnormal{GL } C_M \times_{M \times M} \textnormal{GL } V_M \ra \textnormal{Hom}(\pr_1^*C_M,\pr_2^*V_M).
\end{equation*}
We can think of this map as a type of comparison map. In this case, the zeros of the above map are exactly the cochain maps between $C_M \ra V_M$ and itself. In Section \ref{sec: functorial aspects}, we will see that we can think of this as a map of fat extensions between $\textnormal{GL } C_M$ and $\textnormal{GL } V_M$. On the level of VB-groupoids, this map corresponds to the VB-groupoid map
\begin{equation*}
\partial \times \partial: C_M \times C_M \ra V_M \times V_M.
\end{equation*}
\end{remark}

We now discuss the continuation of Example \ref{exa: subgroupoid of jet}. More precisely, we will come to see what a \textit{fat subextension} and a \textit{fat quotient extension} should be abstractly. The following bundles of Lie groups will play an important role.
\begin{definition}\label{def: restricted bundle of invertible homotopies}
Let $C_M \ra V_M$ be a $2$-term complex, and $C_N \ra V_N$ a subcomplex. The \textit{restricted bundle of invertible homotopies} is the subbundle of of Lie groups of $\textnormal{H}(V_M,C_M)$ given by
\begin{equation*}
\textnormal{H}(V_M,C_M)_{(V_N,C_N)} \coloneq \{h_y \in \textnormal{H}(V_M,C_M) \mid y \in N, h_y|_{(V_N)_y} \subset (C_N)_y\}.
\end{equation*}
\end{definition}

\begin{example}\label{exa: fat extension subgroupoid}
Recall that, given a Lie groupoid $G$ and a subgroupoid $H$, $J^1G$ has the subgroupoid
\begin{equation*}
J^1_H G = \{j^1_y \sigma \in J^1G \mid y \in N, \im \sigma|_N \subset H\}
\end{equation*}
which played a role in the relation between $J^1H$, $J^1G$ and $J^1_\calN(G,H)$ (see Example \ref{exa: subgroupoid of jet}). In general, given a VB-groupoid $V_G$ over $C_M \ra V_M$ and a VB-subgroupoid $V_H$ over $C_N \ra V_N$, $\widehat V_G$ has
\begin{equation*}
\widehat V_{G,H} \coloneq \{H_h \in \widehat V_G \mid h \in H, \im H_h|_{(V_N)_{\bfs h}} \subset (V_H)_h\}
\end{equation*}
as a subgroupoid. It fits into an extension
\begin{center}
\begin{tikzcd}
1 \ar[r] & \textnormal{H}(V_M,C_M)_{(V_N,C_N)} \ar[r] & \widehat V_{G,H} \ar[r] & H \ar[r] & 1
\end{tikzcd}
\end{center}
and $\widehat V_{G,H}$ comes with a cochain complex representation on
\begin{equation*}
C_M|_N \ra V_M|_N.
\end{equation*}
Moreover, the subcomplex $C_N \ra V_N$ is also a cochain complex representation of $\widehat V_{G,H}$, and both cochain complex representations restrict to $\textnormal{H}(V_M,C_M)_{(V_N,C_N)}$. So, while not classifying as a fat extension, $\widehat V_{G,H}$ carries at least a very similar structure. In turn, it induces two actual fat extensions: one recovers $\widehat V_H$ as a quotient:
\begin{center}
\begin{tikzcd}
1 \ar[r] & \textnormal{H}(V_M,C_M)_{(C_N,0_N)} \ar[r] & \widehat V_{G,H} \ar[r] & \widehat V_H \ar[r] & 1
\end{tikzcd}
\end{center}
and one recovers $\widehat{V_G/V_H}$\footnote{We wrote $V_G/V_H$ for $V_G|_H/V_H$.} as a quotient:
\begin{center}
\begin{tikzcd}
1 \ar[r] & \textnormal{H}(V_M,C_M)_{(V_M|_N,C_N)} \ar[r] & \widehat V_{G,H} \ar[r] & \widehat{V_G/V_H} \ar[r] & 1.
\end{tikzcd}
\end{center}
To be clear, also the cochain complex representations of $\widehat V_H$ and $\widehat{V_G/V_H}$ are induced by the cochain complex representations of $\widehat V_{G,H}$.
\end{example}

From the previous example we can infer how to think of fat subextensions and fat quotient extensions abstractly using the following intermediate definition.

\begin{definition}\label{def: restricted fat extension}
Let $F_G$ be a fat extension over $C_M \ra V_M$. Moreover, let $H \rra N$ be a Lie subgroupoid of $G$, and let $C_N \ra V_N$ be a subcomplex of $C_M \ra V_M$. A \textit{restricted fat extension} of $F_G$ consists of a Lie subgroupoid
\begin{equation*}
F_{G,H} \rra N
\end{equation*}
of $F_G$ that comes with a cochain complex representation on $C_M|_N \ra V_M|_N$, which further restricts to $C_N \ra V_N$, and a short exact sequence\footnote{Again, we wrote
\begin{equation*}
\textnormal{H}(V_M,C_M)_{(V_N,C_N)} \coloneq \{h_y \in \textnormal{H}(V_M,C_M) \mid y \in N, \im h_y|_{(V_N)_y} \subset (C_N)_y\}.
\end{equation*}}
\begin{center}
\begin{tikzcd}
1 \ar[r] & \textnormal{H}(V_M,C_M)_{(V_N,C_N)} \ar[r] & F_{G,H} \ar[r] & H \ar[r] & 1.
\end{tikzcd}
\end{center}
Moreover, the cochain complex representation of $F_{G,H}$ on $C_M|_N \ra V_M|_N$ restricts to the canonical cochain complex representation of $\textnormal{H}(V_M,C_M)_{(V_N,C_N)}$.
\end{definition}

We now define:

\begin{definition}\label{def: fat subextension and fat quotient}
Given a restricted fat extension $F_{G,H}$ of a fat extension $F_G$, we call
\begin{center}
\begin{tikzcd}
1 \ar[r] & \textnormal{H}(V_N,C_N) \ar[r] & F_G/\textnormal{H}(V_M,C_M)_{(V_N,0_N)} \ar[r] & H \ar[r] & 1
\end{tikzcd}
\end{center}
the \textit{fat subextension} of $F_{G,H}$, denoted $F_H$, and
\begin{center}
\begin{tikzcd}
1 \ar[r] & \textnormal{H}(V_M|_N/V_N,C_M|_N/C_N) \ar[r] & F_G/\textnormal{H}(V_M,C_M)_{(V_M|_N,C_N)} \ar[r] & H \ar[r] & 1
\end{tikzcd}
\end{center}
the \textit{fat quotient extension} of $F_{G,H}$, denoted $F_G/F_H$.
\end{definition}

\begin{remark}\label{rema: fat subextension and fat quotient extension as fat maps}
We can compare elements of $\widehat V_H$ and $\widehat V_G$, and we can think of this as a map
\begin{equation*}
\iota: \widehat V_H \times_G \widehat V_G\ra \textnormal{Hom}(\bfs^*V_N, \bft^*C_M) \qquad (H_{h,1}, H_{h,2}) \mapsto r_{h^{-1}}(H_{h,1} - H_{h,2}).
\end{equation*}
We can also compare elements of $\widehat V_G$ and $\widehat{V_G/V_H}$ in a similar way
\begin{equation*}
\pi: \widehat V_G \times_G \widehat{V_G/V_H} \ra \textnormal{Hom}(\bfs^*V_M, \bft^*C_M|_N/C_N).
\end{equation*}
Given $H_{h,2} \in \widehat V_{G,H}$, there is a unique $H_{h,1} \in \widehat V_H$ such that
\begin{equation*}
\iota(H_{h,1}, H_{h,2}) = 0,
\end{equation*}
and, similarly, given $H_{h,1} \in \widehat V_{G,H}$, there is a unique $H_{h,2} \in \widehat{V_G/V_H}$ such that
\begin{equation*}
\pi(H_{h,1}, H_{h,2}) = 0.
\end{equation*}
The map $\iota$ should be interpreted as the ``inclusion'' map and $\pi$ as the ``projection'' map. These are examples of maps of fat extensions that we will discuss in Section \ref{sec: functorial aspects}.
\end{remark}

In the the next example we explain how the fiber product of VB-groupoids looks like in terms of fat extensions. Particular examples of interest are the pullback and direct sum constructions, but we first discuss the more general case.

\begin{example}\label{exa: pullback}
Let
\begin{equation*}
G_1 \ra G \la G_2
\end{equation*}
be Lie groupoid maps that cleanly intersect, and let $V_{G_1}$ and $V_{G_2}$ be VB-groupoids over $G_1$ and $G_2$, respectively. Then we can form the fiber product of VB-groupoids (see \cite{VBalejandrohenriquematias}):
\begin{equation*}
V_{G_1} \times_G V_{G_2} \rra V_{M_1} \times_M V_{M_2}.
\end{equation*}
This is a VB-groupoid over $G_1 \times_G G_2 \rra M_1 \times_M M_2$ whose structure is defined ``componentwise''. To understand how the fat groupoid of this VB-groupoid looks like, notice that an element of it can be decomposed into four maps that we organise into a matrix as follows:
\begin{equation*}
\begin{pmatrix} H_{g_1} & h_{12} \\ h_{21} & H_{g_2} \end{pmatrix} = \begin{pmatrix} H_{g_1} & r_{g_1}h_{12} \\ r_{g_2}h_{21} & H_{g_2} \end{pmatrix}: (V_{M_1})_{\bfs g_1} \times (V_{M_2})_{\bfs g_2} \ra (V_{G_1})_{g_1} \times (V_{G_2})_{g_2}.
\end{equation*}
Here, $H_{g_1} \in \widehat V_{G_1}^\textnormal{cat}$, $H_{g_2} \in \widehat V_{G_2}^\textnormal{cat}$ and
\begin{equation*}
h_{12} \in \textnormal{Hom}((V_{M_2})_{\bfs g_2}, (C_{M_1})_{\bft g_1}) \qquad h_{21} \in \textnormal{Hom}((V_{M_1})_{\bfs g_1}, (C_{M_2})_{\bft g_2}).
\end{equation*}
In fact, without extra conditions, such matrices describe exactly
\begin{equation*}
(V_{G_1} \widehat\times_G V_{G_2})^\textnormal{cat} \rra M_1 \times_M M_2.
\end{equation*}
When interpreted in a specific way, the product of this Lie category can really be seen as matrix multiplication:
\begin{equation*}
\begin{pmatrix} H_{g_1} & h_{12} \\ h_{21} & H_{g_2} \end{pmatrix} \cdot \begin{pmatrix} H_{g_1'} & h_{12}' \\ h_{21}' & H_{g_2'} \end{pmatrix} = \begin{pmatrix} H_{g_1} \cdot H_{g_1'} + h_{12} \cdot h_{21}' & H_{g_1} \cdot_{C_M} h_{12}' + h_{12} \cdot_{V_M} H_{g_2'} \\ h_{21} \cdot_{V_M} H_{g_1'} + H_{g_2} \cdot_{C_M} h_{21}' & h_{21} \cdot h_{12}' + H_{g_2} \cdot H_{g_2'} \end{pmatrix}.
\end{equation*}
We used here the notion from Section \ref{sec: the fat category}, but it is also readily verified by a direct computation what the different products are supposed to represent. We obtain the invertible elements of this category, the elements of the fat groupoid $V_{G_1} \times_G V_{G_2}$, by applying the target to each entry and asking the result to be invertible:
\begin{equation*}
\begin{pmatrix} \bft_{g_1}H_{g_1} & \partial h_{12} \\ \partial h_{21} & \bft_{g_2}H_{g_2} \end{pmatrix}: (V_{M_1})_{\bfs g_1} \times (V_{M_2})_{\bfs g_2} \xra{\sim} (V_{M_1})_{\bft g_1} \times (V_{M_2})_{\bft g_2}.
\end{equation*}
Notice that
\begin{equation*}
\textnormal{H}(V_{M_1} \times_M V_{M_2}, C_{M_1} \times_M C_{M_2}) = \{\begin{pmatrix} h_1 & h_{12} \\ h_{21} & h_2 \end{pmatrix} \mid \begin{pmatrix} 1 + \partial h_1 & h_{12} \\ h_{21} & 1 + \partial h_2 \end{pmatrix} \textnormal{ is a linear isomorphism}\}
\end{equation*}
where
\begin{equation*}
h_1 \in \textnormal{Hom}(V_{M_1}, C_{M_1}) \quad h_{12} \in \textnormal{Hom}(V_{M_2}, C_{M_1}) \quad h_{21} \in \textnormal{Hom}(V_{M_1}, C_{M_2}) \quad h_2 \in \textnormal{Hom}(V_{M_2}, C_{M_2}).
\end{equation*}
The diagonal blockmatrices (those for which $h_{12}$ and $h_{21}$ are both zero) form the fiber product of the fat groupoids of $V_{G_1}$ and $V_{G_2}$:
\begin{equation*}
\widehat V_{G_1} \times_G \widehat V_{G_2} \rra M_1 \times_M M_2
\end{equation*}
so the structure is defined componentwise. This Lie groupoid fits into an exact sequence
\begin{center}
\begin{tikzcd}
1 \ar[r] & \textnormal{H}(V_{M_1}, C_{M_1}) \times_M \textnormal{H}(V_{M_2}, C_{M_2}) \ar[r] & \widehat V_{G_1} \times_G \widehat V_{G_2} \ar[r] & G_1 \times_G G_2 \ar[r] & 1.
\end{tikzcd}
\end{center}
We continue this line of thought in the next example.
\end{example}

As an application of the previous example, we describe the fat groupoid of a direct sum of VB-groupoids and of a pullback VB-groupoid.

\begin{example}\label{exa: direct sum in terms of fat extension}
To describe the direct sum, suppose we are given two VB-groupoids $V_{G,1} \rra V_{M,1}$ and $V_{G,2} \rra V_{M,2}$ over the same base groupoid $G$. Then the fat groupoid of
\begin{equation*}
V_{G,1} \oplus V_{G,2} \rra V_{M,1} \oplus V_{M,2}
\end{equation*}
consist of matrices
\begin{equation*}
\begin{pmatrix} H_{g,1} & h_{12} \\ h_{21} & H_{g,2} \end{pmatrix}: (V_{M,1})_{\bfs g} \oplus (V_{M,2})_{\bfs g} \ra (V_{G,1})_g \oplus (V_{G,2})_g.
\end{equation*}
Here, $H_{g,1} \in \widehat V_{G,1}^\textnormal{cat}$, $H_{g,2} \in \widehat V_{G,2}^\textnormal{cat}$ and
\begin{equation*}
h_{12} \in \textnormal{Hom}((V_{M,2})_{\bfs g}, (C_{M,1})_{\bft g}) \qquad h_{21} \in \textnormal{Hom}((V_{M,1})_{\bfs g}, (C_{M,2})_{\bft g}).
\end{equation*}
The diagonal blockmatrices constitute the fiber product of fat groupoids $\widehat V_{G,1} \times_G \widehat V_{G,2}$. This groupoid fits into an exact sequence of the form
\begin{center}
\begin{tikzcd}
1 \ar[r] & \textnormal{H}(V_{M,1}, C_{M,1}) \times_M \textnormal{H}(V_{M,2}, C_{M,2}) \ar[r] & \widehat V_{G,1} \times_G \widehat V_{G,2} \ar[r] & G \ar[r] & 1.
\end{tikzcd}
\end{center}
This groupoid has four natural $3$-term cochain complex representations by taking tensor products. For example, on
\begin{center}
\begin{tikzcd}
\textnormal{Hom}(V_{M,1}, C_{M,2}) \ar[r] & \textnormal{Hom}(V_{M,1}, V_{M,2}) \oplus \textnormal{Hom}(C_{M,1}, C_{M,2}) \ar[r] & \textnormal{Hom}(C_{M,1}, V_{M,2}).
\end{tikzcd}
\end{center}
The actions are defined ``componentwise''. Essentially, $\widehat V_{G,1} \times_G \widehat V_{G,2}$ together with its representations on $\textnormal{Hom}(V_{M,1}, C_{M,2})$ and on $\textnormal{Hom}(V_{M,2}, C_{M,1})$ are what constitutes the product of the upper/lower triangular blockmatrices.
\end{example}

From \cite{BursztynDrummond} we know that multiplicative tensors define homogeneous cochains on Whitney sums of VB-groupoids. In the upcoming work \cite{Multiplicativetensors} we will argue that multiplicative tensors on VB-groupoids $V_{G,1}, \dots, V_{G,n}$ appear naturally as objects defined on $\widehat V_{G,1} \times_G \dots \times_G \widehat V_{G,n}$ together with its canonical representations. A first instance of this can be found in Section \ref{sec: functorial aspects} where we introduce morphisms of fat extensions.

\begin{example}\label{exa: pullback of fat extension}
To describe the pullback, let $V_{G_2} \rra V_{M_2}$, be a VB-groupoid over $G_2$ and $\Phi: G_1 \ra G_2$ a groupoid map. The fat groupoid of the pullback VB-groupoid $\Phi^*V_{G_2}$ is
\begin{equation*}
\Phi^*\widehat V_{G_2} = G_1 \times_{G_2} \widehat V_{G_2} \rra M_1
\end{equation*}
where the multiplication of $\Phi^*\widehat V_{G_2}$ is defined componentwise. If $\Phi$ is the inclusion of a Lie subgroupoid $H$, we simply write
\begin{equation*}
\widehat V_G|_H = \{H_h \in \widehat V_G \mid h \in H\}.
\end{equation*}
\end{example}

To conclude, we formulate the abstract definition of the above examples for fat extensions.

\begin{definition}\label{def: fat fiber product}
Let $F_{G_1}$ and $F_{G_2}$ be fat extensions of $G_1 \rra M_1$ and $G_2 \rra M_2$, over $C_{M_1} \ra V_{M_1}$ and $C_{M_2} \ra V_{M_2}$, respectively, and let
\begin{equation*}
G_1 \ra G \la G_2
\end{equation*}
be two Lie groupoid maps that cleanly intersect. The \textit{fat extension of invertible blockmatrices} of $F_{G_1}$ and $F_{G_2}$ is the following fat extension. The fat groupoid form the invertible elements of the following fat category:
\begin{equation*}
\left\{\begin{pmatrix} H_{g_1} & h_{12} \\ h_{21} & H_{g_2} \end{pmatrix} \mid H_{g_1} \in F_{G_1}^\textnormal{cat}, h_{12} \in \textnormal{Hom}(\bfs^*V_{M_2}, \bft^*C_{M_1}), h_{21} \in \textnormal{Hom}(\bfs^*V_{M_1}, \bft^*C_{M_2}), H_{g_2} \in F_{G_2}^\textnormal{cat}\right\}
\end{equation*}
whose product is defined as
\begin{equation*}
\begin{pmatrix} H_{g_1} & h_{12} \\ h_{21} & H_{g_2} \end{pmatrix} \cdot \begin{pmatrix} H_{g_1'} & h_{12}' \\ h_{21}' & H_{g_2'} \end{pmatrix} = \begin{pmatrix} H_{g_1} \cdot H_{g_1'} + h_{12} \cdot h_{21}' & H_{g_1} \cdot_{C_M} h_{12}' + h_{12} \cdot_{V_M} H_{g_2'} \\ h_{21} \cdot_{V_M} H_{g_1'} + H_{g_2} \cdot_{C_M} h_{21}' & h_{21} \cdot h_{12}' + H_{g_2} \cdot H_{g_2'} \end{pmatrix}.
\end{equation*}
Here, $h_{12} \cdot h_{21}' = h_{12} \partial h_{21}'$, $h_{21} \cdot h_{12}' = h_{21} \partial h_{12}'$, and $\cdot_{V_M}$ and $\cdot_{C_M}$ are the actions induced by the representations. The cochain complex representation on
\begin{equation*}
\partial \times \partial: C_{M_1} \times_M C_{M_2} \ra V_{M_1} \times_M V_{M_2}
\end{equation*}
is given by matrix application: given $(c_1,c_2) \in (C_M)_{\bfs g_1} \times (C_M)_{\bfs g_2}$ and $(v_1,v_2) \in (V_M)_{\bfs g_1} \times (V_M)_{\bfs g_2}$,
\begin{equation*}
\begin{pmatrix} H_{g_1} & h_{12} \\ h_{21} & H_{g_2} \end{pmatrix} \begin{pmatrix} c_1 \\ c_2 \end{pmatrix} = \begin{pmatrix} H_{g_1} \cdot c_1 + h_{12}\partial c_2 \\ h_{21}\partial c_1 + H_{g_2} \cdot c_2 \end{pmatrix} \qquad \begin{pmatrix} H_{g_1} & h_{12} \\ h_{21} & H_{g_2} \end{pmatrix} \begin{pmatrix} v_1 \\ v_2 \end{pmatrix} = \begin{pmatrix} H_{g_1} \cdot v_1 + \partial h_{12}v_2 \\ \partial h_{21} v_1 + H_{g_2} \cdot v_2 \end{pmatrix}.
\end{equation*}
\end{definition}

In the last sections, Sections \ref{sec: other models of the fat groupoid} and \ref{sec: fat pairing}, we make remarks on other models of the fat groupoid of a VB-groupoid. In particular, we will see that dualising a fat extension should be seen as dualising its cochain complex representation. 

\subsection{Other models for the fat groupoid of a VB-groupoid}\label{sec: other models of the fat groupoid}
The fat groupoid of a VB-groupoid $V_G$ appeared in \cite{VBgroupoidsruths} as follows. Consider
\begin{equation*}
\widetilde V_G^\textnormal{cat} = \{\widetilde H_g \subset (V_G)_g \mid (V_G)_g = \ker \bfs_g \oplus \widetilde H_g\}
\end{equation*}
and the following open subset of it:
\begin{equation*}
\widetilde V_G = \{\widetilde H_g \in \widetilde V_G^\textnormal{cat} \mid (V_G)_g = \ker \bft_g \oplus \widetilde H_g\}.
\end{equation*}
Then $\widetilde V_G^\textnormal{cat}$ is a Lie category over $M$ with $\widetilde\bfs \widetilde H_g = \bfs g$, $\widetilde\bft \widetilde H_g = \bft g$, and
\begin{equation*}
\widetilde H_g \cdot \widetilde H_h = \im\bfm(\widetilde H_g, \widetilde H_h) = \{v_g \cdot v_h \in (V_G)_{gh} \mid v_g \in \widetilde H_g, v_h \in \widetilde H_h\},
\end{equation*}
and $\widetilde V_G$ form the invertible elements with
\begin{equation*}
(\widetilde H_g)^{-1} = \im\bfi(\widetilde H_g).
\end{equation*}
It is rather straightforward to see that the canonical map
\begin{equation*}
\widehat V_G^\textnormal{cat} \xra{\im} \widetilde V_G^\textnormal{cat}
\end{equation*}
is an isomorphism of Lie categories (that therefore restricts to an isomorphism of $\widehat V_G$ to $\widetilde V_G$). Still, we discuss the proof, because some interesting things appear. For example, it becomes clear that the fat groupoid $\widehat V_G^*$ of the dual VB-groupoid $V_G^*$ is isomorphic to $\widehat V_G$ as groupoids.

Notice that we can think of elements of $\widehat V_G^*$ as (the dual of) maps
\begin{equation*}
    \omega_g: (V_G)_g \ra (C_M)_{\bfs g}
\end{equation*}
such that precomposing with $\ell_g$ gives the identity map, and precomposing with $r_g$ gives a linear isomorphism.\footnote{As before, we wrote
\begin{equation*}
    \ell_g: (C_M)_{\bfs g} \xra{-0_g \cdot c^{-1}} (V_G)_g \qquad r_g: (C_M)_{\bft g} \xra{c \cdot 0_g} (V_G)_g.
\end{equation*}}
The product of two elements $\omega_g$ and $\omega_h$ is given by
\begin{equation*}
(\omega_g \cdot \omega_h)(v_g \cdot v_h) = \omega_h(r_h\omega_g(v_g)) + \omega_h(v_h)
\end{equation*}
(here $(v_g,v_h) \in V_G^{(2)}$ is a composable element) and the inverse of $\omega_g$ is
\begin{equation*}
\omega_g^{-1}(v_g^{-1}) = -(\omega_g r_g)^{-1}(\omega_gv_g).
\end{equation*}
\begin{proposition}\label{prop: Gracia Saz model fat groupoid and dual}
The maps
\begin{equation*}
    \widehat V_G \ra \widetilde V_G \la \widehat V_G^* \qquad H_g \mapsto \im H_g = \textnormal{ker } \omega_g \textnormal{ \reflectbox{$\mapsto$} } \omega_g
\end{equation*}
are groupoid isomorphisms.
\end{proposition}
\begin{proof}
We construct the inverse of the first map. To do this, take for every $g \in G$ an $H_{g,1} \in \widehat V_G$. Given $\widetilde H_g \in \widetilde V_G$, we can then write, for all $v \in (V_M)_{\bfs g}$,
\begin{equation*}
    H_{g,1}v = v_{\bfs,1} + H_gv
\end{equation*}
where $v_{\bfs,1} \in \ker \bfs_g$ and $H_gv \in \widetilde H_g$. We claim that the map
\begin{equation*}
    \widetilde V_G \ra \widehat V_G \qquad \widetilde H_g \mapsto H_g
\end{equation*}
is a well-defined smooth inverse of the map $\widehat V_G \ra \widetilde V_G$. To see that it is well-defined, take another $H_{g,2} \in \widehat V_G$ and write as above, for all $v \in (V_M)_{\bfs g}$,
\begin{equation*}
H_{g,2}v = v_{\bfs,2} + H_g'v.
\end{equation*}
Then it follows that
\begin{equation*}
H_g'v - H_gv \in \ker \bfs_g \cap \widetilde H_g = 0.
\end{equation*}
A similar argument shows that $\bft_gH_g$ must be invertible: if $\bft_gH_gv = 0$, then
\begin{equation*}
H_gv \in \ker \bft_g \cap \widetilde H_g = 0.
\end{equation*}
The map is smooth, for once we work locally, we can take $H_{g,1}$ to be in the image of a local section of the projection $\widehat V_G \ra G$. For a proof that the maps are groupoid maps, we refer to the proof of Proposition \ref{prop: fat groupoid of dual}.
\end{proof}

\begin{remark}\label{rema: not important to take element of fat groupoid in proof}
Notice that, in the above proof, the auxiliary elements $H_{g,1} \in \widehat V_G$ that are used to define the inverse map $\widetilde V_G \ra \widehat V_G$ can be taken in $\widehat V_G^\textnormal{cat}$ as well. In particular, we could also use a cleavage $\Sigma$ to define, for all $v \in (V_M)_{\bfs g}$, the inverse through the equation
\begin{equation*}
\Sigma_g v = v_\bfs + H_gv
\end{equation*}
where $v_\bfs \in \ker s_g$ and $H_gv \in \widetilde H_g$.
\end{remark}

The isomorphism $\widehat V_G \ra \widetilde V_G \ra \widehat V_G^*$ can be understood through a pairing that we discuss next.

\subsection{The fat pairing}\label{sec: fat pairing}
We first look at the isomorphism $\widehat V_G \ra \widetilde V_G \ra \widehat V_G^*$ more carefully.

\begin{proposition}\label{prop: fat groupoid of dual}
Let $V_G$ be a VB-groupoid and consider $\widehat V_G$ and $\widehat V_G^*$. For $H_g \in \widehat V_G$, write
\begin{equation*}
\omega_g^{H_g}: (V_G)_g \ra (C_M)_{\bfs g} \qquad v \mapsto r_g(H_g^{-1}\bft_gv - v^{-1}).
\end{equation*}
The map
\begin{equation*}
\widehat V_G \ra \widehat V_G^* \qquad H_g \mapsto \omega_g^{H_g}
\end{equation*}
is a Lie groupoid isomorphism. Moreover, the map restricts to the canonical isomorphism $\textnormal{H}(V_M,C_M) \cong \textnormal{H}(C_M^*,V_M^*)$ from Proposition \ref{prop: dual of invertible homotopy bundle}, so the diagram
\begin{center}
\begin{tikzcd}
1 \ar[r] & \textnormal{H}(V_M,C_M) \ar[r] \ar[d, "\cong"] & \widehat V_G \ar[r] \ar[d, "\cong"] & G \ar[r] \ar[d, "="] & 1 \\
1 \ar[r] & \textnormal{H}(C_M^*,V_M^*) \ar[r] & \widehat V_G^* \ar[r] & G \ar[r] & 1
\end{tikzcd}
\end{center}
commutes.
\end{proposition}
\begin{proof}
Notice that $\omega_g^{H_g}$ is the unique element $\omega_g \in \widehat V_G^*$ with the property that $\omega_gH_g = 0$.\footnote{If $\omega_g'$ is another element of $\widehat V_G^*$ such that $\omega_g'H_g=0$, then $\omega_g^{H_g}=\omega_g'$ on $\im H_g$, but also on $\ker \bft_g$, which are elements of the form $\ell_gc$.} To see that the map $\widehat V_G \ra \widehat V_G^*$ is well-defined, observe that
\begin{equation*}
\omega_g^{H_g}r_g = r_g(H_g^{-1}(\partial c) + \ell_g(\partial c - c)) = H_g^{-1} \cdot c
\end{equation*}
has inverse $c \mapsto H_g \cdot c$.

We compute the inverse map $\omega_g \mapsto H_g^{\omega_g}$ explicitly. Pick for this some $H_g' \in \widehat V_G$. Then $H_g^{\omega_g}$ is the unique element $H_g \in \widehat V_G$ such that
\begin{equation*}
    H_g'v = \ell_g\omega_g(H_g'v) + H_g(v + \partial\omega_g(H_g'v)).
\end{equation*}
In fact, given $v \in (V_M)_{\bfs g}$, writing $h_gv \in (C_M)_{\bft g}$ for the inverse of $-\omega_g(H_g'v)$ under $\omega_gr_g$, we have
\begin{equation*}
    H_g^{\omega_g}v = H_g'v + r_g h_gv.
\end{equation*}
We end the proof by showing that $\widehat V_G \ra \widehat V_G^*$ is a groupoid map. For $H_g, H_{g'} \in \widehat V_G$ composable,
\begin{equation*}
\omega_g^{H_g} \cdot \omega_{g'}^{H_{g'}}(v_g \cdot v_{g'}) = (H_{g'}^{-1}((\bft_g H_g)^{-1}\bft_g v_g) + 0_{(g')^{-1}g^{-1}} \cdot (v_g - H_g(\bft_g H_g)^{-1} \bft_g v_g) - v_{g'}^{-1}) \cdot 0_{g'},
\end{equation*}
while
\begin{equation*}
\omega_{gg'}^{H_{gg'}}(v_g \cdot v_{g'}) = (H_{g'}^{-1}H_g^{-1}\bft_g v_g - v_{g'}^{-1}v_g^{-1}) \cdot 0_g \cdot 0_{g'}.
\end{equation*}
If we replace $0_g$ with $\pm v_g \pm H_g(\bft_g H_g)^{-1}\bft_g v_g$ in the last equation, we see that the two expressions agree. This proves the statement.
\end{proof}

\begin{remark}\label{rema: fat groupoid of dual}
Given $H_g$ and $\omega_g = \omega_g^{H_g}$, we constructed in the above proof a map $h_g$ using an auxiliary element $H_g' \in \widehat V_G$:
\begin{equation*}
h_g: (V_M)_{\bfs g} \ra (C_M)_{\bft g} \qquad v \mapsto r_{g^{-1}}(\ell_g + H_g\partial)\omega_g(H_g'v) = r_{g^{-1}}(H_g'v - H_gv).
\end{equation*}
As we mentioned already in Remark \ref{rema: not important to take element of fat groupoid in proof}, we could take such auxiliary elements in $\widehat V_G^\textnormal{cat}$. Given a cleavage $\Sigma$, the equation
\begin{equation*}
    \Sigma_gv = -r_g h_gv + H_gv
\end{equation*}
reveals a new description of the fat groupoid in terms of the associated split ruth. This will be explained in detail in Section \ref{sec: From 2-term ruths to fat extensions}.
\end{remark}

We can understand this isomorphism from Proposition \ref{prop: fat groupoid of dual} more conceptually as follows.
\begin{proposition}\label{prop: fat pairing}
The smooth map
\begin{equation*}
\widehat V_G^* \times_G \widehat V_G \ra \textnormal{H}(V_M,C_M) \qquad (\omega_g,H_g) \mapsto \omega_gH_g
\end{equation*}
is non-degenerate in the sense that, given $H_g \in \widehat V_G$ (resp. $\omega_g \in \widehat V_G^*$) and $h_{\bfs g} \in \textnormal{H}(V_M,C_M)$, there is a unique element $\omega_g \in \widehat V_G^*$ (resp. $H_g \in \widehat V_G$) such that $\omega_g H_g = h_{\bfs g}$.
\end{proposition}
\begin{proof}
Let $H_g \in \widehat V_G^*$ and $h_{\bfs g} \in \textnormal{H}(V_M,C_M)$. If $h_{\bfs g} = 0$, we saw in Proposition \ref{prop: fat groupoid of dual} that
\begin{equation*}
\omega_g^{H_g}v = r_g(H_g^{-1} \bft_g v - v^{-1})
\end{equation*}
is unique with the property that
\begin{equation*}
\omega_g^{H_g}H_g=0.
\end{equation*}
Now, if $h_{\bfs g} \neq 0$, consider
\begin{equation*}
\omega_gv = \omega_g^{H_g}v + h_{\bfs g} (\bft_g H_g)^{-1}\bft_g v.
\end{equation*}
Then $\omega_g \in \widehat V_G^*$. Indeed, the inverse of precomposition with $r_g$ is given by
\begin{equation*}
c \mapsto r_{g^{-1}}(\ell_g c + H_g \partial c) + (h_{\bfs g} \partial c)^{-1}.
\end{equation*}
The element $\omega_g$ is unique with the property that $\omega_gH_g = h_{\bfs g}$, so we proved the statement.
\end{proof}

\begin{remark}\label{rema: fat pairing is a multiplicative tensor}
Given $V_{G,1},\dots,V_{G,n}$ VB-groupoids over $G$, we can see
\begin{equation*}
\widehat V_{G,1} \times_G \cdots \times_G \widehat V_{G,n}
\end{equation*}
as a groupoid over $M$ (with structure defined componentwise; see Example \ref{exa: pullback}). However, with this groupoid structure, the above map $\widehat V_G^* \times_G \widehat V_G \ra \textnormal{H}(V_M,C_M)$ is not a groupoid map (see the discussion after this remark).
Rather, this map should be interpreted as the identity map of fat extensions, and we can interpret the map as a multiplicative tensor if we view $\widehat V_G^* \times_G \widehat V_G$ as a groupoid. That it maps into $\textnormal{H}(V_M,C_M)$ is a feature of the identity map being an isomorphism. We will be more precise in Section \ref{sec: functorial aspects}, but a more general treatment of viewing multiplicative tensors this way is a topic of our upcoming work \cite{Multiplicativetensors}.
\end{remark}

Using $\widehat V_G^* \cong \widehat V_G$, the fat pairing is simply the map
\begin{equation*}
\widehat V_G \times_G \widehat V_G \ra \textnormal{H}(V_M,C_M) \qquad (H_g, H_g') \mapsto -1 + H_g' \cdot H_g^{-1}
\end{equation*}
that describes the inverse to the (principal) action map
\begin{equation*}
\textnormal{H}(V_M,C_M) \ltimes \widehat V_G \ra \widehat V_G \times_G \widehat V_G.
\end{equation*}

\begin{remark}\label{rema: maps of fat extensions are not maps of extensions}
Proposition \ref{prop: fat groupoid of dual} reveals that a map of fat extensions $F_{G,1} \ra F_{G,2}$ (over the same base groupoid $G$) is not simply a commutative diagram
\begin{center}
\begin{tikzcd}
1 \ar[r] & \textnormal{H}(V_{M,1},C_{M,1}) \ar[r] \ar[d] & F_{G,1} \ar[r] \ar[d] & G \ar[r] \ar[d, "="] & 1 \\
1 \ar[r] & \textnormal{H}(V_{M,2},C_{M,2}) \ar[r] & F_{G,2} \ar[r] & G \ar[r] & 1
\end{tikzcd}
\end{center}
Indeed, $\widehat V_G$ and $\widehat V_G^*$ are fat extensions, and there is a commutative diagram
\begin{center}
\begin{tikzcd}
1 \ar[r] & \textnormal{H}(V_M,C_M) \ar[r] \ar[d, "\cong"] & \widehat V_G \ar[r] \ar[d, "\cong"] & G \ar[r] \ar[d, "="] & 1 \\
1 \ar[r] & \textnormal{H}(C_M^*, V_M^*) \ar[r] & \widehat V_G^* \ar[r] & G \ar[r] & 1.
\end{tikzcd}
\end{center}
In fact, we should not even expect that a map of fat extensions comes with a groupoid map $F_{G,1} \ra F_{G,2}$ (see Section \ref{sec: functorial aspects}). Rather, maps of fat extensions are ``comparison maps'' like the ones we encountered in Section \ref{sec: examples of fat extensions}.
\end{remark}

What Proposition \ref{prop: fat groupoid of dual} shows is that the dual of a fat extension should be seen as having the same underlying extension, but considering the dual cochain complex representation:

\begin{definition}\label{def: the dual of a fat extension}
Let $F_G$ be a fat extension of $G$ over $C_M \ra V_M$. The dual of $F_G$ is the extension
\begin{center}
\begin{tikzcd}
1 \ar[r] & \textnormal{H}(V_M,C_M) \ar[r] & F_G \ar[r] & G \ar[r] & 1
\end{tikzcd}
\end{center}
together with the dual cochain complex representation $V_M^* \ra C_M^*$.
\end{definition}

\section{Fat extensions and equivalent formulations}\label{sec: fat extensions}
Having encountered the structure of \textit{fat extension} in the previous section, we proceed to show how to go back and forth between fat extensions, VB-groupoids and $2$-term ruths. We dedicate Section \ref{sec: PB-groupoids} to the correspondence with (general linear) PB-groupoids after this section.

We call the $2$-term ruth associated to a fat extension an \textit{invariant complex}. This type of complex is a \textit{relative complex} as in \cite{Mariarelative}. This new point of view leads to interesting connections between the ``natural candidates'' for tensor product of VB-groupoids and of $2$-term ruths. As already mentioned a few times, this is the topic of the upcoming work \cite{Multiplicativetensors}. Defining the correct notion of morphism for the category of fat extensions (that yields the claimed equivalence of categories between fat extensions and VB-groupoids or $2$-term ruths) is a little more subtle, but similar to how one defines morphisms of (split) $2$-term ruths. We delay the notion of morphism of fat extensions to Section \ref{sec: functorial aspects}, where we also go into the functorial aspects of the equivalence we start describing in this section. The reason for the postponement is that we believe the discussion of this section should help understand the functorial aspects better.

\subsection{From VB-groupoids to fat extensions}
In Section \ref{sec: The fat extension of a VB-groupoid}, we developed the notion of fat extension through the language of VB-groupoids. For referential purposes and convenience, we repeat the definition here.

\begin{definition}
Let $G \rra M$ be a Lie groupoid, and consider a $2$-term complex
\begin{equation*}
C_M \xra{\partial} V_M.
\end{equation*}
A \textit{fat extension} of $G$, with underlying complex $C_M \ra V_M$, consists of a Lie groupoid
\begin{equation*}
F_G \rra M,
\end{equation*}
that comes with a cochain complex representation on $C_M \ra V_M$, and a short exact sequence of Lie groupoids (over $M$)
\begin{center}
\begin{tikzcd}
1 \ar[r] & \textnormal{H}(V_M,C_M) \ar[r] & F_G \ar[r] & G \ar[r] & 1.
\end{tikzcd}
\end{center}
Moreover, the cochain complex representation of $F_G$ on $C_M \ra V_M$ restricts to the canonical cochain complex representation of $\textnormal{H}(V_M,C_M)$, and the two natural conjugation actions of $F_G$ on $\textnormal{H}(V_M,C_M)$ agree.
\end{definition}

We will simply write $F_G$ for a fat extension, and the underlying groupoid $F_G$ is called the fat groupoid of the fat extension. Now, in Sections \ref{sec: the fat groupoid}, \ref{sec: the bundle of invertible homotopies} and \ref{sec: canonical representations of the fat groupoid}, we proved:

\begin{proposition}\label{prop: fat groupoid comes with fat extension}
Let $V_G$ be a VB-groupoid over $G$. Then, the fat groupoid $\widehat V_G$ of $V_G$, together with its cochain complex representation \eqref{eq: representations of the fat groupoid actions}, is a fat extension of $G$ over $C_M \ra V_M$.
\end{proposition}

Similarly, we have:

\begin{proposition}\label{prop: fat category comes with fat extension}
    Let $V_G$ be a VB-groupoid over $G$. Then, the fat category $\widehat V_G^\textnormal{cat}$ of $V_G$, together with its cochain complex representation on $C_M \ra V_M$ and the (comparison) map 
    \begin{equation*}
        h: F_G^\textnormal{cat} \times_G F_G^\textnormal{cat} \ra \textnormal{Hom}(\bft^*V_M,\bfs^*C_M)
    \end{equation*}
    is a fat category extension of $G$ over $C_M \ra V_M$.
\end{proposition}

\subsection{From fat extensions to VB-groupoids}\label{sec: From fat extensions to VB-groupoids}
In this section we prove how to recover the VB-groupoid from a fat extension. Let $F_G$ be a fat extension and consider the VB-groupoid associated to the cochain complex representation $C_M \ra V_M$. We denote this VB-groupoid by
\begin{equation*}
F_G \ltimes (C_M \ra V_M) = C_M \times_M F_G \times_M V_M \rra V_M.
\end{equation*}
The source and target maps are given by $\bfs(c,H_g,v) = v$ and $\bft(c,H_g,v) = \partial c + H_g \cdot v$, and
\begin{equation*}
(c_1,H_g,\partial c_2 + H_h \cdot v) \cdot (c_2,H_h,v) = (c_1 + H_g \cdot c_2, H_g \cdot H_h, v).
\end{equation*}
Now, to recover a VB-groupoid $V_G$ from its fat extension $F_G = \widehat V_G$, observe that an element $H_g \in \widehat V_G$ defines a linear isomorphism
\begin{equation*}
r_g + H_g: (C_M)_{\bft g} \oplus (V_M)_{\bfs g} \ra (V_G)_g \qquad (c,v) \mapsto c \cdot 0_g + H_gv.
\end{equation*}
The inverse is given by $((\omega_g r_g)^{-1}\omega_g,\bfs_g)$, where $\omega_g = \omega_g^{H_g} \in \widehat V_G^*$ is unique with the property that $\omega_gH_g = 0$ (see Proposition \ref{prop: fat groupoid of dual}). This observation leads to the following statement.

\begin{proposition}\label{prop: fat exact sequence of VB-groupoids}
The VB-groupoid map
\begin{equation*}
\widehat V_G \ltimes (C_M \ra V_M) \ra V_G \qquad (c,H_g,v) \mapsto r_gc + H_gv
\end{equation*}
is surjective, and the associated exact sequence (of VB-groupoids) is given by
\begin{center}
\begin{tikzcd}
1 \ar[r] & \textnormal{H}(V_M,C_M) \ltimes V_M \ar[r] & \widehat V_G \ltimes (C_M \ra V_M) \ar[r] & V_G \ar[r] & 1,
\end{tikzcd}
\end{center}
where $\textnormal{H}(V_M,C_M)$ acts trivially on $V_M$, and the inclusion is given by $(h_x,v) \mapsto (-h_xv,1+h_x,v)$.
\end{proposition}

Now, given again an abstract fat extension
\begin{center}
\begin{tikzcd}
1 \ar[r] & \textnormal{H}(V_M,C_M) \ar[r] & F_G \ar[r] & G \ar[r] & 1
\end{tikzcd}
\end{center}
we obtain an injective map of VB-groupoids\footnote{As above, $\textnormal{H}(V_M,C_M)$ acts trivially on $V_M$. Notice that now we identified $h_x \in \textnormal{H}(V_M,C_M)$ with its image in $F_G$; if $F_G = \widehat V_G$, this identifies $h_x$ with its image $1+h_x$.}
\begin{equation*}
\textnormal{H}(V_M,C_M) \ltimes V_M \ra F_G \ltimes (C_M \ra V_M) \qquad (h_x,v) \mapsto (-h_xv, h_x, v).
\end{equation*}
The quotient naturally inherits the structure of a VB-groupoid over $G$ that we denote by $V(F_G)$. The resulting short exact sequence
\begin{center}
\begin{tikzcd}
1 \ar[r] & \textnormal{H}(V_M,C_M) \ltimes V_M \ar[r] & F_G \ltimes (C_M \ra V_M) \ar[r] & V(F_G) \ar[r] & 1
\end{tikzcd}
\end{center}
is a short exact sequence of VB-groupoids over the extension
\begin{center}
\begin{tikzcd}
1 \ar[r] & \textnormal{H}(V_M,C_M) \ar[r] & F_G \ar[r] & G \ar[r] & 1.
\end{tikzcd}
\end{center}
In particular, the underlying cochain complex of $V(F_G)$ is given by $C_M \ra V_M$. We will now show that VB-groupoids and fat extensions are in a one-to-one correspondence with each other. As we delayed the discussion on morphisms of fat extensions to Section \ref{sec: functorial aspects}, we won't go into the functorial aspects here. However, the notion of isomorphism of fat extensions necessarily appears.
\begin{proposition}\label{prop: essential surjectivity of fat construction}
The assignment
\begin{equation*}
\{\textnormal{VB-groupoids}\} \ra \{\textnormal{Fat extensions}\} \qquad V_G \mapsto \widehat V_G
\end{equation*}
sets a one-to-one correspondence, and the assignment
\begin{equation*}
\{\textnormal{Fat extensions}\} \ra \{\textnormal{VB-groupoids}\} \qquad F_G \mapsto V(F_G)
\end{equation*}
is inverse to it (up to canonical isomorphisms).
\end{proposition}
\begin{proof}
That, given a VB-groupoid $V_G$, there is a canonical isomorphism
\begin{equation*}
V(\widehat V_G) \xra{\sim} V_G
\end{equation*}
is a straightforward consequence of Proposition \ref{prop: fat exact sequence of VB-groupoids}.

Given a fat extension $F_G$, observe that the map
\begin{equation*}
F_G \ra \widehat V(F_G) \qquad H_g \mapsto (v \mapsto [0,H_g,v])
\end{equation*}
is a Lie groupoid map that intertwines the representations on $C_M \ra V_M$. In the remainder of the proof, we show that this map is invertible.\footnote{Therefore, we obtain an isomorphism of fat extensions; see Section \ref{sec: functorial aspects} for more details.} Notice for this that elements of $\widehat V(F_G)$ can be represented by maps
\begin{equation*}
(V_M)_{\bfs g} \ra (C_M)_{\bft g} \times (F_G)_g \times (V_M)_{\bfs g} \qquad v \mapsto (hv, H_g^v, v)
\end{equation*}
such that $h$ is linear, and
\begin{equation*}
(V_M)_{\bfs g} \ra (V_M)_{\bft g} \qquad v \mapsto \partial h(v) + H_g^v \cdot v
\end{equation*}
is a well-defined linear isomorphism. We claim that, for all elements of $\widehat V(F_G)$, we can take $h=0$ and $H_g^v=H_g^0$. Indeed, given a representative as above, write, for all $v \in (V_M)_{\bfs g}$,
\begin{equation*}
h_{\bft g}^vw = h((H_g^v)^{-1} \cdot w) \in \textnormal{H}(V_M,C_M).
\end{equation*}
Then, for all $v \in (V_M)_{\bfs g}$,
\begin{equation*}
(h_{\bft g}^v,\partial h(v) + H_g^v \cdot v) \cdot (h(v), H_g^v,v) = (0, h_{\bft g}^v \cdot H_g^v, v),
\end{equation*}
so $v \mapsto (0, h_{\bft g}^v \cdot H_g^v, v)$ represents the same element as $v \mapsto (h(v), H_g^v,v)$ in $\widehat V(F_G)$. Now, for all $v \in (V_M)_{\bfs g}$, we can take the unique element $h_{\bfs g}^v \in \textnormal{H}(V_M,C_M)$ such that
\begin{equation*}
h_{\bft g}^v \cdot H_g^v \cdot h_{\bfs g}^v = H_g^0.
\end{equation*}
Then
\begin{equation*}
(0, h_{\bft g}^v \cdot H_g^v, v) \cdot (h_{\bfs g}^v, v) = (0, H_g^0, v),
\end{equation*}
which shows that, $v \mapsto (0, H_g^0,v)$ represents the same element as $v \mapsto (h(v), H_g^v,v)$ in $\widehat V(F_G)$. This proves the claim, so the map $F_G \ra \widehat V(F_G)$ is indeed invertible.
\end{proof}

This result can be seen as the essential surjectivity of functors that is described in Section \ref{sec: functorial aspects}.

\subsection{From split $2$-term ruths to fat extensions}\label{sec: From 2-term ruths to fat extensions}
Given a VB-groupoid $V_G$, we observed in Remark \ref{rema: fat groupoid of dual} that we can construct an inverse to the canonical map $\widehat V_G \ra \widehat V_G^*$ using a cleavage $\Sigma$. The key ingredient we used was the map
\begin{equation*}
    h_g: (V_M)_{\bfs g} \ra (C_M)_{\bft g} \qquad h_g = r_{g^{-1}}(H_g - \Sigma_g)
\end{equation*}
that satisfies
\begin{equation*}
    \Sigma_g + r_g h_g = H_g.
\end{equation*}
This map $h_g$ is a homotopy, showing that the quasi-action $R_1(g)$ and the action by $H_g$ define the same map in cohomology: for all $v \in (V_M)_{\bfs g}$, we have
\begin{equation*}
    R_1(g)v + [\partial,h_g]v = H_g \cdot v.
\end{equation*}
Defining $H_g$ out of $h_g$ and $\Sigma$ using the above formula, we obtain another description of $\widehat V_G$ in terms of the split $2$-term ruth associated to $V_G$. It should be seen as the fat groupoid associated to a split $2$-term ruth:

\begin{definition}\label{def: fat ruth}
The fat groupoid $\widehat \calE_G$ of a split $2$-term ruth $\calE_G$ is the Lie groupoid
\begin{equation*}
    \widehat \calE_G = \{h_g: (V_M)_{\bfs g} \ra (C_M)_{\bft g} \mid \varphi_g = R_1(g) + \partial h_g \textnormal{ is a linear isomorphism}\}
\end{equation*}
over $M$.
\end{definition}

Notice that, given $h_g \in \widehat\calE_G$, the map
\begin{equation*}
\varphi_g = R_1(g) + [\partial,h_g]
\end{equation*}
is a cochain map that is an isomorphism: it defines the cochain complex representation.\footnote{That $R_1(g) + h_g\partial$ is an isomorphism as well follows from the five-lemma, for $\varphi_g$ is homotopic to $R_1(g)$. Indeed, $\varphi_g$ is therefore a quasi-isomorphism, and since $R_1(g) + \partial h_g$ is an isomorphism, $R_1(g) + h_g\partial$ must be as well.} To clarify, the structure maps are defined again as composition of homotopies (see Lemma \ref{lemm: homotopy composition}): if $h_g$ and $h_{g'}$ are composable (i.e. $g$ and $g'$ are composable), then the homotopy $h_g \cdot h_{g'}$ is such that
\begin{equation*}
    \varphi_g \varphi_{g'} = R_1(gg') + [\partial, h_g \cdot h_{g'}].
\end{equation*}
That is,
\begin{align*}
    h_g \cdot h_{g'} &\coloneq R_2(g,g') + h_g R_1(g') + R_1(g) h_{g'} + h_g \partial h_{g'} \\
h_g^{-1} &\coloneq -(R_2(g^{-1},g) + R_1(g^{-1})h_g)\varphi_g^{-1} = -\varphi_g^{-1} (R_2(g,g^{-1}) + h_g R_1(g^{-1})).
\end{align*}
Alternatively, we may think of $\widehat\calE_G$ as the open set of invertible elements of the Lie category
\begin{equation*}
    \widehat\calE_G^\textnormal{cat} = \textnormal{Hom}(\bfs^*V_M,\bft^*C_M) \rra M
\end{equation*}
with product defined as above. We have an isomorphism of Lie categories
\begin{equation*}
h_\Sigma: \widehat V_G^\textnormal{cat} \ra \widehat\calE_G^\textnormal{cat} \qquad H_g \mapsto r_{g^{-1}}(H_g - \Sigma_g)
\end{equation*}
and this isomorphism restricts to an isomorphism $\widehat V_G \ra \widehat \calE_G$ (see Remark \ref{rema: compare with VB-groupoids: direct sum vs tensor product} and Section \ref{sec: Relation to Representations up to homotopy}).

\begin{remark}\label{rema: twisting by a MC element}
Thinking of $R = (\partial=R_0,R_1,R_2)$ as a Maurer-Cartan element (see Remark \ref{rema: structure maps R}), we can perhaps think of the above product on $\textnormal{Hom}(\bfs^*V_M, \bft^*C_M)$ as a ``twisting'' of the trivial product
\begin{equation*}
h_g \cdot h_{g'} = 0
\end{equation*}
by $R$. The product
\begin{equation*}
h_g \cdot h_{g'} \coloneq h_g \partial h_{g'}
\end{equation*}
we used in Section \ref{sec: the fat category} can also be seen this way by taking $R=(\partial,0,0)$, but notice that this is not a unital ruth.
\end{remark}

\begin{remark}\label{rema: compare with VB-groupoids: direct sum vs tensor product}
The map $h_\Sigma: \widehat V_G \ra \textnormal{Hom}(\bfs^*V_M,\bft^*C_M)$ we constructed above is ``invariant'': given $H_g, H_g' \in \widehat V_G$, we have
\begin{equation*}
h_\Sigma(H_g)-h_\Sigma(H_g') = -h(H_g,H_g')
\end{equation*}
where
\begin{equation*}
h: \widehat V_G \times \widehat V_G \ra \textnormal{Hom}(\bfs^*V_M,\bft^*C_M) \qquad h(H_g,H_g') \coloneq r_{g^{-1}}(H_g'-H_g).
\end{equation*}
Conversely, $\Sigma$ can be recovered as
\begin{equation*}
\Sigma_g = -r_gh_\Sigma(H_g) + H_g
\end{equation*}
where $H_g \in \widehat V_G$ is arbitrary. That $\Sigma$ is unital translates into $h_\Sigma(1+h_x) = h_x$ for all $h_x \in \textnormal{H}(V_M,C_M)$. We can therefore think abstractly of splittings of fat extensions as being invariant maps as above. In Section \ref{sec: Relation to Representations up to homotopy} we put some more details about these \textit{fat splittings}.
\end{remark}

\subsection{From fat extensions to $2$-term ruths}\label{sec: From fat extensions to ruths}
Let $F_G$ be a fat extension of $G$ over $C_M \ra V_M$. Recall that we can construct an exact sequence
\begin{center}
\begin{tikzcd}
1 \ar[r] & \textnormal{H}(V_M,C_M) \ltimes V_M \ar[r] & F_G \ltimes (C_M \ra V_M) \ar[r] & V(F_G) \ar[r] & 1
\end{tikzcd}
\end{center}
of VB-groupoids, where we view $V_M$ as a representation of $\textnormal{H}(V_M,C_M)$ trivially.

\begin{definition}\label{def: invariant complex}
Given a short exact sequence of groupoids
\begin{center}
\begin{tikzcd}
1 \ar[r] & K \ar[r] & G \ar[r] & H \ar[r] & 1
\end{tikzcd}
\end{center}
over $M$,\footnote{All three groupoids are groupoids over $M$ and the maps in the exact sequence are over the identity maps.} the subcomplex\footnote{The condition should be satisfied for all $(g_1,\dots,g_\bullet) \in G^{(\bullet)}$ and $k_1,\dots,k_{\bullet+1} \in K$ with $\bfs g_{\ell-1} = \bfs k_\ell = \bft g_\ell$.}
\begin{equation*}
C^\bullet_{K\textnormal{-rel}} G = \{f \in C^\bullet G \mid f(k_1 g_1 k_2^{-1}, \dots, k_\bullet g_\bullet k_{\bullet+1}^{-1}) = f(g_1,\dots,g_\bullet)\}
\end{equation*}
of $CG$ is called the \textit{$K$-relative complex}.
\end{definition}

The relevance of this complex, that was introduced in general in \cite{Mariarelative}, becomes clear from the following characterisation.

\begin{proposition}\label{prop: invariant complex iso to quotient complex}
Given a short exact sequence of groupoids
\begin{center}
\begin{tikzcd}
1 \ar[r] & K \ar[r] & G \ar[r, "\Phi"] & H \ar[r] & 1
\end{tikzcd}
\end{center}
over $M$, the induced pullback map
\begin{equation*}
\Phi^*: CH \ra CG
\end{equation*}
is an isomorphism onto $C_{K\textnormal{-rel}} G$.
\end{proposition}
\begin{proof}
To see that $\Phi^*$ maps into $C_{K\textnormal{-rel}} G$ is readily verified by direct computation. The inverse is given by choosing a set-theoretic section $\sigma$ of $G \ra H$ (over the identity map) and setting
\begin{equation*}
C^\bullet_{K\textnormal{-rel}} G \ra C^\bullet H \qquad f \mapsto ((h_1,\dots,h_\bullet) \mapsto f(\sigma h_1,\dots, \sigma h_\bullet)).
\end{equation*} 
The map $(h_1,\dots,h_\bullet) \mapsto f(\sigma h_1,\dots, \sigma h_\bullet)$ is indeed smooth, for locally (in $H^{(\bullet)}$) we can replace $\sigma$ in the $\ell$-th component with a smooth section $\sigma_\ell$.
\end{proof}

Returning to the setup where $F_G$ is a fat extension, we write
\begin{equation*}
C_\textnormal{inv}(F_G; V_M^* \ra C_M^*) = C_{(\textnormal{H}(V_M,C_M) \ltimes V_M)\textnormal{-rel}}(F_G; V_M^* \ra C_M^*).
\end{equation*}
(with respect to the short exact sequence of VB-groupoids above) and we call its elements \textit{invariant}. The reason we passed to the dual cochain complex representation is because that corresponds to the VB-complex of the VB-groupoid $F_G \ltimes (C_M \ra V_M)$. Similarly, we write
\begin{equation*}
C_\textnormal{inv}(F_G; C_M \ra V_M) = C_{(\textnormal{H}(V_M,C_M) \ltimes C_M^*)\textnormal{-rel}}(F_G; C_M \ra V_M).
\end{equation*}
As a corollary of Proposition \ref{prop: invariant complex iso to quotient complex}, we obtain:

\begin{corollary}\label{cor: invariant complex fat extension}
Let $F_G$ be a fat extension over $C_M \ra V_M$. Then the pullback induces an isomorphism of differential graded $CG$-modules:
\begin{equation*}
C_\textnormal{VB} V(F_G)^* \xra{\sim} C_\textnormal{inv}(F_G; C_M \ra V_M).
\end{equation*}
\end{corollary}

To make the invariance condition more concrete, thereby also explaining the terminology, notice that the left action of $\textnormal{H}(V_M,C_M) \ltimes V_M$ on $F_G \ltimes (C_M \ra V_M)$ is given by
\begin{equation*}
(h_{\bft g},\partial c + H_g \cdot v) \cdot (c,H_g,v) = (h_{\bft g}(H_g \cdot v) + c, h_{\bft g} \cdot H_g, v)
\end{equation*}
and the right action by
\begin{equation}\label{eq: right action by elements of the bundle of invertible homotopies}
(c,H_g,v) \cdot (h_{\bfs g}, v) = (c - H_g \cdot (h_{\bfs g}v), H_g \cdot h_{\bfs g}, v).
\end{equation}
Working out the condition of being invariant using these explicit expressions, we obtain:

\begin{proposition}\label{prop: invariant cochains in the fat case}
The invariant cochains in $C^\bullet(F_G; C_M \ra V_M)$ are those pairs
\begin{equation*}
(f_0,f_1) \in C^\bullet(F_G; C_M) \oplus C^{\bullet-1}(F_G; V_M)
\end{equation*}
for which 
\begin{equation}\label{eq: fat f_V_M relation to f_C_M 0}
f_1(H_{g_2},\dots,H_{g_\bullet}) = f_1(H_{g_2}',\dots,H_{g_\bullet}')
\end{equation}
and
\begin{equation}\label{eq: fat f_V_M relation to f_C_M 1}
f_0(H_{g_1},\dots,H_{g_\bullet}) - f_0(H_{g_1}',\dots,H_{g_\bullet}') = h_{\bft g}(H_{g,1} \cdot f_1(H_{g_2},\dots,H_{g_\bullet})),
\end{equation}
where $h_{\bft g}$ is the unique element in $\textnormal{H}(V_M,C_M)$ such that $h_{\bft g} \cdot H_{g,1} = H_{g,1}'$ in $F_G$.
\end{proposition}
\begin{proof}
Explicitly, the identification
\begin{equation*}
C(F_G; C_M \ra V_M) \xra{\sim} C_\textnormal{VB}(F_G \ltimes (V_M^* \ra C_M^*)) \qquad (f_0, f_1) \mapsto f
\end{equation*}
is given by
\begin{equation*}
f((f_{\bft g_1}^{V_M}, H_{g_1}, f_{\bfs g_1}^{C_M}), g_2, \dots, g_\bullet) = f_{\bfs g_1}^{C_M}(H_{g_1}^{-1} \cdot f_0(H_{g_1},\dots,H_{g_n})) - f_{\bft g_1}^{V_M}(H_{g_1} \cdot f_1(H_{g_2}, \dots, H_{g_\bullet})).
\end{equation*}
Now, the elements $f \in C_\textnormal{VB}(F_G \ltimes (V_M^* \ra C_M^*))$ that are inside the relative complex with respect to the sequence
\begin{center}
\begin{tikzcd}
1 \ar[r] & \textnormal{H}(V_M,C_M) \ltimes C_M^* \ar[r] & F_G \ltimes (V_M^* \ra C_M^*) \ar[r] & V(F_G) \ar[r] & 1
\end{tikzcd}
\end{center}
are those for which
\begin{equation*}
f((f_{\bft g_1}^{V_M}, H_{g_1}, f_{\bfs g_1}^{C_M}), g_2, \dots, g_\bullet) = f((f_{\bfs g_1}^{C_M}(H_{g_1}^{-1} \cdot h_{\bft g_1}) + f_{\bft g_1}^{V_M}, h_{\bft g_1} \cdot H_{g_1}, f_{\bfs g_1}^{C_M}), g_2, \dots, g_\bullet).
\end{equation*}
Therefore, $f_1$ is independent of its entries:
\begin{equation*}
f_1(H_{g_2}, \dots, H_{g_\bullet}) = f_1(H_{g_2}',\dots, H_{g_\bullet}') = f(g_2, \dots, g_\bullet)
\end{equation*}
and
\begin{align*}
f_0(H_{g_1}, \dots, H_{g_\bullet}) &= f_0(H_{g_1},g_2,\dots, g_\bullet) \\
&= f_0(h_{\bft g_1} \cdot H_{g_1}, g_2,\dots, g_\bullet) - h_{\bft g_1}(H_{g_1} \cdot f_1(g_2, \dots, g_\bullet)).
\end{align*}
Taking $h_{\bft g_1}$ such that $h_{\bft g_1} \cdot H_{g_1} = H_{g_1}'$, we obtain the desired result.
\end{proof}

Notice that the dual result is that
\begin{equation*}
C_\textnormal{VB} V(F_G) \xra{\sim} C_\textnormal{inv}(F_G; V_M^* \ra C_M^*),
\end{equation*}
but this isomorphism can be obtained as the dual of the above isomorphism.

\begin{remark}\label{rema: invariance condition makes f_0 determine f_1}
Given $(f_0,f_1)$ satisfying the conditions from Proposition \ref{prop: invariant cochains in the fat case}, we often write
\begin{equation*}
    f_0(H_{g_1},g_2,\dots,g_\bullet) = f_0(H_{g_1},\dots,H_{g_\bullet}) \qquad f_1(g_2,\dots,g_\bullet) = f_1(H_{g_2},\dots,H_{g_\bullet})
\end{equation*}
as above. We can write the invariance condition also in the following way: for all $h_{\bft g_1} \in \textnormal{H}(V_M,C_M)$,
\begin{equation*}
f_0(h_{\bft g_1} \cdot H_{g_1}, g_2,\dots, g_\bullet) = f_0(H_{g_1}, g_2,\dots, g_\bullet) + h_{\bft g_1}(H_{g_1} \cdot f_1(g_2,\dots,g_\bullet)).
\end{equation*}
Yet another way is that, for all $h_{\bfs g_1} \in \textnormal{H}(V_M,C_M)$,
\begin{equation*}
f_0(H_{g_1} \cdot h_{\bfs g_1}, g_2,\dots, g_\bullet) = f_0(H_{g_1}, g_2,\dots, g_\bullet) + H_{g_1} \cdot (h_{\bfs g_1}f_1(g_2,\dots,g_\bullet)).
\end{equation*}
If $C_M \neq 0$, then invariant cochains $(f_0,f_1)$ are determined by the element $f_0 \in C(F_G; C_M)$. For this reason, we usually write $C_\textnormal{inv}^\bullet(F_G; C_M \ra V_M)$ as the subcomplex
\begin{equation*}
C_\textnormal{inv}^\bullet(F_G; C_M) = \{f_0 \in C^\bullet(F_G; C_M)\mid \exists f_1 \in C^{\bullet-1}(F_G; V_M) \textnormal{ such that } \eqref{eq: fat f_V_M relation to f_C_M 0} \textnormal{ and } \eqref{eq: fat f_V_M relation to f_C_M 1} \textnormal{ hold}\}
\end{equation*}
of $C(F_G; C_M)$. This is done mostly for notational convenience, so usually $C_\textnormal{inv}^\bullet(F_G; C_M \ra V_M)$ is meant when $C_\textnormal{inv}^\bullet(F_G; C_M)$ is written.
\end{remark}

For clarity, we arrived at the following definition for abstract fat extensions:

\begin{definition}\label{def: invariant cochains}
Let $F_G$ be a fat extension of $G$ over $C_M \ra V_M$. The \textit{invariant complex}
\begin{equation*}
C_\textnormal{inv}(F_G; C_M) \subset C(F_G; C_M \ra V_M)
\end{equation*}
consists of pairs $(f_0,f_1) \in C(F_G; C_M \ra V_M)$ such that, for all $h_{\bft g} \in \textnormal{H}(V_M,C_M)$,\footnote{Implicit from the notation is that $f_0$ does not depend on the last $\bullet-1$ entries and $f_1$ is independent of its entries.}
\begin{equation*}
f_0(h_{\bft g} \cdot H_{g_1},g_2,\dots,g_\bullet) = f_0(H_{g_1},g_2,\dots,g_\bullet) + h_{\bft g_1}(H_{g_1} \cdot f_1(g_2,\dots,g_\bullet)).
\end{equation*}
\end{definition}

\begin{remark}\label{rema: invariance condition for fat groupoid of VB-groupoid}
If $F_G = \widehat V_G$ for a VB-groupoid $V_G$, and $H_g,H_g' \in \widehat V_G$,
\begin{equation*}
h_{\bft g}(H_g \cdot v) = h(H_g,H_g') = r_{g^{-1}}(H_g'v - H_gv)
\end{equation*}
where $h_{\bft g} \in \textnormal{H}(V_M,C_M)$ is the map
\begin{equation*}
h_{\bft g} = -1 + H_g' \cdot H_g^{-1}.
\end{equation*}
This can be verified by direct computation.
\end{remark}

\begin{remark}\label{rema: invariant complex fat category extension}
    The invariant complex of a fat category extension is defined in exactly the same way as above.
\end{remark}

For convenience of the reader, we write the VB-complex of a VB-groupoid $V_G$ in terms of the fat groupoid $\widehat V_G$ (as the invariant complex) in the next section explicitly, and we prove this fact directly.

\subsection{From VB-groupoids to $2$-term ruths using the fat groupoid}\label{sec: From fat extensions to ruths using the fat groupoid}
From the previous section, Section \ref{sec: From fat extensions to ruths}, we see that there is a model for $C_{\textnormal{VB}} V_G$ in terms of the fat groupoid $\widehat V_G$. As mentioned there, it is a subcomplex
\begin{equation*}
C_\textnormal{inv}(\widehat V_G; V_M^* \ra C_M^*) \subset C(\widehat V_G; V_M^* \ra C_M^*)
\end{equation*}
of invariant cochains, but one can often view it as a subcomplex
\begin{equation*}
C_\textnormal{inv}(\widehat V_G; V_M^*) \subset C(\widehat V_G; V_M^*)
\end{equation*}
(if $V_M $ is not the zero bundle) which we sometimes use for notational convenience. We write here an explicit isomorphism between $C_\textnormal{VB} V_G$ and $C_\textnormal{inv}(\widehat V_G; V_M^* \ra C_M^*)$.

Consider an element $f$ of $C^\bullet_\textnormal{VB} V_G$. Then we can construct elements
\begin{equation*}
f_0 \in C^\bullet(\widehat V_G; V_M^*)\textnormal{ and } f_1 \in C^{\bullet-1}(\widehat V_G; C_M^*)
\end{equation*}
by setting
\begin{align*}
f_0(H_{g_1},\dots,H_{g_\bullet})v &= f(H_{g_1}(\bft H_{g_1})^{-1}v, g_2, \dots, g_\bullet)
\\
f_1(H_{g_2},\dots,H_{g_\bullet})c &= f(-c^{-1}, g_2, \dots, g_\bullet)
.
\end{align*}
Notice the resemblance with the equations that appeared in the alternative proof of Proposition \ref{prop: VB complex} in Section \ref{sec: Ruths as differential graded modules}.

\begin{proposition}\label{prop: cochain map to representation up to homotopy of fat groupoid}
The map
\begin{align*}
C_{\textnormal{VB}} V_G \ra C(\widehat V_G, V_M^* \ra C_M^*); \quad f \mapsto (f_0, f_1)
\end{align*}
is a cochain map. Moreover, it maps bijectively onto $C_\textnormal{inv}(\widehat V_G, V_M^* \ra C_M^*)$, i.e. the pairs $(f_0,f_1)$ such that 
\begin{equation}\label{eq: f_V_M relation to f_C_M 0}
f_1(H_{g_2},\dots,H_{g_\bullet}) = f_1(H_{g_2}',\dots,H_{g_\bullet}')
\end{equation}
and
\begin{equation}\label{eq: f_V_M relation to f_C_M 1}
f_0(H_{g_1},\dots,H_{g_\bullet})v - f_0(H_{g_1}',\dots,H_{g_\bullet}')v = f_1(H_{g_2},\dots,H_{g_\bullet}) h(H_{g_1}^{-1}, (H_{g_1}')^{-1})v
\end{equation}
where
\begin{equation*}
h(H_{g_1^{-1}},H_{g_1^{-1}}'): (V_M)_{\bft g_1} \ra (C_M)_{\bft g_2}
\end{equation*}
is given by
\begin{equation*}
h(H_{g_1^{-1}},H_{g_1^{-1}}')v = r_{g_1}(H_{g_1^{-1}}'v - H_{g_1^{-1}}v).
\end{equation*}
\end{proposition}
\begin{proof}
The map is a cochain map, for it can be seen as the pullback of the VB-groupoid map
\begin{equation*}
 (C_M \ra V_M) \rtimes \widehat V_G \ra V_G \qquad (v,H_g,c) \mapsto \ell_gc + H_g(\bft H_g)^{-1}v.
\end{equation*}
Notice that we view here $(C_M \ra V_M) \rtimes \widehat V_G$ as a VB-groupoid with $\bfs(v,H_g,c) = \partial c + H_g^{-1} \cdot v$, $\bft(v,H_g,c) = v$, and
\begin{equation*}
(v,H_g,c_1) \cdot (\partial c_1 + H_g^{-1} \cdot v, H_h, c_2) = (v, H_g \cdot H_h, H_h^{-1} \cdot c_1 + c_2).
\end{equation*}
We can recover $f$ from $f_0$ and $f_1$ as follows. Given $(v_1,\dots,v_\bullet) \in V_G^{(\bullet)}$ over $(g_1,\dots,g_\bullet)$, we take any $(H_{g_1},\cdots,H_{g_\bullet})$ and then we set
\begin{equation}\label{eq: f = f_V_M + f_C_M}
f(v_1, \dots, v_\bullet) = f_1(H_{g_2},\dots,H_{g_\bullet})\omega_{g_1}^{H_{g_1}}v_1 + f_0(H_{g_1},\dots,H_{g_\bullet})\bft v_1.
\end{equation}
The map
\begin{equation*}
\omega_{g_1}^{H_{g_1}}v = r_g(H_{g_1}^{-1} \bft v - v^{-1})
\end{equation*}
is as in Proposition \ref{prop: fat groupoid of dual}: the unique map $(V_G)_g \ra (C_M)_{\bfs g}$ such that $\omega_{g_1}^{H_{g_1}} H_{g_1} = 0$.
The right hand side of \eqref{eq: f = f_V_M + f_C_M} is independent of $H_{g_1},\dots,H_{g_\bullet}$. So, if we are given $f_0 \in C^\bullet(\widehat V_G; V_M^*)$ and $f_1 \in C^{\bullet-1}(\widehat V_G; C_M^*)$ such that the sum (as in the right hand side of the equation \eqref{eq: f = f_V_M + f_C_M} above) is independent of $H_{g_1},\dots,H_{g_\bullet}$, then $f \in C_{\textnormal{VB}}V_G$ can be defined by the equation \eqref{eq: f = f_V_M + f_C_M} above. This gives rise to a one-to-one correspondence, and this last condition is exactly the condition as written in the statement.
\end{proof}

As a corollary, we obtain:
\begin{corollary}\label{coro: cochain map to representation}
The map
\begin{equation*}
C_\textnormal{lin} V_G \ra C(\widehat V_G, V_M^*); \quad f \mapsto f_0
\end{equation*}
is a cochain map. If $V_M \neq 0$, then the restriction to $C^\bullet_{\textnormal{VB}} V_G$ maps bijectively onto $C_\textnormal{inv}(\widehat V_G; V_M^*)$, so those elements $f_0 \in C^\bullet(\widehat V_G, V_M^*)$ for which there exists $f_1 \in C^{\bullet-1}(\widehat V_G, C_M^*)$ such that \eqref{eq: f_V_M relation to f_C_M 0} and \eqref{eq: f_V_M relation to f_C_M 1} hold.\footnote{If $V_M=0$, then the map
\begin{equation*}
C^\bullet_\textnormal{lin} V_G \ra C^{\bullet-1}(\widehat V_G, C_M^*); \quad f \mapsto f_1
\end{equation*}
is a cochain map, and the restriction to $C^\bullet_{\textnormal{VB}} V_G$ maps bijectively onto the elements $f_1 \in C^{\bullet-1}(\widehat V_G, C_M^*)$ that are independent of its arguments.}
\end{corollary}

\begin{remark}\label{rema: dual version of VB-complex in terms of fat groupoid}
The VB-complex $C_\textnormal{VB} V_G^*$, in terms of $\widehat V_G$, is given by $C_\textnormal{inv}(\widehat V_G; C_M \ra V_M)$. The latter complex consist of those pairs $(f_0,f_1)$ such that
\begin{equation*}
f_1(H_{g_2},\dots,H_{g_\bullet}) = f_1(H_{g_2}',\dots,H_{g_\bullet}')
\end{equation*}
and
\begin{equation*}
f_0(H_{g_1},\dots,H_{g_\bullet}) - f_0(H_{g_1}',\dots,H_{g_\bullet}') = h(H_{g_1},H_{g_1}')f_1(H_{g_2},\dots,H_{g_\bullet})
\end{equation*}
where again
\begin{equation*}
h(H_{g_1},H_{g_1}'): (V_M)_{\bfs g_1} \ra (C_M)_{\bft g_1}
\end{equation*}
is given by
\begin{equation*}
h(H_{g_1},H_{g_1}')v = r_{g_1^{-1}}(H_{g_1}'v - H_{g_1}v).
\end{equation*}
\end{remark}

We complement the last remark by observing that the isomorphism
\begin{equation*}
C_\textnormal{inv}^\bullet(\widehat V_G; C_M \ra V_M) \ra C_\textnormal{VB} V_G^*
\end{equation*}
is particularly simple to write when using the projectable cochains:
\begin{equation*}
C_\textnormal{VB} V_G^* \cong C_\textnormal{proj} V_G
\end{equation*}
(see Section \ref{sec: Notation for VB-groupoids} for the definition).

\begin{proposition}\label{prop: relation invariant cochains to projectable cochains}
The complex $C_\textnormal{inv}(\widehat V_G; C_M)$ is isomorphic to $C_\textnormal{proj} V_G$ via the map
\begin{equation*}
C_\textnormal{inv}(\widehat V_G; C_M) \ra C_\textnormal{proj} V_G
\end{equation*}
given by
\begin{equation*}
c(g_1,\dots,g_\bullet) = r_{g_1}f_0(H_{g_1},g_2,\dots,g_\bullet) + H_{g_1}f_1(g_2,\dots,g_\bullet).
\end{equation*}
\end{proposition}
\begin{proof}
The map is $CG$-linear and a bijection, for we can read the equation the other way around (using $f_1 = \bfs c$). That the map is a cochain map, is a straightforward computation on generators.
\end{proof}

\begin{remark}
In this remark we will draw some parallels with the discussion presented here and the proofs of Proposition \ref{prop: VB complex} and Proposition \ref{prop: invariant complex iso to quotient complex}. In the discussion of this section and the previous one (Section \ref{sec: From fat extensions to ruths}), we employed exactly the strategy presented in the proof of Proposition \ref{prop: invariant complex iso to quotient complex}. Namely, in the proof of the latter proposition, we constructed an isomorphism $CH \xra{\sim} C_{K\textnormal{-rel}} G$ that is given by the pullback of a groupoid map $G \ra H$ (over the same base) which is a surjective submersion. In the above setup we described, we took one of the surjective VB-groupoid maps
\begin{equation*}
\widehat V_G \ltimes (C_M \ra V_M) \ra V_G \qquad (C_M \ra V_M) \rtimes \widehat V_G \ra V_G
\end{equation*}
In general, the inverse of $CH \xra{\sim} C_{K\textnormal{-rel}} G$ is given explicitly by using local sections of $G \ra H$; the choice of local sections is irrelevant. In our case, the local sections are determined by local sections of $\widehat V_G \ra G$, and this explains the appearance of them in the alternative proof of Proposition \ref{prop: VB complex}. The proof of Proposition \ref{prop: VB complex} using the cleavage is almost like the usage of a set-theoretic section in the proof of Proposition \ref{prop: invariant complex iso to quotient complex} in that it is a global version of the argument. Of course, for the proof of Proposition \ref{prop: VB complex} to work, we need a smooth section, and therefore we have to pass to $\widehat V_G^\textnormal{cat} \ra G$. That is, we drop the invertibility assumption on the values of the section and we need to work ``up to homotopy''. See also Section \ref{sec: Relation to Representations up to homotopy}, where we further clarify the step from fat extensions to the associated split $2$-term ruth.
\end{remark}

\section{Fat extensions and general linear PB-groupoids}\label{sec: PB-groupoids}

We dedicate this section to the correspondence between fat extension and general linear PB-groupoids. In \cite{PB-groupoids}, where the notion of PB-groupoid is introduced, it is already observed that the fat groupoid of a VB-groupoid is useful to construct the general linear PB-groupoid of a VB-groupoid. Following the previous section, Section \ref{sec: fat extensions}, it becomes clear that this should be interpreted as a composition of functors
\begin{equation*}
\{\textnormal{VB-groupoids}\} \ra \{\textnormal{Fat extensions}\} \ra \{\textnormal{General linear PB-groupoids}\}.
\end{equation*}
The functorial aspects are not discussed in \cite{PB-groupoids}, but we will do so in Section \ref{sec: functorial aspects}.

Although the correspondence works as in \cite{PB-groupoids}, we will go a slightly different route. The difference lies in that the general linear PB-groupoids we consider have a different structural strict Lie $2$-groupoid than in \cite{PB-groupoids}. We use the one associated to two vector bundles $C_M$ and $V_M$ instead of the one associated to two \textit{trivial} vector bundles. Moreover, the structure maps of the general linear strict Lie $2$-groupoid are slightly different than in \cite{PB-groupoids}. These changes are done to improve the exposition.

In Appendix \ref{app: double groupoids} we will discuss a related correspondence between a class of double groupoids that we call \textit{core-transitive}, and structures that we call \textit{core extensions}. The discussion presented in the appendix generalises \cite{Corediagram}, but can be seen as a clarification of the remarks made at the end of that work. General linear PB-groupoids come with a gauge double groupoid when the underlying complex is regular, and such double groupoids are core-transitive. The associated core extension is exactly the fat extension. At the end of this section we mention how to think of fat extensions as double groupoids, but the details are left for the appendix.

\subsection{PB-groupoids}
In analogy with VB-groupoids, a PB-groupoid is a diagram
\begin{center}
\begin{tikzcd}
    P_G \ar[r, shift left] \ar[r, shift right] \ar[d] & P_M \ar[d] \\
    G \ar[r, shift left] \ar[r, shift right] & M
\end{tikzcd}
\end{center}
where $P_G$ and $P_M$ are principal groupoid bundles over $G$ and $M$, respectively. The compatibility between the principal bundle structures and the groupoid structures involved can be phrased as follows. The groupoids acting on $P_G$ and on $P_M$ are both groupoids over the same base:
\begin{center}
\begin{tikzcd}
    & & \ar[dl, phantom, "\rotatebox{120}{\Large$\circlearrowright$}"] H_G \ar[dd, shift left] \ar[dd, shift right] \ar[rr, shift left] \ar[rr, shift right] & & \ar[dl, phantom, "\rotatebox{120}{\Large$\circlearrowright$}"] H_M \ar[dd, shift left] \ar[dd, shift right] \\
    & P_G \ar[rr, shift left, crossing over] \ar[rr, shift right, crossing over] \ar[rd] & & P_M \ar[rd] & \\
    & & N \ar[rr, shift left] \ar[rr, shift right] & & N \\
    G \ar[uur, leftarrow] \ar[rr, shift left] \ar[rr, shift right] & & M \ar[uur, crossing over, leftarrow] & & \\
\end{tikzcd}
\end{center}
and these groupoids fit into a double groupoid over the unit groupoid $1_N$:
\begin{center}
\begin{tikzcd}
    H_G \ar[d, shift left] \ar[d, shift right] \ar[r, shift left] \ar[r, shift right] & H_M \ar[d, shift left] \ar[d, shift right] \\
    N \ar[r, shift left] \ar[r, shift right] & N
\end{tikzcd}
\end{center}
That is, it is a strict Lie $2$-groupoid: after appropriate identifications, the structure maps of $H_G \rra H_M$ are groupoid maps over the identity (see also Remark \ref{rema: strict 2-groupoids} and Appendix \ref{app: double groupoids}).

Now, the compatibility between the actions in a PB-groupoid can be phrased as the actions defining a (principal) $2$-action of the structural strict Lie $2$-groupoid.

\begin{definition}\label{def: PB-groupoid}
Let $P_G \rra P_M$ be a Lie groupoid and let $H_G \rra H_M \rra N$ be a strict Lie $2$-groupoid. A \textit{$2$-action} of $H_G$ on $P_G$ consists of two groupoid actions
\begin{equation*}
P_G \textnormal{ } \rotatebox[origin=c]{90}{\large$\circlearrowright$} \textnormal{ } H_G \qquad P_M \textnormal{ } \rotatebox[origin=c]{90}{\large$\circlearrowright$} \textnormal{ } H_M
\end{equation*}
whose moment maps form a groupoid map
\begin{equation*}
P_G \ra 1_N
\end{equation*}
and whose action maps form a groupoid map
\begin{center}
\begin{tikzcd}
    P_G \times_N H_G \ar[d, shift left] \ar[d, shift right] \ar[r] & P_G \ar[d, shift left] \ar[d, shift right] \\
    P_M \times_N H_M \ar[r] & P_M
\end{tikzcd}
\end{center}
where the groupoid structure of $P_G \times_N H_G \rra P_M \times_N H_M$ is defined componentwise. A $2$-action $P_G$ $\rotatebox[origin=c]{90}{\large$\circlearrowright$}$ $H_G$ whose actions are principal is called a \textit{PB-groupoid}, and we say it is over $G \rra M$ if $P_G/H_G \rra P_M/H_M$ is $G \rra M$ (up to isomorphism).
\end{definition}

\begin{example}
    An interesting class of examples is when $H_G = H_M$ is a Lie group $G$. Such PB-groupoids appear in \cite{BruceGrabowskaGrabowski} and are called \textit{$G$-groupoids} there.
\end{example}

To end this small section, we conclude with some comments on strict Lie $2$-groupoids.

\begin{remark}\label{rema: strict 2-groupoids}
Given a strict Lie $2$-groupoid $H_2 \rra H_1 \rra N$, we can think of the elements $h \in H_2$ as abstract homotopy data between the source and target isomorphisms in $H_1$. In this analogy, $H_2$ is a groupoid in two compatible ways: there is one composition of homotopies which should be compared to the fact that being homotopic is a transitive relation, and the other composition of homotopies should be compared to Lemma \ref{lemm: homotopy composition}. In the main example of interest to us, the general linear strict Lie $2$-groupoid, the analogy can be taken literally (see also Remark \ref{rema: 2-groupoids}).

A more geometric perspective is that $H_2 \rra N$ is a Lie groupoid whose source and target fibers have a compatible internal groupoid structure. The foliation by source fibers form a \textit{family of Lie groupoids},\footnote{A family of Lie groupoids is a Lie groupoid together with a groupoid map that is a surjective submersion onto some unit groupoid.} and so does the foliation by target fibers. If two elements are composable in $H_2 \rra H_1$, then they are both elements of a single source fiber (of $H_2 \rra N$), and also of a single target fiber. Their product in $H_2 \rra H_1$, in turn, also lies in these fibers. Now, take two pairs of composable elements in $H_2 \rra H_1$, say, one pair in a source fiber over $x \in N$ and the other pair in a target fiber over $x$. Either, we multiply first in $H_2 \rra H_1$ and then in $H_2 \rra N$, or we multiply first in $H_2 \rra N$ and then in $H_2 \rra H_1$. The results are the same, and this compatibility between the two groupoid structures $H_2 \rra N$ and $H_2 \rra H_1$ is called the \textit{interchange law}.
\end{remark}

\subsection{The general linear strict Lie $2$-groupoid}

General linear PB-groupoids are PB-groupoids with structural strict Lie $2$-groupoid the general linear strict Lie $2$-groupoid. We introduce the general linear strict Lie $2$-groupoid of two vector bundles $C_M$ and $V_M$ next, but we do so in a slightly different way than in \cite{PB-groupoids}. We will comment on this in Remark \ref{rema: matrices out of elements of general linear strict Lie 2-groupoid}. Notice that, below, the groupoid
\begin{equation*}
\textnormal{GL}(C_M,V_M)_M = \textnormal{Hom}(C_M,V_M) \rtimes \textnormal{GL}(C_M,V_M)
\end{equation*}
from Remark \ref{rema: cochain complex representation as groupoid map} is used.

\begin{example}\label{exa: general linear 2-groupoid}
Consider the bundle of Lie groups
\begin{equation*}
\textnormal{Pert}(C_M,V_M) \rra \textnormal{Hom}(C_M,V_M)
\end{equation*}
from Remark \ref{rema: perturbation theory}; its multiplication is given by
\begin{equation*}
(h_x,\partial_x) \cdot (h_x',\partial_x) \coloneq (h_x + h_x' + h_x \partial_x h_x', \partial_x).
\end{equation*}
We then set
\begin{equation*}
    \textnormal{GL}(C_M,V_M)_G \coloneq \textnormal{Pert}(C_M,V_M) \ltimes \textnormal{GL}(C_M,V_M)_M
\end{equation*}
where $\textnormal{Pert}(C_M,V_M)$ acts (on the left) on $\textnormal{GL}(C_M,V_M)_M$ via
\begin{equation*}
    (\partial_y,h_y) \cdot (\partial_y,\Psi_{y,x}) = (\partial_y, (1+ h_y \partial_y)\Psi_{y,x}^{C_M}, (1+\partial_y h_y)\Psi_{y,x}^{V_M}).
\end{equation*}
Following Remark \ref{rema: strict 2-groupoids}, the elements $(h_y,\partial_y,\Psi_{y,x})$ of $\textnormal{GL}(C_M,V_M)_G$ can be interpreted as pointwise homotopy data: we can interpret
\begin{equation*}
    \bfs(h_y, \partial_y, \Psi_{y,x}) = (\partial_y, \Psi_{y,x}^{C_M}, \Psi_{y,x}^{V_M}) \qquad \bft(h_y, \partial_y, \Psi_{y,x}) = (\partial_y, (1+ h_y \partial_y)\Psi_{y,x}^{C_M}, (1+\partial_y h_y)\Psi_{y,x}^{V_M})
\end{equation*}
as isomorphisms of complexes with respect to the differentials
\begin{equation*}
    (\Psi_{y,x}^{V_M})^{-1}\partial_y \Psi_{y,x}^{C_M} \textnormal{ and } \partial_y
\end{equation*}
and they are homotopic via $h_y\Psi_{y,x}^{V_M}$. The two multiplications of the two groupoid structures that $\textnormal{GL}(C_M,V_M)_G$ carries can now both be seen as composition of homotopies (see Lemma \ref{lemm: homotopy composition}). The groupoid structure of
\begin{equation*}
\textnormal{GL}(C_M,V_M)_G \rra \textnormal{GL}(C_M,V_M)_M
\end{equation*}
has the above source and target maps, and
\begin{equation*}
    (h_y, \partial_y, \Psi_{y,x}) \cdot (h_y', \partial_y, \Psi_{y,x}') = (h_y + h_y' + h_y \partial_y h_y', \partial_y, \Psi_{y,x}').
\end{equation*}
The groupoid multiplication of
\begin{equation*}
\textnormal{GL}(C_M,V_M)_G \rra \textnormal{Hom}(C_M,V_M)
\end{equation*}
is again a type of composition of homotopies: we set $\bfs(h_y, \partial_y, \Psi_{y,x}) = (\Psi_{y,x}^{V_M})^{-1} \partial_y \Psi_{y,x}^{C_M}$ and $\bft(h_y, \partial_y, \Psi_{y,x}) = \partial_y$, and
\begin{equation*}
    (h_z, \partial_z, \Psi_{z,y}) \cdot (h_y, \partial_y, \Psi_{y,x}) \coloneq (h_z + \Psi_{z,y}^{C_M}h_y (\Psi_{z,y}^{V_M})^{-1} + h_z \partial_z \Psi_{z,y}^{C_M} h_y (\Psi_{z,y}^{V_M})^{-1}, \partial_z, \Psi_{z,y}\Psi_{y,x}).
\end{equation*}
Precomposing the homotopy on the right hand side with $\Psi_{z,y}^{V_M}\Psi_{y,x}^{V_M}$ is indeed composition of homotopies of $h_z \Psi_{z,y}^{V_M}$ and $h_y\Psi_{y,x}^{V_M}$. Another way of interpreting this multiplication is that it can be seen as being induced by the product of $\textnormal{Pert}(C_M,V_M)$:
\begin{equation*}
 (h_z, \partial_z, \Psi_{z,y}) \cdot (h_y, \partial_y, \Psi_{y,x}) = ((h_z,\partial_z) \cdot (\Psi_{z,y}^{C_M} h_y (\Psi_{z,y}^{V_M})^{-1},\partial_z), \Psi_{z,y}\Psi_{y,x}).
\end{equation*}
Taking $C_M = M \times \bbR^k$ and $V_M = M \times \bbR^\ell$, we write $\textnormal{GL}(k,\ell)_G \rra \textnormal{GL}(k,\ell)_M \rra \textnormal{Hom}(\bbR^k,\bbR^\ell)$ for the resulting strict Lie $2$-groupoid.
\end{example}

\begin{remark}\label{rema: general linear strict Lie 2-groupoid as crossed module}
    In essence, the above object is a strict Lie $2$-groupoid defined from a \textit{crossed module of Lie groupoids} as in \cite{MackenzieClassification,CrossedAndrouli,CamilleWagemann}. Crossed modules and strict Lie $2$-groupoids correspond to each other. In the above case, the structure of crossed module is encoded in the Lie groupoid map from $\textnormal{Pert}(C_M,V_M)$ to $\textnormal{GL}(C_M,V_M)_M$, and the canonical action of $\textnormal{GL}(C_M,V_M)_M$ on $\textnormal{Pert}(C_M,V_M)$ (which is a \textit{bundle of Lie groups representation} of $\textnormal{GL}(C_M,V_M)_M$ on $\textnormal{Pert}(C_M,V_M)$). We refer to Section \ref{sec: From general linear PB-groupoids to general linear double groupoids} and Appendix \ref{app: double groupoids} for more details.
\end{remark}

To lighten the notation, we often write $(h_y,\Psi_{y,x})$, or even just $(h,\Psi)$, for elements
\begin{equation*}
(h_y,\partial_y,\Psi_{y,x}) \in \textnormal{GL}(C_M,V_M)_G.
\end{equation*}
Similarly, we often write $\Psi = \Psi_{y,x} = (\partial_y,\Psi_{y,x}) \in \textnormal{GL}(C_M,V_M)_M$.

\begin{remark}\label{rema: matrices out of elements of general linear strict Lie 2-groupoid}
A useful way to depict elements $(h_y,\Psi_{y,x}) \in \textnormal{GL}(C_M,V_M)_G$ is using blockmatrices. Recall for this the observation we made in Remark \ref{rema: invertible homotopy bundle is semidirect product non-canonically}. There, we observed that an element $(\partial_x,h_x) \in \textnormal{Pert}(V_M,C_M)$ can be represented by a matrix 
\begin{equation*}
\begin{pmatrix}
1+h_x\partial_x & h_x \\ 0 & 1
\end{pmatrix} \in \textnormal{GL}(C_M \oplus V_M).
\end{equation*}
Notice that we keep track of the differential $\partial_x$ separately, but we suppress it from the notation. Then, elements of $\textnormal{GL}(C_M,V_M)_G$ can be represented by pairs of matrices
\begin{equation*}
\left(\begin{pmatrix}
1+h_y\partial_y & h_y \\ 0 & 1
\end{pmatrix},
\begin{pmatrix}
\Psi_{y,x}^{C_M} & 0 \\ 0 & \Psi_{y,x}^{V_M}
\end{pmatrix}\right).
\end{equation*}
This way, the groupoid structures of $\textnormal{GL}(C_M,V_M)_G$ are turned into matrix multiplications;
for example, the multiplication of $\textnormal{GL}(C_M,V_M)_G \rra \textnormal{Hom}(C_M,V_M)$ is given by
\begin{align*}
&\left(\begin{pmatrix}
1+ h_z\partial_z & h_z \\ 0 & 1
\end{pmatrix},
\begin{pmatrix}
\Psi_{z,y}^{C_M} & 0 \\ 0 & \Psi_{z,y}^{V_M}
\end{pmatrix}\right) \cdot \left(\begin{pmatrix}
1+ h_y\partial_y & h_y \\ 0 & 1
\end{pmatrix},
\begin{pmatrix}
\Psi_{y,x}^{C_M} & 0 \\ 0 & \Psi_{y,x}^{V_M}
\end{pmatrix}\right) \\
&= \left(\begin{pmatrix}
1+ h_z\partial_z & h_z \\ 0 & 1
\end{pmatrix} \begin{pmatrix}
\Psi_{z,y}^{C_M} & 0 \\ 0 & \Psi_{z,y}^{V_M}
\end{pmatrix} \begin{pmatrix}
1+ h_y\partial_y & h_y \\ 0 & 1
\end{pmatrix} \begin{pmatrix}
\Psi_{z,y}^{C_M} & 0 \\ 0 & \Psi_{z,y}^{V_M}
\end{pmatrix}^{-1},
\begin{pmatrix}
\Psi_{z,y}^{C_M} & 0 \\ 0 & \Psi_{z,y}^{V_M}
\end{pmatrix} \begin{pmatrix}
\Psi_{y,x}^{C_M} & 0 \\ 0 & \Psi_{y,x}^{V_M}
\end{pmatrix}\right).
\end{align*}
In \cite{PB-groupoids}, the correct matrix to consider does not have the ``$1+ h_x\partial_x$-factor'' in the upper left corner (of the matrices representing the elements of $\textnormal{Pert}(C_M,V_M)$). This is a consequence of the fact that the groupoid structure of $\textnormal{GL}(C_M,V_M)_G \rra \textnormal{Hom}(C_M,V_M)$ we use is slightly different. We find the convention presented above more suitable for our discussion, especially with an eye on the correspondence we will set with fat extensions.
\end{remark}

\begin{remark}\label{rema: representations general linear strict Lie 2-groupoid}
The previous remark already touched on it, but recall the following from \cite{PB-groupoids}: the strict Lie $2$-groupoid $\textnormal{GL}(C_M,V_M)_G$ comes with a ``$2$-representation'' (in \cite{PB-groupoids} called ``anchored $2$-representation'') on the $2$-term complex
\begin{equation*}
\partial \times 1: C_M \times_M \textnormal{Hom}(C_M,V_M) \ra V_M \times_M \textnormal{Hom}(C_M,V_M).
\end{equation*}
One way to put this is that there is a map of strict Lie $2$-groupoids
\begin{equation*}
\textnormal{GL}(C_M,V_M)_G \ra \textnormal{GL}(C_M \times_M \textnormal{Hom}(C_M,V_M),V_M \times_M \textnormal{Hom}(C_M,V_M))_G.
\end{equation*}
It is an embedding determined by the base map, which is the diagonal
\begin{equation*}
\textnormal{Hom}(C_M,V_M) \hookrightarrow \textnormal{Hom}(C_M,V_M) \times_M \textnormal{Hom}(C_M,V_M).
\end{equation*}
Following Proposition 3.10 of \cite{PB-groupoids}, it can also be seen as a linear $2$-action of $\textnormal{GL}(C_M,V_M)_G$ on the above $2$-term complex seen as a VB-groupoid over $1_{\textnormal{Hom}(C_M,V_M)}$:
\begin{equation*}
(C_M \times_M \textnormal{Hom}(C_M,V_M)) \ltimes (V_M \times_M \textnormal{Hom}(C_M,V_M)) \rra V_M \times_M \textnormal{Hom}(C_M,V_M)
\end{equation*}
(see Example \ref{exa: bundle of invertible homotopies as fat extension}). That $\textnormal{GL}(C_M,V_M)_G$ acts ``linearly'' on this VB-groupoid becomes clear when using the matrix notation: the (left) action is given by applying the linear transformation
\begin{equation*}
\begin{pmatrix}
1+ h_y\partial_y & h_y \\ 0 & 1
\end{pmatrix}
\begin{pmatrix}
\Psi_{y,x}^{C_M} & 0 \\ 0 & \Psi_{y,x}^{V_M}
\end{pmatrix} = \begin{pmatrix}
(1+ h_y\partial_y)\Psi_{y,x}^{C_M} & h_y\Psi_{y,x}^{V_M} \\ 0 & \Psi_{y,x}^{V_M}
\end{pmatrix} \in \textnormal{GL}(C_M \oplus V_M)
\end{equation*}
(and conjugating the differential). Moreover, $\textnormal{GL}(C_M,V_M)_M$ acts on $V_M \times_M \textnormal{Hom}(C_M,V_M)$ via
\begin{equation*}
\Psi_{y,x} \cdot v_x = \Psi_{y,x}^{V_M}v_x
\end{equation*}
(and conjugating the differential). Recall from Definition \ref{def: PB-groupoid} that the key feature of these two actions defining a $2$-action is that the action map is a groupoid map.\footnote{Implicitly, we use that the moment maps of the two actions are compatible.} Here, the groupoid
\begin{center}
\begin{tikzcd}
\textnormal{GL}(C_M,V_M)_G \times_{\textnormal{Hom}(C_M,V_M)} ((C_M \times_M \textnormal{Hom}(C_M,V_M)) \ltimes (V_M \times_M \textnormal{Hom}(C_M,V_M))) \ar[d, shift left] \ar[d, shift right] \\
\textnormal{GL}(C_M,V_M)_M \times_{\textnormal{Hom}(C_M,V_M)} (V_M \times_M \textnormal{Hom}(C_M,V_M))
\end{tikzcd}
\end{center}
is used, whose structure is defined componentwise. Notice that this groupoid is a VB-groupoid over $\textnormal{GL}(C_M,V_M)_G \rra \textnormal{GL}(C_M,V_M)_M$, and with respect to this structure the action is linear. These facts are readily verified using the matrix notation from above.

Using the differential $\partial$ as moment map, we see that there is also a $2$-representation of $\textnormal{GL}(C_M,V_M)_G$ on the $2$-term complex $C_M \ra V_M$. The linear $2$-action is defined in an analogous way as above, and it corresponds to the identity map $\textnormal{GL}(C_M,V_M)_G \xra{=} \textnormal{GL}(C_M,V_M)_G$.
\end{remark}

\subsection{General linear PB-groupoids}

In this section, we define what we mean with a general linear PB-groupoid. We also discuss some important properties that we will use.

\begin{definition}\label{def: general linear PB-groupoid}
Let $C_M$ and $V_M$ be vector bundles. An abstract general linear PB-groupoid with underlying vector bundles $C_M$ and $V_M$ is a PB-groupoid $P_G \rra P_M$ with structural strict Lie $2$-groupoid $\textnormal{GL}(C_M,V_M)_G \rra \textnormal{GL}(C_M,V_M)_M \rra \textnormal{Hom}(C_M,V_M)$. It is called trivial if
\begin{equation*}
P_M = \textnormal{GL}(C_M,V_M)
\end{equation*}
is the canonical principal $\textnormal{GL}(C_M,V_M)_M$-bundle. A trivial abstract general linear PB-groupoid is called a \textit{general linear PB-groupoid}.
\end{definition}

To be clear, the (right) principal $\textnormal{GL}(C_M,V_M)_M$-bundle structure on $\textnormal{GL}(C_M,V_M)$ is given by
\begin{equation*}
\Psi_{z,y} \cdot \Psi_{y,x} (= \Psi_{z,y} \cdot ((\Psi_{z,y}^{V_M})^{-1}\partial_z\Psi_{z,y}^{C_M}, \Psi_{y,x})) \coloneq \Psi_{z,y}\Psi_{y,x}.
\end{equation*}

To explain the definition, we want to emphasise two features of general linear PB-groupoids (Proposition \ref{prop: trivial abstract general linear PB-groupoid} and Remark \ref{rema: general linear PB-groupoid definition}).

\begin{remark}\label{rema: trivial abstract general linear PB-groupoid}
We first observe that the moment map of $P_M$ in an abstract general linear PB-groupoid $P_G \rra P_M$ is always given by conjugation in the following sense. Let $p \in P_M$, and let
\begin{equation*}
\Psi_{y,x} = (\partial_y,\Psi_{y,x}) \in \textnormal{GL}(C_M,V_M)_M
\end{equation*}
be composable with $p$. Then $1_y = (\partial_y,1_y)$ is composable with $p$, and therefore also any other $\Psi_{y,x}' = (\partial_y,\Psi_{y,x}') \in \textnormal{GL}(C_M,V_M)_M$ is. Since all $p' \in P_M$ with the same image as $p$ under the projection map $P_M \ra M$ are of the form
\begin{equation*}
p' = p \cdot \Psi_{y,x}',
\end{equation*}
the moment map is given on the fibers of the projection $P_M \ra M$ by conjugation of the differential $\partial_y$ by an element of $\textnormal{GL}(C_M,V_M)$.
\end{remark}

Now, to justify the terminology in the above definition, observe the following.

\begin{proposition}\label{prop: trivial abstract general linear PB-groupoid}
Suppose $P_G \rra P_M$ is an abstract general linear PB-groupoid. If $P_M \ra M$ has a global section, then $P_G$ is isomorphic to a (trivial abstract) general linear PB-groupoid.
\end{proposition}
\begin{proof}
If the projection map $P_M \ra M$ has a global section $\sigma$, then we can send $p \in P_M$ to the unique $\Psi(p) = (\partial_y, \Psi(p)) \in \textnormal{GL}(C_M,V_M)_M$ such that
\begin{equation*}
1_p \cdot (0_y, \Psi(p)^{-1}) = 1_{\sigma \pi p} \in P_G.
\end{equation*}
This turns the image of the section in $P_M$ into the unit section of $\textnormal{GL}(C_M,V_M) \ra M$. In turn, the moment map of $\textnormal{GL}(C_M,V_M)$ becomes the conjugation of a differential, and the action transforms into the canonical action of $\textnormal{GL}(C_M,V_M)_M$ on $\textnormal{GL}(C_M,V_M)$. This proves the statement.
\end{proof}

To be clear, from now on we always assume general linear PB-groupoids are trivial abstract general linear PB-groupoids. As expected, a fundamental example of a general linear PB-groupoid is given by a $2$-term complex $C_M \ra V_M$.

\begin{example}\label{exa: cochain complex defines PB-groupoid}
Given a $2$-term complex $C_M \ra V_M$, observe that $\textnormal{H}(C_M,V_M)$ (see Section \ref{sec: the bundle of invertible homotopies}) acts on $\textnormal{GL}(C_M,V_M)$ via its canonical representation. The Lie groupoid
\begin{equation*}
\textnormal{H}(V_M,C_M) \ltimes \textnormal{GL}(C_M,V_M) \rra \textnormal{GL}(C_M,V_M)
\end{equation*}
then has a structure of general linear PB-groupoid given by
\begin{equation*}
(h_z,\Psi_{z,y}) \cdot (h_y, \Psi_{y,x}) = (h_z \cdot (\Psi_{z,y}^{C_M} h_y (\Psi^{V_M}_{z,y})^{-1}), \Psi_{z,y} \Psi_{y,x}).
\end{equation*}
The moment map is given by conjugating the differential of $C_M \ra V_M$. Observe that, as is implicit from the notation, $\Psi_{z,y}^{C_M} h_y (\Psi^{V_M}_{z,y})^{-1}$ is an element of $\textnormal{H}(V_M,C_M)$. Indeed,
\begin{equation*}
1 + \partial\Psi_{z,y}^{C_M} h_y (\Psi^{V_M}_{z,y})^{-1} = \Psi_{z,y}^{V_M}(1+\partial_y h_y)(\Psi^{V_M}_{z,y})^{-1}
\end{equation*}
where $\partial_y$ is the (suppressed) differential in $(h_y, \Psi_{y,x}) = (h_y, \partial_y, \Psi_{y,x})$, which satisfies
\begin{equation*}
\partial_y = (\Psi_{z,y}^{V_M})^{-1}\partial\Psi_{z,y}^{C_M}
\end{equation*}
by assumption.
\end{example}
\begin{remark}\label{rema: 2-term complex as strict Lie 2-groupoid}
In the previous example, $\textnormal{H}(V_M,C_M) \ltimes \textnormal{GL}(C_M,V_M)$ simply inherits the above structure from the differential
\begin{equation*}
\partial: M \ra \textnormal{Hom}(C_M,V_M)
\end{equation*}
through the embedding of strict Lie $2$-groupoids (a $2$-representation)
\begin{equation*}
(\textnormal{H}(C_M,V_M) \ltimes \textnormal{GL}(C_M,V_M) \rra \textnormal{GL}(C_M,V_M) \rra M) \hookrightarrow \textnormal{GL}(C_M,V_M)_G.
\end{equation*}
As $\textnormal{H}(C_M,V_M)$, together with its representations, carries little information about the differential (see Proposition \ref{prop: bundle of Lie groups is locally trivial}), the strict Lie $2$-groupoid $\textnormal{H}(C_M,V_M) \ltimes \textnormal{GL}(C_M,V_M)$ itself does not carry information about the differential.
\end{remark}

The next fact points towards the relation between general linear PB-groupoids, fat extensions, VB-groupoids and $2$-term ruths. Namely, every general linear PB-groupoid over $C_M \ra V_M$ contains the above example:

\begin{proposition}\label{prop: PB-groupoid has H in it}
Let $P_G \rra \textnormal{GL}(C_M,V_M)$ be a general linear PB-groupoid with underlying complex $C_M \ra V_M$. Then\footnote{We used $1$ for both unit embeddings $\textnormal{GL}(C_M,V_M) \hookrightarrow P_G$ and $M \hookrightarrow \textnormal{GL}(C_M,V_M)$.}
\begin{equation*}
\textnormal{H}(V_M,C_M) \ltimes \textnormal{GL}(C_M,V_M) \hookrightarrow P_G \qquad (h_y,\Psi_{y,x}) \mapsto 1_{1_y} \cdot (h_y,\partial,\Psi_{y,x})
\end{equation*}
is an embedding of general linear PB-groupoids. Moreover, this map fits into a short exact sequence
\begin{center}
\begin{tikzcd}
1 \ar[r] & \textnormal{H}(V_M,C_M) \ltimes \textnormal{GL}(C_M,V_M) \ar[r] & P_G \ar[r] & G \ar[r] & 1
\end{tikzcd}
\end{center}
of Lie groupoids, where $P_G \ra G$ is the projection map.
\end{proposition}
\begin{proof}
That the map is an embedding of general linear PB-groupoids is readily verified and follows from the compatibility assumptions between the Lie groupoid structure of $P_G$ and its principal $\textnormal{GL}(C_M,V_M)_G$-structure. The last claim is a consequence of all elements in the kernel of $P_G \ra G$ being related to $1_{1_y} \in P_G$ by a unique element $(h_y,\partial_y,\Psi_{y,x}) \in \textnormal{GL}(C_M,V_M)_G$. Since then $\partial_y = \partial$ is the differential $C_M \ra V_M$, $h_y \in \textnormal{H}(V_M,C_M)$ and the result follows.
\end{proof}

\begin{remark}\label{rema: short exact sequence for a general linear PB-groupoid}
Notice that the short exact sequence of Lie groupoids
\begin{center}
\begin{tikzcd}
1 \ar[r] & \textnormal{H}(V_M,C_M) \ltimes \textnormal{GL}(C_M,V_M) \ar[r] & P_G \ar[r] & G \ar[r] & 1.
\end{tikzcd}
\end{center}
is not over the identity, but over the sequence
\begin{center}
\begin{tikzcd}
\textnormal{GL}(C_M,V_M) \ar[r, "="] & \textnormal{GL}(C_M,V_M) \ar[r, "\bft"] & M.
\end{tikzcd}
\end{center}
\end{remark}

A general linear PB-groupoid $P_G$ comes with a cochain complex representation on
\begin{equation*}
\partial \times 1: C_M \times_M \textnormal{GL}(C_M,V_M) \ra V_M \times_M \textnormal{GL}(C_M,V_M).
\end{equation*}
This cochain complex representation of $P_G$ can be described as follows:

\begin{definition}\label{def: cochain complex representation PB-groupoid}
Let $P_G$ be a general linear PB-groupoid. The cochain complex representation associated to $P_G \rra \textnormal{GL}(C_M,V_M)$ on
\begin{equation*}
\partial \times 1: C_M \times_M \textnormal{GL}(C_M,V_M) \ra V_M \times_M \textnormal{GL}(C_M,V_M)
\end{equation*}
is defined as follows. Given $p_g \in P_G$, $(c_{\bfs g}, \Psi_{\bfs g,x}) \in C_M \times_M \textnormal{GL}(C_M,V_M)$ and $(v_{\bfs g}, \Psi_{\bfs g,x}) \in V_M \times_M \textnormal{GL}(C_M,V_M)$, we set
\begin{equation*}
p_g \cdot (c_{\bfs g}, \Psi_{\bfs g,x}) \coloneq (\bft(p_g) \Psi_{\bfs g,x}^{-1}c_{\bfs g}, \bft p_g) \qquad p_g \cdot (c_{\bfs g}, \Psi_{\bfs g,x}) \coloneq (\bft(p_g) \Psi_{\bfs g,x}^{-1}c_{\bfs g}, \bft p_g).
\end{equation*}
\end{definition}

\begin{remark}\label{rema: cochain complex representation general linear PB-groupoid}
The cochain complex representation is ``$\textnormal{GL}(C_M,V_M)_M$-equivariant''. More precisely, $\textnormal{GL}(C_M,V_M)_M$ acts on the complex (on the right): if $(c_{\bfs g}, \Psi_{\bfs g,y}) \in C_M \times_M \textnormal{GL}(C_M,V_M)$, $(v_{\bfs g}, \Psi_{\bfs g,y}) \in V_M \times_M \textnormal{GL}(C_M,V_M)$ and $\Psi_{y,x} \in \textnormal{GL}(C_M,V_M)_M$, then
\begin{equation*}
(c_{\bfs g}, \Psi_{\bfs g,y}) \cdot \Psi_{y,x} = (c_{\bfs g}, \Psi_{\bfs g,y}\Psi_{y,x}) \qquad (v_{\bfs g}, \Psi_{\bfs g,y}) \cdot \Psi_{y,x} = (v_{\bfs g}, \Psi_{\bfs g,y}\Psi_{y,x}).
\end{equation*}
Given $p_g \in P_G$, $(c_{\bfs g}, \Psi_{\bfs g,y}) \in C_M \times_M \textnormal{GL}(C_M,V_M)$ and $\Psi_{y,x} \in \textnormal{GL}(C_M,V_M)_M$, we therefore have
\begin{equation*}
(p_g \cdot \Psi_{y,x}) \cdot ((c_{\bfs g}, \Psi_{\bfs g,y}) \cdot \Psi_{y,x}) = (p_g \cdot (c_{\bfs g}, \Psi_{\bfs g,y})) \cdot \Psi_{y,x},
\end{equation*}
and similarly for $(v_{\bfs g}, \Psi_{\bfs g,y}) \in V_M \times_M \textnormal{GL}(C_M,V_M)$. Because of this fact, the cochain complex representation is determined by the induced cochain complex representation of the subgroupoid
\begin{equation*}
F_G \coloneq \{p_g \in P_G \mid \bfs p_g = 1_{\bfs g}\} \subset P_G.
\end{equation*}
The suggestive notation is of course to emphasise that this subgroupoid, together with the above cochain complex representation, reveals the fat extension associated to a general linear PB-groupoid.
\end{remark}

Before we move on to give the correspondences, we make another remark which relates our setup of general linear PB-groupoids back to \cite{PB-groupoids}.

\begin{remark}\label{rema: general linear PB-groupoid definition}
The general linear groupoid $\textnormal{GL}(C_M,V_M) $, seen as a groupoid over $M$, contains the gauge groupoid of the principal $\textnormal{GL}(k,\ell)_M$-bundle
\begin{equation*}
\textnormal{Fr}(C_M,V_M) = \textnormal{Fr } C_M \times_M \textnormal{Fr } V_M = \{\phi_x^{C_M} \times \phi_x^{V_M}: \bbR^k \oplus \bbR^\ell \xra{\sim} (C_M)_x \oplus (V_M)_x \mid x \in M\},
\end{equation*}
if it exists, that is. The frame bundle $\textnormal{Fr}(C_M,V_M)$ has moment map given by conjugating the differential $\partial$ using the frames, and the action is the canonical one. This is completely analogous to how we defined the action of $\textnormal{GL}(C_M,V_M)_M$ on $\textnormal{GL}(C_M,V_M)$. In fact, the gauge groupoids of $\textnormal{GL}(C_M,V_M)$ and $\textnormal{Fr}(C_M,V_M)$ are equal and given by the cochain isomorphisms $\textnormal{Aut}(C_M \ra V_M)$ from Remark \ref{rema: GL of a cochain complex representation}
\begin{center}
\begin{tikzcd}
(C_M)_x \ar[r, "\partial"] \ar[d, "\Phi_{y,x}^{C_M}"] & (V_M)_x \ar[d, "\Phi_{y,x}^{V_M}"]\\ (C_M)_x \ar[r, "\partial"] & (V_M)_x
\end{tikzcd}
\end{center}
(which is smooth only when $\partial$ has constant rank). Now, as follows from \cite{PB-groupoids}, a $2$-term complex of vector bundles $C_M \ra V_M$ is actually equivalently described by a principal $\textnormal{GL}(k,\ell)_M$-bundle (using the associated vector bundle construction in the other direction). Instead of making this point again, we simply encode the $2$-term complex as a map on the gauge groupoid $\textnormal{GL}(C_M,V_M)$. Or rather, we reinterpret this data again as a type of general linear principal groupoid bundle:
\begin{equation*}
\textnormal{GL}(C_M,V_M) \textnormal{ \rotatebox[origin=c]{90}{\large$\circlearrowright$} } \textnormal{GL}(C_M,V_M)_M.
\end{equation*}
Then, the structure of general linear PB-groupoid over it will, as for PB-groupoids with structural strict Lie $2$-groupoid $\textnormal{GL}(k,\ell)_G \rra \textnormal{GL}(k,\ell)_M \rra \textnormal{Hom}(C_M,V_M)$, exactly describe the data of a fat extension. Moreover, whether we use $\textnormal{GL}(k,\ell)_G$ or $\textnormal{GL}(C_M,V_M)_G$, the constructions that appear in the equivalence will be defined in analogous ways.

If the gauge groupoids of $\textnormal{GL}(C_M,V_M)$ (or $\textnormal{Fr}(C_M,V_M)$) is smooth, we can also take a smooth gauge groupoid of the total space $P_G$ of a general linear PB-groupoid. The result is a double groupoid, and this gives rise to another way to encode the structure of a fat extension that will be explained in Section \ref{sec: From general linear PB-groupoids to general linear double groupoids}.
\end{remark}

We will first describe the correspondence between fat extensions and general linear PB-groupoids.

\subsection{From fat extensions to general linear PB-groupoids}\label{sec: From fat extensions to general linear PB-groupoids}

Given a fat extension $F_G$, the associated general linear PB-groupoid is given as follows. We define
\begin{equation*}
P_G \coloneq F_G \ltimes \textnormal{GL}(C_M,V_M)
\end{equation*}
where we used the canonical (left) action of $F_G$ on $\textnormal{GL}(C_M,V_M)$ by using the actions of $F_G$ on $C_M$ and $V_M$. We proceed to define the right $\textnormal{GL}(C_M,V_M)_G$-action on $P_G$ as another canonical action:
\begin{equation*}
(H_g, \Psi_{\bfs g, x}) \cdot (h_x, (\Psi_{\bfs g, x})^{-1} \partial \Psi_{\bfs g, x}, \Psi_{x, y}') \coloneq (H_g \cdot (\Psi_{\bfs g, x}^{C_M} h_x (\Psi_{\bfs g, x}^{V_M})^{-1}), \Psi_{\bfs g, x} \Psi_{x,y}').
\end{equation*} 
Using the explicit descriptions of all the structure involved, we obtain:
\begin{proposition}
Let $F_G$ be a fat extension. Then
\begin{equation*}
F_G \ltimes \textnormal{GL}(C_M,V_M) \rra \textnormal{GL}(C_M,V_M)
\end{equation*}
is a general linear PB-groupoid.
\end{proposition}

\begin{remark}\label{rema: similarities with construction of PB-groupoids}
Notice the similarities with the PB-groupoid associated to the fat groupoid as explained in \cite{PB-groupoids}.
Instead of defining the PB-groupoid
\begin{equation*}
F_G \ltimes \textnormal{GL}(C_M,V_M) \rra \textnormal{GL}(C_M,V_M)
\end{equation*}
we can set\footnote{Recall that $\textnormal{Fr}(C_M,V_M) = \textnormal{Fr } C_M \times_M \textnormal{Fr } V_M$.}
\begin{equation*}
F_G \ltimes \textnormal{Fr}(C_M,V_M) \rra \textnormal{Fr}(C_M,V_M).
\end{equation*}
In exactly the same way as above, this defines a PB-groupoid with structural strict Lie $2$-groupoid $\textnormal{GL}(k,\ell)_G \rra \textnormal{GL}(k,\ell)_M \rra \textnormal{Hom}(\bbR^k,\bbR^\ell)$. Roughly speaking, the shift in perspective is as the difference between $\textnormal{GL } E$ and $\textnormal{Fr } E$ for a vector bundle $E$. Instead of thinking of usual frames of $E$, namely frames modelled on Euclidean space, we think of frames modelled on its own fibers. We therefore need to consider a $\textnormal{GL}(E)$-symmetry instead of a $\textnormal{GL}(k)$-symmetry.
\end{remark}

\subsection{From general linear PB-groupoids to fat extensions}\label{sec: From general linear PB-groupoids to fat extensions}

The construction we did in the previous section, Section \ref{sec: From fat extensions to general linear PB-groupoids}, gives a clear hint on how to recover the fat extension from its general linear PB-groupoid $P_G \rra \textnormal{GL}(C_M,V_M)$. Namely, the subgroupoid we saw in Remark \ref{rema: cochain complex representation general linear PB-groupoid}
\begin{equation}\label{eq: representative set fat groupoid}
\{p_g \in P_G \mid \bfs p_g = 1_{\bfs g}\} \subset P_G
\end{equation}
turns out to form a unique representative set for the Lie groupoid
\begin{equation*}
F_G \coloneq P_G/\textnormal{GL}(C_M,V_M)_M \rra M.
\end{equation*}
We used here that $\textnormal{GL}(C_M,V_M)_M$ acts freely and properly, via the unit embedding, on $P_G$. As observed in Remark \ref{rema: cochain complex representation general linear PB-groupoid}, the description of $F_G$ using the representatives above makes it clear that the cochain complex representation descends to $F_G$: given $H_g = [p_g] \in F_G$, where $\bfs p_g = 1_{\bfs g} \in \textnormal{GL}(C_M,V_M)$, then the cochain complex representation is given by
\begin{equation*}
F_G \ra \textnormal{GL}(C_M, V_M)_M \qquad H_g \mapsto \bft p_g.
\end{equation*}
The surjective submersion $P_G \ra G$ is a groupoid map, and the induced groupoid map $F_G \ra G$ is still a surjective submersion. 
We therefore obtain a short exact sequence of groupoids
\begin{center}
\begin{tikzcd}
1 \ar[r] & \textnormal{H} \ar[r] & F_G \ar[r] & G \ar[r] & 1.
\end{tikzcd}
\end{center}

The cochain complex representation of $F_G$ induces a cochain complex representation of $\textnormal{H}$, and then we conclude with the following corollary of Proposition \ref{prop: PB-groupoid has H in it}:

\begin{corollary}\label{cor: fat extension of a PB-groupoid H}
The embedding $\textnormal{H}(V_M,C_M) \hookrightarrow P_G$ induced by Proposition \ref{prop: PB-groupoid has H in it} is a Lie groupoid isomorphism onto $\textnormal{H}$. Moreover, this map is equivariant with respect to the cochain complex representations.\footnote{We consider here the cochain complex representation on $\textnormal{H} \subset F_G$ that we constructed above, and the canonical cochain complex representation of $\textnormal{H}(V_M,C_M)$.}
\end{corollary}
\begin{proof}
Its inverse is given as follows. Given $h_x = [p_x] \in \textnormal{H}$ (with $\bfs p_x = 1_x$), denote by $\Phi h_x$ the unique element in $\textnormal{H}(V_M,C_M)$ such that
\begin{equation*}
h_x \cdot (\Phi h_x, \partial, 1_x)^{-1} = 1_{1_x} \in P_G.
\end{equation*}
Then, the map
\begin{equation*}
\Phi: \textnormal{H} \ra \textnormal{H}(V_M,C_M)
\end{equation*}
is the desired inverse.
\end{proof}

It is readily verified that we set a true one-to-one correspondence between fat extensions and general linear PB-groupoids:

\begin{proposition}\label{prop: one to one correspondence fat extensions and PB-groupoids}
The assignment
\begin{equation*}
\{\textnormal{Fat extensions}\} \ra \{\textnormal{General linear PB-groupoids}\} \qquad F_G \mapsto F_G \ltimes \textnormal{GL}(C_M,V_M)
\end{equation*}
sets a one-to-one correspondence, and the assignment
\begin{equation*}
\{\textnormal{General linear PB-groupoids}\} \ra \{\textnormal{Fat extensions}\} \qquad P_G \mapsto P_G/\textnormal{GL}(C_M,V_M)_M.
\end{equation*}
is inverse to it (up to isomorphisms).
\end{proposition}

We find it worthwhile to take a moment to also describe the correspondences with VB-groupoids and $2$-term ruths. This is what we discuss next, starting with VB-groupoids.

\subsection{From VB-groupoids to general linear PB-groupoids}\label{sec: From VB-groupoids to general linear PB-groupoids}

As we use a different strict structural Lie $2$-groupoid $\textnormal{GL}(C_M,V_M)_G \rra \textnormal{GL}(C_M,V_M)_M \rra \textnormal{Hom}(C_M,V_M)$ for general linear PB-groupoids, the correspondence between VB-groupoids and general linear PB-groupoids looks slightly different than in \cite{PB-groupoids}. But the shift in perspective is analogous to the one from the previous correspondence between fat extensions and general linear PB-groupoids. Instead of looking at (adapted) frames of a VB-groupoid $V_G$ modelled on $\bbR^k \oplus \bbR^\ell$, we will consider ``frames'' modelled on fibers of $C_M \oplus V_M$:
\begin{equation*}
\phi_g: (C_M)_x \oplus (V_M)_x \xra{\sim} (V_G)_g.
\end{equation*}
In fact, from the previous section, Section \ref{sec: From general linear PB-groupoids to fat extensions}, we can infer what maps we should really be considering: a composition of an isomorphism defined by an element $H_g \in \widehat V_G$:
\begin{equation*}
(C_M)_{\bft g} \oplus (V_M)_{\bfs g} \xra{\sim} (V_G)_g \qquad (c,v) \mapsto r_gc + H_g v
\end{equation*}
with two linear isomorphisms
\begin{equation*}
(C_M)_x \xra{\sim} (C_M)_{\bft g} \qquad (V_M)_x \xra{\sim} (V_M)_{\bfs g}.
\end{equation*}
The collection of such elements form the general linear PB-groupoid
\begin{equation*}
\textnormal{GL}_\bfs V_G \rra \textnormal{GL}(C_M,V_M)
\end{equation*}
of interest. This is analogous to how one defines the (adapted) frame bundle of a VB-groupoid in \cite{PB-groupoids}. More precisely, we define this general linear PB-groupoid, without using the fat groupoid, as follows.\footnote{See for this the older versions of the paper \cite{PB-groupoids} on the arXiv.} We set
\begin{equation*}
\textnormal{GL}_\bfs V_G = \{\phi_g: (C_M)_x \oplus (V_M)_x \xra{\sim} (V_G)_g \mid \bfs_g\phi_g|_{(C_M)_x} = 0, \textnormal{ } \bft_g\phi_g|_{(V_M)_x}, \textnormal{ } \bfs_g\phi_g|_{(V_M)_x} \in \textnormal{GL}(V_M)\}
\end{equation*}
which will be the total groupoid $P_G$ of the general linear PB-groupoid associated to $V_G$. To define the structure maps of this groupoid, we use the following diagram:
\begin{center}
\begin{tikzcd}
(C_M)_x \ar[d, hookrightarrow] & (C_M)_{\bfs g} \ar[rd, "\ell_g"] & & \ar[ld, "r_g"'] (C_M)_{\bft g} \ar[dd, "\partial"] \\
(C_M)_x \oplus (V_M)_x \ar[rr, "\phi_g"] \ar[d, hookleftarrow] & & \ar[ld, "\bfs"] (V_G)_g \ar[rd, "\bft"'] & \\
(V_M)_x & (V_M)_{\bfs g} & & (V_M)_{\bft g}
\end{tikzcd}
\end{center}
Now, we construct four maps that will describe the source and target map of $\textnormal{GL}_\bfs V_G \rra \textnormal{GL}(C_M,V_M)$. Namely, given $\phi_g \in \textnormal{GL}_\bfs V_G$, we write
\begin{equation*}
\phi_{\bft g}^{C_M} \coloneq r_{g^{-1}}\phi_g|_{(C_M)_x} \qquad \phi_{\bft g}^{V_M} \coloneq \bft_g\phi_g|_{(V_M)_x} \qquad \phi_{\bfs g}^{V_M} \coloneq \bfs_g\phi_g|_{(V_M)_x}
\end{equation*}
and\footnote{Notice that
\begin{equation*}
\ell_g^{-1}v = -r_g(v^{-1})
\end{equation*}
where $v \in \ker \bft$.}
\begin{equation*}
\phi_{\bfs g}^{C_M} \coloneq \ell_g^{-1}\phi_g (1,-(\phi_{\bft g}^{V_M})^{-1} \partial \phi_{\bft g}^{C_M}).
\end{equation*}
We then set $\bfs \phi_g = (\phi_{\bfs g}^{C_M}, \phi_{\bfs g}^{V_M})$, $\bft \phi_g = (\phi_{\bft g}^{C_M}, \phi_{\bft g}^{V_M})$, and
\begin{equation*}
\phi_g \cdot \phi_h = r_h\phi_g|_{(C_M)_x} + \phi_g|_{(V_M)_x} \cdot \phi_h|_{(V_M)_x} \qquad \phi_g^{-1} = -(\phi_g(1,-1-(\phi_{\bft g}^{V_M})^{-1} \partial \phi_{\bft g}^{C_M}))^{-1}.
\end{equation*}

Realising elements $(h_y,\Psi_{y,x})$ of $\textnormal{GL}(C_M,V_M)_G$ as block matrices
\begin{equation*}
\begin{pmatrix}
(1+h_y\partial_y)\Psi_{y,x}^{C_M} & h_y\Psi_{y,x}^{V_M} \\ 0 & \Psi_{y,x}^{V_M}
\end{pmatrix} \in \textnormal{GL}(C_M \oplus V_M)
\end{equation*}
the (right) action of $\textnormal{GL}(C_M,V_M)_G \rra \textnormal{Hom}(C_M,V_M)$ on $\textnormal{GL}_\bfs V_G$ is given by matrix application:
\begin{equation*}
\begin{pmatrix} \phi|_{(C_M)_y} & \phi|_{(V_M)_y} \end{pmatrix} \begin{pmatrix}
(1+h_y\partial_y)\Psi_{y,x}^{C_M} & h_y\Psi_{y,x}^{V_M} \\ 0 & \Psi_{y,x}^{V_M}
\end{pmatrix} = \begin{pmatrix}
\phi(1+h_y\partial_y)\Psi_{y,x}^{C_M} & \phi h\Psi_{y,x}^{V_M} + \phi\Psi_{y,x}^{V_M}
\end{pmatrix}.
\end{equation*}
This matrix showed up in Remark \ref{rema: representations general linear strict Lie 2-groupoid} where we discussed the canonical $2$-representations of $\textnormal{GL}(C_M,V_M)_G$. For similar reasons, the action defined above defines a $2$-action, and this is the desired principal $2$-action. We therefore obtain:
\begin{proposition}\label{prop: PB-groupoid of a VB-groupoid}
Let $V_G$ be a VB-groupoid. Then
\begin{equation*}
\textnormal{GL}_\bfs V_G \rra \textnormal{GL}(C_M,V_M)
\end{equation*}
is a general linear PB-groupoid.
\end{proposition}

\begin{remark}
Using the fat groupoid to describe this general linear PB-groupoid, i.e. seeing it as an action groupoid
\begin{equation*}
\widehat V_G \ltimes \textnormal{GL}(C_M,V_M) \rra \textnormal{GL}(C_M,V_M)
\end{equation*}
describes the composition of functors
\begin{equation*}
\{\textnormal{VB-groupoids}\} \ra \{\textnormal{Fat extensions}\} \ra \{\textnormal{General linear PB-groupoids}\}.
\end{equation*}
As we already mentioned, this realises elements $\phi \in \textnormal{GL}_\bfs V_G$ as compositions of maps
\begin{equation*}
r_g\phi_{\bft g}^{C_M} + H_g\phi_{\bfs g}^{V_M}: (C_M)_x \oplus (V_M)_x \xra{\sim} (C_M)_{\bft g} \oplus (V_M)_{\bfs g} \xra{\sim} (V_G)_g.
\end{equation*}
Working with $\textnormal{GL}(k,\ell)_G$ instead of $\textnormal{GL}(C_M,V_M)_G$ (so with $\textnormal{Fr}(C_M,V_M)$ instead of $\textnormal{GL}(C_M,V_M)$), this action groupoid was used in \cite{PB-groupoids}.
\end{remark}

\subsection{From general linear PB-groupoids to VB-groupoids}

Following \cite{PB-groupoids}, it should not come as a surprise now that we can recover a VB-groupoid from its associated general linear PB-groupoid using a type of associated bundle construction. However, since we passed from $\textnormal{Fr}(C_M,V_M)$ to $\textnormal{GL}(C_M,V_M)$, the construction will especially be reminiscent to the association of a VB-groupoid from a fat extension.

Now, given a general linear PB-groupoid $P_G$, recall from Definition \ref{def: cochain complex representation PB-groupoid} that we can view
\begin{equation*}
\partial \times 1: C_M \times_M \textnormal{GL}(C_M,V_M) \ra V_M \times_M \textnormal{GL}(C_M,V_M)
\end{equation*}
as a cochain complex representation of $P_G$. Therefore, it comes with an associated VB-groupoid that we simply denote by
\begin{equation*}
P_G \ltimes (C_M \ra V_M) = C_M \times_M P_G \times_M V_M \rra V_M \times_M \textnormal{GL}(C_M,V_M).
\end{equation*}
Explicitly, the structure maps are as follows: the source and target maps are given, respectively, by $\bfs(c_{\bft g},p_g,v_{\bfs g}) = (v_{\bfs g}, \bfs p_g)$, $\bft(c_{\bft g},p_g,v_{\bfs g}) = (\partial c_{\bft g} + p_g \cdot v_{\bfs g}, \bft p_g)$ and
\begin{equation*}
(c_{\bft g}, p_g, \partial c_{\bfs g} + p_{g'} \cdot v_{\bfs g'}) \cdot (c_{\bfs g}, p_{g'}, v_{\bfs g'}) = (c_{\bft g} + p_g \cdot c_{\bfs g}, p_g \cdot p_{g'}, v_{\bfs g'}).
\end{equation*}
The VB-groupoid inherits a natural free and proper (right $2$-) action of $\textnormal{GL}(C_M,V_M)_G$:
\begin{equation*}
(c_{\bft g},p_g,v_{\bfs g}) \cdot (h_y, \Psi_{y,x}) = (c_{\bft g} - p_g \cdot (\bfs(p_g) h_y\bfs(p_g)^{-1}v_{\bfs g}), p_g \cdot (h_y, \Psi_{y,x}), v_{\bfs g}).
\end{equation*}
Since the action indeed preserves the linear structure, the result $V(P_G)$ is a VB-groupoid over $V_M$.

\begin{proposition}\label{prop: essential surjectivity of GL of VB-groupoid}
The assignment
\begin{equation*}
\{\textnormal{VB-groupoids}\} \ra \{\textnormal{General linear PB-groupoids}\} \qquad V_G \mapsto \textnormal{GL}_\bfs V_G
\end{equation*}
sets a one-to-one correspondence, and the assignment
\begin{equation*}
\{\textnormal{General linear PB-groupoids}\} \ra \{\textnormal{VB-groupoids}\} \qquad P_G \mapsto V(P_G)
\end{equation*}
is inverse to it (up to canonical isomorphisms).
\end{proposition}

\begin{remark}
The above quotient can be taken in steps. Indeed, we can take a quotient of $P_G \ltimes (C_M \ra V_M)$ by $\textnormal{GL}(C_M,V_M)_M$, and then by $\textnormal{H}(C_M,V_M) \ltimes V_M$ (see Section \ref{sec: From fat extensions to VB-groupoids} and \eqref{eq: right action by elements of the bundle of invertible homotopies} for the explicit formula of this action). This way, we can interpret the first quotient as recovering $F_G \ltimes (C_M \ra V_M)$, and then the second quotient is as in Section \ref{sec: From fat extensions to VB-groupoids}. Doing this, we described the composition of functors
\begin{equation*}
\{\textnormal{General linear PB-groupoids}\} \ra \{\textnormal{Fat extensions}\} \ra \{\textnormal{VB-groupoids}\}.
\end{equation*}
\end{remark}

\subsection{The abstract ruth of a general linear PB-groupoid}

We conclude by setting the correspondence between general linear PB-groupoids and $2$-term ruths. 
Consider the abstract ruth associated to $P_G \ltimes (V_M^* \ra C_M^*)$, so
\begin{equation*}
C^\bullet(P_G; C_M \ra V_M) = C^\bullet(P_G; C_M) \oplus C^{\bullet-1}(P_G; V_M).
\end{equation*}
Note the abuse of notation; it would be more correct to write the complex as
\begin{equation*}
C^\bullet(P_G; C_M \times_M \textnormal{GL}(C_M,V_M) \ra V_M \times_M \textnormal{GL}(C_M,V_M)).
\end{equation*}
But if $(f_0,f_1)$ is a cochain in this complex, then the $\textnormal{GL}(C_M,V_M)$ components of $f_0(p_1,\dots,p_\bullet)$ and $f_0(p_2,\dots,p_\bullet)$ are equal to $\bft p_1$ and $\bft p_2$, respectively, so we simplify the notation.

If we argue through the language of fat extensions or VB-groupoids, for example, using Proposition \ref{prop: invariant complex iso to quotient complex}, we see that the following subcomplex should be considered:
\begin{definition}\label{def: equivariant cochains general linear PB-groupoids}
The \textit{$\textnormal{GL}$-equivariant cochains}
\begin{equation*}
(f_0,f_1) \in C(P_G; C_M \ra V_M) = C^\bullet(P_G; C_M) \oplus C^{\bullet-1}(P_G; V_M)
\end{equation*}
are those satisfying, for all $(p_{g_1},\dots,p_{g_\bullet}) \in P^{(\bullet)}$,
\begin{equation*}
f_1(p_{g_2},\dots,p_{g_\bullet}) = f_1(p_{g_2}',\dots,p_{g_\bullet}') 
\end{equation*}
and, for all $(h_y,\Psi_{y,x}) \in \textnormal{GL}(C_M,V_M)_G$,
\begin{equation*}
f_0(p_{g_1} \cdot (h_y,\Psi_{y,x}), \dots, p_{g_\bullet} \cdot (h_y,\Psi_{y,x})) = f_0(p_{g_1}, \dots, p_{g_\bullet}) + p_{g_1} \cdot (\bfs p_{g_1} h_y (\bfs p_{g_1})^{-1} f_1(p_{g_2}, \dots, p_{g_\bullet})).
\end{equation*}
The resulting subcomplex is denoted by $C_\textnormal{GL}(P_G; C_M \ra V_M)$ (or $C_\textnormal{GL}(P_G; C_M)$).
\end{definition}

\begin{remark}\label{rema: dropping p from notation gl euivariant cochains}
Given $(f_0,f_1) \in C_\textnormal{GL}(P_G; C_M)$, we write
\begin{equation*}
f_0(p_{g_1},g_2,\dots,g_\bullet) = f_0(p_{g_1},p_{g_2},\dots,p_{g_\bullet}) \qquad f_1(\bft(p_{g_2}), g_2,\dots,g_\bullet) = f_1(p_{g_2},\dots,p_{g_\bullet}).
\end{equation*}
The above condition then reads
\begin{equation*}
f_0(p_{g_1} \cdot (h_y,\Psi_{y,x}), g_2, \dots, g_\bullet) = f_0(p_{g_1}, g_2, \dots, g_\bullet) + p_{g_1} \cdot (\bfs p_{g_1} h_y (\bfs p_{g_1})^{-1} f_1(\bfs(p_{g_1}), g_2, \dots, g_\bullet)).
\end{equation*}
\end{remark}

From our previous discussions, we can infer that the above abstract ruth is isomorphic to $C_\textnormal{VB} V(P_G)^*$ and to $C_\textnormal{inv}(F_G; C_M)$. For example: the PB-groupoid associated to a fat extension $F_G$ is given by
\begin{equation*}
P_G = F_G \ltimes \textnormal{GL}(C_M,V_M).
\end{equation*}
The representation on $P_G$ is given by the pullback along the canonical groupoid map
\begin{equation*}
P_G \ra F_G.
\end{equation*}
This pullback then sets the desired isomorphism.

\begin{proposition}\label{prop: ruth of PB-groupoid}
Let $F_G$ be a fat extension of $G$ over $C_M \ra V_M$. The pullback
\begin{equation*}
C^\bullet(F_G; C_M \ra V_M) \ra C^\bullet(F_G \ltimes \textnormal{GL}(C_M,V_M); C_M \ra V_M)
\end{equation*} 
restricts to an isomorphism
\begin{equation*}
C^\bullet_\textnormal{inv}(F_G; C_M) \xra{\sim} C^\bullet_\textnormal{GL}(F_G \ltimes \textnormal{GL}(C_M,V_M); C_M).
\end{equation*}
\end{proposition}

\begin{remark}\label{rema: explanation of terminology}
To explain the terminology ``\textnormal{GL}-equivariance'', notice that we can conjugate elements $(f_0,f_1) \in C_\textnormal{GL}(P_G; C_M)$ to
\begin{align*}
\widetilde f_0(p_{g_1},g_2,\dots,g_\bullet) &\coloneq \bft(p_{g_1})^{-1}f_0(p_{g_1},g_2,\dots,g_\bullet) \\ \widetilde f_1(\Psi_{\bft g_2,x},g_2,\dots,g_\bullet) &\coloneq \Psi_{\bft g_2,x}^{-1}f_1(\Psi_{\bft g_2,x}, g_2,\dots,g_\bullet).
\end{align*}
We can think of $(\widetilde f_0, \widetilde f_1)$ as the result of postcomposing $(f_0,f_1)$ with the evaluation/action map
\begin{center}
\begin{tikzcd}
C_M \times_M \textnormal{GL}(C_M,V_M) \ar[r] \ar[d] & V_M \times_M \textnormal{GL}(C_M,V_M) \ar[d] \\
C_M\ar[r] & V_M
\end{tikzcd}
\end{center}
The conditions of Definition \ref{def: equivariant cochains general linear PB-groupoids} in terms of $(\widetilde f_0, \widetilde f_1)$ are as follows:
\begin{equation*}
\begin{pmatrix}
\widetilde f_0(p_{g_1} \cdot (h_y, \Psi_{y,x}), g) \\ \widetilde f_1(\bfs(p_{g_1} \cdot (h_y, \Psi_{y,x})), g)
\end{pmatrix} = \begin{pmatrix} (\Psi_{y,x}^{C_M})^{-1}(1+h_y\partial_y)^{-1} & -(\Psi_{y,x}^{V_M})^{-1}h_y^{-1} \\ 0 & (\Psi_{y,x}^{V_M})^{-1} \end{pmatrix} \begin{pmatrix}
\widetilde f_0(p_{g_1}, g) \\ \widetilde f_1(\bfs p_{g_1}, g)
\end{pmatrix}.
\end{equation*}
The result is therefore a complex $\widetilde C_\textnormal{GL}(P_G; C_M)$ consisting of ``\textnormal{GL}-equivariant'' cochains $(\widetilde f_0, \widetilde f_1)$ with respect to the $2$-actions of $\textnormal{GL}(C_M,V_M)_G$ on $P_G$ and on $C_M \ltimes V_M \rra V_M$ (see Remark \ref{rema: representations general linear strict Lie 2-groupoid}). In terms of this complex, the differential is given by
\begin{align*}
\delta \widetilde f_0(p_{g_1},g_2,\dots,g_{\bullet+1}) &= (-1)^\bullet (\widetilde f_0(p_{g_2}, g_3,\dots, g_{\bullet+1}) - \widetilde f_0(p_{g_1} \cdot p_{g_2}, g_3,\dots,g_{\bullet+1})) \\
&+ (-1)^\bullet\textstyle\sum_{k=2}^{\bullet+1} (-1)^k d_k^* \widetilde f_0(p_{g_1},g_2,\dots,g_{\bullet+1})
\end{align*}
where $p_{g_2}$ is arbitrary and composable with $p_{g_1}$ (and similarly for $\delta \widetilde f_1$).
\end{remark}

\subsection{From general linear PB-groupoids to general linear double groupoids}\label{sec: From general linear PB-groupoids to general linear double groupoids}

There is yet another way to encode the data of a fat extension. This becomes clear from the fact that general linear PB-groupoids sometimes come with a \textit{gauge double groupoid}.

\begin{remark}\label{rema: PB-groupoid come with a gauge double groupoid}
Recall from Remark \ref{rema: general linear PB-groupoid definition} that the groupoid of cochain isomorphisms
\begin{equation*}
\textnormal{Aut}(C_M \ra V_M) \rra M
\end{equation*}
can be seen as the gauge groupoid of the principal groupoid bundle $\textnormal{GL}(C_M,V_M)$ (or $\textnormal{Fr}(C_M,V_M)$). Given a general linear PB-groupoid $P_G$, what if we also take the gauge groupoid of the principal $\textnormal{GL}(C_M,V_M)_G$-bundle $P_G$?

In general, given a PB-groupoid $P_G \rra P_M$ with structural strict Lie $2$-groupoid $H_G \rra H_M \rra N$, then $P_G$ comes with a (not necessarily smooth) \textit{gauge double groupoid}
\begin{center}
\begin{tikzcd}
P_G \times_{H_G} P_G \ar[d, shift left] \ar[d, shift right] \ar[r, shift left] \ar[r, shift right] & P_M \times_{H_M} P_M \ar[d, shift left] \ar[d, shift right] \\ G \ar[r, shift left] \ar[r, shift right] & M
\end{tikzcd}
\end{center}
Doing this for a general linear PB-groupoid (also if we take a general linear PB-groupoid as in \cite{PB-groupoids}), we obtain a double groupoid whose vertical base groupoid is given by $\textnormal{Aut}(C_M \ra V_M)$. A natural condition to put on the PB-groupoid so that the above gauge groupoids are smooth is to assume the moment maps to be surjective submersions onto an embedded submanifold of $N$.\footnote{By the properties of the moment maps, one moment map has this property if the other does.} This happens for general linear PB-groupoids precisely when the differential $\partial$ is of constant rank.
\end{remark}

Now, to be more precise, we introduce a special class of double groupoids. First of all, we use the following notation for double groupoids:

\begin{center}
\begin{tikzcd}
    G \ar[d, shift left] \ar[d, shift right] \ar[r, shift left] \ar[r, shift right] & G_\da \ar[d, shift left] \ar[d, shift right] \\
    G_\ra \ar[r, shift left] \ar[r, shift right] & M
\end{tikzcd}
\end{center}

Here, $G_\ra$ and $G_\da$ are Lie groupoids, and we denote their structure maps using the same arrow-index. Moreover, $G$ is a Lie groupoid in two ways, over $G_\ra$ and over $G_\da$, and we denote the groupoid structure maps of $G \rra G_\da$ with indices $\ra$ and $G \rra G_\ra$ with indices $\da$.

\begin{definition}\label{def: core-transitive double groupoids}
A \textit{double groupoid} is a diagram of groupoids
\begin{center}
\begin{tikzcd}
    G \ar[d, shift left] \ar[d, shift right] \ar[r, shift left] \ar[r, shift right] & G_\da \ar[d, shift left] \ar[d, shift right] \\
    G_\ra \ar[r, shift left] \ar[r, shift right] & M
\end{tikzcd}
\end{center}
such that the structure maps of $G \rra G_\ra$ are groupoid maps over the respective structure maps of $G_\da \rra M$.\footnote{Equivalently, the structure maps of $G \rra G_\da$ are groupoid maps over the respective structure maps of $G_\ra \rra M$} We also assume that the double source map
\begin{equation*}
(\bfs_\ra,\bfs_\da): G \ra G_\da \times_M G_\ra
\end{equation*}
is a (not necessarily surjective) submersion.
\end{definition}

In the compatibility assumption made in this definition, we see, for example, the composable arrows of $G \rra G_\ra$ as a groupoid
\begin{equation*}
G \times_{G_\ra} G \rra G_\da \times_M G_\da
\end{equation*}
with structure defined componentwise. The compatibility of the two products is called the \textit{interchange law} and takes the form
\begin{equation*}
    (g \cdot_\ra h) \cdot_\da (g' \cdot_\ra h') = (g \cdot_\da g') \cdot_\ra (h \cdot_\da h').
\end{equation*}

\begin{remark}\label{rema: double source map}
As in \cite{Stefanini,DoublegroupoidsMehtaTang}, we do not require the double source map $(\bfs_\ra,\bfs_\da)$ to be a \textit{surjective} submersion. However, gauge double groupoids of PB-groupoids (that come with a gauge double groupoid) have this property. Following \cite{DoublegroupoidsMehtaTang}, we call double groupoids with a double source map that is a surjective submersion \textit{full}.
\end{remark}

In general, since the double source map $(\bfs_\ra,\bfs_\da)$ is a submersion, we obtain a \textit{core groupoid} that we suggestively denote by $F \rra M$:

\begin{definition}\label{def: core groupoid of double groupoid}
The \textit{core groupoid} of a double groupoid is given by
\begin{equation*}
    F \coloneq \ker \bfs_\da \cap \ker \bfs_\ra \rra M.
\end{equation*}
\end{definition}

Explicitly, its structure maps are $\bfs(g) = \bfs_\ra\bfs_\da g = \bfs_\da\bfs_\ra g$, $\bft(g) = \bft_\ra\bft_\da g = \bft_\da\bft_\ra g$, and\footnote{That is, in the definition of $g \cdot h$, we right-translate $g$ to be composable with $h$.}
\begin{align*}
    g \cdot h &= (g \cdot_\da 1_\ra(\bft_\ra h)) \cdot_\ra h = (g \cdot_\ra 1_\da(\bft_\da h)) \cdot_\da h \\
    g^{-1} &= (g)_\ra^{-1} \cdot_\da 1_\ra(\bft_\ra g)^{-1} = (g)_\da^{-1} \cdot_\ra 1_\da(\bft_\da g)^{-1}.
\end{align*}

Moreover, $F$ comes with two groupoid maps
\begin{equation*}
    G_\da \xla{\bft_\ra} F \xra{\bft_\da} G_\ra
\end{equation*}
which are just the restriction to the target maps. We now concentrate on the following very particular case of double groupoids:

\begin{definition}\label{def: transitive double Lie groupoid}
A double groupoid as above is called \textit{(vertically) core-transitive} if the map
\begin{equation*}
    \bft_\da: F \ra G_\ra
\end{equation*}
is a surjective submersion.\footnote{Of course, there is an analogous definition of horizontally core-transitive if we interchange the vertical structures with the horizontal ones.}
\end{definition}

Core-transitive double groupoids turn out to be determined by, what we will call, \textit{core extensions}.

\begin{definition}
A \textit{core extension} is a diagram\footnote{The action $G_\da \rotatebox[origin=c]{270}{$\circlearrowright$} \textnormal{ H}$ is a representation; that is, it acts by Lie group automorphisms (See Definition \ref{def: representations of bundles of Lie groups}).}
\begin{center}
\begin{tikzcd}
    1 \ar[r] & \textnormal{H} \ar[r] & \ar[dl] F \ar[r] & G_\ra \ar[r] & 1 \\
    \phantom{A} & G_\da \ar[u, phantom, "\textnormal{\Large$\circlearrowright$}"] & \phantom{A} & \phantom{A} & \phantom{A}
\end{tikzcd}
\end{center}
such that the restriction of the map $F \ra G_\da$ to $\textnormal{H}$ is equivariant (with respect to conjugation in $G_\da$) and the Peiffer identity\footnote{The \textit{Peiffer identity} is a name given to a condition that a \textit{crossed module} of Lie groupoids satisfies, of which the Peiffer identity in our case is a generalisation.} holds: the conjugation action of $F$ on $\textnormal{H}$, viewing $\textnormal{H}$ as a normal subgroupoid of $F$, equals the action by $F$ on $\textnormal{H}$ induced by the representation $G_\da$ $\rotatebox[origin=c]{270}{\large$\circlearrowright$}$ $\textnormal{H}$ and $F \ra G_\da$.
\end{definition}

This is directly related to \cite{Corediagram}, a result that is generalised in Appendix \ref{app: double groupoids}. In the appendix we present a proof, and we explain the already apparent connection with fat extensions. In the rest of this section we focus on the implication that regular fat extensions (fat extensions over regular cochain complexes) are equivalently described by \textit{general linear double groupoids}.

The general linear double groupoids that encode regular $2$-term complexes are strict Lie $2$-groupoids. Such a general linear double groupoid sits naturally inside the strict Lie $2$-groupoid of Remark \ref{rema: 2-term complex as strict Lie 2-groupoid}.

\begin{example}
Given a regular $2$-term complex $C_M \ra V_M$, we denote again by $\textnormal{Aut}(C_M \ra V_M)$ its Lie groupoid of cochain isomorphisms. Consider the double groupoid
\begin{center}
\begin{tikzcd}
\textnormal{H} \times_M \textnormal{Aut}(C_M \ra V_M) \times_M \textnormal{H} \ar[d, shift left] \ar[d, shift right] \ar[r, shift left] \ar[r, shift right] & \textnormal{Aut}(C_M \ra V_M) \ar[d, shift left] \ar[d, shift right] \\
\textnormal{H} \ar[r, shift left] \ar[r, shift right] & M
\end{tikzcd}
\end{center}
The groupoid structures are given as follows. The vertical groupoid structure has source and target map given, respectively, by $\bfs_\da(h_y,\Psi_{y,x},h_x) = h_x$, $\bft_\da(h_y,\Psi_{y,x},h_x) = h_y$ and
\begin{equation*}
(h_z,\Psi_{z,y},h_y) \cdot_\da (h_y,\Psi_{y,x},h_x) = (h_z, \Psi_{z,y}\Psi_{y,x}, h_x).
\end{equation*}
The horizontal groupoid structure is given by $\bfs_\ra(h_y,\Psi_{y,x},h_x) = \Psi_{y,x}$,
\begin{equation*}
\bft_\ra(h_y,\Psi_{y,x},h_x) = ((1+h_y\partial)\Psi_{y,x}^{C_M}(1+h_x\partial)^{-1}, (1+\partial h_y)\Psi_{y,x}^{V_M}(1+\partial h_x)^{-1}),
\end{equation*}
and
\begin{equation*}
(h_y,\Psi_{y,x},h_x) \cdot_\ra (h_y',\Psi_{y,x}',h_x') = (h_yh_y', \Psi_{y,x}', h_xh_x').
\end{equation*}
The vertical groupoid structure has a normal subgroupoid of the form
\begin{equation*}
\textnormal{Aut}(C_M \ra V_M) \ltimes \textnormal{H} \rra \textnormal{H}
\end{equation*}
through the embedding
\begin{equation*}
\textnormal{Aut}(C_M \ra V_M) \ltimes \textnormal{H} \ra \textnormal{H} \times_M \textnormal{Aut}(C_M \ra V_M) \times_M \textnormal{H} \qquad (\Psi_{y,x}, h_x) \mapsto (\Psi_{y,x} h_x \Psi^{-1}_{y,x}, \Psi_{y,x}, h_x).
\end{equation*}
We can then take the quotient with respect to this normal subgroupoid to obtain a double groupoid (a strict Lie $2$-groupoid) of the form
\begin{center}
\begin{tikzcd}
\textnormal{H}_\textnormal{Aut} \ar[d, shift left] \ar[d, shift right] \ar[r, shift left] \ar[r, shift right] & \textnormal{Aut}(C_M \ra V_M) \ar[d, shift left] \ar[d, shift right] \\
M \ar[r, shift left] \ar[r, shift right] & M
\end{tikzcd}
\end{center}
This is the general linear double groupoid associated to a $2$-term complex. It sits naturally inside of $\textnormal{H}(V_M,C_M) \ltimes \textnormal{GL}(C_M,V_M)$ (seen as a strict Lie $2$-groupoid; see Remark \ref{rema: 2-term complex as strict Lie 2-groupoid}), for it is also this strict Lie $2$-groupoid restricted to $\textnormal{Aut}(C_M \ra V_M)$:
\begin{equation*}
\textnormal{H} \ltimes \textnormal{Aut}(C_M \ra V_M) \xra{\sim} \textnormal{H}_\textnormal{Aut} \qquad (h_y,\Psi_{y,x}) \mapsto [h_y, \Psi_{y,x}, 0_x].
\end{equation*}
The reason for writing the double groupoid as above is that this description generalises to produce the general linear double groupoid of a fat extension (see Appendix \ref{app: double groupoids}).
\end{example}

We now define:

\begin{definition}\label{def: general linear double groupoids}
Let $G$ be a Lie groupoid, and let $C_M \ra V_M$ be a regular $2$-term complex. A double groupoid
\begin{center}
\begin{tikzcd}
    G_\textnormal{Aut} \ar[d, shift left] \ar[d, shift right] \ar[r, shift left] \ar[r, shift right] & \textnormal{Aut}(C_M \ra V_M) \ar[d, shift left] \ar[d, shift right] \\
    G \ar[r, shift left] \ar[r, shift right] & M
\end{tikzcd}
\end{center}
that is vertically core-transitive, and whose restriction to the vertical units is the strict Lie $2$-groupoid
\begin{center}
\begin{tikzcd}
    \textnormal{H}_\textnormal{Aut} \ar[d, shift left] \ar[d, shift right] \ar[r, shift left] \ar[r, shift right] & \textnormal{Aut}(C_M \ra V_M) \ar[d, shift left] \ar[d, shift right] \\
    M \ar[r, shift left] \ar[r, shift right] & M
\end{tikzcd}
\end{center}
is called a \textit{general linear double groupoid} of $G$ over $C_M \ra V_M$.
\end{definition}

Our observation in Remark \ref{rema: PB-groupoid come with a gauge double groupoid}, together with Proposition \ref{prop: PB-groupoid has H in it}, reveals:

\begin{proposition}
The gauge double groupoid of a regular general linear PB-groupoid $P_G$
\begin{center}
\begin{tikzcd}
P_G \times_{\textnormal{GL}} P_G \ar[d, shift left] \ar[d, shift right] \ar[r, shift left] \ar[r, shift right] & \textnormal{Aut}(C_M \ra V_M) \ar[d, shift left] \ar[d, shift right] \\ G \ar[r, shift left] \ar[r, shift right] & M
\end{tikzcd}
\end{center}
is a general linear double groupoid.
\end{proposition}

To go from general linear double groupoids to regular general linear PB-groupoids requires a slight generalisation (and clarification) of \cite{Corediagram}.\footnote{To emphasise: most essential points are explained in \cite{Corediagram}, especially in the last remarks made in that work.} If we consider a regular general linear PB-groupoid, then the core extension of its gauge double groupoid recovers precisely the fat extension. So, with this correspondence in mind, general fat extensions can be thought of as ``core extensions'' of general linear PB-groupoids. We delay the discussion now to the appendix, and proceed with making some further comments on the correspondences.

\section{Fat splittings and comments on the correspondences}\label{sec: comments}

Now that we went through the correspondences, we make some final comments.

\subsection{Cohomological vanishing for proper groupoids}\label{sec: Cohomological vanishing for proper groupoids}
We mention here that the cohomological vanishing theorem for $2$-term ruths of proper groupoids is straightforward using fat extensions:

\begin{proposition}\label{prop: vanishing for proper groupoids}
Let $F_G$ be a fat extension over a proper groupoid $G$. Moreover, let $\textstyle\int_x = \textstyle\int_{\bft^{-1}x} d\mu(g)$ be a proper normalised Haar system of $G$. Then the maps
\begin{equation*}
\eta: C_\textnormal{inv}^\bullet(F_G; C_M) \ra C_\textnormal{inv}^{\bullet-1}(F_G; C_M)
\end{equation*}
given by
\begin{equation*}
\eta f(H_{g_1},g_2,\dots,g_{\bullet-1}) = (-1)^\bullet\textstyle\int_{\bfs g_{\bullet-1}} f(H_{g_1},g_2,\dots,g_{\bullet-1},g)
\end{equation*}
define a contraction homotopy $[\delta,\eta]=1$ in degrees $\bullet \ge 2$.
\end{proposition}
\begin{proof}
The proof is a straightforward computation, entirely analogous to the proof that groupoid cohomology with trivial coefficients of $G$ vanishes.
\end{proof}

\subsection{Splittings of fat extensions and trivialisations}\label{sec: Relation to Representations up to homotopy}

In this section we clarify some aspects of ``split'' fat extensions. Naturally, this discussion comes with several clarifications of the correspondences between $2$-term ruths, VB-groupoids and fat extensions. Following Section \ref{sec: From 2-term ruths to fat extensions}, we define fat splittings as follows.

\begin{definition}\label{def: fat splittings}
Let $F_G$ be a fat extension of $G$ over $C_M \ra V_M$. A \textit{fat splitting} of $F_G$ is a map
\begin{equation*}
h: F_G \ra \textnormal{Hom}(\bfs^*V_M,\bft^*C_M)
\end{equation*}
that is invariant: for all $H_g,H_g' \in F_G$, \footnote{Recall that $h(H_g,H_g')$}
\begin{equation*}
h(H_g) - h(H_g') = -h(H_g,H_g') = -H_g' \cdot H_g^{-1}
\end{equation*}
and unital: for all $h_x \in \textnormal{H}(V_M,C_M)$,
\begin{equation*}
    h(h_x) = h_x.
\end{equation*}
It is called multiplicative if
\begin{equation*}
h(H_{g_1} \cdot H_{g_2})v = h(H_{g_1})H_{g_2} \cdot v - h(H_{g_1}) \partial h(H_{g_2})v + H_{g_1} \cdot h(H_{g_2})v.
\end{equation*}
\end{definition}

A disclaimer is that, although the similarity will become apparent, the multiplicativity condition is different than the multiplicativity condition for tensors (that appears in Section \ref{sec: functorial aspects}).

\begin{proposition}\label{prop: cleavage and fat splitting}
Let $V_G$ be a VB-groupoid. The association
\begin{equation*}
\{\textnormal{Cleavages on $V_G$}\} \ra \{\textnormal{Fat splittings on $\widehat V_G$}\} \qquad \Sigma \mapsto h_\Sigma(H_g) \coloneq r_{g^{-1}}(H_g-\Sigma_g)
\end{equation*}
sets a one-to-one correspondence between (unital) cleavages of $V_G$ and (unital) fat splittings. Moreover, for all $H_{g_1}, H_{g_2} \in \widehat V_G$ composable,
\begin{equation*}
h_\Sigma(H_{g_1} \cdot H_{g_2})v = R_2(g_1,g_2)v + H_{g_1} \cdot h_\Sigma(H_{g_2})v - h_\Sigma(H_{g_1}) \partial h_\Sigma(H_{g_2})v + h_\Sigma(H_{g_1})H_{g_2} \cdot v,
\end{equation*}
so under this correspondence, flat cleavages correspond to fat multiplicative splittings.
\end{proposition}
\begin{proof}
As mentioned before, an invariant map $h$ defines a cleavage $\Sigma$ by setting
\begin{equation*}
\Sigma_g v = -r_gh(H_g)v + H_gv
\end{equation*}
where $H_g$ is arbitrary. It is readily verified that the two constructions are mutually inverse. The last statement follows by a direct computation:
\begin{align*}
r_{g_1g_2}(H_{g_1} \cdot h_\Sigma(H_{g_2})v &+ h_\Sigma(H_{g_1})H_{g_2} \cdot v) = H_{g_1}(H_{g_2} \cdot v - R_1(g_2)v) \cdot (H_{g_2}v - \Sigma_{g_2}v) \\
&+(H_{g_1} - \Sigma_{g_1})(H_{g_2} \cdot v) \\
&= (H_{g_1} \cdot H_{g_2})v + r_{g_2}(H_{g_1} - \Sigma_{g_1})(H_{g_2} \cdot v - R_1(g_2)v) - \Sigma_{g_1} \cdot \Sigma_{g_2}v \\
&= r_{g_1g_2}(h_\Sigma(H_{g_1} \cdot H_{g_2}) + h_\Sigma(H_{g_1})\partial h_\Sigma(H_{g_2}) - R_2(g_1,g_2))
\end{align*}
where we added $\pm \Sigma_{g_1} \cdot \Sigma_{g_2} v$ in the second equality and $\pm \Sigma_{g_1g_2}$ in the last.
\end{proof}

\begin{remark}\label{rema: relation to split ruth}
Notice that, given $h_\Sigma$ as above, then, for all $H_{g_1},H_{g_2} \in \widehat V_G$ composable,
\begin{align*}
h(H_{g_1}) \partial h(H_{g_2})v &= h_\Sigma(H_{g_1})H_{g_2} \cdot v - h_\Sigma(H_{g_1})R_1(g_2)v \\
&= H_{g_1} \cdot h_\Sigma(H_{g_2}) v - R_1(g_1)h_\Sigma(H_{g_2})v.
\end{align*}
Therefore, $h_\Sigma(H_{g_1} \cdot H_{g_2})$ equals $h_\Sigma(H_{g_1}) \cdot h_\Sigma(H_{g_2})$, where the product is defined as in Section \ref{sec: From 2-term ruths to fat extensions}.
\end{remark}

\begin{remark}\label{rema: fat category and splitting}
Another way to arrive at the correct notion of fat splitting is Proposition \ref{prop: fat category from fat groupoid} of Section \ref{sec: the fat category}. Namely, since
\begin{equation*}
\widehat V_G^\textnormal{cat} \cong (\textnormal{Hom}(\bfs^*V_M,\bft^*C_M) \times_G \widehat V_G) / H(V_M,C_M) \qquad H_g \mapsto (-h(H_g,H_g'),H_g'),
\end{equation*}
we see that a cleavage $\Sigma$, i.e. a map $\Sigma: G \ra \widehat V_G^\textnormal{cat}$, can also be seen as a fat splitting $h$. Indeed, $h$ defines the map
\begin{equation*}
(h, 1): \widehat V_G \ra \textnormal{Hom}(\bfs^*V_M,\bft^*C_M) \times_G \widehat V_G
\end{equation*}
and, by invariance, this last map descends to a map $G \ra \widehat V_G^\textnormal{cat}$, which is $\Sigma$ if $h=h_\Sigma$. The failure for $h$ to be multiplicative measures the failure of this last map being a Lie category map with respect to the structure of Lie category on $\textnormal{Hom}(\bfs^*V_M,\bft^*C_M)$ as defined in Section \ref{sec: the fat category}.
\end{remark}

\begin{remark}\label{rema: fat splitting as multiplicartive tensor?}
We can translate a fat splitting $h$ to a map
\begin{equation*}
\widetilde h: \widehat V_G \ra \bft^*\textnormal{Hom}(V_M,C_M) = \textnormal{Hom}(\bft^*V_M, \bft^*C_M)
\end{equation*}
by using the representation:
\begin{equation*}
\widetilde h(H_g)v \coloneq h(H_g)(H_g^{-1} \cdot v)
\end{equation*}
Then, $\widetilde h$ satisfies the invariance condition
\begin{equation*}
\widetilde h(H_g)H_g \cdot v - \widetilde h(H_g')H_g' \cdot v = -h(H_g,H_g')v
\end{equation*}
(such that $\widetilde h(h_x) = 0$ for all $h_x \in \textnormal{H}(V_M,C_M)$). We can rewrite the condition for $h$ to be multiplicative in terms of $\widetilde h$ as
\begin{equation*}
\widetilde h(H_{g_1} \cdot H_{g_2})v = H_{g_1} \cdot \widetilde h(H_{g_2}) H_{g_1}^{-1} \cdot v - \widetilde h(H_{g_1}) \partial(H_{g_1} \cdot \widetilde h(H_{g_2}) H_{g_1}^{-1} \cdot v) + \widetilde h(H_{g_1})v.
\end{equation*}
\end{remark}

Using the following terminology, it is straightforward to see that a fat multiplicative splitting ``trivialises'' the associated fat extension.

\begin{definition}\label{def: trivialisation fat extension}
Suppose we are given a cochain complex representation of $G$ on $C_M \ra V_M$. The trivial fat extension is the fat extension corresponding to the fat groupoid
\begin{equation*}
\textnormal{H}(V_M,C_M) \times_M G \rra M
\end{equation*}
with $\widehat\bft(h_{\bft g},g) = \bft g$, $\widehat\bfs(h_{\bft g},g) = \bfs g$, and multiplication and inversion are given by\footnote{Here, $G$ acts on $\textnormal{H}(V_M,C_M)$ by conjugation.}
\begin{equation*}
(h_{\bft g_1}, g_1) \cdot (h_{\bft g_2}, g_2) = (h_{\bft g_1} \cdot (g_1 \cdot h_{\bft g_2}), g_1g_2) \qquad (h_{\bft g_1}, g_1)^{-1} = (g_1^{-1} \cdot h_{\bft g_1}^{-1}, g_1^{-1}).
\end{equation*}
The representation of $\textnormal{H}(V_M,C_M) \times_M G \rra M$ on $C_M \ra V_M$ is given by
\begin{equation*}
(h_{\bft g}, g) \cdot c = (1+ h_{\bft g}\partial)(g \cdot c) \qquad (h_{\bft g}, g) \cdot v = (1+\partial h_{\bft g})(g \cdot v).
\end{equation*}
A trivialisation of a fat extension is a cochain complex representation of $G$ on $C_M \ra V_M$, and an isomorphism of fat extensions
\begin{equation*}
F_G \cong \textnormal{H}(V_M,C_M) \times_M G.
\end{equation*}
\end{definition}

\begin{proposition}\label{prop: fat multiplicative splitting}
Let $F_G$ be a fat extension. A fat multiplicative splitting $h$ is equivalent to a trivialisation of $F_G$. More precisely, the correspondence is given by
\begin{equation*}
F_G \cong \textnormal{H}(V_M,C_M) \times_M G \qquad H_g \mapsto (v \mapsto -h(H_g)(H_g^{-1} \cdot v), g),
\end{equation*}
where the representation of $G$ on $C_M \ra V_M$ is given by
\begin{equation*}
g \cdot c \coloneq H_g \cdot c - h(H_g)\partial c \qquad g \cdot v = H_g \cdot v - \partial h(H_g) v.
\end{equation*}
\end{proposition}

To conclude, since the invariant cochains $C_\textnormal{inv}(F_G; C_M)$ form an abstract ruth, a fat splitting splits the invariant cochains $C_\textnormal{inv}(F_G; C_M)$:

\begin{proposition}\label{prop: invariant cochain complex is a ruth}
The complex $C_\textnormal{inv}(\widehat V_G; C_M)$ is an abstract ruth. Moreover, writing $\calE_G = C^\bullet(G; C_M) \oplus C^{\bullet-1}(G; V_M)$, a fat splitting $h$ defines an isomorphism
\begin{equation*}
\calE_G \xra{\sim} C_\textnormal{inv}(F_G; C_M) \qquad (f_0^\calE,f_1^\calE) \mapsto (f_0,f_1)
\end{equation*}
by setting
\begin{align*}
f_0(H_{g_1},g_2,\dots,g_\bullet) &= f_0^\calE(g_1,\dots,g_\bullet) - h(H_{g_1})f_1^\calE(g_2,\dots,g_\bullet) \\
f_1(g_2,\dots,g_\bullet) &= f_1^\calE(g_2,\dots,g_\bullet).
\end{align*}
The induced operators $R_1$ and $R_2$ on $\calE_G$ are given by
\begin{align*}
R_1(g)c &\coloneq H_g \cdot c - h(H_g)\partial c \qquad R_1(g)v = H_g \cdot v - \partial h(H_g) v \\ R_2(g_1,g_2)v &\coloneq h(H_{g_1} \cdot H_{g_2})v - h(H_{g_1})H_{g_2} \cdot v + h(H_{g_1})\partial h(H_{g_2})v - H_{g_1} \cdot h(H_{g_2})v.
\end{align*}
\end{proposition}
\begin{proof}
It is clear that the given map is a map of $CG$-modules. The inverse of the map is given by reading the equation the other way around (which makes sense by the properties \eqref{eq: f_V_M relation to f_C_M 0} and \eqref{eq: f_V_M relation to f_C_M 1}). The last statement follows from the above discussion in this section.
\end{proof}

\subsection{Fat extensions as a functor to the quasi general linear Lie $2$-groupoid}\label{sec: Fat extensions as a functor to the quasi general linear Lie 2-groupoid}

In this section, we will explain how fat extensions of $G$ relate to ``pseudofunctors'' from $G$ to a type of general linear $2$-groupoid as defined in \cite{GLStefaniMatias}. Here, we will refer to that general linear Lie $2$-groupoid as the quasi-general linear Lie $2$-groupoid. It is a Lie $2$-groupoid that is not strict.

\begin{remark}\label{rema: 2-groupoids}
In the analogy we made in Remark \ref{rema: strict 2-groupoids}, we can relax ``strictness'' as follows. A Lie $2$-groupoid (so not necessarily strict) can similarly be described by a diagram $H_2 \rra H_1 \rra N$, but $H_1 \rra N$ is only a groupoid ``up to homotopy''. More precisely, we relax the condition that $H_2 \rra N$ and $H_1 \rra N$ are Lie groupoids to them being possibly only Lie categories. Still, $g \in H_1$ is assumed to have ``quasi-inverses'' $g^{-1} \in H_1$. The elements of $H_2$, which are invertible in the groupoid $H_2 \rra H_1$, can then again be interpreted as homotopies $h$ that show two elements in $H_1$ are homotopic. In particular, given $g \in H_1$, there will exist a $g^{-1} \in H_1$ and homotopies in $H_2$ that show $gg^{-1}$ and $g^{-1}g$ are homotopic to identity elements of $H_1$.
\end{remark}

We explain next the definition of the quasi general linear groupoid (the general linear groupoid that appeared in \cite{GLStefaniMatias}). We will use the following object (that will also appear in Section \ref{sec: morphisms of general linear PB-groupoids}): given two vector bundles $E_1$ and $E_2$ over $M$, we write
\begin{equation*}
\textnormal{Lin}(E_1,E_2) \coloneq \{E_{1,x} \ra E_{2,y} \mid x,y \in M\}
\end{equation*}
for the vector bundle over $M \times M$ of linear maps between fibers of $E_1$ and $E_2$.

\begin{example}\label{exa: using lax functor}
Recall the general linear strict Lie $2$-groupoid from Example \ref{exa: general linear 2-groupoid}. Relaxing the conditions for elements of $\textnormal{GL}(C_M,V_M)_M$ to be quasi-isomorphisms (see Remark \ref{rema: 2-groupoids}), we will show here that we obtain a well-defined quasi general linear Lie $2$-groupoid (as we call it here)
\begin{equation*}
\textnormal{QGL}(C_M,V_M)_G \rra \textnormal{QGL}(C_M,V_M)_M \rra M.
\end{equation*}
As observed in \cite{GLStefaniMatias}, notice that the quasi-isomorphisms 
\begin{equation*}
    \textnormal{QGL}(C_M,V_M)_M \subset \textnormal{Hom}(C_M,V_M) \times_M \textnormal{End } C_M \times_{M \times M} \textnormal{End } V_M \times_M \textnormal{Hom}(C_M,V_M)
\end{equation*}
form an embedded submanifold. To see this, observe first that, given a cochain map $(\partial_y, \Psi_{y,x}, \partial_x)$, then the existence of maps
\begin{equation*}
[\partial, \Psi^{-1}] = 0 \qquad 1 - \Psi\Psi^{-1} = [\partial, h_y] \qquad 1 - \Psi^{-1}\Psi = [\partial, h_x]
\end{equation*}
such that additionally\footnote{If $\Psi$ is a quasi-isomorphism, it is indeed possible to find homotopy data satisfying these extra conditions.}
\begin{equation*}
    h_y \Psi^{V_M} = \Psi^{C_M} h_x \qquad (\Psi^{C_M})^{-1} h_y = h_x (\Psi^{V_M})^{-1}
\end{equation*}
is equivalent to the existence of an isomorphism
\begin{equation*}
\begin{pmatrix}
\Psi^{C_M}_{y,x} & -h_y \\ \partial_x & (\Psi^{V_M}_{y,x})^{-1}
\end{pmatrix}: (C_M)_x \oplus (V_M)_y \xra{\sim} (C_M)_y \oplus (V_M)_x
\end{equation*}
(so the left column is fixed) with inverse given by
\begin{equation*}
\begin{pmatrix}
(\Psi^{C_M}_{y,x})^{-1} & h_x \\ -\partial_y & \Psi^{V_M}_{y,x}
\end{pmatrix}: (C_M)_y \oplus (V_M)_x \xra{\sim} (C_M)_x \oplus (V_M)_y
\end{equation*}
(where the bottom row is fixed). Therefore, $(\partial_y, \Psi_{y,x}, \partial_x)$ is a quasi-isomorphism if and only if
\begin{equation}\label{eq: open conditions to be a quasi-iso}
\ker \Psi^{C_M}_{y,x} \cap \ker \partial_x = 0 \qquad \im \Psi^{V_M}_{y,x} + \im \partial_y = (V_M)_y.
\end{equation}
In \cite{GLStefaniMatias} (Proposition 5.4 there) it is then shown that, on the open set $U$ of (not-necessarily cochain maps) $(\partial_y, \Psi_{y,x}, \partial_x)$ satisfying \eqref{eq: open conditions to be a quasi-iso}, the map 
\begin{equation*}
    U \ra \textnormal{Lin}(C_M,V_M) \qquad (\partial_y, \Psi_{y,x}, \partial_x) \mapsto \partial_y\Psi_{y,x}^{C_M} - \Psi_{y,x}^{V_M}\partial_x
\end{equation*}
has maximal rank. We now define $\textnormal{QGL}(C_M,V_M)_M \rra \textnormal{Hom}(C_M,V_M)$ as a Lie category with composition, and we set 
\begin{equation*}
\textnormal{GL}(C_M,V_M)_G \coloneq \textnormal{QGL}(C_M,V_M)_M \times_{M \times M} \textnormal{Lin}(V_M,C_M)
\end{equation*}
thought of as elements $(\partial_y, \Psi_{y,x}, \Psi_{y,x}', \partial_x, h_{y,x})$ (or simply $(\partial_y, \Psi_{y,x}, \partial_x, h_{y,x})$) that form homotopy data between the quasi-isomorphisms $\Psi_{y,x}$ and $\Psi_{y,x}'$:\footnote{We sometimes leave $\Psi_{y,x}'$ implicit for it is determined by $\Psi_{y,x}$ and $h_{y,x}$. Moreover, given $\Psi_{y,x}$, any $h_{y,x}$ defines the cochain map $\Psi_{y,x}' \coloneq \Psi_{y,x} - \partial_y h_{y,x} - h_{y,x}\partial_x$ which is homotopic to $\Psi_{y,x}$.}
\begin{equation*}
    \Psi_{y,x} - \Psi_{y,x}' = \partial_xh_{x,y} + h_{x,y}\partial_y.
\end{equation*}
Both products can be interpreted as homotopy composition in the following way. The Lie category structure $\textnormal{QGL}(C_M,V_M)_G \rra \textnormal{Hom}(C_M,V_M)$ is given by\footnote{Notice that $\Psi_{z,y}'h_{y,x} + h_{z,y}\Psi_{y,x} = \Psi_{z,y} h_{y,x} + h_{z,y}\Psi_{y,x}'$ (see Lemma \ref{lemm: homotopy composition}).}
\begin{equation*}
    (\partial_z, \Psi_{z,y}, \partial_y, h_{z,y}) \cdot (\partial_y, \Psi_{y,x}, \partial_x, h_{y,x}) \coloneq (\partial_z, \Psi_{z,y}\Psi_{y,x},  \partial_x, \Psi_{z,y}' h_{y,x} + h_{z,y}\Psi_{y,x}).
\end{equation*}
The groupoid structure $\textnormal{QGL}(C_M,V_M)_G \rra \textnormal{QGL}(C_M,V_M)_M$ has source and target map given by $\bfs(\partial_z, \Psi_{z,y}, \partial_y, h_{z,y}) = (\partial_z, \Psi_{z,y}', \partial_y)$ and $\bft(\partial_z, \Psi_{z,y}, \partial_y, h_{z,y}) = (\partial_z, \Psi_{z,y}, \partial_y)$, and 
\begin{equation*}
    (\partial_z, \Psi_{z,y}, \partial_y, h_{z,y}) \cdot (\partial_z, \Psi_{z,y}', \partial_y, h_{z,y}') \coloneq (\partial_z, \Psi_{z,y}, \Psi_{z,y}'', \partial_y, h_{z,y} + h_{z,y}').
\end{equation*}
Now, an element $(\partial_y, \Psi_{y,x}, \partial_x) \in \textnormal{QGL}(C_M,V_M)_M$ is not always invertible, but, given homotopy data $(h_x, \Psi_{y,x}^{-1}, h_y)$ as above, $(\partial_x, \Psi_{y,x}^{-1}, \partial_y)$ is a ``quasi-inverse'' in the sense of Remark \ref{rema: 2-groupoids}, which shows that $\textnormal{QGL}(C_M,V_M)_G \rra \textnormal{QGL}(C_M,V_M)_M \rra \textnormal{Hom}(C_M,V_M)$ is indeed a Lie $2$-groupoid.
\end{example}

\begin{remark}
    Notice that the we can identify $\textnormal{GL}(C_M,V_M)_G \subset \textnormal{QGL}(C_M,V_M)_G$ as the invertible elements of the Lie category $\textnormal{QGL}(C_M,V_M)_G \rra \textnormal{Hom}(C_M,V_M)$ via 
    \begin{equation*}
        (h_y, \partial_y, \Psi_{y,x}) \mapsto (\partial_y, (1+[\partial_y,h_y])\Psi_{y,x}, (\Psi_{y,x}^{V_M})^{-1}\partial_y\Psi_{y,x}^{C_M}, h_y\Psi_{y,x}^{V_M}).
    \end{equation*}
\end{remark}

We can interpret the step to pseudofunctors efficiently using the fat category extension. A fat category extension $F_G^\textnormal{cat}$ of $G$ over $C_M \ra V_M$ comes with a Lie category map
\begin{equation*}
F_G^\textnormal{cat} \ra \textnormal{QGL}(C_M,V_M)_M
\end{equation*}
which defines the cochain complex representation. Now, a fat category splitting is a section of the projection $F_G^\textnormal{cat} \ra G$ (see Remark \ref{rema: splitting of fat category extension} or Section \ref{sec: Relation to Representations up to homotopy}) and therefore defines a map
\begin{equation*}
R_1: G \ra F_G^\textnormal{cat} \ra \textnormal{QGL}(C_M,V_M)_M
\end{equation*}
which might fail to be a functor. Nevertheless, its image lies in $\textnormal{QGL}(C_M,V_M)_M$, so for all $g \in G$,
\begin{equation*}
[\partial,R_1(g)] = 0.
\end{equation*}
The failure for this map to define a functor can now be seen as a map
\begin{equation*}
R_2: G^{(2)} \ra \textnormal{QGL}(C_M,V_M)_G.
\end{equation*}
That it maps into $\textnormal{QGL}(C_M,V_M)_G$ is saying that, for all $(g_1,g_2) \in G^{(2)}$,
\begin{equation*}
R_1(g_1)R_1(g_2) - R_1(g_1g_2) = [\partial,R_2(g_1,g_2)].
\end{equation*}
The condition that $R_2$ is unital, and that for all $(g_1,g_2,g_3) \in G^{(3)}$,
\begin{equation*}
-R_2(g_1g_2,g_3) + R_2(g_1,g_2g_3) = -R_1(g_1)R_2(g_2,g_3) + R_2(g_1,g_2)R_1(g_3)
\end{equation*}
is then equivalent to the fact that $R_1$ and $R_2$ together define a so-called pseudofunctor.\footnote{The pseudofunctor conditions say that $R_2$ is unital and that the two homotopies
\begin{equation*}
    R_2(g_1g_2,g_3) + R_2(g_1,g_2)R_1(g_3) \qquad R_2(g_1,g_2g_3) + R_1(g_1)R_2(g_2,g_3)
\end{equation*}
are equal.} The functorial aspects of this construction are a bit subtle. It is proved in \cite{GLStefaniMatias} that the category of (split) $2$-terms ruths over a fixed $2$-term complex $C_M \ra V_M$, together with \textit{quasi-isomorphisms} of $2$-term ruths, is equivalent to the category of pseudofunctors together with lax-functors. A strict version of this correspondence is not discussed in \cite{GLStefaniMatias}, but some more comments about this issue are made in that work. 

As we already mentioned, a fat category splitting of a fat category extension turns it into a \textit{weak representation} as in \cite{Wolbert}, and such objects come with a natural notion of morphism that turns the correspondence into an equivalence of categories.

\section{Functorial aspects}\label{sec: functorial aspects}

We will introduce a morphism of fat extensions as a multiplicative tensor. This will be a first instance where we view a multiplicative tensor as an object that naturally occurs in the theory of fat extensions. This should therefore also be seen as part of the bigger project together with Joa\~o Nuno Mestre and Luca Vitagliano \cite{Multiplicativetensors}. A morphism of general linear PB-groupoids is similar in nature, but its properties can succinctly be described as \textit{$\textnormal{GL}$-equivariance}. We introduce both types of morphisms through the language of VB-groupoids. For this reason, in introducing the correct notions of morphisms, we will come to see naturally how the respective categories are equivalent to the category of VB-groupoids.

A notational comment: throughout, we will denote a cochain map as follows:
\begin{center}
\begin{tikzcd}
C_{M,1} \ar[r, "\partial"] \ar[d, "\Phi^{C_M}"] & V_{M,1} \ar[d, "\Phi^{V_M}"] \\
C_{M,2} \ar[r, "\partial"] & V_{M,2}
\end{tikzcd}
\end{center}
but often we simply write $\Phi$. A VB-groupoid map is denoted by
\begin{equation*}
\Phi = \Phi^{V_G}: V_{G,1} \ra V_{G,2}.
\end{equation*}
We focus on maps over the identity map of $G$. Such a map is a groupoid map $V_{G,1} \ra V_{G,2}$ that is also linear (i.e. intertwines the scalar multiplications). It comes with a cochain map by setting
\begin{equation*}
\Phi^{C_M} \coloneq \Phi^{V_G}|_{C_M} \qquad \Phi^{V_M} \coloneq \bfs\Phi^{V_G}|_{V_M}
\end{equation*}
where we used the natural inclusion
\begin{equation*}
C_M \oplus V_M = V_G|_M \hookrightarrow V_G.
\end{equation*}

\begin{remark}\label{rema: fibred categories}
We restrict attention to VB-groupoid maps over the identity for exposition. In order to pass to more general maps (over general groupoid maps), notice that the category of VB-groupoids is naturally fibred, just as, for example, the category of vector bundles. That is, general VB-groupoid maps from $V_G$ to $V_H$ can be seen as VB-groupoid maps
\begin{equation*}
V_G \ra \varphi^*V_H
\end{equation*}
over the identity, where $\varphi: G \ra H$ is a Lie groupoid map, and $\varphi^*V_H$ is the pullback VB-groupoid. The other categories we will consider (of fat extensions and of general linear PB-groupoids) are fibred analogously, and passing to more general morphisms is therefore straightforward and omitted.
\end{remark}

\subsection{Morphisms of fat extensions}\label{sec: morphisms of fat extensions}

Our aim is now to describe, given a map of VB-groupoids $\Phi = \Phi^{V_G}: V_{G,1} \ra V_{G,2}$, a type of morphism $f = f_\Phi: \widehat V_{G,1} \ra \widehat V_{G,2}$ that consists of exactly the same data. However, observe that there is no obvious groupoid map $\widehat V_{G,1} \ra \widehat V_{G,2}$ in general. Rather, we obtain a notion of ``related'' elements. We therefore aid our discussion with the following intermediate definition.

\begin{definition}\label{def: related elements}
Given a VB-groupoid $\Phi: V_{G,1} \ra V_{G,2}$, we say $H_{g,1} \in \widehat V_{G,1}$ and $H_{g,2} \in \widehat V_{G,2}$ are $\Phi$-related if the diagram
\begin{center}
\begin{tikzcd}
(V_{M,1})_{\bfs g} \ar[r, "H_{g,1}"] \ar[d,"\Phi^{V_M}"] & (V_{G,1})_g \ar[d, "\Phi^{V_G}"] \\
(V_{M,2})_{\bfs g} \ar[r, "H_{g,2}"] & (V_{G,2})_g
\end{tikzcd}
\end{center}
commutes. In this case, we write $H_{g,1} \sim_\Phi H_{g,2}$.
\end{definition}

Given $(g,g') \in G^{(2)}$, $H_{g,1} \sim_\Phi H_{g,2}$ and $H_{g',1} \sim_\Phi H_{g',2}$, then also $H_{g,1} \cdot H_{g',1} \sim_\Phi H_{g,2} \cdot H_{g',2}$. But for our purposes such a notion of ``multiplicative relation'' is too narrow. For example, there will often exist $H_{g,1} \in \widehat V_{G,1}$ not $\Phi$-related to any $H_{g,2} \in \widehat V_{G,2}$.\footnote{For example, if there is a $v \in \ker \Phi \subset (V_{M,1})_{\bfs g}$ and $H_{g,1}$ such that $\Phi H_{g,1}v \neq 0$.} In general, however, given $H_{g,1} \in \widehat V_{G,1}$ and $H_{g,2} \in \widehat V_{G,2}$, we can always produce the map
\begin{equation}\label{eq: VB-map to fat map}
f_\Phi(H_{g,1},H_{g,2}): (V_{M,1})_{\bfs g} \ra (C_{M,2})_{\bft g} \qquad v \mapsto r_{g^{-1}}(H_{g,2}\Phi^{V_M} v - \Phi^{V_G} H_{g,1}v)
\end{equation}
that is zero precisely when $H_{g,1} \sim_\Phi H_{g,2}$. The argument that the $\Phi$-relation is multiplicative can be generalised to the following property of $f_\Phi$:

\begin{lemma}\label{lemm: f_Phi is multiplicative 1}
Given a VB-groupoid map $\Phi: V_{G,1} \ra V_{G,2}$,
\begin{equation*}
f_\Phi: \widehat V_{G,1} \times_G \widehat V_{G,2} \ra \textnormal{Hom}(\bfs^*V_M, \bft^*C_M),
\end{equation*}
defined as above in \eqref{eq: VB-map to fat map}, satisfies, for all $(H_{g,1},H_{g',1}) \in \widehat V_{G,1}^{(2)}$ and $(H_{g,2},H_{g',2}) \in \widehat V_{G,2}^{(2)}$,
\begin{equation*}
f_\Phi(H_{g,1} \cdot H_{g',1}, H_{g,2} \cdot H_{g',2})v = f_\Phi(H_{g,1},H_{g,2})H_{g',1} \cdot v + H_{g,2} \cdot f_\Phi(H_{g',1}, H_{g',2})v.
\end{equation*}
\end{lemma}
\begin{proof}
We compute directly:
\begin{align*}
r_{gg'}f_\Phi(H_{g,1} \cdot H_{g',1},H_{g,2} \cdot H_{g',2})v &= H_{g,2}(\bft H_{g',2}\Phi v) \cdot H_{g',2}\Phi v - \Phi(H_{g,1}(\bft H_{g',1}v)) \cdot \Phi(H_{g',1}v) \\
&= r_{g'}(H_{g,2}(\Phi(\bft H_{g',1}v)) - \Phi(H_{g,1}(\bft H_{g',1}v))) + \\
&\phantom{=} + H_{g,2}\bft(H_{g',2}\Phi v - \Phi H_{g',1}v) \cdot (H_{g',2}\Phi v - \Phi H_{g',1}v) \\
&= r_{gg'}f_\Phi(H_{g,1},H_{g,2})(H_{g',1} \cdot v) + r_{gg'}(H_{g,2} \cdot f_\Phi(H_{g',1},H_{g',2})v).
\end{align*}
This proves the statement.
\end{proof}

Aside from this multiplicativity property, which generalises the multiplicativity of the $\Phi$-relation, we also need a generalisation of the following fact. If $H_{g,1} \sim_\Phi H_{g,2}$, then the cochain map $\Phi = (\Phi^{C_M},\Phi^{V_M})$ intertwines the actions by $H_{g,1}$ and $H_{g,2}$, respectively:
\begin{equation*}
H_{g,2} \cdot \Phi^{C_M} c = \Phi^{C_M}(H_{g,1} \cdot c) \qquad H_{g,2} \cdot \Phi^{V_M} v = \Phi^{V_M}(H_{g,1} \cdot v).
\end{equation*}
That is, we observe next that the map $f_\Phi(H_{g,1},H_{g,2})$ can actually be interpreted as the homotopy that measures the failure of $\Phi$ to intertwine the actions by $H_{g,1}$ and $H_{g,2}$:

\begin{lemma}\label{lemm: f_Phi is multiplicative 2}
Given a VB-groupoid map $\Phi: V_{G,1} \ra V_{G,2}$,
\begin{equation*}
f_\Phi: \widehat V_{G,1} \times_G \widehat V_{G,2} \ra \textnormal{Hom}(\bfs^*V_M, \bft^*C_M),
\end{equation*}
defined as above in \eqref{eq: VB-map to fat map}, satisfies, for all $H_{g,1} \in \widehat V_{G,1}$ and $H_{g,2} \in \widehat V_{G,2}$,
\begin{align*}
f_\Phi(H_{g,1},H_{g,2})\partial c &= H_{g,2} \cdot \Phi^{C_M} c - \Phi^{C_M}(H_{g,1} \cdot c), \\
\partial f_\Phi(H_{g,1},H_{g,2})v &= H_{g,2} \cdot \Phi^{V_M} v - \Phi^{V_M}(H_{g,1} \cdot v).
\end{align*}
\end{lemma}

There is one more important property that relate the map $f_\Phi$ and the cochain map $\Phi$: invariance. To introduce it, it is helpful to understand how one recovers $\Phi^{V_G}$ from $f_\Phi$ and the cochain map $\Phi$.

Recall that an element, say, $H_{g,1} \in \widehat V_{G,1}$, defines a linear isomorphism
\begin{equation*}
r_g + H_{g,1}: (C_{M,1})_{\bft g} \oplus (V_{M,1})_{\bfs g} \xra{\sim} (V_{G,1})_g \qquad (c,v) \mapsto c \cdot 0_g + H_{g,1}v.
\end{equation*}
We then have (for future reference, we write $f=f_\Phi$)
\begin{equation}\label{eq: fat map to VB-map}
\Phi_g^{V_G}(r_gc + H_{g,1}v) = r_g(\Phi^{C_M}(c) - f(H_{g,1},H_{g,2})v) + H_{g,2}\Phi^{V_M}(v).
\end{equation}
Interpreting $f_\Phi$ as a ``correcting term'' in this way points to the missing piece of information, invariance, that $f_\Phi$ and the cochain map $\Phi$ satisfy.

\begin{lemma}\label{lemm: f_Phi is invariant}
Given a VB-groupoid map $\Phi: V_{G,1} \ra V_{G,2}$,
\begin{equation*}
f_\Phi: \widehat V_{G,1} \times_G \widehat V_{G,2} \ra \textnormal{Hom}(\bfs^*V_M, \bft^*C_M)
\end{equation*}
defined as above in \eqref{eq: VB-map to fat map} satisfies, for all $H_{g,1}, H_{g,1}' \in \widehat V_{G,1}$ and $H_{g,2}, H_{g,2}' \in \widehat V_{G,2}$,
\begin{equation*}
f_\Phi(H_{g,1},H_{g,2})v - f_\Phi(H_{g,1}',H_{g,2}')v = h_2(H_{g,2},H_{g,2}')\Phi^{V_M} v - \Phi^{C_M} h_1(H_{g,1},H_{g,1}')v
\end{equation*}
where
\begin{equation*}
h_1(H_{g,1},H_{g,1}')v = r_{g^{-1}}(H_{g,1}'v-H_{g,1}v),
\end{equation*}
and $h_2$ is defined similarly.
\end{lemma}

Recognise that, in the above lemma, $h_1$ and $h_2$ are equal to $f_\Phi$ if we take $\Phi$ to be the respective identity maps. The similarity with the invariance condition that appeared in Section \ref{sec: From fat extensions to ruths} is, of course, not coincidental. In Remark \ref{rema: map of fat extensions is a multiplicative tensor} we clarify this connection further.

We organise the properties that $(f,\Phi)$ should satisfy, in order to set the correspondence, in the following way.

\begin{definition}\label{def: fat maps}
Consider fat extensions $F_{G,1}$ and $F_{G,2}$ of $G$, over $C_{M,1} \ra V_{M,1}$ and $C_{M,2} \ra V_{M,2}$, respectively. Consider a pair $(f, \Phi)$, where $\Phi = (\Phi^{C_M}, \Phi^{V_M})$ are two vector bundle maps
\begin{equation*}
C_{M,1} \xra{\Phi^{C_M}} C_{M,2} \qquad V_{M,1} \xra{\Phi^{V_M}} V_{M,2}
\end{equation*}
and $f$ is a smooth map
\begin{equation*}
f: F_{G,1} \times_G F_{G,2} \ra \textnormal{Hom}(\bfs^*V_{M,1},\bft^*C_{M,2}).
\end{equation*}
Now, $(f, \Phi)$ is a \textit{morphism $F_{G,1} \ra F_{G,2}$ of fat extensions} if it is invariant:
\begin{equation*}
f(H_{g,1},H_{g,2})v - f(H_{g,1}',H_{g,2}')v = h_2(H_{g,2},H_{g,2}')\Phi^{V_M} v - \Phi^{C_M} h_1(H_{g,1},H_{g,1}')v
\end{equation*}
and multiplicative: this means that $\Phi$ is a cochain map
\begin{center}
\begin{tikzcd}
C_{M,1} \ar[r, "\partial"] \ar[d, "\Phi^{C_M}"] & V_{M,1} \ar[d, "\Phi^{V_M}"] \\
C_{M,2} \ar[r, "\partial"] & V_{M,2}
\end{tikzcd}
\end{center}
that
\begin{align*}
f(H_{g,1},H_{g,2})\partial c &= H_{g,2} \cdot \Phi^{C_M} c - \Phi^{C_M}(H_{g,1} \cdot c), \\
\partial f(H_{g,1},H_{g,2})v &= H_{g,2} \cdot \Phi^{V_M} v - \Phi^{V_M}(H_{g,1} \cdot v),
\end{align*}
and that
\begin{equation*}
f(H_{g,1} \cdot H_{h,1}, H_{g,2} \cdot H_{h,2})v = f(H_{g,1},H_{g,2})(H_{h,1} \cdot v) + H_{g,2} \cdot f(H_{h,1}, H_{h,2})v.
\end{equation*}
\end{definition}

\begin{remark}\label{rema: invariance usually makes f determine Phi}
We often simplify and write a map of fat extensions $(f,\Phi)$ by $f$, so we leave $\Phi$ implicit from the notation. Similar as happens for the invariant cochains we described in Section \ref{sec: From fat extensions to ruths} (see Remark \ref{rema: invariance condition makes f_0 determine f_1}), $f$ usually determines the cochain map $\Phi$ by invariance. But, aside from suppressing $\Phi$ from the notation sometimes, we keep track of $\Phi$ as part of the data. 
\end{remark}

\begin{remark}\label{rema: map of fat extensions is a multiplicative tensor}
The point of this remark is to clarify why we used the term ``multiplicative'' in the definition above. In fact, we will argue here that we can think of a morphism of fat extensions as a multiplicative tensor.

To do this, notice that we can translate a map of fat extensions $f$ so that it takes values in
\begin{equation*}
\bft^*\textnormal{Hom}(V_M,C_M) = \textnormal{Hom}(\bft^*V_M,\bft^*C_M).
\end{equation*}
Indeed, simply replace $f$ with
\begin{equation*}
\widetilde f(H_{g,1},H_{g,2}): (V_{M,1})_{\bft g} \ra (C_{M,2})_{\bft g} \qquad v \mapsto f(H_{g,1},H_{g,2})(H_{g,1}^{-1} \cdot v).
\end{equation*}
In terms of $\widetilde f$, the invariance condition becomes
\begin{equation*}
\widetilde f(H_{g,1},H_{g,2})(H_{g,1} \cdot v) - \widetilde f(H_{g,1}',H_{g,2}')(H_{g,1}' \cdot v) = h_2(H_{g,2},H_{g,2}') \Phi^{V_M} v - \Phi^{C_M} h_1(H_{g,1},H_{g,1}')v,
\end{equation*}
and the two multiplicativity conditions involving $f$ translate to
\begin{align*}
\partial \widetilde f(H_{g,1},H_{g,2})v &= H_{g,2} \cdot \Phi^{V_M}(H_{g,1}^{-1} \cdot v) - \Phi^{V_M} v \\
\widetilde f(H_{g,1},H_{g,2})\partial c &= H_{g,2} \cdot \Phi^{C_M}(H_{g,1}^{-1} \cdot c) - \Phi^{C_M} c
\end{align*}
and
\begin{equation*}
\widetilde f(H_{g,1} \cdot H_{h,1}, H_{g,2} \cdot H_{h,2})v = \widetilde f(H_{g,1},H_{g,2})v + H_{g,2} \cdot \widetilde f(H_{h,1}, H_{h,2})(H_{g,1}^{-1} \cdot v).
\end{equation*}
This way, $f$ (or rather $(\widetilde f,\Phi)$) can be thought of as an ``invariant'' $1$-cocycle in (the $3$-term ruth corresponding to) the cochain complex representation
\begin{equation*}
\textnormal{Hom}(V_{M,1}, C_{M,2}) \ra \textnormal{Hom}(C_{M,1}, C_{M,2}) \oplus \textnormal{Hom}(V_{M,1}, V_{M,2}) \ra \textnormal{Hom}(C_{M,1}, V_{M,2})
\end{equation*}
of the groupoid $F_{G,1} \times_G F_{G,2}$, whose groupoid structure is defined componentwise. Here, the cochain complex representation can be thought of as a tensor product of the cochain complexes
\begin{equation*}
V_{M,1}^* \ra C_{M,1}^* \qquad C_{M,2} \ra V_{M,2}
\end{equation*}
with the action defined ``componentwise''. Studying more general tensor products this way is part of the objective of the work in progress \cite{Multiplicativetensors}.
\end{remark}

Before we move on, we discuss the composition of morphisms.

\begin{definition}\label{def: composition of morphisms}
Given maps of fat extensions
\begin{equation*}
F_{G,1} \times_G F_{G,2} \xra{f_{21}} \textnormal{Hom}(\bfs^*V_{M,1},\bft^*C_{M,2}) \qquad F_{G,2} \times_G F_{G,3} \xra{f_{32}} \textnormal{Hom}(\bfs^*V_{M,2},\bft^*C_{M,3}),
\end{equation*}
over cochain maps $\Phi_{21}$ and $\Phi_{32}$, respectively, the composition $f_{31} = f_{32} \circ f_{21}$ over $\Phi_{31} = \Phi_{32} \circ \Phi_{21}$ is given by
\begin{equation*}
f_{31}(H_{g,1}, H_{g,3}) = f_{32}(H_{g,2}, H_{g,3})\Phi_{21}^{V_M} + \Phi_{32}^{C_M}f_{21}(H_{g,1}, H_{g,2}),
\end{equation*}
where $H_{g,2} \in F_{G,2}$ is arbitrary.
\end{definition}

\begin{remark}\label{rema: identity map of fat extensions}
The identity map $F_G \ra F_G$ is the identity cochain map between $C_M \ra V_M$ and itself, together with the map
\begin{equation*}
h: F_G \times_G F_G \ra \textnormal{Hom}(\bfs^*V_M,\bft^*C_M) \qquad h(H_g,H_g') = h_{\bft g}(H_g \cdot v),
\end{equation*}
where $h_{\bft g}$ is unique with the property that
\begin{equation*}
h_{\bft g} \cdot H_g = H_g'.
\end{equation*}
If $F_G = \widehat V_G$ for a VB-groupoid $V_G$, we recover the map
\begin{equation*}
h(H_g,H_g') = r_{g^{-1}}(H_g' - H_g)
\end{equation*}
which we already mentioned in Remark \ref{rema: invariance condition for fat groupoid of VB-groupoid}.
\end{remark}

The following Remark should be compared to Remark \ref{rema: fat subextension and fat quotient extension as fat maps}, where we discussed the ``inclusion'' map of a fat subextension and the ``projection'' map of a fat quotient.

\begin{remark}\label{rema: injective and surjective maps of fat extensions}
We observe here that $f^{-1}$ is a left/right inverse to $f$ if and only if
\begin{equation*}
f^{-1}(H_{g,2}, H_{g,1}) = -\Phi^{-1} f(H_{g,1},H_{g,2}) \Phi^{-1},
\end{equation*}
where $\Phi^{-1}$ is a left/right inverse to the cochain map $\Phi$ underlying $f$. In this case, the map that sends $H_{g,1} \in F_{G,1}$ to the unique $H_{g,2} \in F_{G,2}$ for which\footnote{Explicitly, $H_{g,2}$ is given by $H_{g,2}' \cdot h_{\bfs g}$, where $H_{g,2}' \in F_{G,2}$ is arbitrary, and
\begin{equation*}
h_{\bfs g} \coloneq (H_{g,2}')^{-1} \cdot f(H_{g,1}, H_{g,2}')(\Phi^{V_M})^{-1}.
\end{equation*}}
\begin{equation*}
f(H_{g,1}, H_{g,2}) = 0
\end{equation*}
is a groupoid immersion/submersion $F_{G,1} \ra F_{G,2}$ that we can fit into a commutative diagram
\begin{center}
\begin{tikzcd}
1 \ar[r] & \textnormal{H}(C_{M,1}, V_{M,1}) \ar[r] \ar[d] & F_{G,1} \ar[r] \ar[d] & G \ar[r] \ar[d, "="] & 1 \\
1 \ar[r] & \textnormal{H}(C_{M,2}, V_{M,2}) \ar[r] & F_{G,2} \ar[r] & G \ar[r] & 1
\end{tikzcd}
\end{center}
Conversely, all such commutative diagrams arise in this way; that is, those for which the vertical map on the left is given by
\begin{equation*}
\textnormal{H}(C_{M,1}, V_{M,1}) \ra \textnormal{H}(C_{M,2}, V_{M,2}) \qquad h_x \mapsto \Phi^{C_M}h_x(\Phi^{V_M})^{-1}
\end{equation*}
if $\Phi^{-1}$ a left inverse of $\Phi$, and 
\begin{equation*}
    \textnormal{H}(C_{M,1}, V_{M,1}) \ra \textnormal{H}(C_{M,2}, V_{M,2}) \qquad h_x \mapsto (\Phi^{C_M})^{-1}h_x\Phi^{V_M}
    \end{equation*}
if $\Phi^{-1}$ a right inverse of $\Phi$.
\end{remark}

There are of course many other classes of maps of interest. For example, the class of maps for which the underlying cochain map is a quasi-isomorphism. Such maps correspond to VB-Morita maps (as can be inferred from \cite{VBMorita}). This will be further clarified and explored in \cite{Homotopyoperators}.

To conclude this section, we summarise:

\begin{proposition}\label{prop: VB-map to fat map}
Let $\Phi: V_{G,1} \ra V_{G,2}$ be a VB-groupoid map. Then,
\begin{equation*}
f_\Phi: \widehat V_{G,1} \ra \widehat V_{G,2},
\end{equation*}
defined as in \eqref{eq: VB-map to fat map}, is a morphism of fat extensions. Moreover, the assignment $\Phi \mapsto f_\Phi$ respects composition.
\end{proposition}
\begin{proof}
Only the last statement is left to be verified. Given $\Phi_{21}$ and $\Phi_{32}$ VB-groupoid maps, we compute: any $H_{g,2} \in \widehat V_{G,2}$ gives
\begin{align*}
r_g f_{31}(H_{g,1}, H_{g,3}) &= H_{g,3}(\Phi_{32}\Phi_{21}v) - \Phi_{32}\Phi_{21}(H_{g,1}v) \\
&= \Phi_{32}(H_{g,2}\Phi_{21}v - \Phi_{21}(H_{g,1}v)) + (H_{g,3}\Phi_{32} - \Phi_{32}H_{g,2})\Phi_{21}v \\
&= r_g\Phi_{32}f_{21}(H_{g,1},H_{g,2})v + r_gf_{32}(H_{g,2},H_{g,3})\Phi_{21}v.
\end{align*}
This proves the statement.
\end{proof}

Starting with a map of fat extensions $f$, the invariance of $f$ makes sure the map $\Phi$, defined as in \eqref{eq: fat map to VB-map}, is well-defined. The multiplicativity of $f$ ensures this map $\Phi$ is a VB-groupoid map.

We arrived at the main theorem of this work.

\begin{theorem}\label{thm: equivalence of VBLG and fat}
The functor
\begin{equation*}
\{\textnormal{VB-groupoids}\} \ra \{\textnormal{Fat extensions}\} \qquad V_G \ra \widehat V_G \quad \Phi \mapsto f_\Phi
\end{equation*}
sets an equivalence of categories.
\end{theorem}
\begin{proof}
In this section we established the (fully faithful) functoriality, and in Section \ref{sec: fat extensions} (see Section \ref{sec: From fat extensions to VB-groupoids}; also see Section \ref{sec: The fat extension of a VB-groupoid}) we showed the essential surjectivity.
\end{proof}

\begin{remark}
The proof of \cite{VBMorita} that the functor
\begin{equation*}
\{\textnormal{Split $2$-term ruths}\} \ra \{\textnormal{VB-groupoids}\}
\end{equation*}
is an equivalence can now be understood as follows. We call $\{\textnormal{Split VB-groupoids}\}$ the category of VB-groupoids equipped with a cleavage. Then the above functor decomposes as
\begin{equation*}
\{\textnormal{Split $2$-term ruths}\} \ra \{\textnormal{Split VB-groupoids}\} \ra \{\textnormal{VB-groupoids}\}
\end{equation*}
where the latter equivalence is given by the forgetful functor. To see how the first equivalence works, observe that, from a VB-groupoid map $\Phi: V_{G,1} \ra V_{G,2}$, we can construct the map $f=f_\Phi$ as above on $\widehat V_G^\textnormal{cat} \times_G \widehat V_G^\textnormal{cat}$. So, cleavages $\Sigma_1$ and $\Sigma_2$ of $V_{G,1}$ and $V_{G,2}$, respectively, define a map
\begin{equation*}
\mu_g \coloneq f(\Sigma_{g,1},\Sigma_{g,2}): (V_{M,1})_{\bfs g} \ra (C_{M,2})_{\bft g}.
\end{equation*}
In the first multiplicativity condition involving $f$ (see Definition \ref{def: fat maps}) we can replace $f$ with $\mu$, but the second condition is true only up to the commutator $[R_2(g,g'),\Phi]$:\footnote{We simply wrote the associated split ruths of $V_{G,1}$ and $V_{G,2}$ using the same notation.}
\begin{equation*}
\mu_{gg'}v = \mu_gR_1(g')v + R_1(g)\mu_hv + R_2(g,g')\Phi v - \Phi R_2(g,g')v.
\end{equation*}
This precisely means the pair $(\Phi,\mu)$ defines a ruth map, and, as above in \eqref{eq: fat map to VB-map}, the equation
\begin{equation*}
\Phi^{V_G}(r_gc + \Sigma_{g,1}v) = r_g(\Phi^{C_M}(c) - \mu_gv) + \Sigma_{g,2}\Phi^{V_M} v
\end{equation*}
sets a one-to-one correspondence between VB-groupoid maps and ruth maps.
\end{remark}

The notions of morphisms for fat extensions and for fat category extensions are formally identical. Therefore, it is readily verified that we also obtain:

\begin{theorem}\label{thm: equivalence of fat category and fat}
The functor
\begin{equation*}
    \{\textnormal{Fat category extensions}\} \ra \{\textnormal{Fat extensions}\} \qquad F_G^\textnormal{cat} \ra F_G \quad f \mapsto f|_{F_G \times F_G}
\end{equation*}
sets an equivalence of categories.
\end{theorem}

\begin{remark}
Similarly, we can establish equivalences of the form
\begin{align*}
\{\textnormal{Fat extensions}\} &\ra \{\textnormal{VB-groupoids}\} \qquad F_G \mapsto V(F_G) \\
\{\textnormal{Fat extensions}\} &\ra \{\textnormal{Abstract $2$-term ruths}\} \qquad F_G \mapsto C_{\textnormal{inv}}(F_G; C_M) \\
\{\textnormal{Split $2$-term ruths}\} &\ra \{\textnormal{Split fat extensions}\} \qquad \calE_G \mapsto \widehat \calE_G
\end{align*}
but we will not go any further into these details here.
\end{remark}

\subsection{Morphisms of general linear PB-groupoids}\label{sec: morphisms of general linear PB-groupoids}

To understand maps between general linear PB-groupoids, we first go through the equivalence of categories between ordinary vector bundles and ordinary general linear principal bundles. Doing so, we will already come across some of the most essential ingredients to understand the more general case. 

\subsubsection{Morphisms of general linear principal bundles}

Let $E_1 \ra M$ be a vector bundle of rank $k_1$, $E_2 \ra M$ be a vector bundle of rank $k_2$, and consider a vector bundle map $E_1 \ra E_2$ over $M$. If we then consider two frames over the same basepoint $x \in M$, one of $E_1$ and one of $E_2$, we obtain the following diagram
\begin{center}
\begin{tikzcd}
\bbR^{k_1} \ar[r, "\sim"] \ar[d, dashed] & E_{1,x} \ar[d] \\ \bbR^{k_2} \ar[r, "\sim"] & E_{2,x}
\end{tikzcd}
\end{center}
So, a linear map $E_1 \ra E_2$ produces a map
\begin{equation*}
f_\Phi: \textnormal{Fr } E_1 \times_M \textnormal{Fr } E_2 \ra \textnormal{Hom}(\bbR^{k_1}, \bbR^{k_2}) \qquad f_\Phi(\phi_1,\phi_2) \coloneq \phi_2^{-1} \Phi \phi_1.
\end{equation*}
This map has an additional special property, which leads to the following definition.

\begin{definition}\label{def: maps of general linear principal bundles}
A \textit{morphism of principal $\textnormal{GL}$-bundles} is a map
\begin{equation*}
f: P_1 \times_M P_2 \ra \textnormal{Hom}(\bbR^{k_1}, \bbR^{k_2})
\end{equation*}
that is $\textnormal{GL}$-equivariant: given $p_1 \in P_1$, $p_2 \in P_2$, $\Psi_1 \in \textnormal{GL}(k_1)$ and $\Psi_2 \in \textnormal{GL}(k_2)$, we have
\begin{equation*}
f(p_1\Psi_1, p_2\Psi_2) = \Psi_2^{-1} f(p_1,p_2) \Psi_1.
\end{equation*}
Moreover, composition $f_{31} = f_{32} \circ f_{21}: P_1 \ra P_3$ of morphisms $f_{21}: P_1 \ra P_2$ and $f_{32}: P_2 \ra P_3$ is given by
\begin{equation*}
f_{31}(p_1,p_3) = f_{32}(p_2,p_3) \circ f_{21}(p_1,p_2)
\end{equation*}
where $p_2 \in P_2$ is arbitrary.
\end{definition}

The term ''equivariant'' is justified for $\textnormal{GL}(k_1)$ and $\textnormal{GL}(k_2)$ act canonically on $\textnormal{Fr } E_1 \times_M \textnormal{Fr } E_2$ and also on $\textnormal{Hom}(\bbR^{k_1}, \bbR^{k_2})$. The $\textnormal{GL}$-equivariance condition can then literally be understood as the equivariance with respect to these actions.

\begin{remark}\label{rema: isomorphisms of general linear principal bundles}
The identity map $P \ra P$ is the map
\begin{equation*}
P \times_M P \ra \textnormal{Hom}(\bbR^k, \bbR^k)
\end{equation*}
which assigns to $(p_1,p_2) \in P \times_M P$ the unique $\Psi \in \textnormal{GL}(k)$ with the property that $p_1 = p_2\Psi$. In particular, $f$ is an isomorphism if and only if $f$ maps into $\textnormal{GL}(k)$, for then
\begin{equation*}
f^{-1}(p_2,p_1) = f(p_1,p_2)^{-1}
\end{equation*}
where the inverse on the right hand side is taken in $\textnormal{GL}(k)$. In this case, the map sending $p_1 \in P_1$ to the unique $p_2 \in P_2$ for which
\begin{equation*}
f(p_1,p_2) = 1
\end{equation*}
is an equivariant isomorphism $P_1 \xra{\sim} P_2$. Conversely, all such equivariant isomorphisms arise in this way.
\end{remark}

With the correct notion of morphism of general linear principal bundles at hand, we now prove:

\begin{proposition}\label{prop: equivalence ordinary vector bundles principal bundles}
The assignment
\begin{equation*}
\{\textnormal{Vector bundles}\} \ra \{\textnormal{Principal $\textnormal{GL}$-bundles}\} \qquad E \mapsto \textnormal{Fr } E \qquad \Phi \mapsto f_\Phi
\end{equation*}
sets an equivalence of categories.
\end{proposition}
\begin{proof}
Suppose we are given a map
\begin{equation*}
\textnormal{Fr } E_1 \times_M \textnormal{Fr } E_2 \ra \textnormal{Hom}(\bbR^{k_1}, \bbR^{k_2})
\end{equation*}
that is $\textnormal{GL}$-equivariant. We can then define a linear map $E_{1,x} \ra E_{2, x}$ by choosing any two frames:
\begin{center}
\begin{tikzcd}
\bbR^{k_1} \ar[r, "\sim"] \ar[d] & E_{1,x} \ar[d, dashed] \\ \bbR^{k_2} \ar[r, "\sim"] & E_{2,x}
\end{tikzcd}
\end{center}
By the equivariance, this map is independent of the chosen frames. That we set a true one-to-one correspondence now is readily verified.
\end{proof}

We now move to understanding morphisms of general linear PB-groupoids. Although a natural guess arises from the above discussion, to avoid mistakes, it is especially useful to simply go through the above one-to-one correspondence in the more general setup. We start with setting the equivalence between general linear PB-groupoids and VB-groupoid as it seems the most appropriate following the above discussion. However, as we will see, the equivalence on the level of morphisms between general linear PB-groupoids and fat extensions is set in a similar way.

\subsubsection{Morphisms of general linear PB-groupoids}

Recall that the total space of the associated PB-groupoid of a VB-groupoid $V_G$ is the following space of frames:
\begin{equation*}
\textnormal{GL}_\bfs V_G = \{\phi_g: (C_M)_x \oplus (V_M)_x \xra{\sim} (V_G)_g \mid \bfs_g\phi_g|_{(C_M)_x} = 0, \bft_g\phi_g|_{(V_M)_x}, \bfs_g\phi_g|_{(V_M)_x} \in \textnormal{GL}(V_M)\}
\end{equation*}
Given a VB-groupoid map $V_{G,1} \ra V_{G,2}$ over $G$, and picking one such frame $\phi_1$ for $\textnormal{GL}_\bfs V_{G,1}$ and $\phi_2$ for $\textnormal{GL}_\bfs V_{G,2}$, we obtain the following diagram
\begin{center}
\begin{tikzcd}
(C_{M,1})_{x_1} \oplus (V_{M,1})_{x_1} \ar[r, "\sim"] \ar[d, dashed] & (V_{G,1})_g \ar[d] \\
(C_{M,2})_{x_2} \oplus (V_{M,2})_{x_2} \ar[r, "\sim"] & (V_{G,2})_g
\end{tikzcd}
\end{center}
Before we move on, observe that the resulting map
\begin{equation*}
(C_{M,1})_{x_1} \oplus (V_{M,1})_{x_1} \ra (C_{M,2})_{x_2} \oplus (V_{M,2})_{x_2}
\end{equation*}
is not an element of $\textnormal{Hom}(C_{M,1} \oplus V_{M,1}, C_{M,2} \oplus V_{M,2})$, but rather of
\begin{equation*}
\textnormal{Lin}(C_{M,1} \oplus V_{M,1}, C_{M,2} \oplus V_{M,2})
\end{equation*}
where, as in Section \ref{sec: Fat extensions as a functor to the quasi general linear Lie 2-groupoid}, for $E_1$ and $E_2$ vector bundles over $M$,
\begin{equation*}
\textnormal{Lin}(E_1,E_2) \coloneq \{E_{1,x} \ra E_{2,y} \mid x,y \in M\}
\end{equation*}
denotes the vector bundle over $M \times M$ of linear maps between fibers of $E_1$ and $E_2$. We often denote elements of $\textnormal{Lin}(C_{M,1} \oplus V_{M,1}, C_{M,2} \oplus V_{M,2})$ by blockmatrices, and use
\begin{equation*}
\textnormal{U}(C_{M,1} \oplus V_{M,1}, C_{M,2} \oplus V_{M,2})
\end{equation*}
to mean upper triangular blockmatrices. The reason for introducing this subbundle is because the maps from above are of this form, for $\bfs \phi_2 = 0$. We write the components of the above maps now (suggestively) as follows:
\begin{equation*}
\begin{pmatrix} \Phi(\phi_1,\phi_2) & f_\Phi(\phi_1, \phi_2) \\ 0 & \Phi(\phi_1,\phi_2) \end{pmatrix} (C_{M,1})_{x_1} \oplus (V_{M,1})_{x_1} \ra (C_{M,2})_{x_2} \oplus (V_{M,2})_{x_2}.
\end{equation*}
Up to conjugation by frames, $\Phi(\phi_1,\phi_2)$ is precisely the cochain map from $C_{M,1} \ra V_{M,1}$ to $C_{M,2} \ra V_{M,2}$. In analogy with Proposition \ref{prop: equivalence ordinary vector bundles principal bundles}, this does not come as a surprise. Rests to consider the map $f_\Phi: (V_{M,1})_{x_1} \ra (C_{M,2})_{x_2}$. To understand its essential property, we compute:

\begin{proposition}\label{prop: essential property of f for PB-groupoid maps from VB-groupoid map}
Let $\Phi: V_{G,1} \ra V_{G,2}$ be a VB-groupoid map, and consider $\Phi(\phi_1,\phi_2)$ and $f_\Phi(\phi_1,\phi_2)$ defined as above. Then, given $\Psi_1 \in \textnormal{GL}(C_{M,1},V_{M,1})_M$ composable with $\phi_1 \in \textnormal{GL}(C_M,V_M)$ and $\Psi_2 \in \textnormal{GL}(C_{M,2},V_{M,2})_M$ composable with $\phi_2 \in \textnormal{GL}(C_M,V_M)$, we have
\begin{equation}
\Phi(\phi_1 \cdot \Psi_1, \phi_2 \cdot \Psi_2) = \Psi_2^{-1} \Phi(\phi_1,\phi_2) \Psi_1,
\end{equation}
and, given $(h_1, \Psi_1) \in \textnormal{GL}(C_{M,1},V_{M,1})_G$ composable with $\phi_1 \in P_G$ and $(h_2, \Psi_2) \in \textnormal{GL}(C_{M,2},V_{M,2})_G$ composable with $\phi_2 \in P_G$, we have\footnote{The differentials $\partial_1$ and $\partial_2$ that appear are the suppressed differentials in the elements $(h_1, \Psi_1)$ and $(h_2, \Psi_2)$.}
\begin{align}\label{eq: property of f PB-groupoids}
&f_{\Phi}(\phi_1 \cdot (h_1, \Psi_1), \phi_2 \cdot (h_2, \Psi_2)) = \nonumber \\ &= (\Psi_2^{C_M})^{-1} (1+h_2\partial_2)^{-1}(f_{\Phi}(\phi_1, \phi_2) + \Phi^{C_M}(\phi_1,\phi_2)h_1 - h_2\Phi^{V_M}(\phi_1,\phi_2))\Psi_1^{V_M}.
\end{align}
\end{proposition}
\begin{proof}
We prove \eqref{eq: property of f PB-groupoids}. Using the matrix notation we developed, the map
\begin{center}
\begin{tikzcd}
(C_{M,1})_{x_1} \oplus (V_{M,1})_{x_1} \ar[r, "\sim"] \ar[d, dashed] & (V_{G,1})_g \ar[d] \\
(C_{M,2})_{x_2} \oplus (V_{M,2})_{x_2} \ar[r, "\sim"] & (V_{G,2})_g
\end{tikzcd}
\end{center}
that we considered above is given by
\begin{equation*}
\begin{pmatrix} \Phi^{C_M}(\phi_1,\phi_2) & f_\Phi(\phi_1,\phi_2) \\ 0 & \Phi^{V_M}(\phi_1,\phi_2) \end{pmatrix}.
\end{equation*}
The left hand side of \eqref{eq: property of f PB-groupoids} corresponds to the diagram
\begin{center}
\begin{tikzcd}
(C_{M,1})_{y_1} \oplus (V_{M,1})_{y_1} \ar[r, "\sim"] \ar[d, dashed] & (C_{M,1})_{x_1} \oplus (V_{M,1})_{x_1} \ar[r, "\sim"] & (V_{G,1})_g \ar[d] \\
(C_{M,2})_{y_2} \oplus (V_{M,2})_{y_2} \ar[r, "\sim"] & (C_{M,2})_{x_2} \oplus (V_{M,2})_{x_2} \ar[r, "\sim"] & (V_{G,2})_g
\end{tikzcd}
\end{center}
which is the map
\begin{equation*}
\begin{pmatrix} \Phi^{C_M}(\phi_1 \cdot (h_1, \partial_1, \Psi_1), \phi_2 \cdot (h_2, \partial_2, \Psi_2)) & f_\Phi(\phi_1 \cdot (h_1, \partial_1, \Psi_1), \phi_2 \cdot (h_2, \partial_2, \Psi_2)) \\ 0 & \Phi^{V_M}(\phi_1 \cdot (h_1, \partial_1, \Psi_1), \phi_2 \cdot (h_2, \partial_2, \Psi_2)) \end{pmatrix}.
\end{equation*}
But, on the other hand, the diagram also represents the composition
\begin{equation*}
\begin{pmatrix} (1+h_2\partial_2)\Psi_2^{C_M} & h_2\Psi_2^{V_M} \\ 0 & \Psi_2^{V_M} \end{pmatrix}^{-1} \begin{pmatrix} \Phi^{C_M}(\phi_1,\phi_2) & f_\Phi(\phi_1,\phi_2) \\ 0 & \Phi^{V_M}(\phi_1,\phi_2) \end{pmatrix} \begin{pmatrix} (1+h_1\partial_1)\Psi_1^{C_M} & h_1\Psi_1^{V_M} \\ 0 & \Psi_1^{V_M} \end{pmatrix}.
\end{equation*}
This proves the statement.
\end{proof}

Observe that we can view the last matrix multiplication in the above proof as describing two right $2$-actions, one of $\textnormal{GL}(C_{M,1},V_{M,1})_G$ and one of $\textnormal{GL}(C_{M,2},V_{M,2})_G$ (see Remark \ref{rema: representations general linear strict Lie 2-groupoid}). That is, these $2$-actions of $\textnormal{GL}(C_{M,1},V_{M,1})_G$ and of $\textnormal{GL}(C_{M,2},V_{M,2})_G$ --- the two matrix multiplications --- are well-defined $2$-actions on $\textnormal{U}(C_{M,1} \oplus V_{M,1}, C_{M,2} \oplus V_{M,2})$. The above property of such maps
\begin{equation*}
\textnormal{GL}_\bfs V_{G,1} \times_G \textnormal{GL}_\bfs V_{G,2} \ra \textnormal{U}(C_{M,1} \oplus V_{M,1}, C_{M,2} \oplus V_{M,2})
\end{equation*}
can then literally be interpreted as $\textnormal{GL}$-equivariance.

\begin{definition}\label{def: morphism of PB-groupoids}
A \textit{morphism of general linear PB-groupoids} is a map
\begin{equation*}
\begin{pmatrix} \Phi^{C_M} & f_\Phi \\ 0 & \Phi^{V_M} \end{pmatrix}: P_{G,1} \times_G P_{G,2} \ra \textnormal{U}(C_{M,1} \oplus V_{M,1}, C_{M,2} \oplus V_{M,2})
\end{equation*}
that is $\textnormal{GL}$-equivariant; that is, given $(h_1, \Psi_1) \in \textnormal{GL}(C_{M,1},V_{M,1})_G$ composable with $p_1 \in P_G$ and $(h_2, \Psi_2) \in \textnormal{GL}(C_{M,2},V_{M,2})_G$ composable with $p_2 \in P_G$, we have
\begin{align*}
&\begin{pmatrix} \Phi^{C_M}(p_1 \cdot (h_1,\Psi_1),p_2 \cdot (h_2,\Psi_2)) & f_\Phi(p_1 \cdot (h_1,\Psi_1),p_2 \cdot (h_2,\Psi_2)) \\ 0 & \Phi^{V_M}(p_1 \cdot (h_1,\Psi_1),p_2 \cdot (h_2,\Psi_2)) \end{pmatrix} = \\
&= \begin{pmatrix} (1+h_2\partial_2)\Psi_2^{C_M} & h_2\Psi_2^{V_M} \\ 0 & \Psi_2^{V_M} \end{pmatrix}^{-1} \begin{pmatrix} \Phi^{C_M}(p_1,p_2) & f_\Phi(p_1,p_2) \\ 0 & \Phi^{V_M}(p_1,p_2) \end{pmatrix} \begin{pmatrix} (1+h_1\partial_1)\Psi_1^{C_M} & h_1\Psi_1^{V_M} \\ 0 & \Psi_1^{V_M} \end{pmatrix}.
\end{align*}
\end{definition}

Starting with a pair $(\Phi, f_\Phi)$ that is $\textnormal{GL}$-equivariant, we can define $\Phi: V_{G,1} \ra V_{G,2}$ by picking, for all $g \in G$, any two frames and using the diagram
\begin{center}
\begin{tikzcd}
(C_{M,1})_{x_1} \oplus (V_{M,1})_{x_1} \ar[r, "\sim"] \ar[d] & (V_{G,1})_g \ar[d, dashed] \\
(C_{M,2})_{x_2} \oplus (V_{M,2})_{x_2} \ar[r, "\sim"] & (V_{G,2})_g
\end{tikzcd}
\end{center}
that $(\Phi, f_\Phi)$ induces. The resulting map $\Phi_g: (V_{G,1})_g \ra (V_{G,2})_g$ is independent of chosen frames by $\textnormal{GL}$-equivariance. So, similar to Proposition \ref{prop: equivalence ordinary vector bundles principal bundles}, we have:

\begin{theorem}\label{thm: equivalence VB-groupoids and PB-groupoids}
The assignment
\begin{equation*}
\{\textnormal{VB-groupoids}\} \ra \{\textnormal{General linear PB-groupoids}\} \qquad V_G \mapsto \textnormal{GL}_\bfs V_G \qquad \Phi^{V_G} \mapsto \begin{pmatrix} \Phi^{C_M} & f_\Phi \\ 0 & \Phi^{V_M} \end{pmatrix}
\end{equation*}
sets an equivalence of categories.
\end{theorem}
\begin{proof}
In this section we established the (fully faithful) functoriality, and in Section \ref{sec: PB-groupoids} (see Section \ref{sec: From VB-groupoids to general linear PB-groupoids}) we showed the essential surjectivity (see also \cite{PB-groupoids}).
\end{proof}

\begin{remark}\label{rema: isomorphisms of PB-groupoids}
Analogous to Remark \ref{rema: isomorphisms of general linear principal bundles}, the identity map $P_G \ra P_G$ of a PB-groupoid $P_G$ over $C_M \ra V_M$ is the map
\begin{equation*}
P_G \times_G P_G \ra \textnormal{U}(C_M \oplus V_M, C_M \oplus V_M)
\end{equation*}
which assigns to $(p_1,p_2) \in P_G \times_G P_G$ the unique (matrix representation of) $(h,\Psi) \in \textnormal{GL}(C_M,V_M)_G$ with the property that $p_1 = p_2 \cdot (h,\Psi)$. In particular, $f$ is an isomorphism if and only if $f$ maps into $\textnormal{GL}(C_M \oplus V_M)$, for then
\begin{equation*}
f^{-1}(p_2,p_1) = f(p_1,p_2)^{-1}
\end{equation*}
where the inverse on the right hand side is taken in $\textnormal{GL}(C_M \oplus V_M)$. In this case, the map sending $p_1 \in P_{G,1}$ to the unique $p_2 \in P_{G,2}$ for which
\begin{equation*}
f(p_1,p_2) = 1
\end{equation*}
is an equivariant isomorphism $P_1 \xra{\sim} P_2$. Conversely, all such equivariant isomorphisms arise in this way.
\end{remark}

\subsubsection{Morphisms of general linear PB-groupoids using fat extensions}

Given a map of general linear PB-groupoids
\begin{equation*}
\begin{pmatrix} \Phi_{\textnormal{PB}}^{C_M} & f_{\textnormal{PB}} \\ 0 & \Phi_{\textnormal{PB}}^{V_M} \end{pmatrix}: P_{G,1} \times_G P_{G,2} \ra \textnormal{U}(C_{M,1} \oplus V_{M,1}, C_{M,2} \oplus V_{M,2})
\end{equation*}
the cochain map $\Phi$ from $C_{M,1} \ra V_{M,1}$ to $C_{M,2} \ra V_{M,2}$ over $x \in M$ can simply be recovered as
\begin{equation*}
\Phi_x \coloneq \Phi_{\textnormal{PB}}(1_{1_x}, 1_{1_x}).
\end{equation*}
But we can also pick any other elements $\Psi_1 \in \textnormal{GL}(C_{M,1},V_{M,1})$ and $\Psi_2 \in \textnormal{GL}(C_{M,2},V_{M,2})$ over $x$, for then we have
\begin{equation*}
\Phi_x = \Psi_2 \Phi_{\textnormal{PB}}(\Psi_1, \Psi_2) \Psi_1^{-1}.
\end{equation*}
by $\textnormal{GL}$-equivariance.\footnote{The identification we use here uses Proposition \ref{prop: PB-groupoid has H in it}.} Reading this equation the other way around, given $\Phi$, we can generate the map $\Phi_{\textnormal{PB}}$ from the cochain map $\Phi$ by imposing $\textnormal{GL}$-equivariance.

Now, if we realise $P_{G,1}$ and $P_{G,2}$ through their associated fat extensions; that is, writing
\begin{equation*}
P_{G,1} = F_{G,1} \ltimes \textnormal{GL}(C_{M,1}, V_{M,1}) \qquad P_{G,2} = F_{G,2} \ltimes \textnormal{GL}(C_{M,2}, V_{M,2}),
\end{equation*}
then we can similarly recover the map of fat extensions as
\begin{equation*}
f_{\textnormal{fat}}(H_{g,1}, H_{g,2}) \coloneq f_{\textnormal{PB}}(H_{g,1}, 1_{\bfs g}, H_{g,2}, 1_{\bfs g}).
\end{equation*}
A straightforward computation shows that $f_\textnormal{fat}$ is indeed a map of fat extensions. On the other hand, we can generate the map of general linear PB-groupoids $f_{\textnormal{PB}}$ from the map of fat extensions $f_{\textnormal{fat}}$ similar to how we generated $\Phi_{\textnormal{PB}}$ from $\Phi$. We therefore obtain:
\begin{theorem}\label{thm: equivalence fat extensions and PB-groupoids}
The functor
\begin{equation*}
\{\textnormal{Fat extensions}\} \ra \{\textnormal{General linear PB-groupoids}\} \qquad F_G \mapsto F_G \ltimes \textnormal{GL}(C_M,V_M) \qquad f_{\textnormal{fat}} \mapsto f_{\textnormal{PB}}
\end{equation*}
sets an equivalence of categories.
\end{theorem}
\begin{proof}
In this section we established the (fully faithful) functoriality, and in Section \ref{sec: PB-groupoids} (see Sections \ref{sec: From fat extensions to general linear PB-groupoids} and \ref{sec: From general linear PB-groupoids to fat extensions}) we showed the essential surjectivity (see also \cite{PB-groupoids}).
\end{proof}

\begin{remark}\label{rema: equivalence fat extensions and PB-groupoids}
It might be worthwhile to realise that, when before we used VB-groupoids, we built the map of PB-groupoids using the diagram
\begin{center}
\begin{tikzcd}
(C_{M,1})_{x_1} \oplus (V_{M,1})_{x_1} \ar[r, "\sim"] \ar[d, dashed] & (V_{G,1})_g \ar[d] \\
(C_{M,2})_{x_2} \oplus (V_{M,2})_{x_2} \ar[r, "\sim"] & (V_{G,2})_g
\end{tikzcd}
\end{center}
But if we go through the composition of functors
\begin{equation*}
\{\textnormal{VB-groupoids}\} \ra \{\textnormal{Fat extensions}\} \ra \{\textnormal{General linear PB-groupoids}\},
\end{equation*}
we, instead, build the map of PB-groupoids as follows:
\begin{center}
\begin{tikzcd}
(C_{M,1})_{x_1} \oplus (V_{M,1})_{x_1} \ar[r, "\sim"] \ar[d, dashed] & (C_{M,1})_{\bft g} \oplus (V_{M,1})_{\bfs g} \ar[r, "\sim"] \ar[d] & (V_{G,1})_g \\
(C_{M,2})_{x_2} \oplus (V_{M,2})_{x_2} \ar[r, "\sim"] & (C_{M,2})_{\bft g} \oplus (V_{M,2})_{\bfs g} \ar[r, "\sim"] & (V_{G,2})_g
\end{tikzcd}
\end{center}
where we can think of the middle vertical map as taking the right-most horizontal maps as input. The end result is the same, because the diagram
\begin{center}
\begin{tikzcd}
(C_{M,1})_{\bft g} \oplus (V_{M,1})_{\bfs g} \ar[r, "\sim"] \ar[d] & (V_{G,1})_g \ar[d] \\
(C_{M,2})_{\bft g} \oplus (V_{M,2})_{\bfs g} \ar[r, "\sim"] & (V_{G,2})_g
\end{tikzcd}
\end{center}
commutes. The following full commutative diagram is perhaps a nice illustration of how all three different maps fit together (from left to right: a map of PB-groupoids, a map of fat extensions and a map of VB-groupoids): 
\begin{center}
\begin{tikzcd}
(C_{M,1})_{x_1} \oplus (V_{M,1})_{x_1} \ar[r, "\sim"] \ar[d] & (C_{M,1})_{\bft g} \oplus (V_{M,1})_{\bfs g} \ar[r, "\sim"] \ar[d] & (V_{G,1})_g \ar[d] \\
(C_{M,2})_{x_2} \oplus (V_{M,2})_{x_2} \ar[r, "\sim"] & (C_{M,2})_{\bft g} \oplus (V_{M,2})_{\bfs g} \ar[r, "\sim"] & (V_{G,2})_g
\end{tikzcd}
\end{center}
The first vertical map (the map of PB-groupoids) takes as input the two horizontal maps (in both rows), the second vertical map (the map of fat extensions) takes as input the right horizontal maps (in both rows), and the last vertical map is simply defined without extra input.
\end{remark}

\section{The infinitesimal picture}\label{sec: The infinitesimal picture}

We now turn to the infinitesimal theory of what we have discussed until now. As already observed in \cite{VBalgebroidsruths}, there is a fat algebroid $\widehat V_A$ associated to a VB-algebroid $V_A$. It carries canonical representations on $C_M$ and $V_M$, and it is shown in \cite{VBalgebroidsruths} how one obtains the (split) ruth associated to $V_A$ using these representations and a splitting of a canonical exact sequence:
\begin{center}
\begin{tikzcd}\label{eq: short exact sequence fat algebroid}
0 \ar[r] & \textnormal{Hom}(V_M,C_M) \ar[r] & \widehat V_A \ar[r] & A \ar[r] & 0.
\end{tikzcd}
\end{center}
In this section we elaborate on these remarks, explaining the infinitesimal counterparts of the theory we have seen in the above sections. We keep the discussion shorter, and a discussion on ``PB-algebroids'' (a concept which is not defined in the literature yet) is not included. We put an emphasis on describing the correspondences between VB-algebroids, $2$-term ruths, and the newly defined fat extensions. In the section after, Section \ref{sec: Fat Lie theory}, we explain the process of differentiation.

\subsection{Notation for VB-algebroids}
See, for example, \cite{VBalejandrohenriquematias,VanEsthomogeneous} for more details on the theory of VB-algebroids.

We denote the structure maps of a Lie algebroid by $([\cdot,\cdot], \rho)$. Writing $\Omega^\bullet A = \Gamma \textstyle\bigwedge A^*$, we write
\begin{align*}
d: \Omega^\bullet A \ra \Omega^{\bullet+1} A
\end{align*}
for the algebroid differential (the de Rham differential using $[\cdot,\cdot]$ and $\calL = \calL_\rho$). Then $\Omega A$ becomes a (commutative) differential graded algebra with $d$ and the wedge product.

For $A \Ra M$ a Lie algebroid, we denote a VB-algebroid over $A$ by $V_A \Ra V_M$.
\begin{remark}\label{rema: VB algebroid notation}
We think of a VB-algebroid as a diagram
\begin{center}
\begin{tikzcd}
V_A \ar[r,Rightarrow] \ar[d] & V_M \ar[d] \\
A \ar[r,Rightarrow]& M
\end{tikzcd}
\end{center}
so that we think of the algebroid structures as being ``horizontal'', while the other vector bundle structures (on $V_A$ and $V_M$) are ``vertical''. In line with this convention, we write
the projection maps $V_A \ra V_M$ and $A \ra M$ by $\pr_\ra$, and for the projection maps $V_A \ra A$ and $V_M \ra M$ we write $\pr_\da$. Moreover, we write
\begin{equation*}
\Gamma_{\ra} V_A = \Gamma(V_M,V_A) \textnormal{ and } \Gamma_{\da} V_A = \Gamma(A,V_A)
\end{equation*}
for the spaces of sections, and similarly $\Omega_\ra V_A$ and $\Omega_\da V_A$ for the respective forms, and we write the scalar multiplications and additions of $V_A \ra A$ and $V_M \ra M$ by $h^\lambda_\da$ and $+_\da$, respectively, and the scalar multiplications and additions of $V_A \ra V_M$ and $A \ra M$ by $h^\lambda_\ra$ and $+_\ra$. However, we often simplify to
\begin{equation*}
h^\lambda_\da v = \lambda_\da v \qquad h^\lambda_\ra v = \lambda_\ra v
\end{equation*}

In particular, $\lambda_\da$ is infinitesimally multiplicative:
\begin{equation*}
\lambda^*_\da [v,w] = [\lambda_\da^* v,\lambda_\da^* w]
\end{equation*}
(with $\lambda_\da^*v = \lambda^{-1}_\da \circ v \circ \lambda_\da$). That is, the scalar multiplication $\lambda_\da$ by $\lambda \in \bbR_{>0}$ is an algebroid automorphism $V_A \ra V_A$ \cite{VBalejandrohenriquematias}. On the other hand, the scalar multiplication $\lambda_\ra$ by $\lambda \in \bbR_{>0}$ is a vector bundle isomorphism $V_A \ra V_A$ with respect to the vertical vector bundle structures. Recall that on the core
\begin{equation*}
C_M = \ker \pr_\ra \cap \ker \pr_\da \ra M
\end{equation*}
of $V_A$ the two vector bundle structures coincide, i.e. for all $c \in C_M$, $\lambda_\da c = \lambda_\ra c$.
\end{remark}

Since $\lambda_\da$ is an algebroid map, the $1$-homogeneous cochains (with respect to $\lambda_\da$)
\begin{equation}\label{eq: linear complex of a VB algebroid}
\Omega^\bullet_{\textnormal{lin}}V_A = \{\omega \in \Omega_\ra^\bullet V_A \mid \lambda_\da^*\omega = \lambda \cdot \omega\} \subset \Omega_\ra^\bullet V_A
\end{equation}
form a differential graded submodule of $\Omega V_A$ called the linear (sub)complex. More generally, a form $\omega \in \Omega_{\ra} V_A$ is called $k$-homogeneous if
\begin{equation*}
\lambda_\da^*\omega = \lambda^k\omega
\end{equation*}
and the $k$-homogeneous forms form a differential graded submodule of $\Omega V_A$.

Similarly, given a section $\alpha \in \Gamma_\ra V_A$, it is $k$-homogeneous if\footnote{Notice that the appearance of $\lambda_\ra^k$ on the right hand side is similar to the appearance of $\lambda^k$ on the right hand side of the equation of homogeneity of functions on $V_A$. More precisely, $V_A \times \bbR \Ra V_A$ can be seen as a VB-algebroid over $0_{V_M}$ or $0_A$, and $k$-homogeneous sections of these VB-algebroids are exactly the $k+1$-homogeneous functions.
}
\begin{equation*}
\lambda_\da^* \alpha (=\lambda^{-1}_\da \circ \alpha \circ \lambda_\da) = \lambda^k_\ra \alpha
\end{equation*}
and we write $\Gamma_\textnormal{lin}V_A$ for the linear (i.e. $0$-homogeneous) sections.

\begin{remark}\label{rema: homogeneous functions on V_A}
A word on terminology and notation is perhaps on its place. Given a function $f \in C^\infty V_A$, we have two notions of $f$ being $k$-homogeneous, for it can be homogeneous ``horizontally'' or ``vertically'':
\begin{equation*}
\lambda_\ra^* f = \lambda^k f \qquad \lambda_\da^* f = \lambda^k f.
\end{equation*}
We therefore write
\begin{equation*}
C^\infty_{\ra\textnormal{lin}} V_A \qquad C^\infty_{\da\textnormal{lin}} V_A
\end{equation*}
for horizontally linear (i.e. $1$-homogeneous) functions and vertically linear functions respectively.

Notice now that, since\footnote{The notation $V_A^{*_{V_M}}$ is used to indicate that we dualise $V_A$ with respect to the vector bundle structure $V_A \ra V_M$.}
\begin{equation*}
\Omega_\ra^1 V_A = \Gamma V_A^{*_{V_M}} = C^\infty_{\ra \textnormal{lin}} V_A
\end{equation*}
we can think of elements $\omega$ of $\Omega_\ra^1 V_A$ as linear functions on $V_A$ with respect to the horizontal vector bundle structure. If $\omega \in \Omega_\textnormal{lin}^1 V_A$, then the induced function of $V_A$ is also linear with respect to the vertical vector bundle structure. This means that $\omega$ defines an element of $\Omega_\da^1 V_A$ as well. For general $\omega \in \Omega^\bullet_\textnormal{lin} V_A$, we have multilinearity ``horizontally'' and linearity ``vertically'' in the sense that
\begin{equation*}
\omega(\lambda_\da v_1 +_\da \mu_\da w_1, \dots, \lambda_\da v_\bullet +_\da \mu_\da w_\bullet) = \lambda\omega(v_1,\dots,v_\bullet) + \mu\omega(w_1,\dots,w_\bullet).
\end{equation*}
That is, ``linearity'' refers here to the vertical vector bundle structure, and it would, in line with our convention, be sensible to write $\Omega_\textnormal{lin} V_A$ as $\Omega_{\ra, \da\textnormal{lin}} V_A$. We avoid this cumbersome notation as it is clear from the context that, in this case, we are interested in forms on the algebroid $V_A \Ra V_M$ that are linear with respect to the compatible vector bundle structures on $V_A \ra A$ and $V_M \ra M$.
\end{remark}

In local coordinates it becomes clear that $\Gamma_\ra V_A$ is generated, as a $C^\infty V_M$-module, by the $(-1)$-homogeneous sections, called \textit{core sections}, and the $0$-homogeneous sections, called \textit{linear sections}. 

\begin{remark}\label{rema: local coordinate description of linear forms}
It is convenient to think about linear (or more generally homogeneous) forms as follows. The space of homogeneous sections in $\Gamma_{\ra} V_A$ are naturally graded. We can think of a form $\omega \in \Omega_\textnormal{lin} V_A$ therefore as a map
\begin{equation*}
\omega: \Gamma_{\ra} V_A \otimes_{C^\infty V_M} \cdots \otimes_{C^\infty V_M} \Gamma_{\ra} V_A \ra C^\infty V_M
\end{equation*}
that additionally restricts to a map of degree $+1$ between homogeneous elements (considering the total degree on the left hand side). Indeed,
\begin{equation*}
\omega(\alpha_\ell) \circ \lambda_\da = (\lambda_\da^* \omega)(\lambda_\da^* \alpha_\ell)
\end{equation*}
so the functions $\omega(\alpha_\ell)$, for $k_\ell$-homogeneous $\alpha_\ell \in \Gamma_{\ra} V_A$, are $k + \textstyle\sum_\ell k_\ell$-homogeneous if and only if $\omega$ is $k$-homogeneous.
\end{remark}

\subsection{From VB-algebroids to fat extensions}
We begin with a continuation of the previous section.

\subsubsection{The appearance of the fat algebroid}
From Remark \ref{rema: local coordinate description of linear forms}, we can conclude that $\omega \in \Omega_\textnormal{lin} V_A$ is completely determined by a map of the form
\begin{equation*}
\omega_0: \Gamma_\textnormal{lin} V_A \times \cdots \times \Gamma_\textnormal{lin} V_A \ra C^\infty_\textnormal{lin} V_M.
\end{equation*}
Now, $\Gamma_\textnormal{lin} V_A$ is the space of sections of a Lie algebroid $\widehat V_A$ over $M$, called the fat algebroid (see \cite{VBalgebroidsruths}). This algebroid fits into a short exact sequence (of Lie algebroids)
\begin{center}
\begin{tikzcd}
0 \ar[r] & \textnormal{Hom}(V_M,C_M) \ar[r] & \widehat V_A \ar[r] & A \ar[r] & 0
\end{tikzcd}
\end{center}
where the projection $\widehat V_A \ra A$ is induced by the projection $V_A \ra A$ (every linear section is projectable to a section of $A$), and the inclusion is given by
\begin{equation*}
\textnormal{Hom}(V_M,C_M) \ra \widehat V_A \qquad h_x \mapsto (v \mapsto h_x(v) +_\da 0_\ra(v)).
\end{equation*}
As a vector bundle, the fat algebroid is given by
\begin{equation*}
\widehat V_A = \{H_a: (V_M)_{\pr_\ra a} \ra (V_A)_a \mid \pr_\ra H_a = 1\}.
\end{equation*}
The anchor map is induced by the projection:
\begin{equation*}
\widehat\rho H_a = \rho a
\end{equation*}
and the Lie bracket is induced by the bracket of $V_A$ (the bracket of linear sections is linear).

\begin{remark}
In \cite{VBalgebroidsruths} the fat algebroid is written as $\hat A$. However, we prefer to keep track of the dependence on $V_A$.
\end{remark}

The Lie bracket of the bundle of Lie algebras $\textnormal{Hom}(V_M,C_M)$ is given by
\begin{equation*}
[h,h'] \coloneq h'\partial h - h\partial h'.
\end{equation*}
As is clear from its definition, it is defined for every $2$-term complex $C_M \ra V_M$.

\begin{definition}\label{def: Hom(V_M,C_M) bracket}
Let
\begin{equation*}
C_M \xra{\partial} V_M
\end{equation*}
be a $2$-term complex. We call the bundle of Lie algebras $\textnormal{Hom}(V_M,C_M)$ the \textit{bundle of infinitesimal homotopies} of the $2$-term complex $C_M \ra V_M$.
\end{definition}

\begin{remark}\label{rema: canonical cochain complex representation Hom(V_M,C_M)}
The bundle of infinitesimal homotopies $\textnormal{Hom}(V_M,C_M)$ comes canonically with a cochain complex representation on $C_M \ra V_M$: given $\gamma \in \Gamma C_M$ and $v \in \Gamma V_M$, we set
\begin{equation*}
\nabla_h \gamma \coloneq -h \partial \gamma \qquad \nabla_h v \coloneq -\partial h v.
\end{equation*}
\end{remark}

We see that, if $\omega \in \Omega_\textnormal{lin}^\bullet V_A$, then
\begin{equation*}
\omega_0: \Gamma \widehat V_A \times \overset{\bullet}{\cdots} \times \Gamma \widehat V_A \ra \Gamma V_M^* \qquad \omega_0(\alpha_1,\dots,\alpha_\bullet) \coloneq \iota_{\alpha_\bullet}\cdots\iota_{\alpha_1}\omega
\end{equation*}
and
\begin{equation*}
\omega_1: \Gamma \widehat V_A \times \overset{\bullet-1}{\cdots} \times \Gamma \widehat V_A \ra \Gamma C_M^* \qquad \omega_1(\alpha_1,\cdots,\alpha_{\bullet-1})\gamma \coloneq \iota_\gamma\iota_{\alpha_{\bullet-1}}\cdots\iota_{\alpha_1}\omega
\end{equation*}
together satisfy the following property: for all $h \in \textnormal{Hom}(V_M,C_M)$ we have
\begin{equation*}
\iota_h\omega_0 = \omega_1 h.
\end{equation*}
Notice that, in the definition $\omega_1 h$, a sign appears:
\begin{equation*}
(\omega_1 h)(\alpha_2,\dots,\alpha_\bullet)v \coloneq (-1)^{\bullet-1} \omega_1(\alpha_2,\dots,\alpha_\bullet) hv
\end{equation*}
We call $(\omega_0,\omega_1)$ ``invariant'' forms, and in the next sections we will explain how such invariant forms constitute a $2$-term ruth $\Omega_\textnormal{inv}(\widehat V_A; V_M^*)$ that is isomorphic to $\Omega_\textnormal{lin} V_A$.

\subsubsection{The fat algebroid and its canonical representations}

The fat algebroid $\widehat V_A$ has canonical representations on $C_M$ and $V_M^*$\footnote{We use here that $[\cdot,\cdot]$ is a degree $+1$ operator on homogeneous sections and, on homogeneous functions, $\rho_{V_A}(\alpha)$ has degree equal to the degree of $\alpha$.}
\begin{equation}\label{eq: representations of the fat algebroid actions}
\nabla_\alpha \coloneq [\alpha, \cdot]: \Gamma C_M \ra \Gamma C_M \qquad \nabla_\alpha \coloneq \rho_{V_A}\alpha: \Gamma V_M^* \ra \Gamma V_M^*
\end{equation}
and we can put them together to form a cochain complex representation on
\begin{equation*}
\partial: C_M \ra V_M.
\end{equation*}
Dually, we also have a cochain complex representation on
\begin{equation*}
V_M^* \xra{\partial^*} C_M^*.
\end{equation*}
We think of $V_M^*$ as sitting in degree $0$ and $C_M^*$ as sitting in degree $1$. We denote the resulting differential graded $\Omega A$-module by
\begin{equation*}
\Omega(\widehat V_A; V_M^* \ra C_M^*)^\bullet = \Omega^\bullet(\widehat V_A; V_M^*) \oplus \Omega^{\bullet-1}(\widehat V_A; C_M^*).
\end{equation*}
In analogy with Proposition \ref{prop: representation H(V_M,C_M) is action by conjugation}, we have:

\begin{proposition}\label{prop: action on Hom(V_M,C_M) equals representation}
The cochain complex representation of $\widehat V_A$ restricts to the canonical cochain complex representation of $\textnormal{Hom}(V_M,C_M)$. Moreover, the action of $\widehat V_A$ on $\textnormal{Hom}(V_M,C_M)$, viewing $\textnormal{Hom}(V_M,C_M)$ as an ideal of $\widehat V_A$, equals the representation of $\widehat V_A$ on $\textnormal{Hom}(V_M,C_M)$.
\end{proposition}
\begin{proof}
We prove the second statement, and we do a direct computation: locally, $\textnormal{Hom}(V_M,C_M)$ is generated by elements $f \otimes \gamma$, where $f \in \Gamma V_M^*$ and $\gamma \in \Gamma C_M$. Given $\alpha \in \Gamma \widehat V_A$, we then have
\begin{equation*}
[\alpha, f \otimes \gamma] = f \otimes \nabla_\alpha\gamma + \nabla_\alpha f \otimes \gamma. 
\end{equation*}
This proves the statement.
\end{proof}

\begin{remark}\label{rema: cochain complex representation algebroid in terms of gl}
Given a vector bundle $E_M$, algebroid representations of $A$ on $E_M$ correspond to algebroid maps
\begin{equation*}
A \ra \mathfrak{gl}(E_M),
\end{equation*}
where $\mathfrak{gl}(E_M)$ is the algebroid of derivations
\begin{equation*}
\mathfrak{gl}(E_M) \coloneq \{D = (D,X) \mid D(fs) = fD(s) + X(f) s\} \subset \textnormal{End } E_M \oplus TM.
\end{equation*}
Given two vector bundles $C_M$ and $V_M$, we set
\begin{equation*}
\mathfrak{gl}(C_M,V_M) \coloneq \mathfrak{gl}(C_M) \times_{TM} \mathfrak{gl}(V_M).
\end{equation*}
Using the action algebroid
\begin{equation*}
\mathfrak{gl}(C_M,V_M)_M \coloneq \textnormal{Hom}(C_M,V_M) \rtimes \mathfrak{gl}(C_M,V_M)
\end{equation*}
associated to the representation of $\mathfrak{gl}(C_M,V_M)$ on $\textnormal{Hom}(C_M,V_M)$ given by
\begin{equation*}
\nabla_D \partial \coloneq D^{V_M} \partial - \partial D^{C_M},
\end{equation*}
we see that cochain complex representations of $A$ on $C_M \ra V_M$ correspond to algebroid maps
\begin{equation*}
A \ra \mathfrak{gl}(C_M,V_M)_M
\end{equation*}
(see also Remark \ref{rema: PB-algebroids}).
\end{remark}

\subsubsection{Fat extensions}

We discussed now the structure behind the fat algebroid of a VB-algebroid that makes it into a \textit{fat extension}. Here is the abstract definition:

\begin{definition}\label{def: fat extensions for LA}
Let $A \Ra M$ be a Lie algebroid, and consider a $2$-term complex
\begin{equation*}
C_M \xra{\partial} V_M.
\end{equation*}
A \textit{fat extension} of $A$, with underlying complex $C_M \ra V_M$, consists of a Lie algebroid
\begin{equation*}
F_A \rra M,
\end{equation*}
that comes with a cochain complex representation on $C_M \ra V_M$, and a short exact sequence of Lie algebroids (over $M$)
\begin{center}
\begin{tikzcd}
0 \ar[r] & \textnormal{Hom}(V_M,C_M) \ar[r] & F_A \ar[r] & A \ar[r] & 0.
\end{tikzcd}
\end{center}
Moreover, the cochain complex representation of $F_A$ on $C_M \ra V_M$ restricts to the canonical cochain complex representation of $\textnormal{Hom}(V_M,C_M)$, and the two natural actions of $F_A$ on $\textnormal{Hom}(V_M,C_M)$ --- viewing $\textnormal{Hom}(V_M,C_M)$ as an ideal or as a representation --- agree.
\end{definition}

We will simply write $F_A$ for a fat extension, and the underlying algebroid $F_A$ is called the fat algebroid of the fat extension. From the above discussion, we obtain:

\begin{proposition}\label{prop: fat algebroid comes with fat extension}
Let $V_A$ be a VB-algebroid over $A$. Then, the fat algebroid $\widehat V_A$ of $V_A$, together with its cochain complex representation \eqref{eq: representations of the fat algebroid actions}, is a fat extension of $A$ over $C_M \ra V_M$.
\end{proposition}

\begin{remark}\label{rema: structure of dual VB-algebroid}
We observe here that, given a VB-algebroid $V_A$, the fat extension $\widehat V_A$ with its dual cochain complex representation should be seen as the fat extension of the dual VB-algebroid $V_A^* \Ra C_M^*$. First of all, the fat groupoid $\widehat V_A^*$ fits into a short exact sequence 
\begin{center}
\begin{tikzcd}
0 \ar[r] & \textnormal{Hom}(C_M^*,V_M^*) \ar[r] & \widehat V_A^* \ar[r] & A \ar[r] & 0.
\end{tikzcd}
\end{center}
Then, $\widehat V_A$ and $\widehat V_A^*$ are isomorphic as Lie algebroids via\footnote{Here, say, $\alpha \in \Gamma_\textnormal{lin} V_A$ is related to $\alpha \in \Gamma A$ via the projection, $\xi_\alpha \in \Gamma_\textnormal{lin} V_A^*$ is also related to $\alpha \in \Gamma A$ via the projection, $f \in \Gamma C_M^* = C^\infty_\textnormal{lin} C_M$ and $v \in \Gamma_\da V_A$.}
\begin{equation*}
\widehat V_A \ra \widehat V_A^*; \quad \alpha \mapsto \xi_\alpha(f)v = f(v-_\da \alpha \pr_\ra v).
\end{equation*}
Using the isomorphism
\begin{equation*}
\textnormal{Hom}(V_M,C_M) \ra \textnormal{Hom}(C_M^*,V_M^*); \quad h \mapsto -h^*,
\end{equation*}
the short exact sequences of $\widehat V_A$ and $\widehat V_A^*$ correspond to each other.

Notice that the structure of $V_A^* \Ra C_M^*$ is revealed from the fat extension $\widehat V_A^*$ as follows: given $\alpha,\beta \in \Gamma \widehat V_A$ (and writing $\xi_\alpha,\xi_\beta \in \Gamma \widehat V_A^*$ as above), $\gamma \in \Gamma C_M$ and $f \in \Gamma V_M^*$,
\begin{equation*}
\rho_{V_A^*}(\xi_\alpha)\gamma = [\alpha,\gamma] \qquad \rho_{V_A^*}(f)\gamma = \partial^*(f)\gamma \qquad [\xi_\alpha,\xi_\beta] = \xi_{[\alpha,\beta]} \qquad [\xi_\alpha, f] = \rho_{V_A}(\alpha)f.
\end{equation*}
Starting with $V_A$, these equations define the algebroid structure on $V_A^*$.
\end{remark}

\subsubsection{Extended Cartan calculus}\label{sec: extended Cartan calculus}
Before we continue, we develop a type of ``extended'' Cartan calculus on $\Omega_\textnormal{lin} V_A$ and on $\Omega(\widehat V_A; V _M^* \ra C_M^*)$. We find it particularly useful to see how $\Omega_\textnormal{lin} V_A$ can be described in terms of the fat algebroid $\widehat V_A$ using invariant cochains.

We start with the ordinary Cartan calculus on $\Omega_\ra V_A$. We have three types of operators: a differential $d$, and then for each section $\alpha \in \Gamma V_A$ a Lie derivative $\calL_\alpha$ and a contraction $\iota_\alpha$. The operators are derivations and of degree $+1$, $0$ and $-1$, respectively. Together, they satisfy the Cartan calculus identities:
\begin{equation}\label{eq: Cartan calculus identities}
[d,\iota_\alpha] = \calL_\alpha, \quad [d,\calL_\alpha] = 0 \quad [\iota_\alpha,\iota_\beta] = 0 \quad [\calL_\alpha,\calL_\beta] = \calL_{[\alpha,\beta]} \quad [\calL_\alpha,\iota_\beta] = \iota_{[\alpha,\beta]}.
\end{equation}
Similarly, on $\Omega_\textnormal{lin} V_A$ we have a Cartan calculus if we restrict attention to linear sections (see Remark \ref{rema: local coordinate description of linear forms}), and see the operators as satisfying the Leibniz rule.

However, from the inclusion $\Omega_\textnormal{lin} V_A \subset \Omega V_A$, we can also apply contractions and Lie derivatives of (not necessarily linear) sections to linear forms $\omega$. In particular, we can do this for core sections $\gamma \in \Gamma C_M$. Instead of thinking of $\iota_\gamma\omega$ as another form of $V_A$, we know from Remark \ref{rema: local coordinate description of linear forms} that we can see it as a form of $\widehat V_A$ (or even $A$). That is, we can define a contraction of the following form:
\begin{definition}\label{def: contraction C}
Let $V_A$ be a VB-algebroid. Given $\gamma \in \Gamma C_M$, we set
\begin{equation*}
\iota_\gamma: \Omega^\bullet_\textnormal{lin} V_A \ra \Omega^{\bullet-1} A \subset \Omega^{\bullet-1} \widehat V_A \qquad \iota_\gamma\omega(\alpha_2,\dots,\alpha_\bullet) \coloneq \omega(\gamma,\alpha_2,\dots,\alpha_\bullet).
\end{equation*}
\end{definition}
We think of this contraction as having degree $-1$. Notice that it is $\Omega A$-linear. Taking this point of view, we see that there is another type of contraction:
\begin{definition}\label{def: contraction V}
Let $V_A$ be a VB-algebroid. Given $v \in \Gamma V_M$, we set
\begin{equation*}
\iota_v: \Omega^\bullet_\textnormal{lin} V_A \ra \Omega^\bullet \widehat V_A \qquad \iota_v\omega(\alpha_1,\dots,\alpha_\bullet) \coloneq \omega(\alpha_1,\dots,\alpha_\bullet)v.
\end{equation*}
\end{definition}
We think of this contraction as having degree $0$. Also this map is $\Omega A$-linear. In turn, the contractions define Lie derivatives:
\begin{definition}\label{def: Lie derivatives}
Let $V_A$ be a VB-algebroid. Given $\gamma \in C_M$ and $v \in \Gamma V_M$, we set
\begin{equation*}
\calL_\gamma: \Omega^\bullet_\textnormal{lin} V_A \ra \Omega^\bullet \widehat V_A \qquad \calL_v: \Omega^\bullet_\textnormal{lin} V_A \ra \Omega^{\bullet+1} \widehat V_A,
\end{equation*}
(of degree $0$ and $+1$, respectively) given by
\begin{align*}
\calL_\gamma\omega(\alpha_1,\dots,\alpha_\bullet) &= \omega(\alpha_1,\dots,\alpha_\bullet)\partial\gamma - \textstyle\sum_k (-1)^k \omega([\alpha_k,\gamma],\alpha_1,\dots,\widehat\alpha_k,\dots,\alpha_\bullet) \\
\calL_v\omega(\alpha_0,\dots,\alpha_\bullet) &= -\textstyle\sum_k (-1)^k (\rho_{V_A}(\alpha_k)(\omega(\alpha_1,\dots,\widehat\alpha_k,\dots,\alpha_\bullet))v \\
&- \rho_A(\alpha_k)(\omega(\alpha_1,\dots,\widehat\alpha_k,\dots,\alpha_\bullet)v)).
\end{align*}
\end{definition}
Notice that both these maps satisfy the Leibniz rule. We ``extended'' successfully now the Cartan calculus on $\Omega_\textnormal{lin} V_A$: the usual Cartan calculus identities \eqref{eq: Cartan calculus identities} hold, but we should interpret equalities like $[d,\iota_\gamma] = \calL_\gamma$ as equality of maps $\Omega_\textnormal{lin} V_A \ra \Omega \widehat V_A$.\footnote{Note that $\iota_v$ is a degree $0$ operator, and $\calL_v$ is a degree $+1$ operator, so for example $\calL_v = [d,\iota_v] = d\iota_v-\iota_vd$.}

Similarly, on $\Omega(\widehat V_A; V_M^* \ra C_M^*)$ we have a Cartan calculus, for $\widehat V_A$ is a Lie algebroid, and $V_M^*$ and $C_M^*$ are representations: for $\alpha \in \Gamma \widehat V_A$ we have $\iota_\alpha$ as usual and $\calL_\alpha = \nabla_\alpha$. Notice that again it makes sense to apply contractions and Lie derivatives of sections of $C_M$ and $V_M$ in the following sense:

\begin{definition}\label{def: Cartan calculus invariant cochains}
Let $V_A$ be a VB-algebroid. Given $\gamma \in \Gamma C_M$ and $v \in \Gamma V_M$, we set
\begin{align*}
\iota_\gamma&: \Omega(\widehat V_A; V_M^* \ra C_M^*)^\bullet \ra \Omega^{\bullet-1} \widehat V_A \qquad
\iota_\gamma(\omega_0,\omega_1) = \iota_\gamma\omega_1 \\
\iota_v&: \Omega(\widehat V_A; V_M^* \ra C_M^*)^\bullet \ra \Omega^\bullet \widehat V_A \qquad \iota_v(\omega_0,\omega_1) = \iota_v\omega_0.
\end{align*}
where
\begin{align*}
\iota_\gamma\omega_1(\alpha_1,\dots,\alpha_{\bullet-1}) &= (-1)^{\bullet-1}\omega_1(\alpha_1,\dots,\alpha_{\bullet-1})\gamma \\
\iota_v\omega_0(\alpha_1,\dots,\alpha_\bullet) &= \omega_0(\alpha_1,\dots,\alpha_\bullet)v.
\end{align*}
Moreover,
\begin{equation*}
\calL_\gamma: \Omega(\widehat V_A; V_M^* \ra C_M^*)^\bullet \ra \Omega^\bullet \widehat V_A \qquad \calL_v: \Omega(\widehat V_A; V_M^* \ra C_M^*)^\bullet \ra \Omega^{\bullet+1} \widehat V_A,
\end{equation*}
are given, respectively, by
\begin{align*}
\calL_\gamma(\omega_0,\omega_1)(\alpha_1,\dots,\alpha_\bullet) &= \omega_0(\alpha_1,\dots,\alpha_\bullet)\partial\gamma + \textstyle\sum_k (-1)^{k+\bullet} \omega_1(\alpha_1,\dots,\widehat\alpha_k,\dots,\alpha_\bullet)[\alpha_k,\gamma] \\
\calL_v(\omega_0,\omega_1)(\alpha_0,\dots,\alpha_\bullet) &= -\textstyle\sum_k (-1)^k (\rho_{V_A}(\alpha_k)(\omega_0(\alpha_1,\dots,\widehat\alpha_k,\dots,\alpha_\bullet))v \\
&- \rho(\alpha_k)(\omega_0(\alpha_1,\dots,\widehat\alpha_k,\dots,\alpha_\bullet)v)).
\end{align*}
\end{definition}

\subsection{From fat extensions to $2$-term ruths}\label{sec: From fat extensions to $2$-term ruths LA}
We will now show that, as expected, $\Omega_\textnormal{lin} V_A$ defines a subcomplex of $\Omega(\widehat V_A; V_M^* \ra C_M^*)$. We will denote this subcomplex by $\Omega_\textnormal{inv}(\widehat V_A; V_M^*)$ and call its elements ``invariant''. This way, we will introduce the abstract $2$-term ruth $\Omega_\textnormal{inv}(F_A; C_M)$ of an abstract fat extension $F_A$. As for Lie groupoids, we can interpret the invariant complex as a type of relative complex from \cite{Mariarelative}. Notice that for $F_A = J^1A$, this complex appeared in \cite{ShengLAdeformations}.

Consider an element $\omega$ of $\Omega^\bullet_\textnormal{lin} V_A$. From Remark \ref{rema: local coordinate description of linear forms}, we know that linear forms are determined by the induced map\footnote{We assume for simplicity that $V_M$ is not the zero bundle, otherwise it is determined by a form on $A$ with values in $C_M$.}
\begin{equation*}
\omega_0: \Gamma \widehat V_A \times \cdots \times \Gamma \widehat V_A \ra \Gamma V_M^* \qquad \omega_0(\alpha_1,\dots,\alpha_\bullet) \coloneq \iota_{\alpha_\bullet} \cdots \iota_{\alpha_1} \omega.
\end{equation*}
This map is $C^\infty M$-multilinear and can therefore be thought of as an element of $\Omega^\bullet(\widehat V_A; V_M^*)$. In fact, this defines a cochain map
\begin{equation*}
\Omega_\textnormal{lin} V_A \ra \Omega(\widehat V_A; V_M^*) \qquad \omega \mapsto \omega_0
\end{equation*}
using the representation on $V_M^*$ discussed above. Since
\begin{equation*}
\iota_{\alpha_\bullet}\cdots\iota_{\alpha_2}\iota_h\omega(v) = \iota_{\alpha_\bullet}\cdots\iota_{\alpha_2}\iota_{h v}\omega,
\end{equation*}
the elements $\omega_0$ have the property that, for all $h \in \textnormal{Hom}(V_M,C_M)$,
\begin{equation*}\label{eq: relation omega_V_M and omega_C_M}
\iota_h\omega_0 = \omega_1 h,
\end{equation*}
where $\omega_1 \in \Omega^{\bullet-1}(A; C_M^*)$ is given by\footnote{Here, $\gamma \in \Gamma C_M$, and we use that $\iota_\gamma\iota_{\alpha_{\bullet-1}}\cdots\iota_{\alpha_1}\omega \in C^\infty M$ by Remark \ref{rema: local coordinate description of linear forms}.}
\begin{equation*}
\omega_1(\alpha_2,\cdots,\alpha_\bullet)\gamma \coloneq \iota_\gamma\iota_{\alpha_\bullet}\cdots\iota_{\alpha_2}\omega
\end{equation*}
and\footnote{To explain the sign, observe that we can think of $h^*$ as a graded endomorphism of degree $-1$ of the graded vector bundle $E_M = V_M^* \oplus C_M^*[-1]$. The expression $\omega h$ can then be interpreted as $h^* \wedge \omega$, where $\wedge$ is induced by the canonical pairing $\textnormal{End } E_M \otimes E_M \ra E_M$ (see, for example, \cite{LAruths}).}
\begin{equation*}
(\omega_1 h)(\alpha_2,\dots,\alpha_\bullet)v = (-1)^{\bullet-1} \omega_1(\alpha_2,\dots,\alpha_\bullet)h(v).
\end{equation*}
We can see $(\omega_0,\omega_1)$ as an element of the cochain complex representation $\Omega(\widehat V_A, V_M^* \ra C_M^*)^\bullet$. Using the Cartan calculus we developed in the previous section, Section \ref{sec: extended Cartan calculus}, a way to write the above condition is now that, for all $h \in \textnormal{Hom}(V_M,C_M)$ and $v \in \Gamma V_M$,
\begin{equation*}
\iota_v\iota_h = \iota_{h v}.
\end{equation*}
The map $\omega \mapsto (\omega_0,\omega_1)$ is $\Omega A$-linear, and it intertwines the contraction operators and Lie derivatives. Therefore:
\begin{proposition}\label{prop: linear forms as forms on the fat algebroid}
The map
\begin{equation*}
\Omega_\textnormal{lin} V_A \ra \Omega(\widehat V_A; V_M^* \ra C_M^*)
\end{equation*}
is a map differential graded $\Omega A$-modules, and it maps bijectively onto those pairs $(\omega_0,\omega_1)$ such that $\omega_1 \in \Omega^{\bullet-1}(\widehat V_A; C_M^*)$ is independent of its entries\footnote{By independence we mean that $\omega_1(\alpha_2,\dots,\alpha_\bullet)$ only depends on the base sections underlying $\alpha_2,\dots,\alpha_\bullet$.} and \eqref{eq: relation omega_V_M and omega_C_M} holds.
\end{proposition}
\begin{proof}
The proof is a straightforward computation using the Cartan calculus we developed:
\begin{align*}
(d\omega)_0(\alpha_0,\dots,\alpha_\bullet) &= \textstyle\sum_k (-1)^k \calL_{\alpha_k} \iota_{\alpha_\bullet} \cdots \widehat\iota_{\alpha_k} \cdots \iota_{\alpha_0} \omega \\
&+ \textstyle\sum_{k < \ell} (-1)^{k+\ell} \iota_{\alpha_\bullet} \cdots \widehat\iota_{\alpha_\ell} \cdots \widehat\iota_{\alpha_k} \cdots \iota_{\alpha_0} \iota_{[\alpha_k,\alpha_\ell]} \omega = d\omega_0(\alpha_0,\dots,\alpha_\bullet)
\end{align*}
and
\begin{align*}
(d\omega)_1(\alpha_1,\dots,\alpha_\bullet)\gamma &= 
(-1)^\bullet\omega_0(\alpha_1,\dots,\alpha_\bullet)\partial\gamma \\
&-\textstyle\sum_k (-1)^k (\calL_{\alpha_k}\iota_\gamma - \iota_{[\alpha_k,\gamma]})\iota_{\alpha_\bullet} \cdots \widehat\iota_{\alpha_k} \cdots \iota_{\alpha_1} \omega \\
&+ \textstyle\sum_{k < \ell} (-1)^{k+\ell} \iota_\gamma \iota_{\alpha_\bullet} \cdots \widehat\iota_{\alpha_\ell} \cdots \widehat\iota_{\alpha_k} \cdots \iota_{\alpha_1} \iota_{[\alpha_k,\alpha_\ell]} \omega \\
&= \partial^*(\omega_0)(\alpha_1,\dots,\alpha_\bullet)\gamma + d\omega_1(\alpha_0,\dots,\alpha_\bullet)\gamma.
\end{align*}
This proves that the map $\omega \mapsto (\omega_0,\omega_1)$ is a cochain map. That it maps onto those pairs $(\omega_0,\omega_1)$ satisfying \eqref{eq: relation omega_V_M and omega_C_M} is a consequence of Remark \ref{rema: local coordinate description of linear forms}.
\end{proof}

We can immediately conclude the following.

\begin{corollary}\label{coro: linear forms as forms on the fat algebroid}
The map
\begin{equation*}
\Omega_\textnormal{lin} V_A \ra \Omega(\widehat V_A; V_M^*)
\end{equation*}
is a map of differential graded $\Omega A$-modules. If $V_M \neq 0$, it maps bijectively onto $\Omega_\textnormal{inv}(\widehat V_A; V_M^*)$, so those forms $\omega_0 \in \Omega^{\bullet}(\widehat V_A, V_M^*)$ for which there exists $\omega_1 \in \Omega^{\bullet-1}(\widehat V_A, C_M^*)$, independent of its entries, such that \eqref{eq: relation omega_V_M and omega_C_M} holds.\footnote{If $V_M=0$, then the map
\begin{equation*}
C^\bullet_\textnormal{lin} V_A \ra C^{\bullet-1}(A, C_M^*); \quad f \mapsto f_1
\end{equation*}
is a bijective cochain map.}
\end{corollary}
\begin{remark}
Although we will usually write $\Omega_\textnormal{inv}(\widehat V_A; V_M^*)$, we do this for notational convenience, so $\Omega_\textnormal{inv}(\widehat V_A; V_M^* \ra C_M^*)$ is meant when writing $\Omega_\textnormal{inv}(\widehat V_A; V_M^*)$.
\end{remark}

To conclude, we formulate the abstract definition of invariant cochains in terms of fat extensions.

\begin{definition}\label{def: invariant cochains fat extensions LA}
Let $F_A$ be a fat extension of $A$ over $C_M \ra V_M$. The \textit{invariant complex}
\begin{equation*}
\Omega_\textnormal{inv}(F_G; C_M) \subset \Omega(F_G; C_M \ra V_M)
\end{equation*}
consists of pairs $(\omega_0,\omega_1) \in \Omega(F_G; C_M \ra V_M)$ such that $\omega_1$ is independent of its entries, and, for all $h \in \textnormal{Hom}(V_M,C_M)$,
\begin{equation*}
\iota_h\omega_0 = h\omega_1.
\end{equation*}
\end{definition}

\begin{remark}\label{rema: condition for invariant cochain}
Notice that the above condition can be written as
\begin{equation*}
\iota_f\iota_h = \iota_{fh}
\end{equation*}
for all $f \in \Gamma C_M^*$ and $h \in \textnormal{Hom}(V_C,C_M)$.
\end{remark}

\subsection{Fat extensions and split $2$-term ruths}

In \cite{VBalgebroidsruths} it is shown how a VB-algebroid induces a split $2$-term ruth on $A$. This goes as follows:

\begin{remark}\label{rema: split ruth of A}
The choice of a splitting $\sigma: A \ra \widehat V_A$ defines $A$-connections $\nabla_\alpha = \nabla_{\sigma\alpha}$ on $V_M^*$ and $C_M^*$. These connections are usually not flat, but they are up to the homotopy $R_2 \in \Omega^2(A; \textnormal{Hom}(C_M^*,V_M^*) = \textnormal{Hom}(V_M,C_M))$ given by
\begin{equation*}
R_2(\alpha,\beta) = \sigma[\alpha,\beta] - [\sigma\alpha,\sigma\beta].
\end{equation*}
That is, $d = \partial^* + d_\nabla + R_2$\footnote{Here, $d_\nabla$ is given by the De Rham type formulas on $\Omega(A; V_M^*)$ and on $\Omega(A; C_M^*)$ using the $A$-connections.} squares to zero:
\begin{equation*}
[\partial^*, d_\nabla] = 0 \qquad d_\nabla^2 + [\partial^*,R_2] = 0 \qquad [d_\nabla,R_2] = 0. 
\end{equation*} 
\end{remark}

Although it is well-known that $\Omega_\textnormal{lin} V_A$ is isomorphic (using the splitting $\sigma$) to $\Omega(A; V_M^* \ra C_M^*)$, we observe that this is readily verified using the canonical model of invariant cochains from the previous section, Section \ref{sec: From fat extensions to $2$-term ruths LA}. The pullback of the splitting defines a map
\begin{equation*}
\sigma^*: \Omega(\widehat V_A; V_M^* \ra C_M^*) \ra \Omega(A; V_M^* \ra C_M^*).
\end{equation*}
This is usually not a cochain map, but it restricts to one on $\Omega_\textnormal{inv}(\widehat V_A; V_M^*)$ (i.e. those pairs $(\omega_0,\omega_1)$ for which \eqref{eq: relation omega_V_M and omega_C_M} holds).

\begin{proposition}\label{prop: linear forms as ruths}
Let $\sigma: A \ra \widehat V_A$ be a splitting and consider the split ruth $\Omega(A; V_M^* \ra C_M^*)$. Then the map
\begin{equation*}
\sigma^*: \Omega_\textnormal{inv}(\widehat V_A; V_M^*) \ra \Omega(A; V_M^* \ra C_M^*)
\end{equation*}
is an isomorphism of ruths.
\end{proposition}
\begin{proof}
Indeed, for all $\omega=(\omega_0,\omega_1) \in \Omega(\widehat V_A; V_M^* \ra C_M^*)$ we have
\begin{align*}
([d,\sigma^*]\omega)_0(\alpha_0,\dots,\alpha_\bullet) &= \textstyle\sum_{k < \ell} (-1)^{k+\ell+\bullet} \omega_1(\sigma\alpha_0,\dots,\widehat{\sigma\alpha_k},\dots,\widehat{\sigma\alpha_\ell},\dots,\sigma\alpha_\bullet)R_2(\alpha_k,\alpha_\ell) \\
&+ \textstyle\sum_{k < \ell} (-1)^{k+\ell} \omega_0(R_2(\alpha_k,\alpha_\ell),\sigma\alpha_0,\dots,\widehat{\sigma\alpha_k},\dots,\widehat{\sigma\alpha_\ell},\dots,\sigma\alpha_\bullet) \\
([d,\sigma^*]\omega)_1(\alpha_0,\dots,\alpha_{\bullet-1}) &= \textstyle\sum_{k<\ell} (-1)^{k+\ell} \omega_1(R_2(\alpha_k,\alpha_\ell),\sigma\alpha_0,\dots,\widehat{\sigma\alpha_k},\dots,\widehat{\sigma\alpha_\ell},\dots,\sigma\alpha_{\bullet-1})
\end{align*}
which are zero on forms satisfying \eqref{eq: relation omega_V_M and omega_C_M}. On the other hand, using the projection $\widehat V_A \ra A$, its inverse is given by
\begin{equation*}
\Omega(A; V_M^* \ra C_M^*) \ra \Omega(\widehat V_A; V_M^* \ra C_M^*) \qquad (\omega_0,\omega_1) \mapsto (\omega_0 - h(\omega_1), \omega_1)
\end{equation*}
where we wrote
\begin{equation*}
h: \Omega^{\bullet-1}(A; C_M^*) \ra \Omega^\bullet(\widehat V_A; V_M^*) \qquad h(\omega_1)(\alpha_1,\dots,\alpha_\bullet) \coloneq (-1)^{\bullet-1}\omega_1(\alpha_2,\dots,\alpha_\bullet)(\alpha_1-\sigma\alpha_1).
\end{equation*}
It is readily verified that the maps are $\Omega A$-linear, so this proves the statement.
\end{proof}

\begin{remark}
In each step of passing to a different model
\begin{equation*}
\Omega_\textnormal{lin} V_A \ra \Omega_\textnormal{inv}(\widehat V_A; V_M^*) \ra \Omega(A; V_M^* \ra C_M^*)
\end{equation*}
we are realizing that $\omega \in \Omega_{\textnormal{lin}} V_A$ is determined by ``fewer'' sections, at the cost of increasing the combinatorial complexity. In the first step, we can canonically reduce from thinking of forms on $V_A$ to forms on $\widehat V_A$. In the second step, we choose a splitting $\sigma: A \ra \widehat V_A$, which puts us in the position to think of forms on $A$.
\end{remark}

\begin{remark}\label{rema: fat extension not just extension}
As was already observed in \cite{LAruths}, a splitting $\sigma: A \ra \widehat V_A$ splits the groupoid $\widehat V_A$ as follows. As a vector bundle, it takes the form
\begin{equation*}
\widehat V_A \cong \textnormal{Hom}(V_M,C_M) \oplus A,
\end{equation*}
and its Lie bracket transforms into
\begin{equation*}
[(h,\alpha),(h',\alpha')] \coloneq (R_2(\alpha,\alpha') + \nabla_h \alpha' - \nabla_{h'}\alpha + [h,h'], [\alpha,\alpha']).
\end{equation*}
This should be seen as the fat extension $\widehat \calE_A$ associated to a split ruth $\calE_A$. In \cite{LAruths}, after a sequence of observations about what we now recognise as the fat algebroid of a split ruth (Proposition 3.9 --- Proposition 3.12), it is said that ``\textit{...there doesn't seem to be any construction which associates to a Lie algebroid extension of $A$ a representation up to homotopy, so that, applying $J^1(A)$ one recovers the adjoint representation.}''. While that might be true, we now know that the ruth can be recovered from $J^1A$ and its extension when we think of $J^1A$ as a fat extension. That is, we can recover the ruth if we additionally keep track of the cochain complex representation $A \ra TM$ of $J^1A$ and its compatibility with the extension.
\end{remark}

\subsection{From fat extensions to VB-algebroids}

The step we have not discussed yet is how to build the VB-algebroid out of an abstract fat extension. So, let $F_A$ be an abstract fat extension of $A$ over $C_M \ra V_M$. Consider the VB-algebroid
\begin{equation*}
F_A \ltimes (C_M \ra V_M) \Ra V_M
\end{equation*}
associated to the cochain complex representation of $F_A$ on $C_M \ra V_M$. That is, we consider
\begin{equation*}
F_A \ltimes (C_M \ra V_M) \coloneq C_M \times_M F_A \times_M V_M.
\end{equation*}
as a vector bundle, and we define
\begin{equation*}
[(\gamma,\alpha),(\gamma',\alpha')] \coloneq (\nabla_\alpha\gamma' - \nabla_{\alpha'}\gamma, [\alpha,\alpha'])
\end{equation*}
on sections $(\gamma,\alpha) = (\gamma,\alpha,1): V_M \ra F_A \ltimes (C_M \ra V_M)$ that are constant along the fibers of $V_M$.
\begin{remark}\label{rema: VB-algebroid of fat extension in terms of fat groupoid of VB-algebroid}
When $F_G = \widehat V_A$ for a VB-algebroid $V_A$, then $\widehat V_A \ltimes (C_M \ra V_M)$ projects onto $V_A$ via
\begin{equation*}
\widehat V_A \ltimes (C_M \ra V_M) \ra V_A \qquad (\gamma,\alpha,v) \mapsto (\gamma +_\ra 0_\da \alpha) +_\da \alpha v = (\gamma +_\da 0_\ra v) +_\ra \alpha v.
\end{equation*}
As expected, the kernel can be identified with $\textnormal{Hom}(V_M,C_M) \ltimes V_M$, where $V_M$ is seen as a representation of $\textnormal{Hom}(V_M,C_M)$ trivially.
\end{remark}

Following the above remark, we observe next that $F_A \ltimes (C_M \ra V_M)$ contains the ideal
\begin{equation*}
\textnormal{Hom}(V_M,C_M) \ltimes V_M \hookrightarrow F_A \ltimes (C_M \ra V_M) \qquad (h,v) \mapsto (-hv,h,v)
\end{equation*}
and then the quotient is the required VB-algebroid $V(F_G)$:

\begin{proposition}\label{prop: correspondence fat extensions VB-algebroids}
The assignment
\begin{equation*}
\{\textnormal{VB-algebroids}\} \ra \{\textnormal{Fat extensions}\} \qquad V_A \mapsto \widehat V_A
\end{equation*}
sets a one-to-one correspondence, and the assignment
\begin{equation*}
\{\textnormal{Fat extensions}\} \ra \{\textnormal{VB-algebroids}\} \qquad F_A \mapsto V(F_A)
\end{equation*}
is inverse to it (up to canonical isomorphisms).
\end{proposition}

Notice that the above result generalises \cite{ChenLiuSheng}.

\subsection{Comments on PB-algebroids and core-transitive double algebroids}\label{sec: Comments on PB-algebroids and double algebroids}

The notion of PB-algebroid has not been developed in the literature yet. But the discussion in Section \ref{sec: PB-groupoids} gives a good hint on what the ``PB-algebroid'' of a fat extension should be. Also for algebroids we can formulate a category of (vertically/horizontally) \textit{core-transitive} double algebroids, and such double algebroids correspond to \textit{core extensions}. This is a slight generalisation of the result in \cite{CorediagramLA} (similar to the generalisation for groupoids presented in Appendix \ref{app: double groupoids}). The definition is as follows.

\begin{definition}\label{def: core extension LA}
A \textit{core extension} is a diagram
\begin{center}
\begin{tikzcd}
    0 \ar[r] & \mathfrak{h} \ar[r] & \ar[dl] F \ar[r] & A_\ra \ar[r] & 0 \\
    \phantom{A} & A_\da \ar[u, phantom, "\textnormal{\Large$\circlearrowright$}"] & \phantom{A} & \phantom{A} & \phantom{A}
\end{tikzcd}
\end{center}
which is a short exact sequence of Lie algebroids, and the action $A_\da$ $\rotatebox[origin=c]{270}{\large$\circlearrowright$}$ $\mathfrak{h}$ is a \textit{bundle of Lie algebras representation}, i.e. a Lie algebroid action $\nabla$ of $A_\da$ on $\mathfrak{h}$ that is a derivation with respect to the bracket of $\mathfrak{h}$:
\begin{equation*}
\nabla_\alpha[h_1,h_2] = [\nabla_\alpha h_1,h_2] + [h_1, \nabla_\alpha h_2].
\end{equation*}
Moreover, the map $F \ra A_\da$ restricted to $\mathfrak{h}$ intertwines $\nabla_\alpha$ with $[\alpha,\cdot]$, and the \textit{Peiffer identity} holds: the action of $F$ on $\mathfrak{h}$, viewing $\mathfrak{h} \subset F$ as an ideal, agrees with the action of $F$ on $\mathfrak{h}$ induced by $A_\da$ $\rotatebox[origin=c]{270}{\large$\circlearrowright$}$ $\mathfrak{h}$ and the map $F \ra A_\da$.
\end{definition}

\begin{remark}\label{rema: PB-algebroids}
Although we do not discuss it in full detail, we mention here what the main objects are that define the ``general linear PB-algebroid'' associated to a fat extension of algebroids. Since we do not discuss ``$2$-actions'' (and in what way they are ``principal''), we mainly hint to some analogies with our discussion in Section \ref{sec: PB-groupoids}.

If we are given a fat extension $F_A$ of $A$ over $C_M \ra V_M$, we can consider
\begin{equation*}
F_A \ltimes \mathfrak{gl}(C_M,V_M) \Ra \mathfrak{gl}(C_M,V_M)
\end{equation*}
associated to the representation of $F_A$ on $C_M \ra V_M$ (see Remark \ref{rema: cochain complex representation algebroid in terms of gl}). This Lie algebroid carries a ``$2$-action'' of the strict Lie $2$-algebroid
\begin{equation*}
\mathfrak{gl}(C_M,V_M)_A \Ra \mathfrak{gl}(C_M,V_M)_M \Ra \textnormal{Hom}(C_M,V_M).
\end{equation*}
Following the above discussion, this strict Lie $2$-algebroid is actually determined by a \textit{crossed module of Lie algebroids}, so a core extension as above for which $\mathfrak{h}=F$. The bundle of Lie algebras $\mathfrak{h}$ in this case is
\begin{equation*}
\mathfrak{pert}(C_M,V_M) \coloneq \textnormal{Hom}(C_M,V_M) \oplus \textnormal{Hom}(V_M,C_M) \Ra \textnormal{Hom}(C_M,V_M)
\end{equation*}
whose bracket is given by (we write $h_\partial$ for the value of the section $h \in \Gamma \mathfrak{pert}(C_M,V_M)$ at $\partial$):
\begin{equation*}
[h,h']_\partial \coloneq h'_\partial \partial h_\partial - h_\partial \partial h'_\partial.
\end{equation*}
The Lie algebroid $\mathfrak{gl}(C_M,V_M)_M$ (see Remark \ref{rema: cochain complex representation algebroid in terms of gl}) acts on $\mathfrak{pert}(C_M,V_M)$ as a bundle of Lie algebras representation via
\begin{equation*}
(\nabla_Dh)_\partial \coloneq (D^{C_M}h - h D^{V_M})_\partial.
\end{equation*}
The Lie algebroid map that is part of the core extension is given by
\begin{equation*}
\mathfrak{pert}(C_M,V_M) \ra \mathfrak{gl}(C_M,V_M)_M \qquad h \mapsto (D^{C_M}_\partial \coloneq -h_\partial\partial, \textnormal{ } D^{V_M}_\partial \coloneq -\partial h_\partial).
\end{equation*}
As in Example \ref{exa: general linear 2-groupoid}, we can think of the resulting strict Lie $2$-algebroid as being of the form
\begin{equation*}
\mathfrak{gl}(C_M,V_M)_A \coloneq \mathfrak{pert}(C_M,V_M) \ltimes \mathfrak{gl}(C_M,V_M)_M.
\end{equation*}
\end{remark}

\subsection{Morphisms of fat extensions}

We briefly introduce here the notion of morphisms of fat extensions of Lie algebroids. We do this by introducing it directly as an infinitesimally multiplicative tensor. Given two fat extensions $F_{A,1}$ and $F_{A,2}$ of $A$, we have a Lie algebroid
\begin{equation*}
F_{A,1} \times_A F_{A,2} \Ra M.
\end{equation*}
This Lie algebroid carries representations that are the tensor product of a representation of $F_{A,1}$ with a representation of $F_{A,2}$.
\begin{definition}\label{def: morphisms of fat extensions}
Consider fat extensions $F_{A,1}$ and $F_{A,2}$ of $A$, over $C_{M,1} \ra V_{M,1}$ and $C_{M,2} \ra V_{M,2}$, respectively. Consider linear maps
\begin{center}
\begin{tikzcd}
C_{M,1} \ar[d, "\Phi^{C_M}"] & V_{M,1} \ar[d, "\Phi^{V_M}"] \\
C_{M,2} & V_{M,2}
\end{tikzcd}
\end{center}
and a smooth map (that is linear over $M$)
\begin{equation*}
f: F_{A,1} \times_A F_{A,2} \ra \textnormal{Hom}(V_{M,1},C_{M,2}).
\end{equation*}
Then $(f,\Phi)$ is a \textit{morphism of fat extensions} if it is invariant: for all $v \in \Gamma V_M$ and $(h_1,h_2) \in \textnormal{Hom}(V_{M,1},C_{M,1}) \oplus \textnormal{Hom}(V_{M,2},C_{M,2})$,
\begin{equation*}
f(h_1,h_2) = h_2\Phi^{V_M} - \Phi^{C_M} h_1,
\end{equation*}
and infinitesimally multiplicative: $d(f,\Phi)=0$, where $d$ is the differential associated to the $3$-term cochain complex representation
\begin{equation*}
\textnormal{Hom}(V_{M,1}, C_{M,2}) \ra \textnormal{Hom}(C_{M,1}, C_{M,2}) \oplus \textnormal{Hom}(V_{M,1}, V_{M,2}) \ra \textnormal{Hom}(C_{M,1}, V_{M,2})
\end{equation*}
of $F_{A,1} \ra F_{A,2}$.
\end{definition}

\begin{remark}\label{rema: morphisms of fat extensions algebroids more concretely}
The multiplicativity conditions for a morphism $(f,\Phi)$ are given more concretely by the following three conditions. The first one is that the map $\Phi$ is a cochain map with respect to the differential $\partial$. The second condition realises $f(\alpha_1,\alpha_2)$, for all $(\alpha_1,\alpha_2) \in F_{A,1} \times_A F_{A,2}$, as a homotopy measuring the failure of $\Phi$ to intertwine $\nabla_{\alpha_1}$ with $\nabla_{\alpha_2}$:
\begin{align*}
f(\alpha_1,\alpha_2)\partial &= \nabla_{\alpha_1}\Phi^{C_M} - \Phi^{C_M}\nabla_{\alpha_2} \\
\partial f(\alpha_1,\alpha_2) &= \nabla_{\alpha_1}\Phi^{V_M} - \Phi^{V_M}\nabla_{\alpha_2}.
\end{align*}
The last multiplicativity condition can be written as follows: given sections $\alpha, \beta \in \Gamma (F_{A,1} \times_A F_{A,2})$,
\begin{equation*}
f([\alpha,\beta])v = \nabla_{\alpha_2}(f(\beta)v) - f(\beta)(\nabla_{\alpha_1} v) + f(\alpha)(\nabla_{\beta_1} v) - \nabla_{\beta_2}(f(\alpha)v),
\end{equation*}
where we wrote $\alpha = (\alpha_1,\alpha_2)$ and $\beta = (\beta_1,\beta_2)$.
\end{remark}

Composition is defined in the same way as in Definition \ref{def: composition of morphisms}.

\begin{definition}\label{def: composition of morphisms algebroids}
Given maps of fat extensions
\begin{equation*}
F_{A,1} \times_A F_{A,2} \xra{f_{21}} \textnormal{Hom}(V_{M,1},C_{M,2}) \qquad F_{A,2} \times_A F_{A,3} \xra{f_{32}} \textnormal{Hom}(V_{M,2},C_{M,3}),
\end{equation*}
over cochain maps $\Phi_{21}$ and $\Phi_{32}$, respectively, the composition $f_{31} = f_{32} \circ f_{21}$ over $\Phi_{31} = \Phi_{32} \circ \Phi_{21}$ is given by
\begin{equation*}
f_{31}(\alpha_1, \alpha_3) = f_{32}(\alpha_2, \alpha_3)\Phi_{21}^{V_M} + \Phi_{32}^{C_M}f_{21}(\alpha_1, \alpha_2),
\end{equation*}
where $\alpha_2 \in F_{G,2}$ is arbitrary.
\end{definition}

We can arrive at the above definition of morphism by starting with a VB-algebroid map
\begin{equation*}
\Phi^{V_A}: V_{A,1} \ra V_{A,2}
\end{equation*}
and defining a comparison map
\begin{equation*}
f_\Phi: \widehat V_{A,1} \ra \widehat V_{A,2} \qquad f_\Phi(\alpha_1,\alpha_2)v \coloneq (\alpha_2 (\Phi^{V_M}v) -_\da \Phi^{V_A}(\alpha_1v)) -_\ra 0_\da\alpha.
\end{equation*}
The map $\Phi^{V_A}$ is recovered from $(f_\Phi,\Phi)$ by realising that
\begin{equation*}
\Phi^{V_A}((\gamma +_\ra 0_\da\alpha) +_\da \alpha_1 v) = ((\Phi^{C_M}(\gamma) - f_\Phi(\alpha_1,\alpha_2)v) +_\ra 0_\da\alpha) +_\da \alpha_2\Phi^{V_M}(v).
\end{equation*}
We then have:

\begin{theorem}\label{thm: equivalence of VBLA and fat}
The functor
\begin{equation*}
\{\textnormal{VB-algebroids}\} \ra \{\textnormal{Fat extensions}\} \qquad V_A \ra \widehat V_A \quad \Phi \mapsto f_\Phi
\end{equation*}
sets an equivalence of categories.
\end{theorem}

\begin{remark}
Similarly, we can establish equivalences of the form
\begin{align*}
\{\textnormal{Fat extensions}\} &\ra \{\textnormal{VB-algebroids}\} \qquad F_A \mapsto V(F_A) \\
\{\textnormal{Fat extensions}\} &\ra \{\textnormal{Abstract $2$-term ruths}\} \qquad F_A \mapsto C_{\textnormal{inv}}(F_A; C_M) \\
\{\textnormal{Split $2$-term ruths}\} &\ra \{\textnormal{Split fat extensions}\} \qquad \calE_A \mapsto \widehat \calE_A
\end{align*}
but we will not go any further into these details here.
\end{remark}

In the next section, Section \ref{sec: Fat Lie theory}, we will give a detailed description of the differentiation procedure in ``Fat Lie theory''. Apart from differentiating fat extensions, we will work out the details of the differentiation of the fat groupoid of a VB-groupoid to the fat algebroid of a VB-algebroid. The differentiation of the associated abstract $2$-term ruths is, as expected, a type of Van Est map.

\section{Fat Lie theory}\label{sec: Fat Lie theory}

In this section we explain the differentiation procedure of fat extensions. As we will see next, the differentiation procedure for fat extensions is relatively straightforward. Afterwards, we go through the details of differentiating the fat groupoid of a VB-groupoid to the fat algebroid of its VB-algebroid. To do that, some particular identifications have to be made. This step is also discussed in \cite{DrummondEgea} (Proposition 2.7), but we elaborate on some aspects. Afterwards, we comment on the Van Est map in this setup.

\subsection{The Fat Lie functor}

To obtain the fat extension of a the VB-algebroid of a VB-groupoid, we can simply apply the Lie functor multiple times in several ways:

\begin{proposition}\label{prop: Lie functor fat extensions}
Let $F_G$ be a fat extension of $G$ over $C_M \ra V_M$. Set
\begin{equation*}
F_A \coloneq \textnormal{Lie } F_G.
\end{equation*}
Moreover, consider the cochain complex representation $C_M \ra V_M$ of $A$ obtained by differentiation, and the exact sequence
\begin{center}
\begin{tikzcd}
0 \ar[r] & \textnormal{Hom}(V_M,C_M) \ar[r] & F_A \ar[r] & A \ar[r] & 0,
\end{tikzcd}
\end{center}
which is obtained from the fat extension of $G$ by differentiation. Then, together with this structure, $F_A$ is a fat extension of $A$ over $C_M \ra V_M$.
\end{proposition}

The assignment is functorial, for morphisms can be differentiated as well. Following Remark \ref{rema: map of fat extensions is a multiplicative tensor} and Remark \ref{rema: morphisms of fat extensions algebroids more concretely}, this should be seen as an application of a Van Est map. A general form of such Van Est maps will be discussed in \cite{Multiplicativetensors}. Here is the special case of interest:

\begin{proposition}
Let $f_G: F_{G,1} \ra F_{G,2}$ be a morphism of fat extensions over $G$. Then
\begin{equation*}
f_A(\alpha_1,\alpha_2) \coloneq \partial_{\epsilon=0} f(\exp(\epsilon \alpha_1), \exp(\epsilon \alpha_2))
\end{equation*}
defines a morphism of fat extensions $f_A: F_{A,1} \ra F_{A,2}$ over $A$. This defines a functor
\begin{equation*}
\textnormal{Lie}: \{\textnormal{fat extensions of } G\} \ra \{\textnormal{fat extensions of } A\} \qquad F_G \mapsto F_A \quad f_G \mapsto f_A.
\end{equation*}
\end{proposition}

In the next section \ref{sec: The tangent bundle of the fat groupoid of a VB-groupoid} we show that, given a VB-groupoid $V_G$ with VB-algebroid $V_A$, the Lie algebroid of $\widehat V_G$ can be seen as $\widehat V_A$.

\subsection{The tangent bundle of the fat groupoid of a VB-groupoid}\label{sec: The tangent bundle of the fat groupoid of a VB-groupoid}

We often take the kinematic approach to tangent vectors in this section. Given $M \ni x$, we write tangent vectors as
\begin{equation*}
\partial_{\epsilon=0} x_\epsilon = \tfrac{\partial}{\partial \epsilon}\vert_{\epsilon=0} x_\epsilon \in T_xM
\end{equation*}
where $x_\epsilon$ is a curve in $M$ with $x_0=x$. We will show a precise relation relating $T\widehat{V_G}$ and $\widehat{TV_G}$, where we view $TV_G$ as a VB-groupoid over $TG$ (a similar statement holds for VB-algebroids). We start with the following observation about vector bundles.

If $E \ra M$ and $F \ra M$ are vector bundles, then $T\textnormal{Hom}(E,F)$ is a vector bundle both over $\textnormal{Hom}(E,F)$ and over $TM$.\footnote{Given a vector bundle $E \ra M$, we write $T_eE$ for the fiber over $e \in E$ and $(TE)_v$ for the fiber over $v \in TM$.} We consider the vector bundle structure over $TM$, and we will show that it embeds into $\textnormal{Hom}_{TM}(TE,TF)$.

Let $\partial_{\epsilon=0} \varphi_\epsilon \in T_\varphi\textnormal{Hom}(E,F)$ be a tangent vector of $\textnormal{Hom}(E,F)$, say, with
\begin{equation*}
\varphi_\epsilon: E_{x_\epsilon} \ra F_{x_\epsilon},
\end{equation*}
$x=x_0$ and $\varphi=\varphi_0: E_x \ra F_x$. Write $v = \partial_{\epsilon=0} x_\epsilon \in T_xM$ for the projection. Then, we define
\begin{equation*}
(TE)_v \ra (TF)_v \qquad \partial_{\epsilon = 0} e_\epsilon \mapsto \partial_{\epsilon = 0} \varphi_\epsilon e_\epsilon.
\end{equation*} 
This map is injective and linear, so we obtained the following:

\begin{proposition}\label{prop: tangent of hom}
Let $E$ and $F$ be vector bundles over $M$. Then, the map
\begin{equation*}
T\textnormal{Hom}(E,F) \hookrightarrow \textnormal{Hom}_{TM}(TE, TF) \qquad \partial_{\epsilon=0} \varphi_\epsilon \mapsto (\partial_{\epsilon=0} e_\epsilon \mapsto \partial_{\epsilon=0} \varphi_\epsilon e_\epsilon)
\end{equation*}
is a vector bundle embedding over $TM$.
\end{proposition}

This is probably easiest to see by writing the map in (standard) local coordinates. The above map over the zero section of $TM \ra M$ takes the following form.

\begin{remark}\label{rema: tangent of hom}
Here, we write the above embedding for $v=0_x \in T_xM$. We decompose\footnote{Here, we wrote $E_x^\textnormal{vert}$ for the space of vertical vectors, which can be identified with $E_x$.}
\begin{equation*}
(TE)_{0_x} = E_x \oplus E_x^{\textnormal{vert}}
\end{equation*}
canonically. Then, $\textnormal{Hom}((TE)_{0_x},(TF)_{0_x})$ splits canonically as
\begin{equation*}
\textnormal{Hom}(E_x,F_x) \oplus \textnormal{Hom}(E_x^{\textnormal{vert}},F_x) \oplus \textnormal{Hom}(E_x,F_x^{\textnormal{vert}}) \oplus \textnormal{Hom}(E_x^{\textnormal{vert}}, F_x^{\textnormal{vert}}).
\end{equation*}
The image of the fiber $(T\textnormal{Hom}(E,F))_{0_x}$, under the embedding, can be identified with
\begin{equation*}
\textnormal{Hom}(E_x,F_x) \oplus \textnormal{Hom}(E_x^{\textnormal{vert}},F_x^{\textnormal{vert}}).
\end{equation*}
Indeed, let $\partial_{\epsilon=0} \varphi_\epsilon \in T_\varphi\textnormal{Hom}(E,F)$ such that its projection to $TM$ is $0_x$. Plugging in
\begin{equation*}
\partial_{\epsilon=0} e_\epsilon \mapsto \partial_{\epsilon=0} \varphi_\epsilon e_\epsilon
\end{equation*}
an element of $E_x$ gives $\varphi: E_x \ra F_x$. If we plug in a vertical vector, then, using the notation from above, we can take $x_\epsilon = x$, and so $\partial_{\epsilon=0} \varphi_\epsilon e_\epsilon$ is again vertical. The map $(TE)_{0_x} \ra (TF)_{0_x}$ therefore identifies with $\varphi: E_x \ra F_x$ and this map $E_x^{\textnormal{vert}} \ra F_x^{\textnormal{vert}}$.
\end{remark}

Now, let $V_G \rra V_M$ be a VB-groupoid, and consider the VB-groupoids
\begin{center}
\begin{tikzcd}
T\widehat V_G \ar[r, shift left] \ar[r, shift right] \ar[d] & TM \ar[d] \\
\widehat V_G \ar[r, shift left] \ar[r, shift right] & M
\end{tikzcd} \qquad
\begin{tikzcd}
TV_G \ar[r, shift left] \ar[r, shift right] \ar[d] & TV_M \ar[d] \\
TG \ar[r, shift left] \ar[r, shift right] & TM
\end{tikzcd}
\end{center}

\begin{proposition}\label{prop: tangent bundle of fat}
Let $V_G \rra V_M$ be a VB-groupoid. Then the map
\begin{equation*}
T\widehat V_G \ra \widehat{TV_G} \qquad \partial_{\epsilon=0} H_{g_\epsilon} \mapsto (\partial_{\epsilon=0} v_\epsilon \mapsto \partial_{\epsilon=0} H_{g_\epsilon}v_\epsilon)
\end{equation*}
is an embedding of Lie groupoids. In fact, using the canonical embedding of Proposition \ref{prop: tangent of hom}, there is a commutative diagram (of short exact sequences over $TM$)
\begin{center}
\begin{tikzcd}
1 \ar[r] & T\textnormal{H}(V_M,C_M) \ar[r] & T\widehat V_G \ar[r] & TG \ar[r] \ar[d, "="] & 1 \\
1 \ar[r] & \textnormal{H}(TV_M,TC_M) \ar[r] \ar[u, hookleftarrow] & \widehat{TV_G} \ar[r] \ar[u, hookleftarrow] & TG \ar[r] & 1
\end{tikzcd}
\end{center}
\end{proposition}
\begin{proof}
To see that the above map is a groupoid map, observe that we can write the multiplication of $TG \rra TM$ as follows (see \cite{Mackenziebook}). Given $v \in T_gG$ and $w \in T_hG$ composable, we can find $g_\epsilon$ and $h_\epsilon$ curves in $G$, such that $v = \partial_{\epsilon=0} g_\epsilon$ and $w = \partial_{\epsilon=0} h_\epsilon$, while $\bfs g_\epsilon = \bft h_\epsilon$ for all (small) $\epsilon$. Then
\begin{equation*}
v \cdot w = \partial_{\epsilon=0}(g_\epsilon \cdot h_\epsilon).
\end{equation*}
Using this formula, the statement is readily verified.
\end{proof}

\subsection{The fat algebroid of the fat groupoid of a VB-groupoid}\label{sec: The fat algebroid of the fat groupoid of a VB-groupoid}

In this section, we prove that there is a canonical identification
\begin{equation*}
\textnormal{Lie } \widehat V_G = \widehat V_A
\end{equation*}
as vector bundles. Recall that $\widehat V_G$ is an open subset of $\widehat V_G^\textnormal{cat}$, and $\widehat V_G^\textnormal{cat}$ is the preimage of $1$ (the identity section in $\textnormal{End } V_M$) under the surjective linear map (over $\bfs: G \ra M$)
\begin{equation*}
\textnormal{Hom}(\bfs^*V_M, V_G) \ra \textnormal{End } V_M \qquad H_g \mapsto \bfs_g H_g.
\end{equation*}
The tangent space $T_{H_g} \widehat V_G$ can therefore be thought of as elements
\begin{equation*}
\partial_{\epsilon=0} H_{g_\epsilon} \in T_{H_g}\textnormal{Hom}(\bfs^*V_M, V_G)
\end{equation*}
where $H_{g_\epsilon}$ are maps $(V_M)_{\bfs g_\epsilon} \ra (V_G)_{g_\epsilon}$, such that
\begin{equation*}
\partial_{\epsilon=0}\bfs_g H_{g_\epsilon} = 0.
\end{equation*}
Using Proposition \ref{prop: tangent of hom}, this equation can be interpreted inside of $\textnormal{End } (TV_M)_{d\bfs X}$, where
\begin{equation*}
X = \partial_{\epsilon=0} g_\epsilon \in T_gG.
\end{equation*}

\begin{proposition}\label{prop: Lie algebroid of fat groupoid is fat algebroid as vector bundles}
We have
\begin{equation*}
\textnormal{Lie } \widehat V_G = \{H_a: (V_M)_{\pr_\ra a} \ra (V_A)_a \mid \pr_\ra H_a = 1\} = \widehat V_A
\end{equation*}
as vector bundles.
\end{proposition}
\begin{proof}
Since
\begin{equation*}
(d\widehat \bfs)_{H_g}: T_{H_g}\widehat V_G \ra T_{\bfs g}M \qquad \partial_{\epsilon=0} H_{g_\epsilon} \mapsto \partial_{\epsilon=0} \bfs g_\epsilon,
\end{equation*}
the vector $X = \partial_{\epsilon=0} g_\epsilon$ is tangent to the source fiber if $\partial_{\epsilon=0} H_{g_\epsilon} \in \ker (d\widehat\bfs)_{H_g}$. In this situation, we can take $g_\epsilon \in \bfs^{-1}\bfs g$. An element
\begin{equation*}
(X, \partial_{\epsilon = 0} v_\epsilon) \in d\bfs^*TV_M
\end{equation*}
(here, $v_\epsilon \in (V_M)_{\bfs g_\epsilon} = (V_M)_{\bfs g}$) has then the property that $\partial_{\epsilon=0} v_\epsilon$ is vertical. This element $\partial_{\epsilon=0} v_\epsilon$ can now itself be thought of as an element of $(V_M)_{\bfs g}$. With this identification, the embedding of Proposition \ref{prop: tangent bundle of fat} produces a linear map
\begin{equation*}
H_X: (V_M)_{\bfs g} \ra (H_g^*\ker (d\bfs)|_{(V_G)_g})_X \qquad \partial_{\epsilon=0} v_\epsilon \mapsto \partial_{\epsilon=0} H_{g_\epsilon} v_\epsilon.
\end{equation*}
We can see these maps as (horizontal) linear sections of the double vector bundle\footnote{Notice that this makes sense as $H_g$ has full rank.}
\begin{center}
\begin{tikzcd}
H_g^*\ker (d\bfs)|_{(V_G)_g} \ar[r] \ar[d] & (V_M)_{\bfs g} \ar[d] \\
\ker (d\bfs)_g \ar[r] & \{g\}
\end{tikzcd}
\end{center}
(the base map of the section $H_X$ is given by $g \mapsto X \in \ker (d\bfs)_g$). The linear map we constructed
\begin{equation}\label{eq: ker d fat source map}
\ker (d\widehat\bfs)_{H_g} \ra \{H_X: (V_M)_{\bfs g} \ra (H_g^*\ker d\bfs|_{(V_G)_g})_X \mid \pr_\ra H_X = 1\}
\end{equation}
is injective, so by a dimension count a bijection. The restriction to the units is the desired isomorphism of vector bundles. This proves the statement.
\end{proof}

\begin{remark}
The double vector bundle 
\begin{center}
\begin{tikzcd}
H_g^*\ker (d\bfs)|_{(V_G)_g} \ar[r] \ar[d] & (V_M)_{\bfs g} \ar[d] \\
\ker (d\bfs)_g \ar[r] & \{g\}
\end{tikzcd}
\end{center}
can be seen as the pullback along $H_g: (V_M)_{\bfs g} \ra (V_G)_g$ of
\begin{center}
\begin{tikzcd}
\ker (d\bfs)|_{(V_G)_g} \ar[r] \ar[d] & (V_G)_g \ar[d] \\
\ker (d\bfs)_g \ar[r] & \{g\}
\end{tikzcd}
\end{center}
which is a sub double vector bundle of
\begin{center}
\begin{tikzcd}
TV_G \ar[r] \ar[d] & V_G \ar[d] \\
TG \ar[r] & G.
\end{tikzcd}
\end{center}
That is, we view $\ker(d\bfs)|_{(V_G)_g} \ra \ker(d\bfs)_g$ as a vector bundle over $(V_G)_g \ra \{g\}$, and then we change the base to $(V_M)_{\bfs g} \ra \{g\}$ using $H_g$.
\end{remark}

\subsection{Right invariant vector fields of the fat groupoid}\label{sec: Right invariant vector fields of the fat groupoid}
Under the correspondence
\begin{equation*}
\frakX_{\textnormal{R-inv}} V_G \cong \Gamma_\ra V_A,
\end{equation*}
$\frakX_{\textnormal{R-inv,lin}} V_G$ corresponds to $\Gamma_{\textnormal{lin}} V_A = \Gamma \widehat V_A$. Linear vector fields on $V_G$ are (horizontal) linear sections of the double vector bundle
\begin{center}
\begin{tikzcd}
TV_G \ar[r] \ar[d] & V_G \ar[d] \\
TG \ar[r] & G
\end{tikzcd}
\end{center}
so $\Gamma \widehat V_A$, with the induced bracket from $V_A$, is isomorphic to $\frakX_{\textnormal{R-inv,lin}} V_G$ as Lie algebras.

On the other hand, by the discussion from the last section, Section \ref{sec: The fat algebroid of the fat groupoid of a VB-groupoid}, $\Gamma \widehat V_A \cong \mathfrak{X}_{\textnormal{R-inv}} \widehat V_G$. Our goal is now to show that this isomorphism is an isomorphism of Lie algebras, i.e. that the induced bracket on $\textnormal{Lie } \widehat V_G$ is the correct one (the one induced from $V_A$). That is, we want to show that
\begin{equation*}
\mathfrak{X}_{\textnormal{R-inv,lin}} V_G \cong \mathfrak{X}_{\textnormal{R-inv}} \widehat V_G
\end{equation*}
is an isomorphism of Lie algebras. For the following result, see also \cite{DrummondEgea}.

\begin{proposition}\label{prop: Lie algebroid of fat groupoid is fat algebroid}
The canonical identification of vector bundles
\begin{equation*}
\textnormal{Lie } \widehat V_G \cong \widehat V_A
\end{equation*}
of Proposition \ref{prop: Lie algebroid of fat groupoid is fat algebroid as vector bundles} is an isomorphism of Lie algebroids.
\end{proposition}
\begin{proof}
This is straightforward using bisections. Recall that the bisections of $\widehat V_G$ can be seen as linear bisections of $V_G$ (Remark \ref{rema: core and linear bisections}). The exponential map (taking right-invariant flow)\footnote{As usual, this should not be interpreted literally due to completeness issues.}
\begin{equation*}
\exp: \frakX_{\textnormal{R-inv}} V_G \ra \textnormal{Bis } V_G
\end{equation*}
restricts on linear sections to a map $\frakX_{\textnormal{R-inv,lin}} V_G \ra \textnormal{Bis } \widehat V_G$, and then the diagram
\begin{center}
\begin{tikzcd}
\frakX_{\textnormal{R-inv,lin}} V_G \ar[rd, "\exp"] & \phantom{A} \\
\ar[u, "\sim"] \Gamma \widehat V_A \ar[d, "\sim"'] \ar[r, dotted] & \textnormal{Bis } \widehat V_G \\
\frakX_{\textnormal{R-inv}} \widehat V_G \ar[ru, "\exp"'] & \phantom{A}
\end{tikzcd}
\end{center}
commutes. Using the formula for the Lie bracket in terms of the exponential map, we see that the brackets on $\mathfrak{X}_{\textnormal{R-inv,lin}} V_G$ and on $\mathfrak{X}_{\textnormal{R-inv}} \widehat V_G$ coincide.
\end{proof}

\begin{remark}\label{rema: Lie algebroid of fat groupoid is fat algebroid}
In view of the next section, Section \ref{sec: Van Est}, we add some more details. A section $\alpha \in \Gamma \widehat V_A$ defines right-invariant vector fields
\begin{equation*}
X \in \frakX_{\textnormal{R-inv,lin}} V_G \qquad \widehat X \in \frakX_{\textnormal{R-inv}} \widehat V_G,
\end{equation*}
both over $X \in \frakX_{\textnormal{R-inv}} G$. The right-invariant flow of the linear vector field $X$ is determined by a curve of maps
\begin{equation*}
\varphi^X_\epsilon: V_M \ra V_G
\end{equation*}
that are linear over a map $\varphi_\epsilon^X: M \ra G$. Linearity means that $\varphi^X_\epsilon$ restricts to linear maps\footnote{Notice that, given $v \in (V_M)_x$, the curve $(\varphi^X_\epsilon)_xv$ stays within the source fiber $\bfs^{-1}\bfs v$.}
\begin{equation*}
(\varphi^X_\epsilon)_x: (V_M)_x \ra (V_G)_{\varphi^X_\epsilon(x)}.
\end{equation*}
From the previous section, Section \ref{sec: Right invariant vector fields of the fat groupoid}, recall \eqref{eq: ker d fat source map}, i.e. that $\ker (d\widehat\bfs)_{H_g}$ can be identified with (horizontal) linear sections of the double vector bundle
\begin{center}
\begin{tikzcd}
H_g^*\ker (d\bfs)|_{(V_G)_g} \ar[r] \ar[d] & (V_M)_{\bfs g} \ar[d] \\
\ker (d\bfs)_g \ar[r] & \{g\}.
\end{tikzcd}
\end{center}
For this reason, we can see the flow of some section $\widehat X$ of $\ker d\widehat \bfs$ (not necessarily right-invariant)  as a curve of linear maps
\begin{equation*}
\varphi^{\widehat X}_\epsilon: \widehat V_G \ra \widehat V_G \qquad \varphi^{\widehat X}_\epsilon(H_g): (V_M)_{\bfs g} \ra (V_G)_{\varphi^X_\epsilon(g)}
\end{equation*}
Now, if $\widehat X$ is right-invariant, i.e. such that
\begin{equation*}
\widehat X(H_g \cdot H_h) = (dR_{H_h})_{H_g} \widehat X(H_g),
\end{equation*}
then $\widehat X$ is determined on the identity elements of $\widehat V_G$, which defines the section $\alpha \in \Gamma \widehat V_A$ we started with. The flow $\varphi^{\widehat X}$ is therefore right-invariant and determined by its values on the identity elements; these form a family of linear maps
\begin{equation*}
\varphi^{\widehat X}_\epsilon: (V_M)_x \ra (V_G)_{\varphi_\epsilon^X(x)}.
\end{equation*}
The two maps $\varphi^{\widehat X}_\epsilon$ and $\varphi^X_\epsilon$ coincide.
\end{remark}

Next, we explain how the differentiation procedure leads to a Van Est map on invariant cochains.

\subsection{The Van Est map}\label{sec: Van Est}

Before we begin, we briefly mention how the differentiation procedure of abstract ruths looks like.

\subsubsection{Differentiation of abstract ruths}

To emphasise, this is only a short discussion. In \cite{Homotopyoperators}, more details will be included.

As is recently shown in \cite{CabreraDelHoyo}, the differentiation of a Lie groupoid $G$ to a Lie algebroid $A$ can be thought of algebraically as follows. The differential graded algebra $CG$ contains the \textit{differentiating ideal} $J_\textnormal{N} = J + \delta J$, where
\begin{equation*}
J = \textstyle\bigoplus_{\bullet \ge 0} J^\bullet \qquad J^\bullet \coloneq \{f \in C^\bullet G \mid f \textnormal{ vanishes to order $\ge 1$ at } \im u_k \textnormal{ for some } k\}.
\end{equation*}
The differential graded algebra $\Omega A$ is canonically isomorphic to $(CG)_\textnormal{N} / J_\textnormal{N}$. The quotient map
\begin{equation*}
\textnormal{VE}: (CG)_\textnormal{N} \ra \Omega A
\end{equation*}
is the ordinary \textit{Van Est map} of \cite{MariusVanEst} (see \cite{CabreraDelHoyo} for more details).

Now, given a ruth $\calE_G$ of $G$, then we can quotient $(\calE_G)_\textnormal{N}$ by $(\calE_G)_\textnormal{N} \cdot J_\textnormal{N}$. The result
\begin{equation*}
\calE_A \coloneq (\calE_G)_\textnormal{N}/(\calE_G)_\textnormal{N} \cdot J_\textnormal{N} \cong (\calE_G)_\textnormal{N} \otimes_{(CG)_\textnormal{N}} \Omega A
\end{equation*}
is a well-defined ruth of $A$, and accordingly there is a Van Est map
\begin{equation*}
(\calE_G)_\textnormal{N} \ra \calE_A.
\end{equation*}
There is a \textit{Van Est theorem} (for split ruths, see \cite{CamiloFlorianruths}) which extends the ordinary Van Est theorem of \cite{MariusVanEst}. This fact and other further details will be discussed in \cite{Homotopyoperators}. We now move to giving an explicit description of the Van Est map/theorem in terms of invariant cochains.

\subsubsection{The Van Est map of invariant cochains}
Let $F_G$ be a fat extension of a Lie groupoid $G$ over $C_M \ra V_M$, and consider the associated fat extension $F_A$ of $A$ over $C_M \ra V_M$. The Van Est map
\begin{equation*}
\textnormal{VE}: C_\textnormal{inv}(F_G; C_M)_\textnormal{N} \ra \Omega_\textnormal{inv}(F_A; C_M)
\end{equation*}
is simply the restriction of the ordinary Van Est map
\begin{equation*}
\textnormal{VE}: C(F_G; C_M \ra V_M)_\textnormal{N} \ra \Omega(F_A; C_M \ra V_M)
\end{equation*}
defined componentwise from the ordinary Van Est maps associated to the representations on $C_M$ and $V_M$. Following \cite{CamiloFlorianruths,VanEsthomogeneous}, there is a Van Est theorem:

\begin{theorem}\label{thm: Van Est theorem of invariant cochains}
The Van Est map
\begin{equation*}
\textnormal{VE}: C^\bullet(F_G; C_M \ra V_M)_\textnormal{N} \ra \Omega^\bullet(F_A; C_M \ra V_M)
\end{equation*}
restricts to a map of differential graded modules
\begin{equation*}
\textnormal{VE}: C^\bullet_\textnormal{inv}(F_G; C_M)_\textnormal{N} \ra \Omega^\bullet_\textnormal{inv}(F_A; C_M).
\end{equation*}
Moreover, if $G$ has (homologically) $\le n$-connected fibers, then $\textnormal{VE}$ induces an isomorphism in cohomology in degrees $\le n$, and it induces an injective map in cohomology in degree $n+1$.
\end{theorem}

Instead of proving the statement directly, we explain how the Van Est map looks like in terms of $F_G = \widehat V_G$. Following the discussion of the previous sections, we can give a clear description of the formula in this situation.

As explained in the Section \ref{sec: Right invariant vector fields of the fat groupoid}, $\alpha \in \Gamma \widehat V_A$ exponentiates to a curve of linear maps\footnote{Here, $\epsilon > 0$ is small enough so that $\epsilon\alpha$ is close to the zero section.}
\begin{equation*}
\exp_x(\epsilon\alpha) = \exp_{\exp(\epsilon\alpha)x}(\epsilon\alpha): (V_M)_x \ra (V_G)_{\exp(\epsilon\alpha)x}.
\end{equation*}
For $\bullet \ge 0$, consider the maps
\begin{equation*}
\partial_\alpha: C^{\bullet+1}(\widehat V_G; V_M^*) \ra C^\bullet(\widehat V_G; V_M^*) \qquad \partial_\alpha: C^{\bullet+1}(\widehat V_G; C_M^*) \ra C^\bullet(\widehat V_G; C_M^*)
\end{equation*}
given both by
\begin{equation*}
(\partial_\alpha f)(H_{g_1}, \dots, H_{g_\bullet})e = \partial_{\epsilon=0}f(\exp_{1_{\bft g_1}}(\epsilon \alpha), H_{g_1}, \dots, H_{g_\bullet}) \exp_{1_{\bft g_1}}(\epsilon \alpha) \cdot e.
\end{equation*}
Then the Van Est map
\begin{equation*}
\textnormal{VE}: C^\bullet(\widehat V_G; V_M^* \ra C_M^*)_\textnormal{N} \ra \Omega^\bullet(\widehat V_A; V_M^* \ra C_M^*)
\end{equation*}
is given componentwise by\footnote{We used $\textnormal{S}_\bullet$ for the group of permutations on $\bullet$ elements.}
\begin{equation*}
\textnormal{VE} f(\alpha_1,\dots,\alpha_\bullet) = \textstyle\sum_{\sigma \in \textnormal{S}_\bullet} (-1)^\sigma \partial_{\alpha_{\sigma \bullet}} \cdots \partial_{\alpha_{\sigma 1}} f|_M.
\end{equation*}

\begin{proposition}\label{prop: Van Est theorem}
The Van Est map
\begin{equation*}
\textnormal{VE}: C^\bullet(\widehat V_G; V_M^* \ra C_M^*)_\textnormal{N} \ra \Omega^\bullet(\widehat V_A; V_M^* \ra C_M^*)
\end{equation*}
restricts to a cochain map
\begin{equation*}
\textnormal{VE}: C^\bullet_\textnormal{inv}(\widehat V_G; V_M^*)_\textnormal{N} \ra \Omega^\bullet_\textnormal{inv}(\widehat V_A; V_M^*).
\end{equation*}
Moreover, if $G$ has (homologically) $\le n$-connected fibers, then $\textnormal{VE}$ induces an isomorphism in cohomology in degrees $\le n$, and it induces an injective map in cohomology in degree $n+1$.
\end{proposition}
\begin{proof}
Given $(f_0,f_1) \in C^\bullet_\textnormal{inv}(\widehat V_G; V_M^*)$, we want to show that
\begin{equation}\label{eq: VE invariance formula}
\iota_v\iota_h \textnormal{VE} f_0 = \iota_{hv} \textnormal{VE} f_1
\end{equation} 
for all $h \in \Gamma \textnormal{Hom}(V_M,C_M)$ and $v \in \Gamma V_M$. First observe that
\begin{equation*}
\exp_x(\epsilon h) = 1 + \epsilon h: (V_M)_x \ra (V_G)_{1_x}.
\end{equation*}
In particular, as this is a bisection over the unit bisection $1: M \ra G$, we have, for all $\alpha \in \Gamma \widehat V_A$, that $\partial_h\partial_\alpha=0$. Therefore, the only surviving terms $\partial_{\beta_\bullet}\cdots\partial_{\beta_1} f_0$ in
\begin{equation*}
\textnormal{VE} f_0(\alpha_1, \dots, \alpha_\bullet) = \textstyle\sum_{\sigma \in \textnormal{S}_\bullet} (-1)^\sigma \partial_{\alpha_{\sigma \bullet}} \cdots \partial_{\alpha_{\sigma 1}} f_0|_M,
\end{equation*}
say, for $\alpha_1 = h$, are those for which $\beta_1=h$. The result follows using the explicit formula of \cite{MariusVanEst}:
\begin{equation*}
(\partial_{\beta_\bullet} \cdots \partial_{\beta_1} f) v = \partial_{\epsilon_\bullet=0} \cdots \partial_{\epsilon_1=0} f(\exp(\epsilon_1\beta_1), \dots, \exp(\epsilon_\bullet\beta_\bullet)) \exp(\epsilon_1\beta_1) \cdots \exp(\epsilon_\bullet\beta_\bullet) v.
\end{equation*}
The last statement is a consequence of, for example, the result of \cite{VanEsthomogeneous} and our discussion on flows, as it follows from that discussion that we have a commutative diagram
\begin{center}
\begin{tikzcd}
(C_{\textnormal{VB}} V_G)_\textnormal{N} \ar[r, "\sim"] \ar[d, "\textnormal{VE}"] & C_{\textnormal{inv}}(\widehat V_G; V_M^*)_\textnormal{N} \ar[d, "\textnormal{VE}"] \\
\Omega_{\textnormal{lin}} V_A \ar[r, "\sim"] & \Omega_{\textnormal{inv}}(\widehat V_A; V_M^*)
\end{tikzcd}
\end{center}
One could also establish such a commutative diagram using instead the result of \cite{CamiloFlorianruths}.
\end{proof}

One can prove the Van Est theorem also directly using similar techniques as in \cite{MariusVanEst}; this will appear in \cite{Homotopyoperators}.

\section{Deformation theory for Lie groupoids using the jet groupoid}\label{sec: Deformation theory for Lie groupoids using the jet groupoid}
We focus now on the fat extension of the jet groupoid $J^1G$. Recall that
\begin{align*}
J^1G &= \{j^1\sigma \mid \sigma \in \textnormal{Bis}_\textnormal{loc} G\} \\
&= \{H_g: T_{\bfs g} M \ra T_gG \mid d\bfs_g H_g = 1 \textnormal{ and } d\bft_g H_g \textnormal{ is a linear isomorphism}\} \\
&= \{\omega_g: T_gG \ra A_{\bfs g} \mid \omega_g \ell_g = 1 \textnormal{ and } \omega_g r_g \textnormal{ is a linear isomorphism}\}.
\end{align*}
The deformation complex of $G$ is $C_{\textnormal{proj}} TG$, which we now see is isomorphic to $C_\textnormal{inv}(J^1G; A)$. However, as we will see, it is also natural, perhaps even more natural, to work with the \textit{jet category} whose elements are $1$-jets of sections of the source map. 

We already reformulated the vanishing of proper groupoids (see Proposition \ref{prop: vanishing for proper groupoids}) and the Van Est theorem in terms of $C_\textnormal{inv}(J^1G; A)$ (see Proposition \ref{prop: Van Est theorem}).
We will discuss two more reformulations, namely that of describing the deformation class of a deformation of groupoids, and for the deformation cohomology in the regular case. There are many other reformulations possible of other parts of \cite{Deformationsgroupoids}, but we won't discuss further applications here. We think the small discussion presented here gives a good starting point for the interested reader to understand deformation theory of Lie groupoids in this setup. 

\begin{remark}\label{rema: deformation complex using jet algebroid}
Before we continue, notice that in \cite{ShengLAdeformations} the deformation complex of Lie algebroids is discussed in terms of the jet algebroid $J^1A$. That is, $\Omega_\textnormal{inv}(J^1A; A)$ is the deformation complex of $A$ (see also \cite{Deformationsalgebroids}). The dgLa structure is given as follows: given $\omega_0 \in \Omega_\textnormal{inv}^{\bullet+1}(J^1A; A)$ and $\omega_0' \in \Omega_\textnormal{inv}^{\bullet'+1}(J^1A; A)$, we first define
\begin{align*}
(\omega_0 \circ \omega_0')(\alpha_1,\dots,\alpha_{\bullet+\bullet'+1}) = \textstyle\sum_\sigma (-1)^\sigma \omega_0(\omega_0'(\alpha_{\sigma 1}, \dots, \alpha_{\sigma(\bullet'+1)}), \alpha_{\sigma(1+\bullet'+1)}, \dots, \alpha_{\sigma(\bullet + \bullet'+1)})
\end{align*}
and then put
\begin{equation*}
[\omega,\omega'] = (-1)^{\bullet\bullet'}\omega \circ \omega' - \omega' \circ \omega.
\end{equation*}
Notice that we can see elements $\omega_0$ in degree $2$ satisfying the Maurer Cartan equation $[\omega_0,\omega_0] = 0$ as Lie brackets by setting
\begin{equation*}
[\alpha,\beta]_\omega \coloneq \omega_0(d\alpha,d\beta).
\end{equation*}
The Leibniz rule follows by the existence of the form $\omega_1$, which can be seen as the anchor map. In fact, for any form $\omega \in \Omega^\bullet_\textnormal{inv}(J^1A; A)$ we can define a multiderivation by
\begin{equation*}
D(\alpha_1,\dots,\alpha_\bullet) = \omega_0(d\alpha_1,\dots,d\alpha_\bullet)
\end{equation*}
whose symbol is given by $\omega_1$. This way, we obtain an isomorphism of differential graded Lie algebras between $\Omega_{\textnormal{def}} A$ of \cite{Deformationsalgebroids} and $\Omega_\textnormal{inv}(J^1A; A)$ of \cite{ShengLAdeformations}.
\end{remark}

\subsection{Deformations of Lie groupoids}

We define a deformation of Lie groupoids as in \cite{Deformationsgroupoids}.
\begin{definition}\label{def: deformation of LG}
A deformation of a Lie groupoid $G$ is a Lie groupoid $G_\bbR \rra M_\bbR$ together with a Lie groupoid submersion
\begin{equation*}
G_\bbR \ra 1_\bbR
\end{equation*}
such that $G_0 = G$ (the preimage of zero).
\end{definition}
\begin{example}\label{exam: deformation to the normal cone}
Let $G \rra M$ be a Lie groupoid, and let $H \rra N$ be a Lie subgroupoid of $G$. Consider the deformation to the normal cone (see for instance \cite{DNCconstructionmanifolds,Meinrenkenlecturenotes,Blup,lobsterblup})\footnote{As sets, $\calD(M,N) = M \times (\bbR \backslash \{0\}) \sqcup \calN(M,N)$. The smooth structure is canonically defined, and it is unique with the property that $M \times (\bbR \backslash \{0\})$ is an embedded submanifold.}
\begin{equation*}
\calD(G,H) \rra \calD(M,N).
\end{equation*}
This is a deformation, in the above sense, of $\calN(G,H) \rra \calN(M,N)$, and it is central to the linearisation problem for Lie groupoids.
\end{example}
Deformations can be seen as families of Lie groupoids $G_\epsilon \rra M_\epsilon$. For the deformation theory of $G$, both $C_{\textnormal{def}} G$ and $C_{\textnormal{def}} G_\bbR$ are relevant. We can see $TG_\bbR$ as a deformation of $TG$, and elements of $J^1G_\bbR$ take the form
\begin{equation*}
H_g: T_{\bfs g}M_\bbR \ra T_gG_\bbR
\end{equation*}
that we should see as a family of maps $T_{\bfs g_\epsilon} M_\epsilon \ra T_{g_\epsilon} G_\epsilon$.\footnote{Note that if $X \in T_{\bfs g} M$ projects to $\lambda \tfrac{\partial}{\partial t}$, then $H_g X$ also projects to $\lambda \tfrac{\partial}{\partial t}$.}

\subsection{The deformation class of a deformation}
Given a deformation $G_\bbR$ of $G$, we can lift the vector field $\tfrac{\partial}{\partial t} \in \frakX \bbR$ to a vector field $X_\bbR$ of $M_\bbR$. Then, take a fat splitting $h$ (see Section \ref{sec: Relation to Representations up to homotopy}) and define the map
\begin{equation*}
f_0: J^1G_\bbR \ra \bft^*A_\bbR \qquad f_0(H_{g}) = -h(H_g) X_\bbR(\bfs g).
\end{equation*}
This defines an element of $C_\textnormal{inv}^1(J^1G_\bbR; A_\bbR)$ with $f_1 = X_\bbR \in \frakX M_\bbR$:
\begin{align*}
f_0(H_g) - f_0(H_{g}') = h(H_g,H_{g}')f_1(\bfs g).
\end{align*}
\begin{remark}
It is good to realise that, using the model $C_{\textnormal{proj}} TG_\bbR$, and using a cleavage $\Sigma$ of $TG_\bbR$, we can lift $X_\bbR$ to an $\bfs$-projectable vector field $\Sigma X_\bbR$ of $G_\bbR$ (i.e. $g \mapsto \Sigma_{g} X_\bbR(\bfs g)$) and it can directly be seen as a $1$-cochain, for it is $\bfs$-projectable. On the one hand, the $1$-cochain we write here makes it clear that, when we pass to the adjoint (split) ruth using $\Sigma$, the $1$-cochain is nothing but $X_\bbR \in C^0(G_\bbR;TM_\bbR)$. On the other hand, the usage of the cleavage shows that the map $f$ above is more naturally defined on the fat category. 
\end{remark}
Taking $\delta f_0$ defines a cocycle that we can restrict to a cocycle in $C_\textnormal{inv}^2(J^1G; A)$. Indeed, since
\begin{equation*}
\delta f_0(H_{g_1},g_2) = h(H_{g_1})X_\bbR(\bfs g_1) - (h(H_{g_1} \cdot H_{g_2}) - H_{g_1} \cdot h(H_{g_2}))X_\bbR(\bfs g_2),
\end{equation*}
where $H_{g_2}$ is arbitrary, and $X_\bbR$ projects to $\tfrac{\partial}{\partial t}$, the resulting element in $(A_\bbR)_{\bft g}$ is vertical with respect to the projection to $\bbR$. Therefore, the same equation defines a cocycle in $C_\textnormal{inv}^2(J^1G; A)$.

The resulting cohomology class is independent of the choice of vector field $X_\bbR$ or the fat splitting $h$. The reason is that $f_0-f_0'$ (where $f_0'$ is defined using a different choice $X_\bbR'$ and $h'$) restricts to $C_\textnormal{inv}^1(J^1G; A)$ (the same reasoning as before applies). Since
\begin{equation*}
\delta f_0 = \delta f_0' + \delta(f_0-f_0'),
\end{equation*}
the cohomology classes defined by $f_0$ and by $f_0'$ are the same.

\begin{remark}
As can be inferred from \cite{Eulerlike,Deformationsgroupoids}, we can linearise a proper Lie groupoid $G$ along an invariant submanifold $N \subset M$ using $\calD(G,G_N)$ and the homotopy operator $\eta$ from Proposition \ref{prop: vanishing for proper groupoids}. The projectable complex from \cite{Deformationsgroupoids} is natural for such an application, for it is easy to see that its $1$-cocycles comprise the multiplicative vector fields. The above model, especially when put in terms of the jet category, seems to also give a natural point of view. 
\end{remark}

\subsection{Deformation cohomology in the regular case}
Given a regular Lie groupoid $G$, $TG$ is also regular. A VB-groupoid $V_G$ over $G$ whose underlying complex $C_M \ra V_M$ is regular, comes with a short exact sequence of complexes
\begin{center}
\begin{tikzcd}
\ker \partial \ar[r] \ar[d] & C_M \ar[r] \ar[d, "\partial"] & 0 \ar[d] \\
0 \ar[r] & V_M \ar[r] & \textnormal{coker } \partial
\end{tikzcd}
\end{center}
all three of which we can see as cochain complex representations of $\widehat V_G$. Notice that the invariant complex of the outer two representations are canonically isomorphic to the corresponding representations of $G$. We therefore obtain a long exact sequence
\begin{center}
\begin{tikzcd}
\cdots \ar[r] & H^\bullet(G; \ker \partial) \ar[r] & H^\bullet_{\textnormal{inv}}(\widehat V_G; C_M) \ar[r] & H^{\bullet-1}(G; \textnormal{coker } \partial) \ar[r] & \cdots.
\end{tikzcd}
\end{center}
If $G=G(P)$ is the Gauge groupoid of a principal bundle $P$, we recover the result that
\begin{equation*}
H^\bullet_{\textnormal{def}} G(P) \cong H^\bullet(G(P); P \times_G \frakg),
\end{equation*}
where $P \times_G \frakg$ is the adjoint bundle. If $G=G(\calF)$ is a foliation groupoid of a foliation $\calF$, we recover the result that
\begin{equation*}
H^\bullet_{\textnormal{def}} G(\calF) \cong H^{\bullet-1}(G(\calF); \calN),
\end{equation*}
where $\calN$ is the normal bundle of the foliation. Similar arguments work to prove analogous results for VB-algebroids.

\appendix
\section{Double groupoids and core extensions}\label{app: double groupoids}
Recall the following notation we use for double groupoids:

\begin{center}
\begin{tikzcd}
    G \ar[d, shift left] \ar[d, shift right] \ar[r, shift left] \ar[r, shift right] & G_\da \ar[d, shift left] \ar[d, shift right] \\
    G_\ra \ar[r, shift left] \ar[r, shift right] & M
\end{tikzcd}
\end{center}

We denote the structure maps of $G_\ra$ and $G_\da$ using the same arrow-index. We denote the groupoid structure maps of $G \rra G_\da$ with indices $\ra$ and $G \rra G_\ra$ with indices $\da$.

The core groupoid is given by
\begin{equation*}
    F \coloneq \ker \bfs_\da \cap \ker \bfs_\ra \rra M,
\end{equation*}
whose structure maps are given by $\bfs(g) = \bfs_\ra\bfs_\da g = \bfs_\da\bfs_\ra g$, $\bft(g) = \bft_\ra\bft_\da g = \bft_\da\bft_\ra g$, and\footnote{That is, we right-translate $g$ to be composable with $h$.}
\begin{align*}
    g \cdot h &= (g \cdot_\da 1_\ra(\bft_\ra h)) \cdot_\ra h = (g \cdot_\ra 1_\da(\bft_\da h)) \cdot_\da h \\
    g^{-1} &= (g)_\ra^{-1} \cdot_\da 1_\ra(\bft_\ra g)^{-1} = (g)_\da^{-1} \cdot_\ra 1_\da(\bft_\da g)^{-1}.
\end{align*}
It comes with two groupoid maps
\begin{equation*}
    G_\da \xla{\bft_\ra} F \xra{\bft_\da} G_\ra.
\end{equation*}
Then, (vertically) core-transitive double groupoids are those for which the map
\begin{equation*}
    \bft_\da: F \ra G_\ra
\end{equation*}
is a surjective submersion.

Following \cite{Corediagram} (or rather generalising their result) we will show that this type of double groupoid can be described in a much more efficient way using, what we call here, a \textit{core extension}.

\subsection{From core-transitive double groupoids to core extensions}

Let $G$ be a vertically core-transitive double groupoid. Then $G$ comes with a bundle of Lie groups suggestively denoted by $\textnormal{H}$, which is the kernel of the map $F \ra G_\ra$. Notice that there is a natural conjugation action of $G_\da$ on $\textnormal{H}$. Namely, for $g \in G_\da$ and $h_x \in \textnormal{H}$, such that $\bfs_\da g = x$, we can set
\begin{equation*}
    g \cdot h_x \coloneq 1_\ra(g) \cdot_\da h_x \cdot_\da 1_\ra(g^{-1}).
\end{equation*}
Notice that this action defines a ``bundle of Lie groups representation'' on the bundle of Lie groups $\textnormal{H}$ in the following sense.

\begin{definition}\label{def: representations of bundles of Lie groups}
Let $G \rra M$ be a Lie groupoid, and consider a bundle of Lie groups $\textnormal{H} \rra M$. An action of $G$ on $\textnormal{H}$ is called a \textit{bundle of Lie groups representation} if the action is by Lie group automorphisms. That is, given $g \in G$ and $h_x,h_x' \in \textnormal{H}$ with $\bfs g=x$,
\begin{equation*}
    g \cdot (h_x \cdot h_x') = (g \cdot h_x) \cdot (g \cdot h_x').
\end{equation*}
\end{definition}

We actually defined now three structures from a vertically core-transitive double groupoid $G$ --- a short exact sequence of Lie groupoids
\begin{center}
\begin{tikzcd}
    1 \ar[r] & \textnormal{H} \ar[r] & F \ar[r] & G_\ra \ar[r] & 1,
\end{tikzcd}
\end{center}
a Lie groupoid map $F \ra G_\da$, and a bundle of Lie groups representation $G_\da$ $\rotatebox[origin=c]{270}{\large$\circlearrowright$}$ $\textnormal{H}$. We summarise this data into a diagram as follows
\begin{center}
\begin{tikzcd}
    1 \ar[r] & \textnormal{H} \ar[r] & \ar[dl] F \ar[r] & G_\ra \ar[r] & 1. \\
    \phantom{A} & G_\da \ar[u, phantom, "\textnormal{\Large$\circlearrowright$}"] & \phantom{A} & \phantom{A} & \phantom{A}
\end{tikzcd}
\end{center}
There are two compatibility assumptions between the three structures described above that are reminiscent of the axioms of a so-called \textit{crossed module} of Lie groupoids. Firstly, the map $F \ra G_\da$ restricted to $\textnormal{H}$ is \textit{equivariant}. By this we mean that the map turns the action of $G_\da$ on $\textnormal{H}$ into the usual conjugation of $G_\da$:\footnote{Notice that $\textnormal{H}$ maps into the isotropy of $G_\da$.}
\begin{equation*}
    \bft_\ra(g \cdot h_x) = g \cdot (\bft_\ra h_x) \cdot g^{-1}.
\end{equation*}
Secondly, the map $F \ra G_\da$ defines an action of $F$ on $\textnormal{H}$. The following lemma shows that this action equals the conjugation action of $F$ on $\textnormal{H}$ (viewing $\textnormal{H} \subset F$ as a normal subgroupoid).

\begin{lemma}\label{lemm: Peiffer identity}
Given a vertically core-transitive double groupoid $G$, we have for all $g \in F$ and $h_x \in \textnormal{H}$ with $\bfs g = x$,\footnote{On the right hand side the product of $F$ is used.}
\begin{equation*}
    1_\ra(\bft_\ra g) \cdot_\da h_x \cdot_\da 1_\ra(\bft_\ra g)^{-1} = g \cdot h \cdot g^{-1}.
\end{equation*}
We call this the \textit{Peiffer identity}.
\end{lemma}%
\begin{proof}
This follows by the interchange law. We do a direct computation:
\begin{align*}
    g \cdot h_x \cdot g^{-1} &= (g \cdot_\da h_x \cdot_\da 1_\ra(\bft_\ra g)^{-1}) \cdot_\ra ((g)_\ra^{-1} \cdot_\da 1_\ra(\bft_\ra g)^{-1}) \\
    &= (g \cdot_\ra (g)_\ra^{-1}) \cdot_\da h_x \cdot_\da 1_\ra(\bft_\ra g)^{-1} = 1_\ra(\bft_\ra g) \cdot_\da h_x \cdot_\da 1_\ra(\bft_\ra g)^{-1}.
\end{align*}
This proves the statement.
\end{proof}

The terminology is chosen to reflect that the identity closely resembles the Peiffer identity for crossed modules of Lie groupoids. That is, the structure we arrived at is both a generalisation of that of core diagrams in \cite{Corediagram} and that of crossed modules (see, for example, \cite{MackenzieClassification,CrossedAndrouli,CamilleWagemann}). We call it a \textit{core extension} in analogy with our definition of fat extension.

\begin{definition}\label{def: core extension}
A \textit{core extension} is a diagram
\begin{center}
    \begin{tikzcd}
        1 \ar[r] & \textnormal{H} \ar[r] & \ar[dl] F \ar[r] & G_\ra \ar[r] & 1 \\
        \phantom{A} & G_\da \ar[u, phantom, "\textnormal{\Large$\circlearrowright$}"] & \phantom{A} & \phantom{A} & \phantom{A}
    \end{tikzcd}
\end{center}
such that the restriction of the map $F \ra G_\da$ to $\textnormal{H}$ is equivariant (with respect to conjugation in $G_\da$) and the Peiffer identity holds: the conjugation action of $F$ on $\textnormal{H}$, viewing $\textnormal{H}$ as a normal subgroupoid of $F$, equals the action by $F$ on $\textnormal{H}$ induced by the bundle of Lie groups representation $G_\da$ $\rotatebox[origin=c]{270}{\large$\circlearrowright$}$ $\textnormal{H}$ and the map $F \ra G_\da$.
\end{definition}

We usually simply write $F$ for a core extension. Notice that, in a core extension $F$, the orbits of $F$ are the same as the orbits of $G_\ra$.

Since a double groupoid might be both horizontally and vertically core-transitive, we sometimes call the associated core extension of a vertically core-transitive double groupoid the \textit{vertical core extension}.

\begin{remark}\label{rema: core extension of a strict Lie 2-groupoid}
It is worthwhile to observe that a strict Lie $2$-groupoid $H_G \rra H_M \rra N$ is always vertically core-transitive. Writing $F_H$ for its core, the above exact sequence for $H_G$ is given by
\begin{center}
\begin{tikzcd}
    1 \ar[r] & F_H \ar[r] & F_H \ar[r] & N \ar[r] & 1.
\end{tikzcd}
\end{center}
So, the vertical core extension of a strict Lie $2$-groupoid is completely determined by a bundle of Lie groups $F_H \rra N$ and a Lie groupoid $H_M \rra N$, together with a Lie groupoid map $F_H \ra H_M$, and a bundle of Lie groups representation $H_M$ $\rotatebox[origin=c]{270}{\large$\circlearrowright$}$ $F_H$. This recovers the concept of crossed module of Lie groupoids. This definition is used in, for example, \cite{MackenzieClassification,CrossedAndrouli,CamilleWagemann}. As we will see in the next section, Section \ref{sec: From core extensions to double groupoids}, the one-to-one correspondence we set between vertically core-transitive double groupoids and core extensions reduces to a one-to-one correspondence between strict Lie $2$-groupoids and crossed modules of Lie groupoids $(F_H,H_M)$ (where $F_H$ is a bundle of Lie groups). See also Remark \ref{rema: correspondence strict Lie 2-groupoids and crossed modules}.
\end{remark}

\subsection{From core extensions to core-transitive double groupoids}\label{sec: From core extensions to double groupoids}

In this section we show that we can recover a vertically core-transitive double groupoid from its vertical core extension. As we already mentioned before, this sets a one-to-one correspondence. For a large and essential part, this was proved already in \cite{Corediagram} (see especially the last comments in that work). There, the focus lies on transitive double groupoids.\footnote{In \cite{Corediagram}, transitive groupoids are called locally trivial groupoids.} These are double groupoids such that all four groupoids defining the double groupoid are assumed to be transitive. Additionally, they are assumed to be horizontally core-transitive and vertically core-transitive. However, it is observed at the end of the work \cite{Corediagram} that generalisations should be possible for vertically core-transitive double groupoids that satisfy another assumption. Writing $G$ for the double groupoid as before, the assumption is written in \cite{Corediagram} as follows: all elements of $G$ can be written as
\begin{equation*}
    g_1 \cdot_\da 1_\ra(g_2) \cdot_\da (g_3)^{-1}_\da
\end{equation*}
where $g_1,g_3 \in F$ and $g_2 \in G_\da$. However, essentially, this assumption is equivalent to the vertically core-transitive assumption:\footnote{A similar remark was already made in \cite{CorediagramLA}.}

\begin{lemma}\label{lemm: vertically core-transitive is equivalent to assumption}
Let $G$ be a double groupoid. Then $G$ is vertically core-transitive if and only if
\begin{equation*}
    F \times_M G_\da \times_M F \ra G \qquad (g_1,g_2,g_3) \mapsto g_1 \cdot_\da 1_\ra(g_2) \cdot_\da (g_3)^{-1}_\da
\end{equation*}
is a surjective submersion.
\end{lemma}
\begin{proof}
A local section $\sigma$ of $\bft_\da: F \ra G_\ra$ defines a local section of the above map by setting
\begin{equation*}
    g \mapsto (g \cdot_\da \sigma(\bfs_\da g) \cdot_\da 1_\ra(\bfs_\ra g)^{-1}, \bfs_\ra g, \sigma(\bfs_\da g)).
\end{equation*}
Given $(g_1,g_2,g_3) \in F \times_M G_\da \times_M F$, a local section of $F \ra G_\ra$ through $g_3$ defines (around the element $g_1 \cdot_\da 1_\ra(g_2) \cdot_\da (g_3)_
\da^{-1} \in G$) a local section through $(g_1,g_2,g_3)$. A local section $\sigma$ of the above map defines a local section of $F \ra G_\ra$ by setting
\begin{equation*}
    g \mapsto \pr_1\sigma 1_\da(g) = \pr_3\sigma 1_\da(g).
\end{equation*}
This proves the statement.
\end{proof}
That is, in Remark 4.3 of \cite{Corediagram}, assumptions (4.3.1) and (4.3.2) are equivalent. The relevance of the above lemma becomes clear after the following observation. The pullback groupoid of $G_\da$ along the source map of $F$ defines a double groupoid
\begin{center}
\begin{tikzcd}
    F \times_M G_\da \times_M F \ar[r, shift left] \ar[r, shift right] \ar[d, shift left] \ar[d, shift right] & G_\da \ar[d, shift left] \ar[d, shift right] \\
    F \ar[r, shift left] \ar[r, shift right] & M
\end{tikzcd}
\end{center}
The vertical groupoid structure takes the form $\bft_\da(g_1,g_2,g_3) = g_1$, $\bfs_\da(g_1,g_2,g_3) = g_3$, and
\begin{equation*}
    (g_1,g_2,g_3) \cdot_\da (g_3, g_4, g_5) = (g_1, g_2 g_4, g_5).
\end{equation*}
The horizontal groupoid structure $F \times_M G_\da \times_M F \rra G_\da$ is given by $\bft_\ra(g_1,g_2,g_3) = (\bft_\ra g_1) g_2 (\bft_\ra g_3)^{-1}$, $\bfs_\ra(g_1,g_2,g_3) = g_2$, and
\begin{equation*}
    (g_1,g_2,g_3) \cdot_\ra (h_1,h_2,h_3) = (g_1h_1, (\bft_\ra h_1)^{-1} g_2 (\bft_\ra h_3) = h_2, g_3h_3).
\end{equation*}

A direct computation shows:

\begin{lemma}\label{lem: horizontally transitive morphism before taking quotient}
Let $G$ be a vertically core-transitive double groupoid. Then the map
\begin{equation*}
    F \times_M G_\da \times_M F \ra G \qquad (g_1,g_2,g_3) \mapsto g_1 \cdot_\da 1_\ra(g_2) \cdot_\da (g_3)^{-1}_\da
\end{equation*}
is a map of double groupoids that is a surjective submersion.
\end{lemma}
\begin{proof}
That the map intertwines the vertical products is readily verified. We do a direct computation to verify that the map intertwines the horizontal products. Suppose $(g_1,g_2,g_3)$ and $(h_1,h_2,h_3)$ in $F \times_M G_\da \times_M F$ are horizontally composable. Then, first of all,
\begin{align*}
    (g_1 \cdot_\da 1_\ra(g_2) &\cdot_\da (g_3)_\da^{-1}) \cdot_\ra (h_1 \cdot_\da 1_\ra(h_2) \cdot_\da (h_3)_\da^{-1}) \\
    &= ((g_1 \cdot_\da 1_\ra(\bft_\ra h_1)) \cdot_\ra h_1) \cdot_\da ((1_\ra((\bft_\ra h_1)^{-1}g_2) \cdot_\da (g_3)_\da^{-1}) \cdot_\ra (1_\ra(h_2) \cdot_\da (h_3)_\da^{-1})) \\
    &= (g_1h_1) \cdot_\da ((1_\ra((\bft_\ra h_1)^{-1}g_2) \cdot_\da (g_3)_\da^{-1}) \cdot_\ra (1_\ra(h_2) \cdot_\da (h_3)_\da^{-1})),
\end{align*}
where in the first equality we replaced $g_1 \cdot_\da 1_\ra(g_2)$ with
\begin{equation*}
    g_1 \cdot_\da 1_\ra(\bft_\ra h_1) \cdot_\da 1_\ra((\bft_\ra h_1)^{-1}g_2)
\end{equation*}
and used the interchange law. We now use that $(\bft_\ra h_1)^{-1}g_2 = h_2(\bft_\ra h_3)^{-1}$ and conclude:
\begin{align*}
    (g_1 \cdot_\da 1_\ra(g_2) &\cdot_\da (g_3)_\da^{-1}) \cdot_\ra (h_1 \cdot_\da 1_\ra(h_2) \cdot_\da (h_3)_\da^{-1}) \\
    &= (g_1h_1) \cdot_\da ((1_\ra(h_2) \cdot_\da 1_\ra(\bft_\ra h_3)^{-1} \cdot_\da (g_3)_\da^{-1}) \cdot_\ra (1_\ra(h_2) \cdot_\da (h_3)_\da^{-1})) \\
    &= (g_1h_1) \cdot_\da 1_\ra(h_2) \cdot_\da ((g_3 \cdot_\da 1_\ra(\bft_\ra h_3))^{-1}_\da \cdot_\ra (h_3)_\da^{-1}) \\
    &= (g_1h_1) \cdot_\da 1_\ra(h_2) \cdot_\da (g_3h_3)_\da^{-1}.
\end{align*}
That the map is a surjective submersion was already proved in Lemma \ref{lemm: vertically core-transitive is equivalent to assumption}.
\end{proof}

Now, we can realise $G$ as a quotient of $F \times_M G_\da \times_M F$ by the horizontal kernel of the above map.

\begin{remark}\label{rema: horizontal kernel is double groupoid}
Given a map of double groupoids $\Phi: G \ra G'$ that is a surjective submersion, notice that its horizontal kernel $K \coloneq \ker_\ra \Phi$ can be seen as a double subgroupoid of $G$. Here,
\begin{equation*}
\ker_\ra \Phi \coloneq \{g \in G \mid \Phi(g) = \im 1_\ra\}.
\end{equation*}
This is a consequence of the interchange law. We can then see the resulting exact sequence
\begin{center}
\begin{tikzcd}
    1 \ar[r] & K \ar[r] & G \ar[r, "\Phi"] & G' \ar[r] & 1
\end{tikzcd}
\end{center}
as a short ``horizontally'' exact sequence of double groupoids. More precisely, all maps in the exact sequence are maps of double groupoids, but the sequence is only exact with respect to the horizontal products.
\end{remark}

The map $F \times_M G_\da \times_M F \ra G$ from Lemma \ref{lemm: vertically core-transitive is equivalent to assumption}, seen as a groupoid map of the horizontal products, has kernel isomorphic to the action groupoid
\begin{equation*}
    G_\da \ltimes \textnormal{H} \rra \textnormal{H},
\end{equation*}
where the action is given by
\begin{equation*}
    g \cdot h_x = 1_\ra(g) \cdot_\da h_x \cdot_\da 1_\ra(g)^{-1}.
\end{equation*}
Its structure of double groupoid is
\begin{center}
\begin{tikzcd}
G_\da \ltimes \textnormal{H} \ar[r, shift left] \ar[r, shift right] \ar[d, shift left] \ar[d, shift right] & G_\da \ar[d, shift left] \ar[d, shift right] \\ \textnormal{H} \ar[r, shift left] \ar[r, shift right] & M
\end{tikzcd}
\end{center}
where $G_\da \times_M \textnormal{H} \rra G_\da$ is seen as a bundle of Lie groups.

\begin{proposition}\label{prop: kernel of defining map between core extension and double groupoid}
Consider the map $F \times_M G_\da \times_M F \ra G$ from Lemma \ref{lem: horizontally transitive morphism before taking quotient}. This map, together with the map
\begin{equation*}
G_\da \ltimes \textnormal{H} \ra F \times_M G_\da \times_M F \qquad (g,h_x) \mapsto (1_\ra(g) \cdot_\da h_x \cdot_\da 1_\ra(g)^{-1}, g, h_x)
\end{equation*}
fit into a short horizontally exact sequence of double Lie groupoids (see Remark \ref{rema: horizontal kernel is double groupoid})
\begin{center}
\begin{tikzcd}
1 \ar[r] & G_\da \ltimes \textnormal{H} \ar[r] & F \times_M G_\da \times_M F \ar[r] & G \ar[r] & 1.
\end{tikzcd}
\end{center}
\end{proposition}
\begin{proof}
The horizontal kernel consists of elements $(g_1,g_2,g_3)$ satisfying
\begin{equation*}
    g_1 \cdot_\da 1_\ra(g_2) \cdot_\da (g_3)_\da^{-1} = 1_\ra(g_2),
\end{equation*}
so $g_1,g_3 \in H$, and
\begin{equation*}
g_1 = 1_\ra(g_2) \cdot_\da g_3 \cdot_\da 1_\ra(g_2)^{-1}.
\end{equation*}
This proves that the horizontal kernel can indeed be identified with $G_\da \ltimes \textnormal{H}$ as above. The map of Lemma \ref{lem: horizontally transitive morphism before taking quotient} induces an isomorphism of double groupoids
\begin{equation*}
    G(F) \ra G \qquad [g_1,g_2,g_3] \mapsto g_1 \cdot_\da 1_\ra(g_2) \cdot_\da (g_3)_\da^{-1},
\end{equation*}
where $G(F)$ is the double groupoid
\begin{equation*}
    G(F) \coloneq (F \times_M G_\da \times_M F) / (G_\da \ltimes \textnormal{H}) \rra F/\textnormal{H} = G_\ra
\end{equation*}
over $G_\da \rra M$. This proves the statement.
\end{proof}
Now, the structure being used to define
\begin{equation*}
    G(F) = (F \times_M G_\da \times_M F) / (G_\da \ltimes \textnormal{H})
\end{equation*}
uses precisely the data underlying the vertical core extension of $G$. To ensure $G(F)$ is defined for any core extension $F$, the compatibility assumptions are used in the following way. Firstly, the action groupoid $G_\da \ltimes \textnormal{H} \rra \textnormal{H}$ is a double groupoid over $G_\da$ precisely because the action defines a bundle of Lie groups representation (see Definition \ref{def: representations of bundles of Lie groups}). The equivariance condition makes sure that the embedding
\begin{equation*}
    G_\da \ltimes \textnormal{H} \ra F \times_M G_\da \times_M F
\end{equation*}
is a map of double groupoids. The Peiffer identity ensures this groupoid is actually normal inside of $F \times_M G_\da \times_M F$. We now conclude:

\begin{theorem}\label{thm: horizontally core transtive double Lie groupoids equivalent to core extensions}
The association
\begin{equation*}
    \{\textnormal{Vertically core-transitive double Lie groupoids}\} \xra{\sim} \{\textnormal{Vertical core extensions}\} \qquad G \mapsto F
\end{equation*}
sets a one-to-one correspondence, and the assignment
\begin{equation*}
    \{\textnormal{Vertical core extensions}\} \xra{\sim} \{\textnormal{Vertically core-transitive double Lie groupoids}\} \qquad F \mapsto G(F)
\end{equation*}
is inverse to it (up to canonical isomorphisms).
\end{theorem}
\begin{proof}
We already showed that, given $G$ a vertically core-transitive double groupoid, $G$ and $G(F)$ are isomorphic, where $F$ is the vertical core extension of $G$. So, it remains to observe that, given a core extension $F$, the core groupoid of $G(F)$ can be identified with $F \ni g$ as the classes of elements represented by $[g, 1_{\bfs g}, 1_{\bfs g}]$. That we recover the core extension is now readily verified.
\end{proof}

\begin{remark}\label{rema: correspondence strict Lie 2-groupoids and crossed modules}
Restricting attention to strict Lie $2$-groupoids yields an interesting consequence: they correspond to crossed modules of Lie groupoids $(F_H,H_M)$ for which $F_H$ is a bundle of Lie groups (see Remark \ref{rema: core extension of a strict Lie 2-groupoid}).
\end{remark}

The whole discussion also applies to horizontally core-transitive double Lie groupoids (by interchanging horizontal and vertical structures everywhere). From our discussion it follows that \textit{double core-transitive groupoids} (so double groupoids that are both horizontally and vertically core-transitive) are isomorphic to either of the two double groupoids that its core extensions define. However, starting with two core extensions, it is not immediate that the two resulting double groupoids one constructs are isomorphic. Therefore, an extra compatibility assumption between the two core extensions appears. This leads to the following definition:

\begin{definition}\label{def: double core-transitive}
A \textit{double core extension}, or core diagram, is a diagram
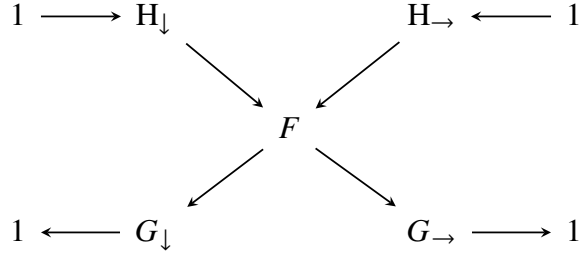
\begin{center}
\begin{tikzcd}
    1 \ar[r] & \textnormal{H}_\da \ar[rd] & \phantom{A} & \ar[ld] \textnormal{H}_\ra & \ar[l] 1 \\
    \phantom{A} & \phantom{A} & \ar[ld] F \ar[rd] & \phantom{A} & \phantom{A} \\
    1 & \ar[l] G_\da & \phantom{A} & G_\ra \ar[r] & 1
\end{tikzcd}
\end{center}
of exact sequences such that $\textnormal{H}_\da$ and $\textnormal{H}_\ra$ commute elementwise in $F$: for all $h_{x,\da} \in \textnormal{H}_\da$ and $h_{x,\ra} \in \textnormal{H}_\ra$ we have that
\begin{equation*}
    h_{x,\da} h_{x,\ra} = h_{x,\ra} h_{x,\da} 
\end{equation*}
as elements of $F$.
\end{definition}

\begin{remark}\label{rema: double core-transitive}
The above compatibility condition (of ``commuting elementwise'') is written in \cite{Corediagram}. We did not specify the actions in the above definition, for specifying the actions in a double core extension is unnecessary (this should be compared to Lemma 2.6 of \cite{Corediagram}). Indeed, given two core extensions
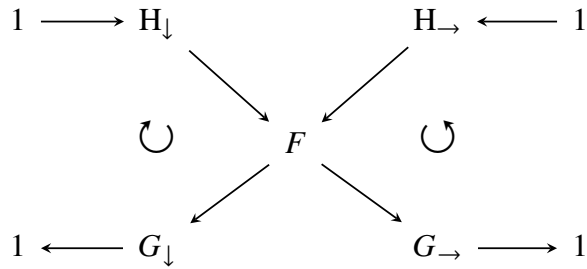
\begin{center}
\begin{tikzcd}
    1 \ar[r] & \textnormal{H}_\da \ar[rd] & \phantom{A} & \ar[ld] \textnormal{H}_\ra & \ar[l] 1 \\
    \phantom{A} & \textnormal{\Large$\circlearrowright$} & \ar[ld] F \ar[rd] & \textnormal{\Large$\circlearrowleft$} & \phantom{A} \\
    1 & \ar[l] G_\da & \phantom{A} & G_\ra \ar[r] & 1
\end{tikzcd}
\end{center}
satisfying that elements of $\textnormal{H}_\da$ and $\textnormal{H}_\ra$ commute elementwise in $F$, then, for example, the action of $G_\ra$ on $H_\ra$ is given by
\begin{equation*}
    g \cdot h_x = \widetilde g \cdot h_x \cdot \widetilde g^{-1}
\end{equation*}
where $\widetilde g \in F$ is arbitrary with the property that $\widetilde g$ maps to $g$ under the map $F \ra G_\ra$. That this action is well-defined uses the compatibility between the two core extensions. The newly defined actions in terms of lifts therefore determine the actions that are part of the data of the given core extensions. As equivariance and the Peiffer identity are immediate from this new description of the action, a double core extension can be defined as in Definition \ref{def: double core-transitive}.
\end{remark}

From our discussion we can conclude (this result was already proved in \cite{CorediagramLA}):

\begin{corollary}\label{coro: double core-transitive}
There is a one-to-one correspondence
\begin{equation*}
    \{\textnormal{Double core-transitive groupoids}\} \xra{\sim} \{\textnormal{Double core extensions}\} \qquad G \mapsto F.
\end{equation*}
\end{corollary}

\begin{remark}\label{rema: recovering corediagram result}
In \cite{Corediagram}, the focus lies on ``double transitive groupoids'', so horizontally and vertically transitive (and core-transitive) double groupoids. To recover the main result of \cite{Corediagram} on double transitive groupoids, recall that, in a core extension
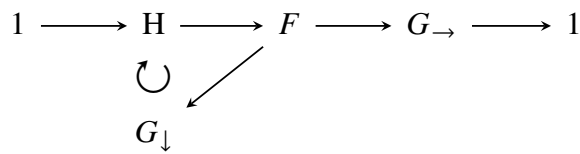
\begin{center}
\begin{tikzcd}
    1 \ar[r] & \textnormal{H} \ar[r] & \ar[dl] F \ar[r] & G_\ra \ar[r] & 1 \\
    \phantom{A} & G_\da \ar[u, phantom, "\textnormal{\Large$\circlearrowright$}"] & \phantom{A} & \phantom{A} & \phantom{A}
\end{tikzcd}
\end{center}
the orbits of $F$ are equal to the orbits of $G_\ra$. Therefore, $F$ is transitive if and only if $G_\ra$ is, and we call core extensions for which this holds transitive. The resulting double groupoid is vertically transitive, and this sets an equivalence between transitive vertical core extensions and vertically transitive (that are also vertically core-transitive) double groupoids. Using Corollary \ref{coro: double core-transitive}, this reproduces the result of \cite{Corediagram} on double transitive groupoids (namely that they correspond to double transitive core extensions).
\end{remark}

\begin{remark}\label{rema: principal bibundle as double groupoid}
A class of examples of double groupoids which are usually not (core-)transitive are given by commuting groupoid actions $G$ \rotatebox[origin=c]{270}{\large$\circlearrowright$} $P$ \rotatebox[origin=c]{90}{\large$\circlearrowright$} $H$. Given such commuting actions, we have a double groupoid
\begin{center}
\begin{tikzcd}
    G \ltimes P \rtimes H \ar[r, shift left] \ar[r, shift right] \ar[d, shift left] \ar[d, shift right] & P \rtimes H \ar[d, shift left] \ar[d, shift right] \\
    G \ltimes P \ar[r, shift left] \ar[r, shift right] & P
\end{tikzcd}
\end{center}
The core is the unit groupoid $1_P$, and the core data consists further of the unit embeddings
\begin{equation*}
    G \ltimes P \hookleftarrow P \hookrightarrow P \rtimes H.
\end{equation*}
These maps carry no information about the actions. In this case, rather than considering the core, we can view $G \ltimes P \rtimes H$ as a groupoid over $P$, and there are two Lie groupoid maps
\begin{equation*}
    G \la G \ltimes P \rtimes H \ra H.
\end{equation*}
It is well-known that there is a natural way to set an equivalence between left/right principal bibundles (commutating actions, one of which is principal) and such diagrams $G \la K \ra H$ where one is a Morita map (see \cite{Matiasorbispaces}).
\end{remark}

We now go back to PB-groupoids, and show that PB-groupoids (that come with a gauge double groupoid) can be efficiently described using core extensions.

\subsection{PB-groupoids and core extensions}

Principal groupoid bundles do not always come with a smooth gauge groupoid. A typical assumption that guarantees the smoothness of the gauge groupoid is that the moment map is a surjective submersion (see e.g. \cite{Matiasorbispaces}). General linear PB-groupoids come with a gauge double groupoid if the differential has constant rank. We focus here on the following class of PB-groupoids:

\begin{definition}\label{def: PB-groupoid with gauge double groupoid}
let $P_G$ be a PB-groupoid. If the moment map(s) of $P_G$ are surjective submersions onto a smooth embedded submanifold, then we call $P_G$ a \textit{PB-groupoid with gauge double groupoid}.
\end{definition}

To be clear, PB-groupoids with gauge double groupoid really come with a gauge double groupoid:

\begin{proposition}\label{prop: PB-groupoid with gauge double groupoid}
Let $P_G \rra P_M$ be a PB-groupoid with gauge double groupoid. Then
\begin{equation*}
    P_G \times_{H_G} P_G \rra P_M \times_{H_M} P_M
\end{equation*}
is a Lie groupoid with $\bfs[p_g,p_h] \coloneq [\bfs p_g, \bfs p_h]$, $\bft[p_g,p_h] = [\bft p_g, \bft p_h]$, and
\begin{equation*}
    [p_g,p_h] \cdot [p_{g'},p_{h'}] \coloneq [p_g \cdot p_{g'}, p_h \cdot p_{h'}] \qquad [p_g,p_h]^{-1} = [p_g^{-1},p_h^{-1}].
\end{equation*}
(here, $\bfs p_g = \bft p_{g'}$ and $\bfs p_h = \bft p_{h'}$). Moreover,
\begin{center}
\begin{tikzcd}
    P_G \times_{H_G} P_G \ar[r, shift left] \ar[r, shift right] \ar[d, shift left] \ar[d, shift right] & P_M \times_{H_M} P_M \ar[d, shift left] \ar[d, shift right] \\
    G \ar[r, shift left] \ar[r, shift right] & M
\end{tikzcd}
\end{center}
is a full double groupoid, called the \textit{gauge double groupoid} of $P_G$.
\end{proposition}
\begin{proof}
We prove here that the above double groupoid is full. That is, the double source map
\begin{equation*}
    (\bfs_\ra, \bfs_\da): P_G \times_{H_G} P_G \ra (P_M \times_{H_M} P_M) \times_M G
\end{equation*}
is a surjective submersion. Indeed, to construct a section through $[p_g,p_h] \in P_G \times_{H_G} P_G$, first pick a local section $\sigma$ of $P_M \ra M$. Then, take a local section $\sigma$ of $P_G \ra G$, which is projectable to the local section $\sigma$ of $P_M \ra M$, and goes through
\begin{equation*}
    p_h \cdot 1_{\Psi(\sigma(\bfs h), \bfs p_h)^{-1}},
\end{equation*}
where we used
\begin{equation*}
    \Psi: P_M \times_M P_M \xla{\sim} P_M \times_N H_M \ra H_M.
\end{equation*}
Lastly, we pick a local section $\sigma_\mu$ of the moment map $\mu: P_G \ra N$ through $p_g$ (defined on an open set in the image of $\mu$). This way, we can build the local section
\begin{equation*}
    (P_M \times_N P_M) \times_M G \ra P_G \times_N P_G \qquad (p_x,p_{\bfs k},k) \mapsto (\sigma_\mu(\mu p_x) \cdot 1_{\Psi(\bfs \sigma_\mu(\mu p_x), p_x)} , \sigma(k) \cdot 1_{\Psi(\sigma(\bfs k), p_{\bfs k})})
\end{equation*}
which induces the desired well-defined local section through $[p_g,p_h] \in P_G \times_{H_G} P_G$.
\end{proof}

We will argue next that the gauge groupoid of a PB-groupoid with gauge double groupoid is always vertically core-transitive. Basically, this follows from the fact that the structural strict Lie $2$-groupoid is always vertically core-transitive (see Remark \ref{rema: core extension of a strict Lie 2-groupoid}).

\begin{lemma}\label{lemm: core-transitive PB-groupoid}
Let $P_G \rra P_M$ be a PB-groupoid with gauge double groupoid, and denote by $H_G \rra H_M \rra N$ its structural strict Lie $2$-groupoid. Then its gauge double groupoid
\begin{center}
\begin{tikzcd}
    P_G \times_{H_G} P_G \ar[r, shift left] \ar[r, shift right] \ar[d, shift left] \ar[d, shift right] & P_M \times_{H_M} P_M \ar[d, shift left] \ar[d, shift right] \\
    G \ar[r, shift left] \ar[r, shift right] & M
\end{tikzcd}
\end{center}
is vertically core-transitive.
\end{lemma}
\begin{proof}
We write $F_H$ for the core of $H_G$ and $F$ for the core of the gauge double groupoid. Note that
\begin{equation*}
    F = \{[p_g,q_{1_{\bfs g}}] \in P_G \times_{H_G} P_G \mid \bfs p = \bfs q \in P_M\}.
\end{equation*}
Now, let $\sigma$ be a (local) section of $P_G \times_N P_G \ra G \times G$ that is $\bfs \times \bfs$-projectable to a section of $P_M \times_N P_M \ra M \times M$, and consider a section of $\bft_\da: F_H \ra N$. Define
\begin{equation*}
    a: G \times G \xra{\sigma} P_G \times_N P_G \ra N \ra F_H.
\end{equation*}
Given $[p_g,q_{1_{\bfs g}}] \in F$, we take $\sigma$ defined on an open set of $(g, 1_{\bfs g})$, and such that $\sigma(g,1_{\bfs g}) = (p_g,q_{1_{\bfs g}})$. Then the pointwise multiplication map
\begin{equation*}
    \sigma \cdot a: G \times G \ra P_G \times_{H_G} P_G
\end{equation*}
defines through
\begin{equation*}
    G \hookrightarrow G \times 1 \xra{\sigma \cdot a} F
\end{equation*}
the desired section of $F \ra G$ through $[p_g,q_{1_{\bfs g}}] \in F$. This proves the statement.
\end{proof}

In turn, a PB-groupoid with gauge double groupoid comes with a vertical core extension. We can describe it as follows.

\begin{definition}\label{def: core extension of a PB-groupoid}
Let $P_G \rra P_M$ be a PB-groupoid with gauge double groupoid. The core extension of $P_G$ is the principal groupoid bundle $P_M$ together with the core extension
\begin{center}
\begin{tikzcd}
    1 \ar[r] & \textnormal{H} \ar[r] & \ar[dl] F \ar[r] & G \ar[r] & 1. \\
    \phantom{A} & P_M \times_{H_M} P_M \ar[u, phantom, "\textnormal{\Large$\circlearrowright$}"] & \phantom{A} & \phantom{A} & \phantom{A}
\end{tikzcd}
\end{center}
\end{definition}

\begin{remark}
In the above discussion, the core groupoid $F$ is isomorphic to
\begin{equation*}
    P_G/H_M \rra P_M/H_M = M
\end{equation*}
where $H_M$ acts on $P_G$ via the unit embedding $H_M \hookrightarrow H_G$. The isomorphism is given by
\begin{equation*}
    P_G/H_M \ra F \qquad [p] \mapsto [p,1_{\bfs p}].
\end{equation*}
This should be compared with Section \ref{sec: From general linear PB-groupoids to fat extensions}.
\end{remark}

If $P_G$ is a regular general linear PB-groupoid, then the gauge groupoid associated to the principal groupoid bundle $P_M = \textnormal{GL}(C_M,V_M)$ is the Lie groupoid $\textnormal{Aut}(C_M \ra V_M)$. To tie everything together, we end this section with the following result.

\begin{proposition}\label{prop: fat extension as core extension}
Let $F_G \rra M$ be a Lie groupoid, and let $C_M \ra V_M$ be a regular $2$-term complex. Consider the canonical (conjugation) action $\textnormal{Aut}(C_M \ra V_M)$ $\rotatebox[origin=c]{270}{\large$\circlearrowright$}$ $\textnormal{H}(C_M,V_M)$. Then $F_G$ defines a fat extension of $G$ over $C_M \ra V_M$ if and only if
\begin{center}
    \begin{tikzcd}
        1 \ar[r] & \textnormal{H}(V_M,C_M) \ar[r] & \ar[dl] F_G \ar[r] & G \ar[r] & 1, \\
        \phantom{A} & \textnormal{Aut}(C_M \ra V_M) \ar[u, phantom, "\textnormal{\Large$\circlearrowright$}"] & \phantom{A} & \phantom{A} & \phantom{A}
    \end{tikzcd}
\end{center}
defines a core extension.
\end{proposition}
\begin{proof}
The equivariance condition is equivalent to $F_G$ restricting to the canonical cochain complex representation, and the Peiffer identity translates to the fact that the two natural conjugation actions of $F_G$ on $\textnormal{H}(V_M,C_M)$ agree. This proves the statement.
\end{proof}

\begin{remark}
That the action in a core extension
\begin{center}
\begin{tikzcd}
    1 \ar[r] & \textnormal{H} \ar[r] & \ar[dl] F \ar[r] & G_\ra \ar[r] & 1 \\
    \phantom{A} & G_\da \ar[u, phantom, "\textnormal{\Large$\circlearrowright$}"] & \phantom{A} & \phantom{A} & \phantom{A}
\end{tikzcd}
\end{center}
defines a bundle of Lie groups representation uses, of course, the group structure of $G_\da$ explicitly. But the equivariance and Peiffer identity use the groupoid structure $G_\da$ in terms of the groupoid map $F \ra G_\da$ only.

While the regularity of the complex in the above statement is used to ensure that $\textnormal{Aut}(C_M \ra V_M)$ is a smooth Lie groupoid, this is irrelevant for the description of a fat extension $F_G$. The reason is that we can ``bypass'' $\textnormal{Aut}(C_M \ra V_M)$ by embedding $\textnormal{Aut}(C_M \ra V_M)$ into, for example, $\textnormal{GL}(C_M,V_M)$. Although the conjugation action of $\textnormal{GL}(C_M,V_M)$ on $\textnormal{H}(V_M,C_M)$ is not well-defined, it is well-defined on the image of $F_G \ra \textnormal{GL}(C_M,V_M)$. So, a fat extension is, in essence, a type of core extension, but only when interpreting the concept in such a more general form.
\end{remark}

In general, the gauge double groupoid of a PB-groupoid with gauge double groupoid is encoded by a core extension. But the above remark suggests that at least general linear PB-groupoids are also encoded by a type of ``core extension'': fat extensions. Investigating this definition for PB-groupoids in more general setups (in the style of the previous remark) seems an interesting direction that will be pursued elsewhere.


\bibliographystyle{alpha}
\bibliography{bib}

@article{Eulerlike,
 author = {Bischoff, F. and Bursztyn, H. and Lima, H. and Meinrenken, E.},
 title = {Deformation spaces and normal forms around transversals},
 fjournal = {Compositio Mathematica},
 journal = {Compos. Math.},
 issn = {0010-437X},
 volume = {156},
 number = {4},
 pages = {697--732},
 year = {2020},
 language = {English},
 doi = {10.1112/S0010437X1900784X},
 zbMATH = {7169455},
 Zbl = {1434.53029}
}

@article{Matiasorbispaces,
 author = {del Hoyo, M. L.},
 title = {Lie groupoids and their orbispaces},
 fjournal = {Portugaliae Mathematica. Nova S{\'e}rie},
 journal = {Port. Math. (N.S.)},
 issn = {0032-5155},
 volume = {70},
 number = {2},
 pages = {161--209},
 year = {2013},
 language = {English},
 doi = {10.4171/PM/1930},
 zbMATH = {6215352},
 Zbl = {1277.22005}
}

@article{VBalejandrohenriquematias,
 author = {Bursztyn, H. and Cabrera, A. and del Hoyo, M. L.},
 title = {Vector bundles over {Lie} groupoids and algebroids},
 fjournal = {Advances in Mathematics},
 journal = {Adv. Math.},
 issn = {0001-8708},
 volume = {290},
 pages = {163--207},
 year = {2016},
 language = {English},
 doi = {10.1016/j.aim.2015.11.044},
 zbMATH = {6538715},
 Zbl = {1391.53092}
}

@article{VBMorita,
 author = {del Hoyo, M. L. and Ortiz, C.},
 title = {Morita equivalences of vector bundles},
 fjournal = {IMRN. International Mathematics Research Notices},
 journal = {Int. Math. Res. Not.},
 issn = {1073-7928},
 volume = {2020},
 number = {14},
 pages = {4395--4432},
 year = {2020},
 language = {English},
 doi = {10.1093/imrn/rny149},
 zbMATH = {7441784},
 Zbl = {1478.18012}
}

@article{Highervectorbundles,
    author = {del Hoyo, M. L. and Trentinaglia, G.},
    title = {Higher Vector Bundles},
    journal = {International Mathematics Research Notices},
    volume = {2025},
    number = {11},
    pages = {rnaf128},
    year = {2025},
    month = {05},
    issn = {1073-7928},
    doi = {10.1093/imrn/rnaf128},
    url = {https://doi.org/10.1093/imrn/rnaf128},
    eprint = {https://academic.oup.com/imrn/article-pdf/2025/11/rnaf128/63326912/rnaf128.pdf},
}

@article{VBalgebroidsruths,
 author = {Gracia-Saz, A. and Mehta, R. A.},
 title = {Lie algebroid structures on double vector bundles and representation theory of {Lie} algebroids},
 fjournal = {Advances in Mathematics},
 journal = {Adv. Math.},
 issn = {0001-8708},
 volume = {223},
 number = {4},
 pages = {1236--1275},
 year = {2010},
 language = {English},
 doi = {10.1016/j.aim.2009.09.010},
 zbMATH = {5671649},
 Zbl = {1183.22002}
}

@article{VBgroupoidsruths,
 author = {Gracia-Saz, A. and Mehta, R. A.},
 title = {{{\(\mathcal{VB}\)}}-groupoids and representation theory of {Lie} groupoids},
 fjournal = {The Journal of Symplectic Geometry},
 journal = {J. Symplectic Geom.},
 issn = {1527-5256},
 volume = {15},
 number = {3},
 pages = {741--783},
 year = {2017},
 language = {English},
 doi = {10.4310/JSG.2017.v15.n3.a5},
 zbMATH = {6826876},
 Zbl = {1387.18009}
}

@article{Blup,
 author = {Debord, C. and Skandalis, G.},
 title = {Blow-up constructions for {Lie} groupoids and a {Boutet} de {Monvel} type calculus},
 fjournal = {M{\"u}nster Journal of Mathematics},
 journal = {M{\"u}nster J. Math.},
 issn = {1867-5778},
 volume = {14},
 number = {1},
 pages = {1--40},
 year = {2021},
 language = {English},
 doi = {10.17879/59019640550},
 zbMATH = {7324080},
 Zbl = {1460.58013}
}

@article{VanEsthomogeneous,
 author = {Cabrera, A. and Drummond, T.},
 title = {Van {Est} isomorphism for homogeneous cochains},
 fjournal = {Pacific Journal of Mathematics},
 journal = {Pac. J. Math.},
 issn = {1945-5844},
 volume = {287},
 number = {2},
 pages = {297--336},
 year = {2017},
 language = {English},
 doi = {10.2140/pjm.2017.287.297},
 zbMATH = {6705118},
 Zbl = {1360.22006}
}

@article{LAruthsasfgp,
 author = {Mehta, R. A.},
 title = {Lie algebroid modules and representations up to homotopy},
 fjournal = {Indagationes Mathematicae. New Series},
 journal = {Indag. Math., New Ser.},
 issn = {0019-3577},
 volume = {25},
 number = {5},
 pages = {1122--1134},
 year = {2014},
 language = {English},
 doi = {10.1016/j.indag.2014.07.013},
 zbMATH = {6351313},
 Zbl = {1301.58009}
}

@article{LGruths,
 author = {Arias Abad, C. and Crainic, M.},
 title = {Representations up to homotopy and {Bott}'s spectral sequence for {Lie} groupoids},
 fjournal = {Advances in Mathematics},
 journal = {Adv. Math.},
 issn = {0001-8708},
 volume = {248},
 pages = {416--452},
 year = {2013},
 language = {English},
 doi = {10.1016/j.aim.2012.12.022},
 zbMATH = {6264524},
 Zbl = {1284.55018}
}

@article{Tensorruth,
 author = {Arias Abad, C. and Crainic, M. and Dherin, B.},
 title = {Tensor products of representations up to homotopy},
 fjournal = {Journal of Homotopy and Related Structures},
 journal = {J. Homotopy Relat. Struct.},
 issn = {2193-8407},
 volume = {6},
 number = {2},
 pages = {239--288},
 year = {2011},
 language = {English},
 url = {tcms.org.ge/Journals/JHRS/},
 zbMATH = {6240899},
 Zbl = {1278.18029}
}

@article{LAruths,
 author = {Arias Abad, C. and Crainic, M.},
 title = {Representations up to homotopy of {Lie} algebroids},
 fjournal = {Journal f{\"u}r die Reine und Angewandte Mathematik},
 journal = {J. Reine Angew. Math.},
 issn = {0075-4102},
 volume = {663},
 pages = {91--126},
 year = {2012},
 language = {English},
 zbMATH = {6012495},
 Zbl = {1238.58010}
}

@article{CamiloFlorianruths,
 author = {Arias Abad, C. and Sch{\"a}tz, F.},
 title = {Deformations of {Lie} brackets and representations up to homotopy},
 fjournal = {Indagationes Mathematicae. New Series},
 journal = {Indag. Math., New Ser.},
 issn = {0019-3577},
 volume = {22},
 number = {1-2},
 pages = {27--54},
 year = {2011},
 language = {English},
 doi = {10.1016/j.indag.2011.07.003},
 url = {www.zora.uzh.ch/id/eprint/58215/4/1006.1550v1.pdf},
 zbMATH = {5982442},
 Zbl = {1235.53085}
}

@article{Deformationsalgebroids,
 author = {Crainic, M. and Moerdijk, I.},
 title = {Deformations of {Lie} brackets: {Cohomological} aspects},
 fjournal = {Journal of the European Mathematical Society (JEMS)},
 journal = {J. Eur. Math. Soc. (JEMS)},
 issn = {1435-9855},
 volume = {10},
 number = {4},
 pages = {1037--1059},
 year = {2008},
 language = {English},
 doi = {10.4171/JEMS/139},
 url = {www.ems-ph.org/journals/show_pdf.php?issn=1435-9855&vol=10&iss=4&rank=7},
 zbMATH = {5365144},
 Zbl = {1159.58011}
}

@article{Deformationsgroupoids,
 author = {Crainic, M. and Mestre, J. N. and Struchiner, I.},
 title = {Deformations of {Lie} groupoids},
 fjournal = {IMRN. International Mathematics Research Notices},
 journal = {Int. Math. Res. Not.},
 issn = {1073-7928},
 volume = {2020},
 number = {21},
 pages = {7662--7746},
 year = {2020},
 language = {English},
 doi = {10.1093/imrn/rny221},
 zbMATH = {7319827},
 Zbl = {1470.58020}
}

@misc{Fat,
  author = {Obster, L.},
  howpublished = {In preparation},
  title = {Fat Lie theory}
}

@misc{Multiplicativetensors,
  author = {Mestre, J. N. and Obster, L. and Vitagliano, L.},
  howpublished = {In preparation},
  title = {A cohomology theory for (infinitesimally) multiplicative tensors}
}

@misc{Homotopyoperators,
  author = {Marcut, I. and Mestre, J. N. and Obster, L.},
  howpublished = {In preparation},
  title = {Homotopy operators for representations up to homotopy}
}

@misc{lobsterblup,
      title={Blow-ups of {L}ie groupoids and {L}ie algebroids}, 
      author={Obster, L.},
      year={2021},
      eprint={2110.12247},
      archivePrefix={arXiv},
      primaryClass={math.DG},
      url={https://arxiv.org/abs/2110.12247}, 
}

@article{ShengLAdeformations,
 author = {Sheng, Y.},
 title = {On deformations of {Lie} algebroids},
 fjournal = {Results in Mathematics},
 journal = {Result. Math.},
 issn = {1422-6383},
 volume = {62},
 number = {1-2},
 pages = {103--120},
 year = {2012},
 language = {English},
 doi = {10.1007/s00025-011-0133-x},
 zbMATH = {6112882},
 Zbl = {1318.53094}
}

@misc{Stefanithesis,
      title={Representations up to homotopy and perfect complexes over differentiable stacks}, 
      howpublished = {PhD Thesis},
      author={Stefani, D.},
      year={2019}
}

@article{Zanbroscats,
 author = {Grad, {\v{Z}}.},
 title = {Fundamentals of {Lie} categories},
 fjournal = {Journal of Noncommutative Geometry},
 journal = {J. Noncommut. Geom.},
 issn = {1661-6952},
 volume = {19},
 number = {1},
 pages = {211--248},
 year = {2025},
 language = {English},
 doi = {10.4171/JNCG/563},
 zbMATH = {8005353},
 Zbl = {1560.22008}
}

@article{LiblandMeinrenkenVanEst,
 author = {Li-Bland, D. and Meinrenken, E.},
 title = {On the {Van} {Est} homomorphism for {Lie} groupoids},
 fjournal = {L'Enseignement Math{\'e}matique. 2e S{\'e}rie},
 journal = {Enseign. Math. (2)},
 issn = {0013-8584},
 volume = {61},
 number = {1-2},
 pages = {93--137},
 year = {2015},
 language = {English},
 doi = {10.4171/LEM/61-1/2-5},
 zbMATH = {6545827},
 Zbl = {1337.55017}
}

@article{MariaMeinrenkenVanEst,
 author = {Meinrenken, E. and Salazar, M. A.},
 title = {Van {Est} differentiation and integration},
 fjournal = {Mathematische Annalen},
 journal = {Math. Ann.},
 issn = {0025-5831},
 volume = {376},
 number = {3-4},
 pages = {1395--1428},
 year = {2020},
 language = {English},
 doi = {10.1007/s00208-019-01917-1},
 zbMATH = {7184947},
 Zbl = {1443.22004}
}

@misc{PerturbationMarius,
 author = {Crainic, M.},
 title = {On the perturbation lemma, and deformations},
 year = {2004},
 howpublished = {Preprint, {arXiv}:math/0403266 [math.{AT}]},
 url = {https://arxiv.org/abs/math/0403266},
 arXiv = {arXiv:math/0403266}
}

@article{PB-groupoids,
 author = {Cattafi, F. and Garmendia, A.},
 title = {{PB}-groupoids vs {VB}-groupoids},
 journal = {Rev. Mat. Iberoam.},
 volume = {42},
 number = {1},
 pages = {345--392},
 year = {2026},
 doi = {10.4171/RMI/1580}
}

@article{Corediagram,
 author = {Brown, R. and Mackenzie, K. C. H.},
 title = {Determination of a double {Lie} groupoid by its core diagram},
 fjournal = {Journal of Pure and Applied Algebra},
 journal = {J. Pure Appl. Algebra},
 issn = {0022-4049},
 volume = {80},
 number = {3},
 pages = {237--272},
 year = {1992},
 language = {English},
 doi = {10.1016/0022-4049(92)90145-6},
 zbMATH = {89584},
 Zbl = {0766.22001}
}

@article{CorediagramLA,
 author = {Jotz-Lean, M. and Mackenzie, K. C. H.},
 title = {Transitive double {Lie} algebroids via core diagrams},
 fjournal = {Journal of Geometric Mechanics},
 journal = {J. Geom. Mech.},
 issn = {1941-4889},
 volume = {13},
 number = {3},
 pages = {403--457},
 year = {2021},
 language = {English},
 doi = {10.3934/jgm.2021023},
 zbMATH = {7453485},
 Zbl = {1484.22015}
}

@article{GLStefaniMatias,
 author = {del Hoyo, M. L. and Stefani, D.},
 title = {The general linear 2-groupoid},
 fjournal = {Pacific Journal of Mathematics},
 journal = {Pac. J. Math.},
 issn = {1945-5844},
 volume = {298},
 number = {1},
 pages = {33--57},
 year = {2019},
 language = {English},
 doi = {10.2140/pjm.2019.298.33},
 zbMATH = {7020375},
 Zbl = {1441.18031}
}

@article{DoublegroupoidsMehtaTang,
 author = {Mehta, R. A. and Tang, X.},
 title = {From double {Lie} groupoids to local {Lie} 2-groupoids},
 fjournal = {Bulletin of the Brazilian Mathematical Society. New Series},
 journal = {Bull. Braz. Math. Soc. (N.S.)},
 issn = {1678-7544},
 volume = {42},
 number = {4},
 pages = {651--681},
 year = {2011},
 language = {English},
 doi = {10.1007/s00574-011-0033-4},
 zbMATH = {6037492},
 Zbl = {1242.53104}
}

@article{MariusVanEst,
 author = {Crainic, M.},
 title = {Differentiable and algebroid cohomology, {Van} {Est} isomorphisms, and characteristic classes},
 fjournal = {Commentarii Mathematici Helvetici},
 journal = {Comment. Math. Helv.},
 issn = {0010-2571},
 volume = {78},
 number = {4},
 pages = {681--721},
 year = {2003},
 language = {English},
 doi = {10.1007/s00014-001-0766-9},
 zbMATH = {2038403},
 Zbl = {1041.58007}
}

@incollection{DNCconstructionmanifolds,
 author = {Carrillo Rouse, P.},
 title = {A {Schwartz} type algebra for the tangent groupoid},
 booktitle = {\(K\)-theory and noncommutative geometry. Proceedings of the ICM 2006 satellite conference, Valladolid, Spain, August 31--September 6, 2006},
 isbn = {978-3-03719-060-9},
 pages = {181--199},
 year = {2008},
 publisher = {Z{\"u}rich: European Mathematical Society (EMS)},
 language = {English},
 zbMATH = {5381937},
 Zbl = {1165.58003}
}

@misc{Meinrenkenlecturenotes,
  author = {Meinrenken, E.},
  howpublished = {Lecture notes},
  title = {Lie groupoids and Lie algebroids lecture notes, fall 2017},
  year = {2017}
}

@unpublished{CrossedAndrouli,
  TITLE = {{Crossed modules and the integrability of Lie brackets}},
  AUTHOR = {Androulidakis, I.},
  URL = {https://hal.science/hal-00010028},
  NOTE = {41 pages},
  YEAR = {2005},
  HAL_ID = {hal-00010028},
  HAL_VERSION = {v1},
}

@article{DrummondEgea,
title = {Differential forms with values in VB-groupoids and its Morita invariance},
journal = {Journal of Geometry and Physics},
volume = {135},
pages = {42-69},
year = {2019},
issn = {0393-0440},
doi = {https://doi.org/10.1016/j.geomphys.2018.08.019},
url = {https://www.sciencedirect.com/science/article/pii/S0393044018303024},
author = {T. Drummond and L. Egea}
}

@article{BlaomCartangroupoids,
 author = {Blaom, A. D.},
 title = {Cartan connections on {Lie} groupoids and their integrability},
 fjournal = {SIGMA. Symmetry, Integrability and Geometry: Methods and Applications},
 journal = {SIGMA, Symmetry Integrability Geom. Methods Appl.},
 issn = {1815-0659},
 volume = {12},
 pages = {paper 114, 26},
 year = {2016},
 language = {English},
 doi = {10.3842/SIGMA.2016.114},
 zbMATH = {6662713},
 Zbl = {1353.53026}
}

@article{BlaomInfinitesimal,
 author = {Blaom, A. D.},
 title = {Geometric structures as deformed infinitesimal symmetries},
 fjournal = {Transactions of the American Mathematical Society},
 journal = {Trans. Am. Math. Soc.},
 issn = {0002-9947},
 volume = {358},
 number = {8},
 pages = {3651--3671},
 year = {2006},
 language = {English},
 doi = {10.1090/S0002-9947-06-04057-8},
 zbMATH = {5023934},
 Zbl = {1100.53029}
}

@incollection{Stefanini,
 author = {Stefanini, L.},
 title = {On the integration of {{\(\mathfrak{LA}\)}}-groupoids and duality for {Poisson} groupoids},
 booktitle = {Proceedings of the school on Poisson geometry and related topics, Yokohama, Japan, May 31--June 2, 2006},
 isbn = {978-2-87971-092-1},
 pages = {39--59},
 year = {2007},
 publisher = {Luxembourg: University of Luxembourg, Faculty of Science, Technology {and} Communication},
 language = {English},
 zbMATH = {5302796},
 Zbl = {1156.58008}
}

@misc{MariaThesis,
 author = {Salazar, M. A.},
 title = {Pfaffian groupoids},
 year = {2013},
 howpublished = {PhD thesis, {arXiv}:1306.1164 [math.{DG}] (2013)},
 url = {https://arxiv.org/abs/1306.1164},
 arXiv = {arXiv:1306.1164}
}

@article{IvanMariaMarius,
 author = {Crainic, M. and Salazar, M. A. and Struchiner, I.},
 title = {Multiplicative forms and {Spencer} operators},
 fjournal = {Mathematische Zeitschrift},
 journal = {Math. Z.},
 issn = {0025-5874},
 volume = {279},
 number = {3-4},
 pages = {939--979},
 year = {2015},
 language = {English},
 doi = {10.1007/s00209-014-1398-z},
 zbMATH = {6422645},
 Zbl = {1408.58015}
}

@article{LucaLaPastina,
 author = {La Pastina, P. P. and Vitagliano, L.},
 title = {Deformations of vector bundles over {Lie} groupoids},
 fjournal = {Revista Matem{\'a}tica Complutense},
 journal = {Rev. Mat. Complut.},
 issn = {1139-1138},
 volume = {36},
 number = {3},
 pages = {933--971},
 year = {2023},
 language = {English},
 doi = {10.1007/s13163-022-00441-2},
 zbMATH = {7741377},
 Zbl = {1540.22006}
}

@article{ShengCategorification,
 author = {Sheng, Y.},
 title = {Categorification of {{\(\mathsf{VB}\)}}-{Lie} algebroids and {{\(\mathsf{VB}\)}}-{Courant} algebroids},
 fjournal = {Journal of Geometric Mechanics},
 journal = {J. Geom. Mech.},
 issn = {1941-4889},
 volume = {15},
 number = {1},
 pages = {27--58},
 year = {2023},
 language = {English},
 doi = {10.3934/jgm.2023002},
 zbMATH = {7701145},
 Zbl = {1517.53073}
}

@book{Papantonis,
 author = {Papantonis, T.},
 title = {{{\(Z\)}}-graded supergeometry: differential graded modules, higher algebroid representations, and linear structures},
 year = {2021},
 publisher = {G{\"o}ttingen: Univ. G{\"o}ttingen (Diss.)},
 language = {English},
 doi = {10.53846/goediss-8678},
 zbMATH = {8090284},
 Zbl = {1571.58002}
}

@article{MeinrenkenPike,
 author = {Meinrenken, E. and Pike, J.},
 title = {The {Weil} algebra of a double {Lie} algebroid},
 fjournal = {IMRN. International Mathematics Research Notices},
 journal = {Int. Math. Res. Not.},
 issn = {1073-7928},
 volume = {2021},
 number = {11},
 pages = {8550--8622},
 year = {2021},
 language = {English},
 doi = {10.1093/imrn/rnz361},
 zbMATH = {7398553},
 Zbl = {1486.53098}
}

@article{LiBlandMeinrenkenManin,
 author = {Li-Bland, D. and Meinrenken, E.},
 title = {On the integration of {Manin} pairs},
 fjournal = {Differential Geometry and its Applications},
 journal = {Differ. Geom. Appl.},
 issn = {0926-2245},
 volume = {99},
 pages = {32},
 note = {Id/No 102246},
 year = {2025},
 language = {English},
 doi = {10.1016/j.difgeo.2025.102246},
 zbMATH = {8041169},
 Zbl = {1565.53067}
}

@article{CamiloMariusWeil,
 author = {Arias Abad, C. and Crainic, M.},
 title = {The {Weil} algebra and the {Van} {Est} isomorphism},
 fjournal = {Annales de l'Institut Fourier},
 journal = {Ann. Inst. Fourier},
 issn = {0373-0956},
 volume = {61},
 number = {3},
 pages = {927--970},
 year = {2011},
 language = {English},
 doi = {10.5802/aif.2633},
 url = {https://eudml.org/doc/219714},
 zbMATH = {6002988},
 Zbl = {1237.58021}
}

@article{BSS,
 author = {Bott, R. and Shulman, H. and Stasheff, J.ames},
 title = {On the de {Rham} theory of certain classifying spaces},
 fjournal = {Advances in Mathematics},
 journal = {Adv. Math.},
 issn = {0001-8708},
 volume = {20},
 pages = {43--56},
 year = {1976},
 language = {English},
 doi = {10.1016/0001-8708(76)90169-9},
 zbMATH = {3533786},
 Zbl = {0342.57016}
}

@misc{Mariarelative,
 author = {Salazar, M. A.},
 title = {On {Relative} {Cohomology} in {Lie} {Theory}},
 year = {2024},
 howpublished = {Preprint, {arXiv}:2407.00169 [math.{DG}]},
 url = {https://arxiv.org/abs/2407.00169},
 arXiv = {arXiv:2407.00169}
}

@article{MackenzieClassification,
 author = {Mackenzie, K. C. H.},
 title = {Classification of principal bundles and {Lie} groupoids with prescribed gauge group bundle},
 fjournal = {Journal of Pure and Applied Algebra},
 journal = {J. Pure Appl. Algebra},
 issn = {0022-4049},
 volume = {58},
 number = {2},
 pages = {181--208},
 year = {1989},
 language = {English},
 doi = {10.1016/0022-4049(89)90157-6},
 zbMATH = {4102066},
 Zbl = {0673.55015}
}

@article{CamilleWagemann,
 author = {Laurent-Gengoux, C. and Wagemann, F.},
 title = {Obstruction classes of crossed modules of {Lie} algebroids and {Lie} groupoids linked to existence of principal bundles},
 fjournal = {Annals of Global Analysis and Geometry},
 journal = {Ann. Global Anal. Geom.},
 issn = {0232-704X},
 volume = {34},
 number = {1},
 pages = {21--37},
 year = {2008},
 language = {English},
 doi = {10.1007/s10455-007-9098-0},
 zbMATH = {5306115},
 Zbl = {1146.55012}
}

@article{BursztynDrummond,
 author = {Bursztyn, H. and Drummond, T.},
 title = {Lie theory of multiplicative tensors},
 fjournal = {Mathematische Annalen},
 journal = {Math. Ann.},
 issn = {0025-5831},
 volume = {375},
 number = {3-4},
 pages = {1489--1554},
 year = {2019},
 language = {English},
 doi = {10.1007/s00208-019-01881-w},
 zbMATH = {7126538},
 Zbl = {1472.22002}
}

@misc{CabreraDelHoyo,
 author = {Cabrera, A. and del Hoyo, M. L.},
 title = {Geometric differentiation of simplicial manifolds},
 year = {2026},
 howpublished = {Preprint, {arXiv}:2602.09885 [math.{DG}]},
 url = {https://arxiv.org/abs/2602.09885},
 arXiv = {arXiv:2602.09885}
}

@article{Bouaziz,
 author = {Bouaziz, E.},
 title = {Deformations and representations of {Lie} algebroids},
 fjournal = {Communications in Algebra},
 journal = {Commun. Algebra},
 issn = {0092-7872},
 volume = {50},
 number = {2},
 pages = {880--888},
 year = {2022},
 language = {English},
 doi = {10.1080/00927872.2021.1975730},
 zbMATH = {7517845},
 Zbl = {1487.58012}
}

@incollection{CrainicFernandeschar,
 author = {Crainic, M. and Fernandes, R. L.},
 title = {Secondary characteristic classes of {Lie} algebroids},
 booktitle = {Quantum field theory and noncommutative geometry. Based on the workshop, Sendai, Japan, November 2002},
 isbn = {3-540-23900-6},
 pages = {157--176},
 year = {2005},
 publisher = {Berlin: Springer},
 language = {English},
 zbMATH = {2182877},
 Zbl = {1068.22004}
}

@article{Transversemeasures,
 author = {Evens, S and Lu, J-H and Weinstein, A.},
 title = {Transverse measures, the modular class and a cohomology pairing for {Lie} algebroids},
 fjournal = {The Quarterly Journal of Mathematics. Oxford Second Series},
 journal = {Q. J. Math., Oxf. II. Ser.},
 issn = {0033-5606},
 volume = {50},
 number = {200},
 pages = {417--436},
 year = {1999},
 language = {English},
 doi = {10.1093/qjmath/50.200.417},
 zbMATH = {1396426},
 Zbl = {0968.58014}
}

@article{PradinesVBLA,
 author = {Pradines, J.},
 title = {Repr{\'e}sentation des jets non holonomes par des morphismes vectoriels doubles soudes},
 fjournal = {Comptes Rendus Hebdomadaires des S{\'e}ances de l'Acad{\'e}mie des Sciences, S{\'e}rie A},
 journal = {C. R. Acad. Sci., Paris, S{\'e}r. A},
 issn = {0366-6034},
 volume = {278},
 pages = {1523--1526},
 year = {1974},
 language = {French},
 zbMATH = {3447808},
 Zbl = {0285.58002}
}

@article{PradinesVBLG,
 author = {Pradines, J.},
 title = {Remarque sur le groupo{\"{\i}}de cotangent de {Weinstein}-{Dazord}. ({A} remark about {Weinstein}-{Dazord} cotangent groupoid)},
 fjournal = {Comptes Rendus de l'Acad{\'e}mie des Sciences. S{\'e}rie I},
 journal = {C. R. Acad. Sci., Paris, S{\'e}r. I},
 issn = {0764-4442},
 volume = {306},
 number = {13},
 pages = {557--560},
 year = {1988},
 language = {French},
 zbMATH = {4077516},
 Zbl = {0659.18009}
}

@book{Mackenziebook,
 author = {Mackenzie, K. C. H.},
 title = {The general theory of {Lie} groupoids and {Lie} algebroids},
 fseries = {London Mathematical Society Lecture Note Series},
 series = {Lond. Math. Soc. Lect. Note Ser.},
 issn = {0076-0552},
 volume = {213},
 isbn = {0-521-49928-3},
 year = {2005},
 publisher = {Cambridge: Cambridge University Press},
 language = {English},
 zbMATH = {822678},
 Zbl = {1078.58011}
}

@misc{Wolbert,
 author = {Wolbert, S.},
 title = {Weak representations, representations up to homotopy, and {VB}-groupoids},
 year = {2017},
 howpublished = {Preprint, {arXiv}:1704.05019 [math.{DG}]},
 url = {https://arxiv.org/abs/1704.05019},
 arXiv = {arXiv:1704.05019}
}

@misc{TangVillatoro,
 author = {Tang, X. and Villatoro, J.},
 title = {Simplicial sheaves of modules and {Morita} invariance of groupoid cohomology},
 year = {2025},
 howpublished = {Preprint, {arXiv}:2509.07285 [math.{DG}]},
 url = {https://arxiv.org/abs/2509.07285},
 arXiv = {arXiv:2509.07285}
}

@article{Matchedpairs,
 author = {Jotz-Lean, M.},
 title = {Lie 2-algebroids and matched pairs of 2-representations: a geometric approach},
 fjournal = {Pacific Journal of Mathematics},
 journal = {Pac. J. Math.},
 issn = {1945-5844},
 volume = {301},
 number = {1},
 pages = {143--188},
 year = {2019},
 language = {English},
 doi = {10.2140/pjm.2019.301.143},
 zbMATH = {7178871},
 Zbl = {1443.53047}
}

@misc{Zanbrocov,
 author = {Grad, {\v{Z}}.},
 title = {Covariant derivatives in the representation-valued {Bott}-{Shulman}-{Stasheff} and {Weil} complex},
 year = {2025},
 howpublished = {Preprint, {arXiv}:2503.08873 [math.{DG}]},
 url = {https://arxiv.org/abs/2503.08873},
 arXiv = {arXiv:2503.08873}
}

@misc{AlvarezCueca,
      title={Homological vector fields over differentiable stacks}, 
      author={{\'{A}}lvarez, D. and Cueca, M.},
      year={2024},
      eprint={2403.14871},
      archivePrefix={arXiv},
      primaryClass={math.DG},
      url={https://arxiv.org/abs/2403.14871}, 
}

@article{BruceGrabowskaGrabowski,
 author = {Bruce, A.J. and Grabowska, K. and Grabowski, J.},
 title = {Remarks on contact and {Jacobi} geometry},
 fjournal = {SIGMA. Symmetry, Integrability and Geometry: Methods and Applications},
 journal = {SIGMA, Symmetry Integrability Geom. Methods Appl.},
 issn = {1815-0659},
 volume = {13},
 pages = {paper 059, 22},
 year = {2017},
 language = {English},
 doi = {10.3842/SIGMA.2017.059},
 zbMATH = {6756045},
 Zbl = {1369.53057}
}

@article{GrabowskiRotkiewicz,
 author = {Grabowski, J. and Rotkiewicz, M.},
 title = {Higher vector bundles and multi-graded symplectic manifolds},
 fjournal = {Journal of Geometry and Physics},
 journal = {J. Geom. Phys.},
 issn = {0393-0440},
 volume = {59},
 number = {9},
 pages = {1285--1305},
 year = {2009},
 language = {English},
 doi = {10.1016/j.geomphys.2009.06.009},
 zbMATH = {5598302},
 Zbl = {1171.58300}
}

@article{GrabowskaGrabowskiRavanpak,
 author = {Grabowska, K. and Grabowski, J. and Ravanpak, Z.},
 title = {{VB}-structures and generalizations},
 fjournal = {Annals of Global Analysis and Geometry},
 journal = {Ann. Global Anal. Geom.},
 issn = {0232-704X},
 volume = {62},
 number = {1},
 pages = {235--284},
 year = {2022},
 language = {English},
 doi = {10.1007/s10455-022-09847-z},
 zbMATH = {7539230},
 Zbl = {1508.53032}
}

@article{ChenLiuSheng,
 author = {Chen, Z. and Liu, Z.J. and Sheng, Y.H.},
 title = {On double vector bundles},
 fjournal = {Acta Mathematica Sinica. English Series},
 journal = {Acta Math. Sin., Engl. Ser.},
 issn = {1439-8516},
 volume = {30},
 number = {10},
 pages = {1655--1673},
 year = {2014},
 language = {English},
 doi = {10.1007/s10114-014-2412-4},
 zbMATH = {6378963},
 Zbl = {1318.58001}
}

@misc{RonchiZhu,
 author = {Ronchi, S. and Zhu, C.},
 title = {Duals of {Higher} {Vector} {Spaces}},
 year = {2025},
 howpublished = {Preprint, {arXiv}:2407.03306 [math.{DG}] (2025)},
 url = {https://arxiv.org/abs/2407.03306},
 arXiv = {arXiv:2407.03306}
}

@article{Ruioan,
 author = {Fernandes, R.L. and M{\u{a}}rcu{\c{t}}, I.},
 title = {Multiplicative {Ehresmann} connections},
 fjournal = {Advances in Mathematics},
 journal = {Adv. Math.},
 issn = {0001-8708},
 volume = {427},
 pages = {84},
 note = {Id/No 109124},
 year = {2023},
 language = {English},
 doi = {10.1016/j.aim.2023.109124},
 zbMATH = {7695596},
 Zbl = {1516.58004}
}

@misc{LaPastinathesis,
      title={Deformations of vector bundles in the categories of Lie algebroids and groupoids}, 
      author={La Pastina, P. P.},
      year={2020},
      eprint={2001.07559},
      archivePrefix={arXiv},
      primaryClass={math.DG},
      url={https://arxiv.org/abs/2001.07559}, 
}

@article{BursztynCabreraOrtiz,
 author = {Bursztyn, H. and Cabrera, A. and Ortiz, C.},
 title = {Linear and multiplicative 2-forms},
 fjournal = {Letters in Mathematical Physics},
 journal = {Lett. Math. Phys.},
 issn = {0377-9017},
 volume = {90},
 number = {1-3},
 pages = {59--83},
 year = {2009},
 language = {English},
 doi = {10.1007/s11005-009-0349-9},
 zbMATH = {5659882},
 Zbl = {1206.58005}
}

@article{BursztynCabrera,
 author = {Bursztyn, H. and Cabrera, A.},
 title = {Multiplicative forms at the infinitesimal level},
 fjournal = {Mathematische Annalen},
 journal = {Math. Ann.},
 issn = {0025-5831},
 volume = {353},
 number = {3},
 pages = {663--705},
 year = {2012},
 language = {English},
 doi = {10.1007/s00208-011-0697-5},
 zbMATH = {6050579},
 Zbl = {1247.58014}
}

@article{EilenbergMaclane,
    AUTHOR = {Eilenberg, S. and MacLane, S.},
     TITLE = {Cohomology theory in abstract groups. {I}},
   JOURNAL = {Ann. of Math. (2)},
  FJOURNAL = {Annals of Mathematics. Second Series},
    VOLUME = {48},
      YEAR = {1947},
     PAGES = {51--78},
      ISSN = {0003-486X},
   MRCLASS = {20.0X},
  MRNUMBER = {19092},
MRREVIEWER = {P.\ A.\ Smith},
       DOI = {10.2307/1969215},
       URL = {https://doi.org/10.2307/1969215},
}

@article{Dold,
 author = {Dold, A.},
 title = {Homology of symmetric products and other functors of complexes},
 fjournal = {Annals of Mathematics. Second Series},
 journal = {Ann. Math. (2)},
 issn = {0003-486X},
 volume = {68},
 pages = {54--80},
 year = {1958},
 language = {English},
 doi = {10.2307/1970043},
 zbMATH = {3135507},
 Zbl = {0082.37701}
}

@article{Kan,
 author = {Kan, D. M.},
 title = {Functors involving c. s. s. complexes},
 fjournal = {Transactions of the American Mathematical Society},
 journal = {Trans. Am. Math. Soc.},
 issn = {0002-9947},
 volume = {87},
 pages = {330--346},
 year = {1958},
 language = {English},
 doi = {10.2307/1993103},
 zbMATH = {3148099},
 Zbl = {0090.39001}
}

\end{document}